\font \sevenrm=cmr7
\font \fiverm=cmr5

\documentclass[11pt]{smfart}
\usepackage{txfonts}
\usepackage{amsfonts}
\usepackage{amssymb}
\usepackage{amsmath}
\usepackage{wasysym}
\usepackage{color}
\usepackage{graphicx}
\usepackage{xspace}
\usepackage{axodraw}

 \newcommand{\nc}{\newcommand}
 \nc{\butcher}{{\scriptstyle\circleright}}

\newenvironment{disarray}%
 {\everymath{\textstyle\everymath{}}\array}%
 {\endarray}

\newtheorem{thm}{Th\'eor\`eme}

\newtheorem{lem}[thm]{Lemme}
\newtheorem{prop}[thm]{Proposition}
\newtheorem{defn}{D\'efinition}
\newtheorem{rmk}[thm]{Remarque}

\nc{\comment}[1]{[[{\tt #1}]] }
\nc{\Cal}[1]{{\mathcal {#1}}}
\nc{\mop}[1]{\mathop{\hbox {\rm #1} }\nolimits}
\nc{\gmop}[1]{\mathop{\hbox {\bf #1} }\nolimits}

\nc{\smop}[1]{\mathop{\hbox {\sevenrm #1} }\nolimits}
\nc{\ssmop}[1]{\mathop{\hbox {\fiverm #1} }\nolimits}
\nc{\mopl}[1]{\mathop{\hbox {\rm #1} }\limits}
\nc{\smopl}[1]{\mathop{\hbox {\sevenrm #1} }\limits}
\nc{\ssmopl}[1]{\mathop{\hbox {\fiverm #1} }\limits}
\nc{\frakg}{{\frak g}}
\nc{\g}[1]{{\frak {#1}}}
\def \restr#1{\mathstrut_{\textstyle |}\raise-6pt\hbox{$\scriptstyle #1$}}
\def \srestr#1{\mathstrut_{\scriptstyle |}\hbox to
  -1.5pt{}\raise-4pt\hbox{$\scriptscriptstyle #1$}}
\nc{\wt}{\widetilde} \nc{\wh}{\widehat}

\nc{\redtext}[1]{\textcolor{red}{#1}}
\nc{\bluetext}[1]{\textcolor{blue}{#1}}

\nc\fleche[1]{\mathop{\hbox to #1 mm{\rightarrowfill}}\limits}
\nc{\ignore}[1]{}
\def\semi{\mathrel{\times}\kern -.85pt\joinrel\mathrel{\raise
    1.4pt\hbox{${\scriptscriptstyle |}$}}}
\nc\R{{\mathbb R}}
\nc\N{{\mathbb N}}
\nc\inver{^{-1}}
\nc\point{\hbox{\bf .}}
\nc\un{\hbox{\bf 1}}
\nc{\adi}{\atop\displaystyle}

\def\racine{{\scalebox{0.3}{ 
\begin{picture}(12,12)(38,-38)
\SetWidth{0.5} \SetColor{Black} \Vertex(45,-37){6}
\end{picture}}}\,}
  
 \def\arbrea{\,{\scalebox{0.15}{ 
  \begin{picture}(8,65) (370,-248)
    \SetWidth{2}
    \SetColor{Black}
    \Line(374,-244)(374,-200)
    \Vertex(374,-197){9}
    \Vertex(374,-244){12}
  \end{picture}
}}\,}

 \def\arbreba{\,{\scalebox{0.15}{ 
\begin{picture}(8,106) (370,-197)
    \SetWidth{2}
    \SetColor{Black}
    \Line(374,-193)(374,-146)
    \Vertex(374,-146){9}
    \Vertex(374,-194){12}
    \Line(374,-146)(374,-98)
    \Vertex(374,-98){9}
  \end{picture}
}}\,}

 \def\arbrebb{\,{\scalebox{0.15}{ 
  \begin{picture}(48,48) (349,-255)
    \SetWidth{2}
    \SetColor{Black}
    \Vertex(375,-252){12}
    \Line(375,-252)(395,-211)
    \Line(375,-252)(354,-211)
    \Vertex(353,-211){9}
    \Vertex(395,-211){9}
  \end{picture}
}}\,}

\def\arbreca{\,{\scalebox{0.15}{
\begin{picture}(8,156) (370,-147)
    \SetWidth{2}
    \SetColor{Black}
    \Line(374,-143)(374,-99)
    \Vertex(374,-96){9}
    \Vertex(374,-144){12}
    \Line(374,-96)(374,-45)
    \Vertex(374,-45){9}
    \Line(374,-45)(374,2)
    \Vertex(374,2){9}
  \end{picture}
}}\,}

\def\arbrecb{\,{\scalebox{0.15}{
\begin{picture}(48,94) (349,-255)
\SetWidth{2}
\SetColor{Black}
\Line(374,-205)(395,-165)
\Line(374,-205)(354,-165)
\Vertex(353,-165){9}
\Vertex(395,-165){9}
\Vertex(374,-205){9}
\Line(374,-252)(374,-205)
\Vertex(374,-252){12}
\end{picture}}}\,}

\def\arbrecc{\,{\scalebox{0.15}{
 \begin{picture}(48,98) (349,-205)
    \SetWidth{2}
    \SetColor{Black}
    \Vertex(375,-202){12}
    \Line(375,-202)(395,-163)
    \Line(375,-202)(354,-163)
    \Vertex(353,-163){9}
    \Vertex(395,-163){9}
    \Line(353,-163)(353,-111)
    \Vertex(353,-111){9}
  \end{picture}
}}\,}

\def\arbrecd{\,{\scalebox{0.15}{
\begin{picture}(48,52) (349,-251)
    \SetWidth{2}
    \SetColor{Black}
    \Vertex(376,-248){12}
    \Line(376,-248)(399,-205)
    \Line(376,-248)(353,-205)
    \Vertex(353,-205){9}
    \Vertex(399,-205){9}
    \Line(376,-248)(376,-205)
    \Vertex(376,-203){9}
  \end{picture}
 }}\,}

\def\arbreda{\,{\scalebox{0.15}{
\begin{picture}(8,204) (370,-99)
    \SetWidth{2}
    \SetColor{Black}
    \Line(374,-96)(374,-48)
    \Vertex(374,-48){9}
    \Vertex(374,-96){12}
    \Line(374,-48)(374,3)
    \Vertex(374,3){9}
    \Line(374,3)(374,53)
    \Vertex(374,53){9}
    \Line(374,53)(374,101)
    \Vertex(374,101){9}
  \end{picture}
}}\,}

\def\arbredb{\,{\scalebox{0.15}{
\begin{picture}(48,135) (349,-255)
    \SetWidth{2}
    \SetColor{Black}
    \Line(374,-164)(395,-126)
    \Line(374,-164)(354,-126)
    \Vertex(353,-126){9}
    \Vertex(395,-126){9}
    \Vertex(374,-164){9}
    \Line(374,-205)(374,-168)
    \Vertex(374,-207){9}
    \Line(374,-252)(374,-207)
    \Vertex(374,-252){12}
  \end{picture}
}}\,}

\def\arbredc{\,{\scalebox{0.15}{
 \begin{picture}(48,150) (349,-205)
    \SetWidth{2}
    \SetColor{Black}
    \Line(374,-148)(395,-111)
    \Line(374,-149)(354,-111)
    \Vertex(353,-111){9}
    \Vertex(395,-111){9}
    \Line(353,-111)(353,-59)
    \Vertex(353,-59){9}
    \Line(374,-202)(374,-149)
    \Vertex(374,-149){9}
    \Vertex(374,-202){12}
  \end{picture}
}}\,}

\def\arbredd{\,{\scalebox{0.15}{
 \begin{picture}(48,99) (349,-251)
    \SetWidth{2}
    \SetColor{Black}
    \Line(376,-199)(395,-164)
    \Line(376,-200)(353,-163)
    \Vertex(353,-162){9}
    \Vertex(399,-162){9}
    \Vertex(376,-160){9}
    \Vertex(376,-248){12}
    \Line(376,-245)(375,-200)
    \Line(376,-200)(376,-160)
    \Vertex(376,-200){9}
  \end{picture}
}}\,}

\def\arbrede{\,{\scalebox{0.15}{
 \begin{picture}(48,153) (349,-150)
    \SetWidth{2}
    \SetColor{Black}
    \Vertex(375,-147){12}
    \Line(375,-147)(395,-108)
    \Line(375,-146)(354,-108)
    \Vertex(353,-108){9}
    \Vertex(395,-108){9}
    \Line(353,-108)(353,-56)
    \Vertex(353,-56){9}
    \Line(353,-56)(353,-1)
    \Vertex(353,-1){9}
  \end{picture}
}}\,}

\def\arbredf{\,{\scalebox{0.15}{
\begin{picture}(48,98) (349,-205)
    \SetWidth{2}
    \SetColor{Black}
    \Vertex(375,-202){12}
    \Line(375,-202)(395,-163)
    \Line(375,-202)(354,-163)
    \Vertex(353,-163){9}
    \Vertex(395,-163){9}
    \Line(353,-163)(353,-111)
    \Vertex(353,-111){9}
    \Line(395,-163)(395,-111)
    \Vertex(395,-111){9}
  \end{picture}
}}\,}

\def\arbredz{\,{\scalebox{0.15}{
  \begin{picture}(68,88) (329,-215)
    \SetWidth{2}
    \SetColor{Black}
    \Vertex(375,-212){12}
    \Line(375,-212)(395,-173)
    \Line(375,-212)(353,-173)
    \Vertex(353,-173){9}
    \Vertex(395,-173){9}
    \Line(353,-173)(333,-131)
    \Line(353,-173)(374,-131)
    \Vertex(333,-131){9}
    \Vertex(374,-131){9}
  \end{picture}
}}\,}

\def\arbredg{\,{\scalebox{0.15}{
\begin{picture}(48,98) (359,-205)
    \SetWidth{2}
    \SetColor{Black}
    \Vertex(376,-202){12}
    \Line(376,-202)(399,-160)
    \Line(376,-202)(353,-160)
    \Vertex(353,-160){9}
    \Vertex(399,-160){9}
    \Line(376,-202)(376,-158)
    \Vertex(376,-158){9}
    \Vertex(353,-111){9}
    \Line(353,-160)(353,-111)
  \end{picture}
}}\,}

\def\arbredh{\,{\scalebox{0.15}{
 \begin{picture}(90,46) (330,-257)
    \SetWidth{2}
    \SetColor{Black}
    \Vertex(375,-254){12}
    \Line(375,-254)(395,-215)
    \Vertex(395,-215){9}
    \Line(375,-254)(334,-224)
    \Vertex(334,-224){9}
    \Line(375,-254)(355,-215)
    \Vertex(355,-215){9}
    \Line(375,-254)(418,-225)
    \Vertex(418,-225){9}
  \end{picture}
}}\,}

\def\arbreeb{\,{\scalebox{0.15}{
 \begin{picture}(48,153) (349,-150)
    \SetWidth{2}
    \SetColor{Black}
    \Vertex(375,-147){12}
    \Line(375,-147)(395,-108)
    \Line(375,-147)(354,-108)
    \Vertex(353,-108){9}
    \Vertex(395,-108){9}
    \Vertex(395,-56){9}
    \Line(353,-108)(353,-56)
    \Line(395,-108)(395,-56)
    \Vertex(353,-56){9}
    \Line(353,-56)(353,-1)
    \Vertex(353,-1){9}
  \end{picture}
}}\,}

\nc{\matulab}{\racine}
\nc{\matulac}{\arbrea}
\nc{\matulae}{\arbreba}
\nc{\matulag}{\arbrebb}
\nc{\matulaaa}{\arbreca}
\nc{\matulaag}{\arbrecb}
\def\matulaagmince{\,{\scalebox{0.15}{
\begin{picture}(48,94) (349,-255)
\SetWidth{2}
\SetColor{Black}
\Line(374,-205)(389,-165)
\Line(374,-205)(360,-165)
\Vertex(359,-165){9}
\Vertex(389,-165){9}
\Vertex(374,-205){9}
\Line(374,-252)(374,-205)
\Vertex(374,-252){12}
\end{picture}}}\,}
\nc{\matulaac}{\arbrecc}
\nc{\matulaai}{\arbrecd}
\nc{\matulaca}{\arbreda}
\nc{\matulaei}{\arbredb}
\nc{\matulada}{\arbredc}
\nc{\matulafg}{\arbredd}
\nc{\matulabi}{\arbrede}
\nc{\matulabc}{\arbredf}
\nc{\matuladc}{\arbredz}
\nc{\matulacg}{\arbredg}
\nc{\matulaec}{\arbredh}
\nc{\matuladg}{\arbreeb}

\def\matulafa{\,{\scalebox{0.15}{
\begin{picture}(48,98) (385,-205)
    \SetWidth{2}
    \SetColor{Black}
    \Vertex(375,-202){12}
    \Line(375,-202)(395,-163)
    \Line(375,-202)(353,-163)
    \Vertex(353,-163){9}
    \Vertex(395,-163){9}
    \Line(353,-163)(353,-111)
    \Vertex(353,-111){9}
    \Line(395,-163)(395,-111)
    \Vertex(395,-111){9}
    \Line(375,-202)(430,-163)
    \Vertex(430,-163){9}
\end{picture}
}}\,}

\def\matulaga{\,{\scalebox{0.15}{
\begin{picture}(48,98) (349,-205)
    \SetWidth{2}
    \SetColor{Black}
    \Vertex(375,-202){12}
    \Line(375,-202)(399,-161)
    \Line(375,-202)(353,-161)
    \Vertex(353,-161){9}
    \Vertex(399,-161){9}
    \Line(375,-202)(375,-157)
    \Vertex(375,-157){9}
    \Vertex(353,-115){9}
    \Line(353,-155)(353,-115)
    \Vertex(353,-70){9}
    \Line(353,-115)(353,-70)
  \end{picture}
}}\,}

\def\matulagc{\,{\scalebox{0.15}{
\begin{picture}(48,98) (379,-205)
    \SetWidth{2}
    \SetColor{Black}
    \Vertex(375,-202){12}
    \Line(375,-202)(395,-165)
    \Line(375,-202)(348,-164)
    \Vertex(348,-163){9}
    \Vertex(395,-163){9}
    \Line(348,-160)(348,-113)
    \Vertex(348,-111){9}
    \Line(395,-159)(380,-113)
    \Vertex(380,-113){9}
    \Line(395,-163)(415,-113)
    \Vertex(415,-113){9}
\end{picture}
}}\,}

\def\matulagi{\,{\scalebox{0.15}{
 \begin{picture}(48,153) (349,-150)
    \SetWidth{2}
    \SetColor{Black}
    \Vertex(375,-147){12}
    \Line(375,-147)(395,-108)
    \Line(375,-147)(354,-108)
    \Vertex(353,-108){9}
    \Vertex(395,-108){9}
    \Line(353,-108)(353,-56)
    \Vertex(353,-56){9}
    \Line(353,-56)(353,-10)
    \Vertex(353,-10){9}
    \Line(353,-10)(353,35)
    \Vertex(353,35){9}
  \end{picture}
}}\,}

\def\matulahc{\,{\scalebox{0.15}{
 \begin{picture}(48,150) (349,-205)
    \SetWidth{2}
    \SetColor{Black}
    \Line(374,-149)(395,-111)
    \Line(374,-149)(354,-111)
    \Vertex(353,-111){9}
    \Vertex(395,-111){9}
    \Line(353,-111)(353,-59)
    \Vertex(353,-59){9}
    \Line(374,-202)(374,-149)
    \Vertex(374,-149){9}
    \Vertex(374,-202){12}
    \Line(395,-111)(395,-59)
    \Vertex(395,-59){9}
  \end{picture}
}}\,}

\def\matulahi{\,{\scalebox{0.15}{
\begin{picture}(48,98) (379,-205)
    \SetWidth{2}
    \SetColor{Black}
    \Vertex(375,-202){12}
    \Line(375,-202)(399,-161)
    \Line(375,-202)(350,-163)
    \Vertex(350,-163){9}
    \Vertex(399,-161){9}
    \Line(375,-202)(375,-160)
    \Vertex(375,-161){9}
    \Vertex(425,-163){9}
    \Vertex(350,-110){9}
    \Line(375,-202)(425,-165)
    \Line(350,-161)(350,-110)
  \end{picture}
}}\,}

\def\matulaig{\,{\scalebox{0.15}{
 \begin{picture}(48,153) (349,-150)
    \SetWidth{2}
    \SetColor{Black}
    \Vertex(375,-147){12}
    \Line(375,-147)(395,-108)
    \Line(375,-147)(354,-108)
    \Vertex(353,-108){9}
    \Vertex(395,-108){9}
    \Vertex(395,-58){9}
    \Line(353,-105)(353,-58)
    \Line(395,-105)(395,-58)
    \Vertex(353,-58){9}
    \Line(353,-52)(353,-5)
    \Vertex(353,-5){9}
    \Line(395,-52)(395,-5)
    \Vertex(395,-5){9}
  \end{picture}
}}\,}

\def\matulaiglarge{\,{\scalebox{0.15}{
 \begin{picture}(48,153) (349,-150)
    \SetWidth{2}
    \SetColor{Black}
    \Vertex(375,-147){12}
    \Line(375,-147)(415,-108)
    \Line(375,-147)(333,-108)
    \Vertex(333,-108){9}
    \Vertex(415,-108){9}
    \Vertex(415,-58){9}
    \Line(333,-105)(333,-58)
    \Line(415,-105)(415,-58)
    \Vertex(333,-58){9}
    \Line(333,-52)(333,-5)
    \Vertex(333,-5){9}
    \Line(415,-52)(415,-5)
    \Vertex(415,-5){9}
  \end{picture}
}}\,}

\def\matulaaoa{\,{\scalebox{0.15}{
 \begin{picture}(58,150) (359,-205)
    \SetWidth{2}
    \SetColor{Black}
    \Line(374,-149)(395,-111)
    \Line(374,-149)(353,-111)
    \Vertex(353,-111){9}
    \Vertex(395,-111){9}
    \Line(353,-111)(353,-59)
    \Vertex(353,-59){9}
    \Line(395,-202)(374,-149)
    \Vertex(374,-149){9}
    \Vertex(395,-202){12}
    \Vertex(416,-149){9}
    \Line(395,-202)(416,-149)
  \end{picture}
}}\,}

\def\matulaaoc{\,{\scalebox{0.15}{
\begin{picture}(48,98) (389,-205)
    \SetWidth{2}
    \SetColor{Black}
    \Vertex(375,-202){12}
    \Line(375,-202)(395,-163)
    \Line(375,-202)(353,-163)
    \Vertex(353,-163){9}
    \Vertex(395,-163){9}
    \Line(353,-160)(353,-113)
    \Vertex(353,-113){9}
    \Line(395,-163)(395,-113)
    \Vertex(395,-113){9}
    \Line(375,-202)(430,-163)
    \Vertex(430,-163){9}
    \Line(430,-163)(430,-113)
    \Vertex(430,-113){9}
\end{picture}
}}\,}

\def\matulaaog{\,{\scalebox{0.15}{
 \begin{picture}(48,99) (379,-251)
    \SetWidth{2}
    \SetColor{Black}
    \Line(376,-201)(395,-162)
    \Line(376,-201)(354,-162)
    \Vertex(353,-162){9}
    \Vertex(399,-162){9}
    \Vertex(400,-201){9}
    \Vertex(400,-248){12}
    \Line(400,-248)(376,-201)
    \Line(400,-248)(424,-201)
    \Line(400,-248)(400,-201)
    \Vertex(376,-201){9}
    \Vertex(424,-201){9}
  \end{picture}
}}\,}

\def\matulaaoi{\,{\scalebox{0.15}{
 \begin{picture}(48,150) (349,-205)
    \SetWidth{2}
    \SetColor{Black}
    \Line(374,-149)(395,-111)
    \Line(374,-149)(353,-111)
    \Vertex(353,-111){9}
    \Vertex(395,-111){9}
    \Line(353,-111)(353,-59)
    \Vertex(353,-59){9}
    \Line(353,-59)(353,-15)
    \Vertex(353,-15){9}
    \Line(374,-202)(374,-149)
    \Vertex(374,-149){9}
    \Vertex(374,-202){12}
  \end{picture}
}}\,}

\def\matulaaac{\,{\scalebox{0.15}{
\begin{picture}(58,98) (379,-205)
    \SetWidth{2}
    \SetColor{Black}
    \Vertex(375,-202){12}
    \Line(375,-202)(395,-163)
    \Line(375,-202)(354,-163)
    \Vertex(353,-163){9}
    \Vertex(395,-163){9}
    \Line(353,-163)(353,-115)
    \Vertex(353,-115){9}
    \Line(395,-163)(395,-115)
    \Vertex(395,-115){9}
    \Line(375,-202)(430,-163)
    \Vertex(430,-163){9}
    \Line(353,-115)(353,-70)
    \Vertex(353,-70){9}
    \end{picture}
}}\,}

\def\matulaabg{\,{\scalebox{0.15}{
\begin{picture}(8,204) (370,-99)
    \SetWidth{2}
    \SetColor{Black}
    \Line(374,-95)(374,-51)
    \Vertex(374,-48){9}
    \Vertex(375,-96){12}
    \Line(374,-44)(374,0)
    \Vertex(374,3){9}
    \Line(374,6)(374,50)
    \Vertex(374,53){9}
    \Line(374,53)(374,95)
    \Vertex(374,95){9}
     \Line(374,95)(374,130)
    \Vertex(374,130){9}
  \end{picture}
}}\,}

\def\matulaaca{\,{\scalebox{0.15}{
 \begin{picture}(90,46) (330,-257)
    \SetWidth{2}
    \SetColor{Black}
    \Vertex(375,-254){12}
    \Line(375,-254)(395,-215)
    \Vertex(395,-215){9}
    \Line(375,-254)(334,-224)
    \Vertex(334,-224){9}
    \Line(375,-254)(355,-215)
    \Vertex(355,-215){9}
    \Line(375,-254)(418,-224)
    \Vertex(418,-224){9}
    \Line(375,-254)(375,-212)
    \Vertex(375,-212){9}
  \end{picture}
}}\,}

\def\matulaacg{\,{\scalebox{0.15}{
 \begin{picture}(48,153) (349,-150)
    \SetWidth{2}
    \SetColor{Black}
    \Vertex(375,-147){12}
    \Line(375,-147)(395,-108)
    \Line(375,-147)(354,-108)
    \Vertex(353,-108){9}
    \Vertex(395,-108){9}
    \Vertex(395,-56){9}
    \Line(353,-108)(353,-56)
    \Line(395,-108)(395,-56)
    \Vertex(353,-56){9}
    \Line(353,-56)(353,-5)
    \Vertex(353,-5){9}
    \Line(353,-5)(353,40)
    \Vertex(353,40){9}
  \end{picture}
}}\,}

\def\matulaaci{\,{\scalebox{0.15}{
 \begin{picture}(48,153) (349,-150)
    \SetWidth{2}
    \SetColor{Black}
    \Vertex(375,-147){12}
    \Line(375,-147)(395,-108)
    \Line(375,-147)(354,-108)
    \Vertex(353,-108){9}
    \Vertex(395,-108){9}
    \Line(353,-108)(353,-56)
    \Vertex(353,-56){9}
    \Line(353,-56)(333,-8)
    \Vertex(333,-8){9}
    \Line(353,-56)(373,-8)
    \Vertex(373,-8){9}
  \end{picture}
}}\,}

\def\matulaadi{\,{\scalebox{0.15}{
\begin{picture}(48,98) (349,-205)
    \SetWidth{2}
    \SetColor{Black}
    \Vertex(375,-202){12}
    \Line(375,-202)(395,-163)
    \Line(375,-202)(348,-163)
    \Vertex(348,-163){9}
    \Vertex(395,-163){9}
    \Line(348,-163)(348,-111)
    \Vertex(348,-111){9}
    \Line(348,-111)(348,-65)
    \Vertex(348,-65){9}
    \Line(395,-163)(380,-111)
    \Vertex(380,-111){9}
    \Line(395,-163)(415,-111)
    \Vertex(415,-111){9}
\end{picture}
}}\,}

\def\matulaaea{\,{\scalebox{0.15}{
\begin{picture}(78,98) (389,-205)
    \SetWidth{2}
    \SetColor{Black}
    \Vertex(375,-202){12}
    \Line(375,-202)(390,-163)
    \Line(375,-202)(349,-163)
    \Vertex(349,-163){9}
    \Vertex(390,-163){9}
    \Line(349,-163)(349,-111)
    \Vertex(349,-111){9}
    \Line(390,-163)(390,-111)
    \Vertex(390,-111){9}
    \Line(375,-202)(425,-163)
    \Vertex(425,-163){9}
    \Line(375,-202)(460,-163)
    \Vertex(460,-163){9}
\end{picture}
}}\,}

\def\matulaaeg{\,{\scalebox{0.15}{
 \begin{picture}(48,99) (349,-251)
    \SetWidth{2}
    \SetColor{Black}
    \Line(376,-199)(399,-160)
    \Line(373,-200)(353,-160)
    \Vertex(353,-160){9}
    \Vertex(399,-160){9}
    \Vertex(375,-156){9}
    \Vertex(375,-248){12}
    \Line(375,-248)(375,-201)
    \Line(375,-201)(375,-156)
    \Line(353,-160)(353,-115)
    \Vertex(353,-115){9}
    \Vertex(375,-201){9}
  \end{picture}
}}\,}

\def\matulaafc{\,{\scalebox{0.15}{
  \begin{picture}(68,88) (329,-215)
    \SetWidth{2}
    \SetColor{Black}
    \Vertex(375,-212){12}
    \Line(375,-212)(395,-175)
    \Line(375,-212)(354,-174)
    \Vertex(353,-171){9}
    \Vertex(395,-171){9}
    \Line(353,-171)(330,-131)
    \Line(353,-171)(376,-131)
    \Line(353,-171)(353,-129)
    \Vertex(353,-129){9}
    \Vertex(330,-131){9}
    \Vertex(376,-131){9}
  \end{picture}
}}\,}

\def\matulaafg{\,{\scalebox{0.15}{
 \begin{picture}(58,150) (379,-205)
    \SetWidth{2}
    \SetColor{Black}
    \Line(376,-148)(395,-113)
    \Line(373,-149)(354,-112)
    \Vertex(353,-111){9}
    \Vertex(395,-111){9}
    \Line(353,-108)(353,-65)
    \Vertex(353,-65){9}
    \Line(400,-202)(374,-153)
    \Vertex(374,-149){9}
    \Vertex(400,-202){12}
    \Vertex(426,-149){9}
    \Line(400,-200)(426,-153)
    \Vertex(426,-111){9}
    \Line(426,-149)(426,-111)
  \end{picture}
}}\,}

\def\matulaagc{\,{\scalebox{0.15}{
\begin{picture}(48,98) (379,-205)
    \SetWidth{2}
    \SetColor{Black}
    \Vertex(375,-202){12}
    \Line(375,-202)(399,-161)
    \Line(375,-202)(354,-164)
    \Vertex(350,-163){9}
    \Vertex(399,-161){9}
    \Line(375,-201)(375,-160)
    \Vertex(376,-161){9}
    \Vertex(425,-163){9}
    \Vertex(350,-110){9}
    \Line(375,-202)(425,-165)
    \Line(350,-161)(350,-110)
    \Line(350,-110)(350,-65)
    \Vertex(350,-65){9}
  \end{picture}
}}\,}

\def\matulaagi{\,{\scalebox{0.15}{
\begin{picture}(48,135) (349,-255)
    \SetWidth{2}
    \SetColor{Black}
    \Line(374,-164)(395,-128)
    \Line(374,-164)(354,-127)
    \Vertex(353,-124){9}
    \Vertex(395,-124){9}
    \Vertex(374,-164){9}
    \Line(374,-207)(374,-164)
    \Vertex(374,-207){9}
    \Line(374,-252)(374,-207)
    \Vertex(374,-252){12}
    \Vertex(354,-80){9}
    \Line(353,-124)(354,-80)
  \end{picture}
}}\,}

\def\matulaaha{\,{\scalebox{0.15}{
\begin{picture}(58,98) (379,-205)
    \SetWidth{2}
    \SetColor{Black}
    \Vertex(375,-202){12}
    \Line(375,-202)(395,-163)
    \Line(375,-202)(348,-163)
    \Vertex(348,-163){9}
    \Vertex(395,-163){9}
    \Line(348,-163)(348,-113)
    \Vertex(348,-113){9}
    \Line(395,-163)(380,-113)
    \Vertex(380,-113){9}
    \Line(395,-163)(415,-113)
    \Vertex(415,-113){9}
    \Line(375,-202)(430,-165)
    \Vertex(430,-165){9}
\end{picture}
}}\,}

\def\matulaaia{\,{\scalebox{0.15}{
 \begin{picture}(48,150) (349,-205)
    \SetWidth{2}
    \SetColor{Black}
    \Line(374,-149)(395,-111)
    \Line(374,-149)(353,-111)
    \Vertex(353,-111){9}
    \Vertex(395,-111){9}
    \Line(353,-109)(334,-65)
    \Line(353,-109)(372,-65)
    \Vertex(334,-65){9}
    \Vertex(372,-65){9}
    \Line(374,-202)(374,-149)
    \Vertex(374,-149){9}
    \Vertex(374,-202){12}
  \end{picture}
}}\,}

\def\matulaaic{\,{\scalebox{0.15}{
\begin{picture}(48,98) (349,-205)
    \SetWidth{2}
    \SetColor{Black}
    \Vertex(375,-202){12}
    \Line(375,-202)(399,-163)
    \Line(375,-202)(353,-163)
    \Vertex(353,-163){9}
    \Vertex(399,-163){9}
    \Line(375,-201)(375,-160)
    \Vertex(376,-157){9}
    \Vertex(353,-115){9}
    \Line(353,-155)(353,-115)
    \Vertex(353,-70){9}
    \Line(353,-115)(353,-70)
    \Vertex(353,-30){9}
    \Line(353,-70)(353,-30)
  \end{picture}
}}\,}

\def\matulaaig{\,{\scalebox{0.15}{
\begin{picture}(48,98) (379,-205)
    \SetWidth{2}
    \SetColor{Black}
    \Vertex(375,-202){12}
    \Line(375,-202)(395,-163)
    \Line(375,-202)(353,-163)
    \Vertex(353,-163){9}
    \Vertex(395,-163){9}
    \Line(353,-160)(353,-113)
    \Vertex(353,-113){9}
    \Line(395,-163)(395,-113)
    \Vertex(395,-113){9}
    \Line(375,-202)(430,-163)
    \Vertex(430,-163){9}
    \Line(430,-163)(430,-113)
    \Vertex(430,-113){9}
     \Line(353,-113)(353,-70)
    \Vertex(353,-70){9}
\end{picture}
}}\,}

\def\matulaaii{\,{\scalebox{0.15}{
 \begin{picture}(48,150) (369,-205)
    \SetWidth{2}
    \SetColor{Black}
    \Line(374,-149)(395,-111)
    \Line(374,-149)(354,-111)
    \Vertex(353,-111){9}
    \Vertex(395,-111){9}
    \Line(353,-111)(353,-59)
    \Vertex(353,-59){9}
    \Line(395,-202)(374,-149)
    \Line(395,-202)(415,-149)
    \Vertex(374,-149){9}
    \Vertex(415,-149){9}
    \Vertex(395,-202){12}
    \Line(395,-108)(395,-61)
    \Vertex(395,-59){9}
  \end{picture}
}}\,}

\def\matulabaa{\,{\scalebox{0.15}{
 \begin{picture}(48,153) (349,-190)
    \SetWidth{2}
    \SetColor{Black}
    \Vertex(375,-190){12}
    \Line(375,-190)(375,-147)
    \Vertex(375,-147){9}
    \Line(375,-145)(395,-108)
    \Line(375,-146)(354,-108)
    \Vertex(353,-108){9}
    \Vertex(395,-108){9}
    \Vertex(395,-56){9}
    \Line(353,-105)(353,-56)
    \Line(395,-105)(395,-56)
    \Vertex(353,-56){9}
    \Line(353,-56)(353,-1)
    \Vertex(353,-1){9}
  \end{picture}
}}\,}

\def\matulabbc{\,{\scalebox{0.15}{
 \begin{picture}(90,46) (330,-257)
    \SetWidth{2}
    \SetColor{Black}
    \Vertex(375,-254){12}
    \Line(375,-254)(395,-215)
    \Vertex(395,-215){9}
    \Line(375,-254)(335,-226)
    \Vertex(334,-224){9}
    \Line(375,-254)(355,-215)
    \Vertex(334,-180){9}
    \Line(334,-224)(334,-180)
    \Vertex(355,-215){9}
    \Line(375,-254)(417,-227)
    \Vertex(418,-225){9}
    \Line(375,-254)(375,-212)
    \Vertex(375,-212){9}
  \end{picture}
}}\,}

\def\matulabbg{\,{\scalebox{0.15}{
\begin{picture}(48,98) (349,-205)
    \SetWidth{2}
    \SetColor{Black}
    \Vertex(375,-202){12}
    \Line(375,-202)(402,-163)
    \Line(375,-202)(348,-163)
    \Vertex(348,-163){9}
    \Vertex(402,-163){9}
    \Line(348,-163)(360,-116)
    \Line(348,-163)(332,-116)
    \Vertex(360,-116){9}
    \Vertex(332,-116){9}
    \Line(402,-163)(392,-116)
    \Vertex(392,-116){9}
    \Line(402,-163)(420,-116)
    \Vertex(420,-116){9}
\end{picture}
}}\,}

\def\matulabbi{\,{\scalebox{0.15}{
 \begin{picture}(58,153) (359,-150)
    \SetWidth{2}
    \SetColor{Black}
    \Vertex(375,-147){12}
    \Line(375,-147)(395,-108)
    \Line(375,-147)(354,-108)
    \Line(375,-147)(437,-108)
    \Vertex(353,-108){9}
    \Vertex(395,-108){9}
    \Vertex(437,-108){9}
    \Vertex(395,-56){9}
    \Line(353,-105)(353,-58)
    \Line(395,-105)(395,-58)
    \Vertex(353,-58){9}
    \Line(353,-52)(353,-5)
    \Vertex(353,-5){9}
    \Line(395,-52)(395,-5)
    \Vertex(395,-5){9}
  \end{picture}
}}\,}

\def\matulabcc{\,{\scalebox{0.15}{
 \begin{picture}(48,153) (349,-150)
    \SetWidth{2}
    \SetColor{Black}
    \Vertex(375,-147){12}
    \Line(375,-147)(395,-108)
    \Line(375,-147)(354,-108)
    \Vertex(353,-108){9}
    \Vertex(395,-108){9}
    \Vertex(395,-58){9}
    \Line(395,-108)(395,-56)
    \Line(353,-108)(353,-56)
    \Vertex(353,-56){9}
    \Line(353,-56)(333,-8)
    \Vertex(333,-8){9}
    \Line(353,-56)(373,-8)
    \Vertex(373,-8){9}
  \end{picture}
}}\,}

\def\matulabci{\,{\scalebox{0.15}{
 \begin{picture}(58,99) (359,-251)
    \SetWidth{2}
    \SetColor{Black}
    \Line(370,-201)(390,-164)
    \Line(370,-201)(349,-164)
    \Vertex(349,-164){9}
    \Vertex(390,-164){9}
    \Vertex(349,-120){9}
    \Line(349,-164)(349,-120)
    \Vertex(405,-201){9}
    \Vertex(405,-248){12}
    \Line(405,-248)(370,-201)
    \Line(405,-248)(434,-201)
    \Line(405,-248)(405,-201)
    \Vertex(370,-201){9}
    \Vertex(434,-201){9}
  \end{picture}
}}\,}

\def\matulabda{\,{\scalebox{0.15}{
 \begin{picture}(90,46) (330,-300)
    \SetWidth{2}
    \SetColor{Black}
    \Vertex(375,-300){12}
    \Line(375,-300)(375,-254)
    \Vertex(375,-254){9}
    \Line(375,-254)(395,-215)
    \Vertex(395,-215){9}
    \Line(375,-254)(334,-224)
    \Vertex(334,-224){9}
    \Line(375,-254)(355,-215)
    \Vertex(355,-215){9}
    \Line(375,-255)(418,-225)
    \Vertex(418,-225){9}
  \end{picture}
}}\,}

\def\matulabea{\,{\scalebox{0.15}{
\begin{picture}(88,98) (349,-205)
    \SetWidth{2}
    \SetColor{Black}
    \Vertex(375,-202){12}
    \Line(375,-202)(395,-163)
    \Line(375,-202)(353,-163)
    \Vertex(353,-163){9}
    \Vertex(310,-163){9}
    \Line(375,-202)(310,-163)
    \Vertex(395,-163){9}
    \Line(353,-163)(353,-113)
    \Vertex(353,-113){9}
    \Line(395,-163)(395,-113)
    \Vertex(395,-113){9}
    \Line(375,-202)(430,-163)
    \Vertex(430,-163){9}
    \Line(310,-164)(310,-113)
    \Vertex(310,-113){9}
\end{picture}
}}\,}

\def\matulabeg{\,{\scalebox{0.15}{
 \begin{picture}(48,153) (349,-150)
    \SetWidth{2}
    \SetColor{Black}
    \Vertex(375,-147){12}
    \Line(375,-147)(395,-108)
    \Line(375,-147)(354,-108)
    \Vertex(353,-108){9}
    \Vertex(395,-108){9}
    \Vertex(395,-56){9}
    \Line(353,-105)(353,-56)
    \Line(395,-105)(395,-56)
    \Vertex(353,-56){9}
    \Vertex(395,-5){9}
    \Line(395,-56)(395,-5)
    \Line(353,-56)(353,-5)
    \Vertex(353,-5){9}
    \Line(353,-5)(353,40)
    \Vertex(353,40){9}
  \end{picture}
}}\,}

\def\matulabfc{\,{\scalebox{0.15}{
\begin{picture}(48,98) (349,-205)
    \SetWidth{2}
    \SetColor{Black}
    \Vertex(375,-202){12}
    \Line(375,-202)(399,-161)
    \Line(375,-202)(350,-163)
    \Vertex(350,-163){9}
    \Vertex(399,-161){9}
    \Line(375,-202)(375,-160)
    \Vertex(375,-161){9}
    \Vertex(425,-163){9}
    \Vertex(365,-110){9}
    \Vertex(335,-110){9}
    \Line(375,-202)(425,-165)
    \Line(350,-161)(335,-110)
    \Line(350,-161)(365,-110)
  \end{picture}
}}\,}

\def\matulabfi{\,{\scalebox{0.15}{
  \begin{picture}(68,88) (329,-215)
    \SetWidth{2}
    \SetColor{Black}
    \Vertex(375,-212){12}
    \Line(375,-212)(405,-171)
    \Line(375,-212)(354,-171)
    \Vertex(353,-171){9}
    \Vertex(405,-171){9}
    \Vertex(405,-131){9}
    \Line(405,-171)(405,-131)
    \Line(353,-171)(330,-131)
    \Line(353,-171)(376,-131)
    \Line(353,-171)(353,-129)
    \Vertex(353,-129){9}
    \Vertex(330,-131){9}
    \Vertex(376,-131){9}
  \end{picture}
}}\,}

\def\matulabga{\,{\scalebox{0.15}{
 \begin{picture}(58,150) (359,-205)
    \SetWidth{2}
    \SetColor{Black}
    \Line(376,-148)(395,-111)
    \Line(373,-149)(354,-111)
    \Vertex(353,-111){9}
    \Vertex(395,-111){9}
    \Line(353,-111)(353,-59)
    \Vertex(353,-59){9}
    \Line(353,-59)(353,-15)
    \Vertex(353,-15){9}
    \Line(395,-202)(374,-149)
    \Vertex(374,-149){9}
    \Vertex(395,-202){12}
    \Vertex(416,-149){9}
    \Line(395,-202)(416,-149)
  \end{picture}
}}\,}

\def\matulabgg{\,{\scalebox{0.15}{
\begin{picture}(48,135) (349,-303)
    \SetWidth{2}
    \SetColor{Black}
    \Line(376,-163)(395,-126)
    \Line(373,-164)(354,-126)
    \Vertex(353,-126){9}
    \Vertex(395,-126){9}
    \Vertex(374,-164){9}
    \Line(374,-205)(374,-164)
    \Vertex(374,-207){9}
    \Line(374,-252)(374,-207)
    \Vertex(374,-252){9}
    \Vertex(374,-300){12}
    \Line(374,-300)(374,-252)
  \end{picture}
}}\,}

\def\matulabha{\,{\scalebox{0.15}{
\begin{picture}(68,98) (369,-205)
    \SetWidth{2}
    \SetColor{Black}
    \Vertex(410,-202){12}
    \Line(410,-202)(395,-163)
    \Line(410,-202)(354,-163)
    \Vertex(353,-163){9}
    \Vertex(395,-163){9}
    \Line(353,-160)(353,-115)
    \Vertex(353,-115){9}
    \Line(395,-159)(395,-115)
    \Vertex(395,-115){9}
    \Line(410,-202)(430,-163)
    \Vertex(430,-163){9}
    \Vertex(465,-163){9}
    \Line(465,-163)(410,-202)
     \Line(353,-115)(353,-70)
    \Vertex(353,-70){9}
    \end{picture}
}}\,}

\def\matulabhc{\,{\scalebox{0.15}{
\begin{picture}(56,98) (357,-253)
    \SetWidth{2}
    \SetColor{Black}
    \Vertex(375,-202){9}
    \Vertex(375,-250){12}
    \Line(375,-250)(375,-202)
    \Line(375,-202)(395,-163)
    \Line(375,-202)(353,-163)
    \Vertex(353,-163){9}
    \Vertex(395,-163){9}
    \Line(353,-163)(353,-116)
    \Vertex(353,-116){9}
    \Line(395,-163)(395,-116)
    \Vertex(395,-116){9}
    \Line(375,-202)(430,-163)
    \Vertex(430,-163){9}
\end{picture}
}}\,}

\def\matulabic{\,{\scalebox{0.15}{
\begin{picture}(18,204) (380,-99)
    \SetWidth{2}
    \SetColor{Black}
    \Line(395,-95)(375,-51)
    \Vertex(375,-48){9}
    \Vertex(395,-96){12}
    \Vertex(415,-48){9}
    \Line(395,-96)(415,-48)
    \Line(375,-44)(375,0)
    \Vertex(374,3){9}
    \Line(375,6)(375,50)
    \Vertex(375,53){9}
    \Line(375,53)(375,95)
    \Vertex(375,95){9}
     \Line(375,95)(375,130)
    \Vertex(375,130){9}
  \end{picture}
}}\,}

\def\matulacog{\,{\scalebox{0.15}{
\begin{picture}(58,98) (359,-205)
    \SetWidth{2}
    \SetColor{Black}
    \Vertex(375,-202){12}
    \Line(375,-202)(390,-163)
    \Line(375,-202)(343,-163)
    \Vertex(343,-163){9}
    \Vertex(390,-163){9}
    \Line(343,-163)(328,-111)
    \Line(343,-163)(358,-111)
    \Vertex(328,-111){9}
    \Line(390,-159)(390,-111)
    \Vertex(390,-111){9}
    \Line(375,-202)(430,-163)
    \Vertex(430,-163){9}
    \Vertex(430,-111){9}
    \Vertex(358,-111){9}
    \Line(430,-163)(430,-111)
\end{picture}
}}\,}

\def\matulacaa{\,{\scalebox{0.15}{
 \begin{picture}(95,46) (335,-257)
    \SetWidth{2}
    \SetColor{Black}
    \Vertex(385,-254){12}
    \Line(385,-254)(395,-202)
    \Vertex(395,-202){9}
    \Line(385,-254)(335,-215)
    \Vertex(335,-215){9}
    \Line(385,-254)(356,-206)
    \Vertex(356,-206){9}
    \Line(385,-254)(417,-206)
    \Vertex(417,-206){9}
    \Line(385,-254)(375,-202)
    \Vertex(375,-202){9}
     \Line(385,-254)(439,-215)
    \Vertex(439,-215){9}
  \end{picture}
}}\,}

\def\matulacac{\,{\scalebox{0.15}{
 \begin{picture}(58,150) (359,-205)
    \SetWidth{2}
    \SetColor{Black}
    \Line(376,-148)(395,-113)
    \Line(373,-149)(354,-112)
    \Vertex(353,-111){9}
    \Vertex(395,-111){9}
    \Line(353,-108)(353,-61)
    \Vertex(353,-59){9}
    \Line(400,-202)(374,-153)
    \Vertex(374,-149){9}
    \Vertex(400,-202){12}
    \Vertex(426,-149){9}
    \Line(400,-200)(426,-153)
    \Vertex(426,-111){9}
    \Vertex(426,-59){9}
    \Line(426,-111)(426,-59)
    \Line(426,-149)(426,-111)
  \end{picture}
}}\,}

\def\matulacag{\,{\scalebox{0.15}{
\begin{picture}(58,98) (359,-205)
    \SetWidth{2}
    \SetColor{Black}
    \Vertex(375,-202){12}
    \Line(376,-200)(395,-165)
    \Line(373,-201)(354,-164)
    \Vertex(353,-163){9}
    \Vertex(395,-163){9}
    \Line(353,-160)(353,-115)
    \Vertex(353,-115){9}
    \Line(395,-159)(395,-115)
    \Vertex(395,-115){9}
    \Line(376,-201)(430,-164)
    \Vertex(430,-163){9}
     \Line(353,-115)(353,-70)
    \Vertex(353,-70){9}
    \Line(353,-70)(353,-30)
    \Vertex(353,-30){9}
    \end{picture}
}}\,}

\def\matulacca{\,{\scalebox{0.15}{
 \begin{picture}(48,99) (349,-303)
    \SetWidth{2}
    \SetColor{Black}
    \Line(376,-199)(395,-164)
    \Line(376,-200)(353,-163)
    \Vertex(353,-162){9}
    \Vertex(399,-162){9}
    \Vertex(376,-160){9}
    \Vertex(376,-248){9}
    \Vertex(376,-300){12}
    \Line(376,-300)(376,-248)
    \Line(376,-245)(375,-200)
    \Line(376,-200)(376,-160)
    \Vertex(376,-200){9}
  \end{picture}
}}\,}

\def\matulaccg{\,{\scalebox{0.15}{
\begin{picture}(48,98) (349,-205)
    \SetWidth{2}
    \SetColor{Black}
    \Vertex(375,-202){12}
    \Line(375,-202)(399,-163)
    \Line(375,-202)(354,-163)
    \Vertex(353,-163){9}
    \Vertex(399,-163){9}
    \Line(375,-202)(375,-160)
    \Vertex(375,-157){9}
    \Vertex(353,-115){9}
    \Line(353,-155)(353,-115)
    \Vertex(340,-70){9}
    \Line(353,-115)(340,-70)
    \Vertex(366,-70){9}
    \Line(353,-115)(366,-70)
  \end{picture}
}}\,}

\def\matulacdg{\,{\scalebox{0.15}{
 \begin{picture}(58,150) (359,-205)
    \SetWidth{2}
    \SetColor{Black}
    \Line(376,-148)(395,-113)
    \Line(373,-149)(354,-112)
    \Vertex(353,-111){9}
    \Vertex(395,-111){9}
    \Vertex(395,-65){9}
    \Line(395,-111)(395,-65)
    \Line(353,-108)(353,-65)
    \Vertex(353,-65){9}
    \Line(400,-202)(374,-153)
    \Vertex(374,-149){9}
    \Vertex(400,-202){12}
    \Vertex(426,-149){9}
    \Line(400,-200)(426,-153)
    \Vertex(426,-111){9}
    \Line(426,-149)(426,-111)
  \end{picture}
}}\,}

\def\matulacdi{\,{\scalebox{0.15}{
\begin{picture}(58,98) (359,-205)
    \SetWidth{2}
    \SetColor{Black}
    \Vertex(375,-202){12}
    \Line(376,-200)(395,-165)
    \Line(373,-201)(348,-165)
    \Vertex(348,-163){9}
    \Vertex(395,-163){9}
    \Line(348,-160)(348,-113)
    \Vertex(348,-113){9}
    \Line(348,-113)(348,-70)
    \Vertex(348,-70){9}
    \Line(395,-159)(380,-113)
    \Vertex(380,-113){9}
    \Line(395,-163)(415,-113)
    \Vertex(415,-113){9}
    \Line(375,-202)(430,-165)
    \Vertex(430,-165){9}
\end{picture}
}}\,}

\def\matulacec{\,{\scalebox{0.15}{
\begin{picture}(48,98) (349,-253)
    \SetWidth{2}
    \SetColor{Black}
    \Vertex(375,-202){9}
     \Vertex(375,-250){12}
     \Line(375,-250)(375,-202)
    \Line(375,-202)(399,-161)
    \Line(375,-202)(353,-161)
    \Vertex(353,-161){9}
    \Vertex(399,-161){9}
    \Line(375,-202)(375,-157)
    \Vertex(375,-157){9}
    \Vertex(353,-115){9}
    \Line(353,-155)(353,-115)
    \Vertex(353,-70){9}
    \Line(353,-115)(353,-70)
  \end{picture}
}}\,}

\def\matulacei{\,{\scalebox{0.15}{
\begin{picture}(48,98) (349,-205)
    \SetWidth{2}
    \SetColor{Black}
    \Vertex(375,-202){12}
    \Line(375,-202)(399,-159)
    \Line(375,-202)(350,-159)
    \Vertex(350,-159){9}
    \Line(375,-202)(325,-163)
    \Vertex(325,-163){9}
    \Vertex(325,-110){9}
    \Line(325,-163)(325,-110)
    \Vertex(399,-159){9}
    \Line(375,-202)(375,-158)
    \Vertex(375,-158){9}
    \Vertex(425,-163){9}
    \Vertex(350,-110){9}
    \Line(375,-202)(425,-165)
    \Line(350,-159)(350,-110)
  \end{picture}
}}\,}

\def\matulacfg{\,{\scalebox{0.15}{
\begin{picture}(48,98) (349,-248)
    \SetWidth{2}
    \SetColor{Black}
    \Vertex(375,-202){9}
    \Vertex(375,-245){12}
    \Line(375,-202)(375,-245)
    \Line(375,-202)(395,-165)
    \Line(375,-202)(348,-164)
    \Vertex(348,-163){9}
    \Vertex(395,-163){9}
    \Line(348,-160)(348,-113)
    \Vertex(348,-111){9}
    \Line(395,-159)(380,-112)
    \Vertex(380,-111){9}
    \Line(395,-163)(415,-113)
    \Vertex(415,-113){9}
\end{picture}
}}\,}

\def\matulacgc{\,{\scalebox{0.15}{
  \begin{picture}(68,88) (329,-215)
    \SetWidth{2}
    \SetColor{Black}
    \Vertex(375,-212){12}
    \Line(376,-210)(395,-175)
    \Line(373,-211)(354,-174)
    \Vertex(353,-171){9}
    \Vertex(395,-171){9}
    \Line(353,-171)(330,-131)
    \Line(353,-171)(376,-131)
    \Line(353,-171)(353,-129)
    \Vertex(353,-129){9}
    \Vertex(330,-131){9}
    \Vertex(376,-131){9}
    \Vertex(330,-85){9}
    \Line(330,-131)(330,-85)
  \end{picture}
}}\,}

\def\matulacgi{\,{\scalebox{0.15}{
 \begin{picture}(60,153) (361,-150)
    \SetWidth{2}
    \SetColor{Black}
    \Vertex(375,-147){12}
    \Line(375,-147)(395,-108)
    \Line(375,-147)(354,-108)
    \Line(375,-147)(437,-108)
    \Vertex(353,-108){9}
    \Vertex(395,-108){9}
    \Vertex(437,-108){9}
    \Vertex(437,-58){9}
    \Line(437,-108)(437,-58)
    \Vertex(395,-56){9}
    \Line(353,-105)(353,-58)
    \Line(395,-105)(395,-58)
    \Vertex(353,-58){9}
    \Line(353,-52)(353,-5)
    \Vertex(353,-5){9}
    \Line(395,-52)(395,-5)
    \Vertex(395,-5){9}
  \end{picture}
}}\,}

\def\matulachc{\,{\scalebox{0.15}{
  \begin{picture}(78,88) (339,-215)
    \SetWidth{2}
    \SetColor{Black}
    \Vertex(375,-212){12}
    \Line(375,-212)(385,-171)
    \Line(375,-212)(353,-171)
    \Vertex(353,-171){9}
    \Vertex(385,-171){9}
    \Vertex(417,-171){9}
    \Line(375,-212)(417,-171)
    \Line(353,-171)(330,-131)
    \Line(353,-171)(376,-131)
    \Line(353,-171)(353,-129)
    \Vertex(353,-129){9}
    \Vertex(330,-131){9}
    \Vertex(376,-131){9}
  \end{picture}
}}\,}

\def\matulachi{\,{\scalebox{0.15}{
\begin{picture}(48,98) (349,-205)
    \SetWidth{2}
    \SetColor{Black}
    \Vertex(375,-202){12}
    \Line(375,-202)(395,-163)
    \Line(375,-202)(348,-163)
    \Vertex(348,-163){9}
    \Vertex(395,-163){9}
    \Line(348,-160)(348,-113)
    \Vertex(348,-111){9}
    \Line(348,-111)(348,-65)
    \Vertex(348,-65){9}
    \Line(348,-65)(348,-20)
    \Vertex(348,-20){9}
    \Line(395,-159)(380,-112)
    \Vertex(380,-111){9}
    \Line(395,-163)(415,-113)
    \Vertex(415,-113){9}
\end{picture}
}}\,}

\def\matulacig{\,{\scalebox{0.15}{
 \begin{picture}(73,150) (374,-205)
    \SetWidth{2}
    \SetColor{Black}
    \Line(376,-148)(395,-113)
    \Line(373,-149)(354,-112)
    \Vertex(353,-111){9}
    \Vertex(395,-111){9}
    \Line(353,-108)(353,-65)
    \Vertex(353,-65){9}
    \Line(400,-202)(374,-153)
    \Vertex(374,-149){9}
    \Vertex(400,-202){12}
    \Vertex(426,-149){9}
    \Vertex(470,-149){9}
    \Line(400,-202)(470,-149)
    \Line(400,-200)(426,-153)
    \Vertex(426,-111){9}
    \Line(426,-149)(426,-111)
  \end{picture}
}}\,}

\def\matuladoa{\,{\scalebox{0.15}{
 \begin{picture}(48,190) (349,-193)
    \SetWidth{2}
    \SetColor{Black}
    \Vertex(375,-147){9}
    \Vertex(375,-190){12}
    \Line(375,-190)(375,-147)
    \Line(375,-147)(395,-106)
    \Line(375,-147)(353,-106)
    \Vertex(353,-106){9}
    \Vertex(395,-106){9}
    \Line(353,-105)(353,-58)
    \Vertex(353,-56){9}
    \Line(353,-52)(353,-10)
    \Vertex(353,-10){9}
    \Line(353,-10)(353,35)
    \Vertex(353,35){9}
  \end{picture}
}}\,}

\def\matuladoi{\,{\scalebox{0.15}{
 \begin{picture}(90,46) (330,-257)
    \SetWidth{2}
    \SetColor{Black}
    \Vertex(375,-254){12}
    \Line(375,-254)(395,-215)
    \Vertex(395,-215){9}
    \Line(375,-254)(335,-226)
    \Vertex(334,-224){9}
    \Line(375,-254)(355,-215)
    \Vertex(334,-180){9}
    \Line(334,-224)(334,-180)
    \Vertex(334,-140){9}
    \Line(334,-140)(334,-180)
    \Vertex(355,-215){9}
    \Line(375,-254)(417,-227)
    \Vertex(418,-225){9}
    \Line(375,-254)(375,-212)
    \Vertex(375,-212){9}
  \end{picture}
}}\,}

\def\matuladai{\,{\scalebox{0.15}{
\begin{picture}(48,98) (349,-205)
    \SetWidth{2}
    \SetColor{Black}
    \Vertex(375,-202){12}
    \Line(375,-202)(395,-165)
    \Line(375,-202)(354,-164)
    \Vertex(353,-163){9}
    \Vertex(310,-163){9}
    \Vertex(310,-113){9}
    \Line(310,-163)(310,-113)
    \Line(375,-202)(310,-165)
    \Vertex(395,-163){9}
    \Line(353,-160)(353,-113)
    \Vertex(353,-113){9}
    \Line(395,-159)(395,-113)
    \Vertex(395,-113){9}
    \Line(375,-202)(430,-164)
    \Vertex(430,-163){9}
    \Line(430,-164)(430,-113)
    \Vertex(430,-113){9}
\end{picture}
}}\,}

\def\matuladba{\,{\scalebox{0.15}{
 \begin{picture}(48,153) (349,-150)
    \SetWidth{2}
    \SetColor{Black}
    \Vertex(375,-147){12}
    \Line(375,-147)(395,-108)
    \Line(375,-147)(354,-108)
    \Vertex(353,-108){9}
    \Vertex(395,-108){9}
    \Line(353,-108)(353,-58)
    \Vertex(353,-56){9}
    \Line(353,-56)(333,-8)
    \Vertex(333,-8){9}
    \Vertex(333,35){9}
    \Line(333,-8)(333,35)
    \Line(353,-56)(373,-8)
    \Vertex(373,-8){9}
  \end{picture}
}}\,}

\def\matuladca{\,{\scalebox{0.15}{
 \begin{picture}(48,150) (349,-253)
    \SetWidth{2}
    \SetColor{Black}
    \Line(376,-148)(395,-113)
    \Line(373,-149)(354,-112)
    \Vertex(353,-109){9}
    \Vertex(395,-111){9}
    \Line(353,-108)(353,-65)
    \Vertex(353,-65){9}
    \Line(374,-200)(374,-153)
    \Vertex(374,-149){9}
    \Vertex(374,-202){9}
     \Vertex(374,-250){12}
     \Line(374,-250)(374,-202)
    \Line(395,-108)(395,-65)
    \Vertex(395,-65){9}
  \end{picture}
}}\,}

\def\matuladcc{\,{\scalebox{0.15}{
\begin{picture}(62,98) (363,-205)
    \SetWidth{2}
    \SetColor{Black}
    \Vertex(375,-202){12}
    \Line(375,-202)(390,-165)
    \Line(375,-202)(349,-164)
    \Vertex(349,-163){9}
    \Vertex(390,-163){9}
    \Line(349,-163)(332,-111)
    \Line(349,-163)(366,-111)
    \Vertex(332,-111){9}
    \Vertex(366,-111){9}
    \Line(390,-159)(390,-112)
    \Vertex(390,-111){9}
    \Line(375,-202)(425,-164)
    \Vertex(425,-163){9}
    \Line(375,-202)(460,-164)
    \Vertex(460,-163){9}
\end{picture}
}}\,}

\def\matuladci{\,{\scalebox{0.15}{
 \begin{picture}(48,153) (349,-150)
    \SetWidth{2}
    \SetColor{Black}
    \Vertex(375,-147){12}
    \Line(375,-147)(400,-108)
    \Line(375,-147)(354,-108)
    \Vertex(353,-108){9}
    \Vertex(400,-108){9}
    \Vertex(400,-58){9}
    \Vertex(400,-8){9}
    \Line(400,-108)(400,-58)
    \Line(400,-58)(400,-8)
    \Line(353,-108)(353,-58)
    \Vertex(353,-56){9}
    \Line(353,-56)(336,-8)
    \Vertex(336,-8){9}
    \Line(353,-56)(370,-8)
    \Vertex(370,-8){9}
  \end{picture}
}}\,}

\def\matuladdc{\,{\scalebox{0.15}{
 \begin{picture}(58,150) (359,-205)
    \SetWidth{2}
    \SetColor{Black}
    \Line(376,-148)(395,-111)
    \Line(373,-149)(354,-111)
    \Vertex(353,-111){9}
    \Vertex(395,-111){9}
    \Line(353,-111)(370,-64)
    \Line(353,-111)(336,-64)
    \Vertex(370,-64){9}
    \Vertex(336,-64){9}
    \Line(395,-202)(374,-153)
    \Vertex(374,-149){9}
    \Vertex(395,-202){12}
    \Vertex(416,-149){9}
    \Line(395,-200)(416,-153)
  \end{picture}
}}\,}

\def\matuladdi{\,{\scalebox{0.15}{
 \begin{picture}(63,150) (364,-205)
    \SetWidth{2}
    \SetColor{Black}
    \Line(376,-148)(395,-113)
    \Line(373,-149)(354,-112)
    \Vertex(353,-111){9}
    \Vertex(395,-111){9}
    \Line(353,-108)(353,-65)
    \Vertex(353,-65){9}
    \Vertex(353,-20){9}
    \Line(353,-65)(353,-20)
    \Line(400,-202)(374,-153)
    \Vertex(374,-149){9}
    \Vertex(400,-202){12}
    \Vertex(426,-149){9}
    \Line(400,-200)(426,-153)
    \Vertex(426,-111){9}
    \Line(426,-149)(426,-111)
  \end{picture}
}}\,}

\def\matuladeg{\,{\scalebox{0.15}{
\begin{picture}(53,98) (354,-205)
    \SetWidth{2}
    \SetColor{Black}
    \Vertex(375,-202){12}
    \Line(376,-200)(399,-161)
    \Line(373,-201)(354,-164)
    \Vertex(350,-163){9}
    \Vertex(399,-161){9}
    \Line(375,-201)(375,-160)
    \Vertex(376,-161){9}
    \Vertex(425,-163){9}
    \Vertex(350,-110){9}
    \Line(375,-202)(425,-165)
    \Line(350,-161)(350,-110)
    \Line(350,-110)(350,-65)
    \Vertex(350,-65){9}
    \Vertex(350,-20){9}
    \Line(350,-65)(350,-20)
  \end{picture}
}}\,}

\def\matuladfa{\,{\scalebox{0.15}{
\begin{picture}(55,98) (356,-253)
    \SetWidth{2}
    \SetColor{Black}
    \Vertex(375,-202){9}
    \Vertex(375,-250){12}
    \Line(375,-250)(375,-202)
    \Line(375,-202)(399,-161)
    \Line(375,-202)(350,-163)
    \Vertex(350,-163){9}
    \Vertex(399,-161){9}
    \Line(375,-202)(375,-160)
    \Vertex(375,-161){9}
    \Vertex(425,-163){9}
    \Vertex(350,-110){9}
    \Line(375,-202)(425,-165)
    \Line(350,-161)(350,-110)
  \end{picture}
}}\,}

\def\matuladfc{\,{\scalebox{0.15}{
\begin{picture}(48,98) (349,-205)
    \SetWidth{2}
    \SetColor{Black}
    \Vertex(375,-202){12}
    \Line(375,-202)(395,-165)
    \Line(375,-202)(354,-164)
    \Vertex(353,-163){9}
    \Vertex(310,-163){9}
    \Line(375,-202)(310,-165)
    \Vertex(395,-163){9}
    \Line(353,-160)(353,-113)
    \Vertex(353,-113){9}
    \Line(395,-159)(395,-113)
    \Vertex(395,-113){9}
    \Line(375,-202)(430,-164)
    \Vertex(430,-163){9}
    \Line(310,-164)(310,-113)
    \Vertex(310,-113){9}
    \Line(310,-113)(310,-70)
    \Vertex(310,-70){9}
\end{picture}
}}\,}

\def\matuladfg{\,{\scalebox{0.15}{
\begin{picture}(48,98) (349,-205)
    \SetWidth{2}
    \SetColor{Black}
    \Vertex(375,-202){12}
    \Line(375,-202)(402,-163)
    \Line(375,-202)(348,-163)
    \Vertex(348,-163){9}
    \Vertex(402,-163){9}
    \Line(348,-163)(360,-116)
    \Line(348,-163)(332,-116)
    \Vertex(360,-116){9}
    \Vertex(332,-116){9}
    \Vertex(332,-71){9}
    \Line(332,-116)(332,-71)
    \Line(402,-163)(392,-116)
    \Vertex(392,-116){9}
    \Line(402,-163)(420,-116)
    \Vertex(420,-116){9}
\end{picture}
}}\,}

\def\matuladgi{\,{\scalebox{0.15}{
 \begin{picture}(58,99) (359,-251)
    \SetWidth{2}
    \SetColor{Black}
    \Line(376,-201)(395,-164)
    \Line(376,-201)(354,-163)
    \Vertex(353,-162){9}
    \Vertex(399,-162){9}
    \Vertex(353,-120){9}
    \Vertex(399,-120){9}
    \Line(399,-162)(399,-120)
    \Line(353,-162)(353,-120)
    \Vertex(400,-201){9}
    \Vertex(400,-248){12}
    \Line(400,-248)(375,-204)
    \Line(400,-248)(424,-204)
    \Line(400,-248)(400,-201)
    \Vertex(376,-201){9}
    \Vertex(424,-201){9}
  \end{picture}
}}\,}

\def\matuladhg{\,{\scalebox{0.15}{
 \begin{picture}(48,163) (349,-150)
    \SetWidth{2}
    \SetColor{Black}
    \Vertex(375,-147){12}
    \Line(376,-145)(395,-110)
    \Line(373,-146)(354,-109)
    \Vertex(353,-106){9}
    \Vertex(395,-108){9}
    \Vertex(395,-56){9}
    \Line(353,-105)(353,-58)
    \Line(395,-105)(395,-58)
    \Vertex(353,-56){9}
    \Line(353,-52)(353,-5)
    \Vertex(353,-5){9}
    \Line(353,-5)(353,40)
    \Vertex(353,40){9}
    \Line(353,40)(353,80)
    \Vertex(353,80){9}
  \end{picture}
}}\,}

\def\matuladia{\,{\scalebox{0.15}{
 \begin{picture}(58,150) (359,-205)
    \SetWidth{2}
    \SetColor{Black}
    \Line(376,-148)(395,-113)
    \Line(373,-149)(354,-112)
    \Vertex(353,-109){9}
    \Vertex(395,-111){9}
    \Line(353,-108)(353,-61)
    \Vertex(353,-59){9}
    \Vertex(353,-14){9}
    \Line(353,-59)(353,-14)
    \Line(395,-200)(374,-153)
    \Line(395,-200)(415,-153)
    \Vertex(374,-149){9}
    \Vertex(415,-149){9}
    \Vertex(395,-202){12}
    \Line(395,-108)(395,-61)
    \Vertex(395,-59){9}
  \end{picture}
}}\,}

\def\matuladii{\,{\scalebox{0.15}{
  \begin{picture}(75,88) (336,-215)
    \SetWidth{2}
    \SetColor{Black}
    \Vertex(375,-212){12}
    \Line(375,-212)(405,-171)
    \Line(375,-212)(354,-171)
    \Vertex(353,-171){9}
    \Vertex(405,-171){9}
    \Vertex(405,-131){9}
    \Vertex(405,-91){9}
    \Line(405,-131)(405,-91)
    \Line(405,-171)(405,-131)
    \Line(353,-171)(330,-131)
    \Line(353,-171)(376,-131)
    \Line(353,-171)(353,-129)
    \Vertex(353,-129){9}
    \Vertex(330,-131){9}
    \Vertex(376,-131){9}
  \end{picture}
}}\,}

\def\matulaeoc{\,{\scalebox{0.15}{
 \begin{picture}(95,46) (335,-257)
    \SetWidth{2}
    \SetColor{Black}
    \Vertex(385,-254){12}
    \Line(385,-254)(395,-202)
    \Vertex(395,-202){9}
    \Line(385,-254)(335,-215)
    \Vertex(335,-215){9}
    \Vertex(335,-165){9}
    \Line(335,-215)(335,-165)
    \Line(385,-254)(356,-206)
    \Vertex(356,-206){9}
    \Line(385,-254)(417,-206)
    \Vertex(417,-206){9}
    \Line(385,-254)(375,-202)
    \Vertex(375,-202){9}
     \Line(385,-254)(439,-215)
    \Vertex(439,-215){9}
  \end{picture}
}}\,}

\def\matulaeoi{\,{\scalebox{0.15}{
 \begin{picture}(48,153) (349,-200)
    \SetWidth{2}
    \SetColor{Black}
    \Vertex(375,-147){9}
    \Vertex(375,-197){12}
    \Line(375,-197)(375,-147)
    \Line(376,-145)(395,-110)
    \Line(373,-146)(354,-109)
    \Vertex(353,-108){9}
    \Vertex(395,-108){9}
    \Vertex(395,-56){9}
    \Line(353,-105)(353,-58)
    \Line(395,-105)(395,-58)
    \Vertex(353,-58){9}
    \Line(353,-52)(353,-5)
    \Vertex(353,-5){9}
    \Line(395,-52)(395,-5)
    \Vertex(395,-5){9}
  \end{picture}
}}\,}

\def\matulaeba{\,{\scalebox{0.15}{
\begin{picture}(61,98) (361,-205)
    \SetWidth{2}
    \SetColor{Black}
    \Vertex(375,-202){12}
    \Line(375,-202)(402,-163)
    \Line(375,-202)(348,-163)
    \Vertex(348,-163){9}
    \Vertex(402,-163){9}
    \Vertex(448,-163){9}
    \Line(448,-163)(375,-202)
    \Line(348,-163)(360,-116)
    \Line(348,-163)(332,-116)
    \Vertex(360,-116){9}
    \Vertex(332,-116){9}
    \Line(402,-163)(392,-116)
    \Vertex(392,-116){9}
    \Line(402,-163)(420,-116)
    \Vertex(420,-116){9}
\end{picture}
}}\,}

\def\matulaebc{\,{\scalebox{0.15}{
\begin{picture}(58,98) (359,-205)
    \SetWidth{2}
    \SetColor{Black}
    \Vertex(375,-202){12}
    \Line(375,-202)(395,-163)
    \Line(375,-202)(353,-163)
    \Vertex(353,-163){9}
    \Vertex(395,-163){9}
    \Line(353,-160)(353,-113)
    \Vertex(353,-113){9}
    \Line(395,-159)(395,-113)
    \Vertex(395,-113){9}
    \Line(376,-201)(430,-163)
    \Vertex(430,-163){9}
    \Line(430,-163)(430,-113)
    \Vertex(430,-113){9}
     \Line(353,-113)(353,-70)
    \Vertex(353,-70){9}
    \Line(353,-70)(353,-30)
    \Vertex(353,-30){9}
\end{picture}
}}\,}

\def\matulaeda{\,{\scalebox{0.15}{
\begin{picture}(68,98) (369,-205)
    \SetWidth{2}
    \SetColor{Black}
    \Vertex(410,-202){12}
    \Line(410,-202)(395,-165)
    \Line(410,-202)(354,-164)
    \Vertex(353,-163){9}
    \Vertex(395,-163){9}
    \Line(353,-160)(353,-115)
    \Vertex(353,-115){9}
    \Line(395,-159)(395,-115)
    \Vertex(395,-115){9}
    \Line(395,-115)(395,-70)
    \Vertex(395,-70){9}
    \Line(410,-202)(430,-164)
    \Vertex(430,-163){9}
    \Vertex(465,-163){9}
    \Line(465,-163)(410,-202)
     \Line(353,-115)(353,-70)
    \Vertex(353,-70){9}
    \end{picture}
}}\,}

\def\matulaedg{\,{\scalebox{0.15}{
 \begin{picture}(58,150) (359,-255)
    \SetWidth{2}
    \SetColor{Black}
    \Line(376,-148)(395,-113)
    \Line(373,-149)(354,-112)
    \Vertex(353,-109){9}
    \Vertex(395,-111){9}
    \Line(353,-108)(353,-61)
    \Vertex(353,-59){9}
    \Line(390,-200)(374,-153)
    \Vertex(374,-149){9}
    \Vertex(395,-202){9}
    \Vertex(395,-252){12}
    \Line(395,-252)(395,-202)
    \Vertex(416,-149){9}
    \Line(395,-200)(416,-153)
  \end{picture}
}}\,}

\def\matulaeeg{\,{\scalebox{0.15}{
 \begin{picture}(48,153) (349,-150)
    \SetWidth{2}
    \SetColor{Black}
    \Vertex(375,-147){12}
    \Vertex(430,-108){9}
    \Line(375,-147)(430,-108)
    \Line(376,-145)(395,-108)
    \Line(373,-146)(354,-108)
    \Vertex(353,-108){9}
    \Vertex(395,-108){9}
    \Vertex(395,-58){9}
    \Line(395,-108)(395,-58)
    \Line(353,-108)(353,-58)
    \Vertex(353,-56){9}
    \Line(353,-56)(333,-8)
    \Vertex(333,-8){9}
    \Line(353,-56)(373,-8)
    \Vertex(373,-8){9}
  \end{picture}
}}\,}

\def\matulaefc{\,{\scalebox{0.15}{
\begin{picture}(58,98) (359,-255)
    \SetWidth{2}
    \SetColor{Black}
    \Vertex(375,-252){12}
    \Vertex(375,-202){9}
    \Line(375,-252)(375,-202)
    \Line(376,-200)(395,-165)
    \Line(373,-201)(354,-164)
    \Vertex(353,-163){9}
    \Vertex(395,-163){9}
    \Line(353,-160)(353,-113)
    \Vertex(353,-113){9}
    \Line(395,-159)(395,-113)
    \Vertex(395,-113){9}
    \Line(376,-201)(430,-164)
    \Vertex(430,-163){9}
    \Line(430,-164)(430,-113)
    \Vertex(430,-113){9}
\end{picture}
}}\,}

\def\matulaefi{\,{\scalebox{0.15}{
\begin{picture}(48,98) (349,-205)
    \SetWidth{2}
    \SetColor{Black}
    \Vertex(375,-202){12}
    \Line(375,-202)(399,-161)
    \Line(375,-202)(350,-163)
    \Vertex(350,-163){9}
    \Vertex(399,-161){9}
    \Line(375,-202)(375,-160)
    \Vertex(375,-161){9}
    \Vertex(425,-163){9}
    \Vertex(365,-110){9}
    \Vertex(335,-110){9}
    \Vertex(335,-65){9}
    \Line(335,-110)(335,-65)
    \Line(375,-202)(425,-165)
    \Line(350,-161)(335,-110)
    \Line(350,-161)(365,-110)
  \end{picture}
}}\,}

\def\matulaega{\,{\scalebox{0.15}{
\begin{picture}(68,98) (369,-205)
    \SetWidth{2}
    \SetColor{Black}
    \Vertex(395,-202){12}
    \Line(395,-202)(395,-165)
    \Line(395,-202)(348,-165)
    \Vertex(348,-163){9}
    \Vertex(395,-163){9}
    \Line(348,-160)(348,-115)
    \Vertex(348,-115){9}
    \Line(348,-115)(348,-70)
    \Vertex(348,-70){9}
    \Line(395,-159)(380,-115)
    \Vertex(380,-115){9}
    \Line(395,-163)(415,-115)
    \Vertex(415,-115){9}
    \Line(395,-202)(442,-165)
    \Vertex(442,-165){9}
    \Vertex(442,-115){9}
    \Line(442,-165)(442,-115)
\end{picture}
}}\,}

\def\matulaegg{\,{\scalebox{0.15}{
 \begin{picture}(100,46) (340,-300)
    \SetWidth{2}
    \SetColor{Black}
    \Vertex(410,-300){12}
    \Line(410,-300)(375,-254)
    \Vertex(375,-254){9}
    \Vertex(445,-254){9}
    \Line(410,-300)(445,-254)
    \Line(375,-252)(395,-217)
    \Vertex(395,-215){9}
    \Line(375,-254)(335,-226)
    \Vertex(334,-224){9}
    \Line(375,-252)(356,-215)
    \Vertex(355,-215){9}
    \Line(375,-255)(417,-227)
    \Vertex(418,-225){9}
  \end{picture}
}}\,}

\def\matulaehg{\,{\scalebox{0.15}{
 \begin{picture}(58,99) (359,-301)
    \SetWidth{2}
    \SetColor{Black}
    \Line(376,-201)(395,-164)
    \Line(376,-201)(354,-163)
    \Vertex(353,-162){9}
    \Vertex(399,-162){9}
    \Vertex(400,-201){9}
    \Vertex(400,-248){9}
    \Vertex(400,-298){12}
    \Line(400,-298)(400,-248)
    \Line(400,-248)(375,-204)
    \Line(400,-248)(424,-204)
    \Line(400,-248)(400,-201)
    \Vertex(376,-201){9}
    \Vertex(424,-201){9}
  \end{picture}
}}\,}

\def\matulaeic{\,{\scalebox{0.15}{
\begin{picture}(48,98) (349,-205)
    \SetWidth{2}
    \SetColor{Black}
    \Vertex(375,-202){12}
    \Line(375,-202)(399,-159)
    \Line(375,-202)(350,-159)
    \Vertex(350,-159){9}
    \Line(375,-202)(325,-163)
    \Vertex(325,-163){9}
    \Vertex(325,-110){9}
    \Line(325,-163)(325,-110)
    \Vertex(399,-159){9}
    \Line(375,-202)(375,-158)
    \Vertex(375,-158){9}
    \Vertex(375,-110){9}
    \Line(375,-158)(375,-110)
    \Vertex(425,-163){9}
    \Vertex(350,-110){9}
    \Line(375,-202)(425,-165)
    \Line(350,-159)(350,-110)
  \end{picture}
}}\,}

\def\matulaeii{\,{\scalebox{0.15}{
 \begin{picture}(48,200) (349,-255)
    \SetWidth{2}
    \SetColor{Black}
    \Line(376,-148)(395,-113)
    \Line(373,-149)(354,-112)
    \Vertex(353,-109){9}
    \Vertex(395,-111){9}
    \Line(353,-108)(353,-61)
    \Vertex(353,-59){9}
    \Line(353,-61)(353,-15)
    \Vertex(353,-15){9}
    \Line(374,-200)(374,-153)
    \Vertex(374,-149){9}
    \Vertex(374,-202){9}
    \Vertex(374,-252){12}
    \Line(374,-202)(374,-252)
  \end{picture}
}}\,}

\def\matulafoa{\,{\scalebox{0.15}{
 \begin{picture}(58,153) (359,-150)
    \SetWidth{2}
    \SetColor{Black}
    \Vertex(375,-147){12}
    \Line(375,-147)(395,-108)
    \Line(375,-147)(354,-108)
    \Line(375,-147)(437,-108)
    \Vertex(353,-108){9}
    \Vertex(395,-108){9}
    \Vertex(437,-108){9}
    \Vertex(395,-56){9}
    \Line(353,-105)(353,-58)
    \Line(395,-105)(395,-58)
    \Vertex(353,-58){9}
    \Line(353,-52)(353,-5)
    \Vertex(353,-5){9}
    \Line(353,-5)(353,40)
    \Vertex(353,40){9}
    \Line(395,-52)(395,-5)
    \Vertex(395,-5){9}
  \end{picture}
}}\,}

\def\matulafog{\,{\scalebox{0.15}{
  \begin{picture}(78,88) (339,-215)
    \SetWidth{2}
    \SetColor{Black}
    \Vertex(375,-212){12}
    \Line(375,-212)(410,-175)
    \Line(375,-212)(354,-174)
    \Vertex(353,-171){9}
    \Vertex(410,-171){9}
    \Vertex(410,-131){9}
    \Line(410,-171)(410,-131)
    \Line(353,-171)(330,-131)
    \Line(353,-171)(376,-131)
    \Line(353,-171)(353,-129)
    \Vertex(353,-129){9}
    \Vertex(330,-131){9}
    \Vertex(376,-131){9}
    \Vertex(330,-85){9}
    \Line(330,-131)(330,-85)
  \end{picture}
}}\,}

\def\matulafac{\,{\scalebox{0.15}{
 \begin{picture}(90,46) (330,-257)
    \SetWidth{2}
    \SetColor{Black}
    \Vertex(375,-254){12}
    \Line(375,-254)(395,-215)
    \Vertex(395,-215){9}
    \Line(375,-254)(335,-226)
    \Vertex(334,-224){9}
    \Line(375,-254)(355,-215)
    \Vertex(350,-180){9}
    \Vertex(318,-180){9}
    \Line(334,-224)(318,-180)
    \Line(334,-224)(350,-180)
    \Vertex(355,-215){9}
    \Line(375,-254)(417,-227)
    \Vertex(418,-225){9}
    \Line(375,-254)(375,-212)
    \Vertex(375,-212){9}
  \end{picture}
}}\,}

\def\matulafag{\,{\scalebox{0.15}{
\begin{picture}(58,98) (359,-255)
    \SetWidth{2}
    \SetColor{Black}
    \Vertex(375,-202){9}
    \Vertex(375,-252){12}
    \Line(375,-202)(375,-252)
    \Line(375,-202)(395,-163)
    \Line(375,-202)(354,-163)
    \Vertex(353,-163){9}
    \Vertex(395,-163){9}
    \Line(353,-160)(353,-115)
    \Vertex(353,-115){9}
    \Line(395,-159)(395,-115)
    \Vertex(395,-115){9}
    \Line(375,-202)(430,-163)
    \Vertex(430,-163){9}
     \Line(353,-115)(353,-70)
    \Vertex(353,-70){9}
    \end{picture}
}}\,}

\def\matulafai{\,{\scalebox{0.15}{
  \begin{picture}(86,88) (347,-215)
    \SetWidth{2}
    \SetColor{Black}
    \Vertex(375,-212){12}
    \Line(375,-212)(405,-171)
    \Line(375,-212)(354,-171)
    \Vertex(353,-171){9}
    \Vertex(405,-171){9}
    \Vertex(440,-171){9}
    \Line(440,-171)(375,-212)
    \Vertex(405,-131){9}
    \Line(405,-171)(405,-131)
    \Line(353,-171)(330,-131)
    \Line(353,-171)(376,-131)
    \Line(353,-171)(353,-129)
    \Vertex(353,-129){9}
    \Vertex(330,-131){9}
    \Vertex(376,-131){9}
  \end{picture}
}}\,}

\def\matulafca{\,{\scalebox{0.15}{
 \begin{picture}(63,150) (364,-205)
    \SetWidth{2}
    \SetColor{Black}
    \Line(376,-148)(395,-113)
    \Line(373,-149)(354,-112)
    \Vertex(353,-111){9}
    \Vertex(395,-111){9}
    \Vertex(395,-65){9}
    \Line(395,-111)(395,-65)
    \Line(353,-108)(353,-65)
    \Vertex(353,-65){9}
    \Line(400,-202)(374,-153)
    \Vertex(374,-149){9}
    \Vertex(400,-202){12}
    \Vertex(426,-149){9}
    \Line(400,-200)(426,-153)
    \Vertex(426,-111){9}
    \Line(426,-149)(426,-111)
    \Vertex(426,-65){9}
    \Line(426,-65)(426,-111)
  \end{picture}
}}\,}

\def\matulafda{\,{\scalebox{0.15}{
 \begin{picture}(64,99) (365,-251)
    \SetWidth{2}
    \SetColor{Black}
    \Line(370,-201)(390,-164)
    \Line(370,-201)(349,-164)
    \Vertex(349,-164){9}
    \Vertex(390,-164){9}
    \Vertex(349,-120){9}
    \Line(349,-162)(349,-120)
    \Vertex(349,-75){9}
    \Line(349,-75)(349,-120)
    \Vertex(405,-201){9}
    \Vertex(405,-248){12}
    \Line(405,-248)(370,-204)
    \Line(405,-248)(434,-204)
    \Line(405,-248)(405,-201)
    \Vertex(370,-201){9}
    \Vertex(434,-201){9}
  \end{picture}
}}\,}

\def\matulafdc{\,{\scalebox{0.15}{
 \begin{picture}(73,150) (374,-205)
    \SetWidth{2}
    \SetColor{Black}
    \Line(376,-148)(395,-113)
    \Line(373,-149)(354,-112)
    \Vertex(353,-111){9}
    \Vertex(395,-111){9}
    \Line(353,-108)(353,-65)
    \Vertex(353,-65){9}
    \Line(400,-202)(374,-153)
    \Vertex(374,-149){9}
    \Vertex(400,-202){12}
    \Vertex(426,-149){9}
    \Vertex(470,-149){9}
     \Vertex(470,-111){9}
     \Line(470,-111)(470,-149)
    \Line(400,-202)(470,-149)
    \Line(400,-200)(426,-153)
    \Vertex(426,-111){9}
    \Line(426,-149)(426,-111)
  \end{picture}
}}\,}

\def\matulafdg{\,{\scalebox{0.15}{
 \begin{picture}(48,153) (349,-150)
    \SetWidth{2}
    \SetColor{Black}
    \Vertex(375,-147){12}
    \Line(376,-145)(395,-110)
    \Line(373,-146)(354,-109)
    \Vertex(353,-106){9}
    \Vertex(395,-108){9}
    \Line(353,-105)(353,-58)
    \Vertex(353,-56){9}
    \Line(353,-52)(353,-10)
    \Vertex(353,-10){9}
    \Line(353,-10)(333,35)
    \Vertex(333,35){9}
    \Line(353,-10)(373,35)
    \Vertex(373,35){9}
  \end{picture}
}}\,}

\def\matulafec{\,{\scalebox{0.15}{
\begin{picture}(48,98) (349,-205)
    \SetWidth{2}
    \SetColor{Black}
    \Vertex(375,-202){12}
    \Line(375,-202)(398,-163)
    \Line(375,-202)(342,-163)
    \Vertex(342,-163){9}
    \Vertex(398,-163){9}
    \Line(342,-160)(342,-113)
    \Vertex(342,-111){9}
    \Line(342,-111)(362,-65)
    \Vertex(362,-65){9}
    \Line(342,-111)(322,-65)
    \Vertex(322,-65){9}
    \Line(398,-159)(383,-111)
    \Vertex(383,-111){9}
    \Line(398,-163)(418,-111)
    \Vertex(418,-111){9}
\end{picture}
}}\,}

\def\matulafei{\,{\scalebox{0.15}{
\begin{picture}(48,98) (349,-205)
    \SetWidth{2}
    \SetColor{Black}
    \Vertex(375,-202){12}
    \Line(375,-202)(399,-159)
    \Line(375,-202)(350,-159)
    \Vertex(350,-159){9}
    \Line(375,-202)(325,-163)
    \Vertex(325,-163){9}
    \Vertex(325,-110){9}
    \Line(325,-163)(325,-110)
    \Vertex(325,-65){9}
    \Line(325,-65)(325,-110)
    \Vertex(399,-159){9}
    \Line(375,-202)(375,-158)
    \Vertex(375,-158){9}
    \Vertex(425,-163){9}
    \Vertex(350,-110){9}
    \Line(375,-202)(425,-165)
    \Line(350,-159)(350,-110)
  \end{picture}
}}\,}

\def\matulaffa{\,{\scalebox{0.15}{
 \begin{picture}(48,153) (349,-150)
    \SetWidth{2}
    \SetColor{Black}
    \Vertex(375,-147){12}
    \Line(375,-147)(395,-108)
    \Line(375,-147)(353,-108)
    \Vertex(353,-108){9}
    \Vertex(395,-108){9}
    \Vertex(395,-56){9}
     \Vertex(353,-108){9}
    \Vertex(353,-56){9}
    \Line(353,-108)(353,-56)
    \Line(395,-108)(395,-56)
    \Vertex(353,-56){9}
    \Vertex(395,-5){9}
    \Line(395,-56)(395,-5)
    \Line(353,-56)(353,-5)
    \Vertex(353,-5){9}
    \Line(353,-5)(353,40)
    \Vertex(353,40){9}
    \Line(395,-5)(395,40)
    \Vertex(395,40){9}
  \end{picture}
}}\,}

\def\matulafgc{\,{\scalebox{0.15}{
\begin{picture}(56,98) (357,-253)
    \SetWidth{2}
    \SetColor{Black}
    \Vertex(375,-202){9}
    \Vertex(375,-250){12}
    \Line(375,-250)(375,-202)
    \Vertex(430,-202){9}
    \Line(430,-202)(375,-250)
    \Line(375,-202)(395,-163)
    \Line(375,-202)(354,-163)
    \Vertex(353,-163){9}
    \Vertex(395,-163){9}
    \Line(353,-160)(353,-116)
    \Vertex(353,-116){9}
    \Line(395,-159)(395,-116)
    \Vertex(395,-116){9}
    \Line(375,-202)(430,-163)
    \Vertex(430,-163){9}
\end{picture}
}}\,}

\def\matulafgg{\,{\scalebox{0.15}{
 \begin{picture}(48,153) (349,-150)
    \SetWidth{2}
    \SetColor{Black}
    \Vertex(375,-147){12}
    \Line(375,-147)(395,-108)
    \Line(375,-147)(354,-108)
    \Vertex(353,-108){9}
    \Vertex(395,-108){9}
    \Vertex(395,-56){9}
    \Line(395,-108)(395,-56)
    \Line(353,-108)(353,-58)
    \Vertex(353,-56){9}
    \Line(353,-56)(333,-8)
    \Vertex(333,-8){9}
    \Vertex(333,35){9}
    \Line(333,-8)(333,35)
    \Line(353,-56)(373,-8)
    \Vertex(373,-8){9}
  \end{picture}
}}\,}

\def\matulafhc{\,{\scalebox{0.15}{
\begin{picture}(48,98) (349,-205)
    \SetWidth{2}
    \SetColor{Black}
    \Vertex(375,-202){12}
    \Line(376,-200)(399,-165)
    \Line(373,-201)(354,-164)
    \Vertex(353,-161){9}
    \Vertex(399,-163){9}
    \Line(375,-201)(375,-160)
    \Vertex(376,-157){9}
    \Vertex(353,-115){9}
    \Line(353,-155)(353,-115)
    \Vertex(353,-70){9}
    \Line(353,-115)(353,-70)
    \Vertex(353,-30){9}
    \Line(353,-70)(353,-30)
    \Vertex(353,5){9}
    \Line(353,-30)(353,5)
  \end{picture}
}}\,}

\def\matulafia{\,{\scalebox{0.15}{
 \begin{picture}(76,153) (365,-150)
    \SetWidth{2}
    \SetColor{Black}
    \Vertex(375,-147){12}
    \Line(375,-147)(395,-108)
    \Line(375,-147)(354,-108)
    \Line(375,-147)(437,-108)
    \Vertex(353,-108){9}
    \Vertex(395,-108){9}
    \Vertex(437,-108){9}
    \Vertex(437,-58){9}
    \Vertex(437,-5){9}
    \Line(437,-58)(437,-5)
    \Line(437,-108)(437,-58)
    \Vertex(395,-56){9}
    \Line(353,-105)(353,-58)
    \Line(395,-105)(395,-58)
    \Vertex(353,-58){9}
    \Line(353,-52)(353,-5)
    \Vertex(353,-5){9}
    \Line(395,-52)(395,-5)
    \Vertex(395,-5){9}
  \end{picture}
}}\,}

\def\matulagoa{\,{\scalebox{0.15}{
\begin{picture}(68,98) (369,-205)
    \SetWidth{2}
    \SetColor{Black}
    \Vertex(375,-202){12}
    \Line(375,-202)(390,-163)
    \Line(375,-202)(343,-163)
    \Vertex(343,-163){9}
    \Vertex(390,-163){9}
    \Line(343,-163)(328,-111)
    \Line(343,-163)(358,-111)
    \Vertex(328,-111){9}
    \Line(390,-159)(390,-111)
    \Vertex(390,-111){9}
    \Line(375,-202)(430,-163)
    \Vertex(430,-163){9}
    \Vertex(470,-163){9}
    \Line(470,-163)(375,-202)
    \Vertex(430,-111){9}
    \Vertex(358,-111){9}
    \Line(430,-163)(430,-111)
\end{picture}
}}\,}

\def\matulagoi{\,{\scalebox{0.15}{
\begin{picture}(8,204) (370,-99)
    \SetWidth{2}
    \SetColor{Black}
    \Line(374,-95)(374,-51)
    \Vertex(374,-48){9}
    \Vertex(375,-96){12}
    \Line(374,-44)(374,0)
    \Vertex(374,3){9}
    \Line(374,6)(374,50)
    \Vertex(374,53){9}
    \Line(374,53)(374,95)
    \Vertex(374,95){9}
     \Line(374,95)(374,130)
    \Vertex(374,130){9}
     \Line(374,160)(374,130)
    \Vertex(374,160){9}
  \end{picture}
}}\,}

\def\matulagai{\,{\scalebox{0.15}{
 \begin{picture}(90,46) (330,-257)
    \SetWidth{2}
    \SetColor{Black}
    \Vertex(375,-254){12}
    \Line(375,-254)(395,-190)
    \Vertex(395,-190){9}
    \Line(375,-254)(334,-199)
    \Vertex(334,-199){9}
    \Line(375,-254)(355,-190)
    \Vertex(355,-190){9}
    \Line(375,-254)(418,-199)
    \Vertex(418,-199){9}
    \Line(375,-254)(375,-187)
    \Vertex(375,-187){9}
    \Vertex(438,-212){9}
    \Vertex(314,-212){9}
    \Line(375,-254)(438,-212)
    \Line(375,-254)(314,-212)
  \end{picture}
}}\,}

\def\matulagbg{\,{\scalebox{0.15}{
 \begin{picture}(63,150) (364,-205)
    \SetWidth{2}
    \SetColor{Black}
    \Line(376,-148)(395,-111)
    \Line(373,-149)(354,-111)
    \Vertex(353,-111){9}
    \Vertex(395,-111){9}
    \Line(353,-111)(370,-64)
    \Line(353,-111)(336,-64)
    \Vertex(370,-64){9}
    \Vertex(336,-64){9}
    \Line(405,-202)(374,-153)
    \Vertex(374,-149){9}
    \Vertex(405,-202){12}
    \Vertex(436,-149){9}
    \Line(405,-202)(436,-153)
    \Vertex(436,-111){9}
    \Line(436,-111)(436,-149)
  \end{picture}
}}\,}

\def\matulagcc{\,{\scalebox{0.15}{
 \begin{picture}(68,150) (369,-205)
    \SetWidth{2}
    \SetColor{Black}
    \Line(376,-148)(395,-113)
    \Line(373,-149)(354,-112)
    \Vertex(353,-111){9}
    \Vertex(395,-111){9}
    \Line(353,-108)(353,-61)
    \Vertex(353,-59){9}
    \Line(400,-202)(374,-153)
    \Vertex(374,-149){9}
    \Vertex(400,-202){12}
    \Vertex(426,-149){9}
    \Vertex(466,-149){9}
    \Line(466,-149)(400,-202)
    \Line(400,-200)(426,-153)
    \Vertex(426,-111){9}
    \Vertex(426,-59){9}
    \Line(426,-111)(426,-59)
    \Line(426,-149)(426,-111)
  \end{picture}
}}\,}

\def\matulagci{\,{\scalebox{0.15}{
 \begin{picture}(90,46) (330,-307)
    \SetWidth{2}
    \SetColor{Black}
    \Vertex(375,-254){9}
    \Vertex(375,-304){12}
    \Line(375,-304)(375,-254)
    \Line(375,-254)(395,-215)
    \Vertex(395,-215){9}
    \Line(375,-254)(334,-224)
    \Vertex(334,-224){9}
    \Line(375,-254)(355,-215)
    \Vertex(355,-215){9}
    \Line(375,-254)(418,-224)
    \Vertex(418,-224){9}
    \Line(375,-254)(375,-212)
    \Vertex(375,-212){9}
  \end{picture}
}}\,}

\def\matulagdc{\,{\scalebox{0.15}{
\begin{picture}(68,98) (369,-205)
    \SetWidth{2}
    \SetColor{Black}
    \Vertex(410,-202){12}
    \Line(410,-202)(395,-163)
    \Line(410,-202)(354,-163)
    \Vertex(353,-163){9}
    \Vertex(395,-163){9}
    \Line(353,-160)(353,-115)
    \Vertex(353,-115){9}
    \Line(395,-159)(395,-115)
    \Vertex(395,-115){9}
    \Line(410,-202)(430,-163)
    \Vertex(430,-163){9}
    \Vertex(465,-163){9}
    \Line(465,-163)(410,-202)
     \Line(353,-115)(353,-70)
    \Vertex(353,-70){9}
    \Line(353,-30)(353,-70)
    \Vertex(353,-30){9}
    \end{picture}
}}\,}

\def\matulagea{\,{\scalebox{0.15}{
\begin{picture}(48,98) (349,-205)
    \SetWidth{2}
    \SetColor{Black}
    \Vertex(375,-202){12}
    \Line(375,-202)(402,-163)
    \Line(375,-202)(328,-163)
    \Vertex(328,-163){9}
    \Vertex(402,-163){9}
    \Line(328,-163)(350,-116)
    \Line(328,-163)(306,-116)
    \Vertex(350,-116){9}
    \Vertex(328,-114){9}
    \Line(328,-114)(328,-163)
    \Vertex(306,-116){9}
    \Line(402,-163)(392,-116)
    \Vertex(392,-116){9}
    \Line(402,-163)(420,-116)
    \Vertex(420,-116){9}
\end{picture}
}}\,}

\def\matulageg{\,{\scalebox{0.15}{
 \begin{picture}(48,153) (349,-150)
    \SetWidth{2}
    \SetColor{Black}
    \Vertex(375,-147){12}
    \Line(375,-147)(395,-108)
    \Line(375,-147)(354,-108)
    \Vertex(353,-108){9}
    \Vertex(395,-108){9}
    \Line(353,-108)(353,-56)
    \Vertex(353,-56){9}
    \Line(353,-56)(331,-10)
    \Vertex(331,-10){9}
    \Line(353,-56)(353,-8)
    \Vertex(353,-8){9}
    \Line(353,-56)(375,-10)
    \Vertex(375,-10){9}
  \end{picture}
}}\,}

\def\matulagfa{\,{\scalebox{0.15}{
\begin{picture}(48,98) (349,-205)
    \SetWidth{2}
    \SetColor{Black}
    \Vertex(375,-202){12}
    \Line(375,-202)(395,-165)
    \Line(375,-202)(354,-164)
    \Vertex(353,-163){9}
    \Vertex(310,-163){9}
    \Vertex(310,-113){9}
    \Line(310,-163)(310,-113)
    \Vertex(310,-66){9}
    \Line(310,-66)(310,-113)
    \Line(375,-202)(310,-165)
    \Vertex(395,-163){9}
    \Line(353,-160)(353,-113)
    \Vertex(353,-113){9}
    \Line(395,-159)(395,-113)
    \Vertex(395,-113){9}
    \Line(375,-202)(430,-164)
    \Vertex(430,-163){9}
    \Line(430,-164)(430,-113)
    \Vertex(430,-113){9}
\end{picture}
}}\,}

\def\matulagfi{\,{\scalebox{0.15}{
\begin{picture}(48,98) (349,-205)
    \SetWidth{2}
    \SetColor{Black}
    \Vertex(375,-202){12}
    \Line(375,-202)(399,-161)
    \Line(375,-202)(354,-164)
    \Vertex(350,-163){9}
    \Vertex(399,-161){9}
    \Line(375,-201)(375,-160)
    \Vertex(376,-161){9}
    \Vertex(425,-163){9}
    \Vertex(350,-110){9}
    \Line(375,-202)(425,-165)
    \Line(350,-161)(350,-110)
    \Line(350,-110)(330,-65)
    \Vertex(330,-65){9}
    \Line(350,-110)(370,-65)
    \Vertex(370,-65){9}
  \end{picture}
}}\,}

\def\matulaggc{\,{\scalebox{0.15}{
 \begin{picture}(48,153) (349,-200)
    \SetWidth{2}
    \SetColor{Black}
    \Vertex(375,-147){9}
    \Vertex(375,-197){12}
    \Line(375,-197)(375,-147)
    \Line(375,-147)(395,-108)
    \Line(375,-147)(354,-108)
    \Vertex(353,-108){9}
    \Vertex(395,-108){9}
    \Vertex(395,-56){9}
    \Line(353,-108)(353,-56)
    \Line(395,-108)(395,-56)
    \Vertex(353,-56){9}
    \Line(353,-56)(353,-5)
    \Vertex(353,-5){9}
    \Line(353,-5)(353,40)
    \Vertex(353,40){9}
  \end{picture}
}}\,}

\def\matulaghg{\,{\scalebox{0.15}{
 \begin{picture}(68,150) (369,-205)
    \SetWidth{2}
    \SetColor{Black}
    \Line(376,-148)(395,-113)
    \Line(373,-149)(354,-112)
    \Vertex(353,-111){9}
    \Vertex(395,-111){9}
    \Vertex(395,-65){9}
    \Line(395,-111)(395,-65)
    \Line(353,-108)(353,-65)
    \Vertex(353,-65){9}
    \Line(400,-202)(374,-153)
    \Vertex(374,-149){9}
    \Vertex(400,-202){12}
    \Vertex(426,-149){9}
    \Vertex(466,-149){9}
    \Line(466,-149)(400,-202)
    \Line(400,-200)(426,-153)
    \Vertex(426,-111){9}
    \Line(426,-149)(426,-111)
  \end{picture}
}}\,}

\def\matulagig{\,{\scalebox{0.15}{
 \begin{picture}(48,153) (349,-200)
    \SetWidth{2}
    \SetColor{Black}
    \Vertex(375,-147){9}
    \Vertex(375,-197){12}
    \Line(375,-197)(375,-147)
    \Line(375,-147)(395,-108)
    \Line(375,-147)(354,-108)
    \Vertex(353,-108){9}
    \Vertex(395,-108){9}
    \Line(353,-108)(353,-56)
    \Vertex(353,-56){9}
    \Line(353,-56)(333,-8)
    \Vertex(333,-8){9}
    \Line(353,-56)(373,-8)
    \Vertex(373,-8){9}
  \end{picture}
}}\,}

\def\matulahoi{\,{\scalebox{0.15}{
\begin{picture}(70,98) (371,-205)
    \SetWidth{2}
    \SetColor{Black}
    \Vertex(395,-202){12}
    \Line(395,-202)(395,-165)
    \Line(395,-202)(348,-165)
    \Vertex(348,-163){9}
    \Vertex(395,-163){9}
    \Line(348,-160)(348,-113)
    \Vertex(348,-113){9}
    \Line(348,-113)(348,-70)
    \Vertex(348,-70){9}
    \Line(395,-159)(380,-113)
    \Vertex(380,-113){9}
    \Line(395,-163)(415,-113)
    \Vertex(415,-113){9}
    \Line(395,-202)(430,-165)
    \Vertex(430,-165){9}
    \Vertex(465,-165){9}
    \Line(465,-165)(395,-202)
\end{picture}
}}\,}

\def\matulahaa{\,{\scalebox{0.15}{
 \begin{picture}(54,150) (355,-205)
    \SetWidth{2}
    \SetColor{Black}
    \Line(374,-149)(395,-111)
    \Line(374,-149)(354,-111)
    \Vertex(353,-111){9}
    \Vertex(395,-111){9}
    \Line(353,-111)(353,-59)
    \Vertex(353,-59){9}
    \Line(353,-14)(353,-59)
    \Vertex(353,-14){9}
    \Line(395,-202)(374,-149)
    \Line(395,-202)(420,-149)
    \Vertex(374,-149){9}
    \Vertex(420,-149){9}
    \Vertex(420,-111){9}
    \Line(420,-149)(415,-111)
    \Vertex(395,-202){12}
    \Line(395,-108)(395,-61)
    \Vertex(395,-59){9}
  \end{picture}
}}\,}

\def\matulahba{\,{\scalebox{0.15}{
  \begin{picture}(68,88) (329,-215)
    \SetWidth{2}
    \SetColor{Black}
    \Vertex(375,-212){12}
    \Line(376,-210)(395,-175)
    \Line(373,-211)(354,-174)
    \Vertex(353,-171){9}
    \Vertex(395,-171){9}
    \Line(353,-171)(330,-131)
    \Line(353,-171)(376,-131)
    \Line(353,-171)(353,-129)
    \Vertex(353,-129){9}
    \Vertex(330,-131){9}
    \Vertex(376,-131){9}
    \Vertex(330,-85){9}
    \Line(330,-131)(330,-85)
    \Vertex(330,-45){9}
    \Line(330,-45)(330,-85)
  \end{picture}
}}\,}

\def\matulahbc{\,{\scalebox{0.15}{
 \begin{picture}(58,150) (359,-205)
    \SetWidth{2}
    \SetColor{Black}
    \Line(376,-148)(395,-113)
    \Line(373,-149)(354,-112)
    \Vertex(353,-111){9}
    \Vertex(395,-111){9}
    \Line(353,-108)(353,-65)
    \Vertex(353,-65){9}
    \Line(400,-202)(374,-153)
    \Vertex(374,-149){9}
    \Vertex(400,-202){12}
    \Vertex(426,-149){9}
    \Line(400,-200)(426,-153)
    \Vertex(426,-111){9}
    \Vertex(426,-65){9}
    \Line(426,-111)(426,-65)
    \Vertex(426,-19){9}
    \Line(426,-19)(426,-65)
    \Line(426,-149)(426,-111)
  \end{picture}
}}\,}

\def\matulahbg{\,{\scalebox{0.15}{
 \begin{picture}(95,46) (335,-257)
    \SetWidth{2}
    \SetColor{Black}
    \Vertex(385,-254){12}
    \Line(385,-254)(395,-202)
    \Vertex(395,-202){9}
    \Line(385,-254)(335,-215)
    \Vertex(335,-215){9}
    \Vertex(335,-160){9}
    \Line(335,-215)(335,-160)
    \Vertex(356,-160){9}
    \Line(356,-215)(356,-160)
    \Line(385,-254)(356,-206)
    \Vertex(356,-206){9}
    \Line(385,-254)(417,-206)
    \Vertex(417,-206){9}
    \Line(385,-254)(375,-202)
    \Vertex(375,-202){9}
     \Line(385,-254)(439,-215)
    \Vertex(439,-215){9}
  \end{picture}
}}\,}

\def\matulahbi{\,{\scalebox{0.15}{
 \begin{picture}(58,150) (359,-205)
    \SetWidth{2}
    \SetColor{Black}
    \Line(376,-148)(395,-113)
    \Line(373,-149)(354,-112)
    \Vertex(353,-111){9}
    \Vertex(395,-111){9}
    \Line(353,-108)(353,-65)
    \Vertex(353,-65){9}
    \Line(400,-202)(374,-153)
    \Vertex(374,-149){9}
    \Vertex(400,-202){12}
    \Vertex(426,-149){9}
    \Line(400,-200)(426,-153)
    \Vertex(426,-111){9}
    \Vertex(426,-65){9}
    \Line(426,-111)(426,-65)
    \Vertex(353,-19){9}
    \Line(353,-19)(353,-65)
    \Line(426,-149)(426,-111)
  \end{picture}
}}\,}

\def\matulahci{\,{\scalebox{0.15}{
 \begin{picture}(58,150) (359,-205)
    \SetWidth{2}
    \SetColor{Black}
    \Line(376,-148)(395,-111)
    \Line(373,-149)(354,-111)
    \Vertex(353,-111){9}
    \Vertex(395,-111){9}
    \Line(353,-111)(370,-64)
    \Line(353,-111)(336,-64)
    \Vertex(370,-64){9}
    \Vertex(336,-64){9}
    \Line(395,-202)(374,-149)
    \Vertex(374,-149){9}
    \Vertex(395,-64){9}
    \Line(395,-111)(395,-64)
    \Vertex(395,-202){12}
    \Vertex(416,-149){9}
    \Line(395,-200)(416,-153)
  \end{picture}
}}\,}

\def\matulahec{\,{\scalebox{0.15}{
\begin{picture}(61,98) (361,-205)
    \SetWidth{2}
    \SetColor{Black}
    \Vertex(375,-202){12}
    \Line(375,-202)(402,-163)
    \Line(375,-202)(348,-163)
    \Vertex(348,-163){9}
    \Vertex(402,-163){9}
    \Vertex(448,-163){9}
    \Vertex(448,-116){9}
    \Line(448,-163)(448,-116)
    \Line(448,-163)(375,-202)
    \Line(348,-163)(360,-116)
    \Line(348,-163)(332,-116)
    \Vertex(360,-116){9}
    \Vertex(332,-116){9}
    \Line(402,-163)(392,-116)
    \Vertex(392,-116){9}
    \Line(402,-163)(420,-116)
    \Vertex(420,-116){9}
\end{picture}
}}\,}

\def\matulaheg{\,{\scalebox{0.15}{
  \begin{picture}(78,88) (339,-215)
    \SetWidth{2}
    \SetColor{Black}
    \Vertex(375,-212){12}
    \Line(375,-212)(385,-171)
    \Line(375,-212)(353,-171)
    \Vertex(353,-171){9}
    \Vertex(385,-171){9}
    \Vertex(417,-171){9}
    \Line(375,-212)(417,-171)
    \Line(353,-171)(330,-131)
    \Line(353,-171)(376,-131)
    \Line(353,-171)(353,-129)
    \Vertex(353,-129){9}
    \Vertex(330,-131){9}
    \Vertex(330,-86){9}
    \Line(330,-131)(330,-86)
    \Vertex(376,-131){9}
  \end{picture}
}}\,}

\def\matulahei{\,{\scalebox{0.15}{
\begin{picture}(48,98) (349,-255)
    \SetWidth{2}
    \SetColor{Black}
    \Vertex(375,-202){9}
    \Vertex(375,-252){12}
    \Line(375,-252)(375,-202)
    \Line(375,-202)(395,-163)
    \Line(375,-202)(348,-163)
    \Vertex(348,-163){9}
    \Vertex(395,-163){9}
    \Line(348,-163)(348,-111)
    \Vertex(348,-111){9}
    \Line(348,-111)(348,-65)
    \Vertex(348,-65){9}
    \Line(395,-163)(380,-111)
    \Vertex(380,-111){9}
    \Line(395,-163)(415,-111)
    \Vertex(415,-111){9}
\end{picture}
}}\,}

\def\matulahfc{\,{\scalebox{0.15}{
\begin{picture}(68,98) (369,-205)
    \SetWidth{2}
    \SetColor{Black}
    \Vertex(410,-202){12}
    \Line(410,-202)(395,-165)
    \Line(410,-202)(354,-164)
    \Vertex(353,-163){9}
    \Vertex(395,-163){9}
    \Line(353,-160)(353,-115)
    \Vertex(353,-115){9}
    \Line(395,-159)(395,-115)
    \Vertex(395,-115){9}
    \Line(395,-115)(395,-70)
    \Vertex(395,-70){9}
    \Line(410,-202)(430,-164)
    \Vertex(430,-163){9}
     \Vertex(430,-115){9}
     \Line(430,-163)(430,-115)
    \Vertex(465,-163){9}
    \Line(465,-163)(410,-202)
     \Line(353,-115)(353,-70)
    \Vertex(353,-70){9}
    \end{picture}
}}\,}

\def\matulahgg{\,{\scalebox{0.15}{
\begin{picture}(58,98) (359,-255)
    \SetWidth{2}
    \SetColor{Black}
    \Vertex(375,-202){9}
    \Vertex(375,-252){12}
    \Line(375,-252)(375,-202)
    \Line(375,-202)(390,-163)
    \Line(375,-202)(349,-163)
    \Vertex(349,-163){9}
    \Vertex(390,-163){9}
    \Line(349,-163)(349,-111)
    \Vertex(349,-111){9}
    \Line(390,-163)(390,-111)
    \Vertex(390,-111){9}
    \Line(375,-202)(425,-163)
    \Vertex(425,-163){9}
    \Line(375,-202)(460,-163)
    \Vertex(460,-163){9}
\end{picture}
}}\,}

\def\matulahha{\,{\scalebox{0.15}{
\begin{picture}(48,98) (349,-205)
    \SetWidth{2}
    \SetColor{Black}
    \Vertex(375,-202){12}
    \Line(375,-202)(399,-161)
    \Line(375,-202)(350,-163)
    \Vertex(350,-163){9}
    \Vertex(399,-161){9}
    \Line(375,-202)(375,-160)
    \Vertex(375,-161){9}
    \Vertex(425,-163){9}
    \Vertex(372,-112){9}
    \Vertex(328,-112){9}
    \Vertex(350,-110){9}
    \Line(375,-202)(425,-165)
    \Line(350,-161)(328,-112)
    \Line(350,-161)(372,-112)
    \Line(350,-161)(350,-110)
  \end{picture}
}}\,}

\def\matulahhc{\,{\scalebox{0.15}{
 \begin{picture}(48,153) (349,-150)
    \SetWidth{2}
    \SetColor{Black}
    \Vertex(375,-147){12}
    \Vertex(430,-108){9}
    \Vertex(430,-58){9}
    \Line(430,-108)(430,-58)
    \Line(375,-147)(430,-108)
    \Line(376,-145)(395,-108)
    \Line(373,-146)(354,-108)
    \Vertex(353,-108){9}
    \Vertex(395,-108){9}
    \Vertex(395,-58){9}
    \Line(395,-108)(395,-58)
    \Line(353,-108)(353,-58)
    \Vertex(353,-56){9}
    \Line(353,-56)(333,-8)
    \Vertex(333,-8){9}
    \Line(353,-56)(373,-8)
    \Vertex(373,-8){9}
  \end{picture}
}}\,}

\def\matulahhg{\,{\scalebox{0.15}{
\begin{picture}(58,98) (359,-205)
    \SetWidth{2}
    \SetColor{Black}
    \Vertex(375,-202){12}
    \Line(376,-200)(395,-165)
    \Line(373,-201)(348,-165)
    \Vertex(348,-163){9}
    \Vertex(395,-163){9}
    \Line(348,-160)(348,-113)
    \Vertex(348,-113){9}
    \Line(348,-113)(348,-70)
    \Vertex(348,-30){9}
    \Line(348,-30)(348,-70)
    \Vertex(348,-70){9}
    \Line(395,-159)(380,-113)
    \Vertex(380,-113){9}
    \Line(395,-163)(415,-113)
    \Vertex(415,-113){9}
    \Line(375,-202)(430,-165)
    \Vertex(430,-165){9}
\end{picture}
}}\,}

\def\matulaiog{\,{\scalebox{0.15}{
 \begin{picture}(48,190) (349,-150)
    \SetWidth{2}
    \SetColor{Black}
    \Vertex(375,-147){12}
    \Line(376,-145)(395,-108)
    \Line(373,-146)(354,-108)
    \Vertex(353,-108){9}
    \Vertex(395,-108){9}
    \Vertex(395,-56){9}
    \Vertex(395,-5){9}
    \Line(395,-56)(395,-5)
    \Line(353,-105)(353,-58)
    \Line(395,-105)(395,-58)
    \Vertex(353,-56){9}
    \Line(353,-52)(353,-5)
    \Vertex(353,-5){9}
    \Line(353,-5)(353,40)
    \Vertex(353,40){9}
    \Line(353,40)(353,80)
    \Vertex(353,80){9}
  \end{picture}
}}\,}

\def\matulaiaa{\,{\scalebox{0.15}{
\begin{picture}(62,98) (363,-205)
    \SetWidth{2}
    \SetColor{Black}
    \Vertex(375,-202){12}
    \Line(375,-202)(390,-165)
    \Line(375,-202)(349,-164)
    \Vertex(349,-163){9}
    \Vertex(390,-163){9}
    \Line(349,-163)(332,-111)
    \Line(349,-163)(366,-111)
    \Vertex(332,-111){9}
    \Vertex(332,-61){9}
    \Line(332,-61)(332,-111)
    \Vertex(366,-111){9}
    \Line(390,-159)(390,-112)
    \Vertex(390,-111){9}
    \Line(375,-202)(425,-164)
    \Vertex(425,-163){9}
    \Line(375,-202)(460,-164)
    \Vertex(460,-163){9}
\end{picture}
}}\,}

\def\matulaiai{\,{\scalebox{0.15}{
 \begin{picture}(48,99) (349,-301)
    \SetWidth{2}
    \SetColor{Black}
    \Line(376,-199)(399,-160)
    \Line(373,-200)(353,-160)
    \Vertex(353,-160){9}
    \Vertex(399,-160){9}
    \Vertex(375,-156){9}
    \Vertex(375,-248){9}
    \Vertex(375,-298){12}
    \Line(375,-298)(375,-248)
    \Line(375,-248)(375,-201)
    \Line(375,-201)(375,-156)
    \Line(353,-160)(353,-115)
    \Vertex(353,-115){9}
    \Vertex(375,-201){9}
  \end{picture}
}}\,}

\def\matulaibi{\,{\scalebox{0.15}{
 \begin{picture}(58,150) (359,-205)
    \SetWidth{2}
    \SetColor{Black}
    \Line(376,-148)(395,-111)
    \Line(373,-149)(354,-111)
    \Vertex(353,-111){9}
    \Vertex(395,-111){9}
    \Line(353,-111)(353,-59)
    \Vertex(353,-59){9}
    \Line(353,-15)(353,25)
    \Vertex(353,25){9}
    \Line(353,-59)(353,-15)
    \Vertex(353,-15){9}
    \Line(395,-202)(374,-149)
    \Vertex(374,-149){9}
    \Vertex(395,-202){12}
    \Vertex(416,-149){9}
    \Line(395,-202)(416,-149)
  \end{picture}
}}\,}

\def\matulaicg{\,{\scalebox{0.15}{
 \begin{picture}(100,46) (340,-300)
    \SetWidth{2}
    \SetColor{Black}
    \Vertex(410,-300){12}
    \Line(410,-300)(375,-254)
    \Vertex(375,-254){9}
    \Vertex(445,-254){9}
    \Vertex(445,-215){9}
    \Line(445,-254)(445,-215)
    \Line(410,-300)(445,-254)
    \Line(375,-252)(395,-217)
    \Vertex(395,-215){9}
    \Line(375,-254)(335,-226)
    \Vertex(334,-224){9}
    \Line(375,-252)(356,-215)
    \Vertex(355,-215){9}
    \Line(375,-255)(417,-227)
    \Vertex(418,-225){9}
  \end{picture}
}}\,}

\def\matulaida{\,{\scalebox{0.15}{
 \begin{picture}(95,46) (335,-257)
    \SetWidth{2}
    \SetColor{Black}
    \Vertex(385,-254){12}
    \Line(385,-254)(395,-202)
    \Vertex(395,-202){9}
    \Line(385,-254)(335,-215)
    \Vertex(335,-215){9}
    \Vertex(335,-165){9}
    \Line(335,-215)(335,-165)
    \Vertex(335,-120){9}
    \Line(335,-120)(335,-165)
    \Line(385,-254)(356,-206)
    \Vertex(356,-206){9}
    \Line(385,-254)(417,-206)
    \Vertex(417,-206){9}
    \Line(385,-254)(375,-202)
    \Vertex(375,-202){9}
     \Line(385,-254)(439,-215)
    \Vertex(439,-215){9}
  \end{picture}
}}\,}

\def\matulaidg{\,{\scalebox{0.15}{
\begin{picture}(48,98) (349,-205)
    \SetWidth{2}
    \SetColor{Black}
    \Vertex(375,-202){12}
    \Line(375,-202)(402,-163)
    \Line(375,-202)(348,-163)
    \Vertex(348,-163){9}
    \Vertex(402,-163){9}
    \Line(348,-163)(360,-116)
    \Line(348,-163)(332,-116)
    \Vertex(360,-116){9}
    \Vertex(360,-71){9}
    \Line(360,-116)(360,-71)
    \Vertex(332,-116){9}
    \Vertex(332,-71){9}
    \Line(332,-116)(332,-71)
    \Line(402,-163)(392,-116)
    \Vertex(392,-116){9}
    \Line(402,-163)(420,-116)
    \Vertex(420,-116){9}
\end{picture}
}}\,}

\def\matulaiec{\,{\scalebox{0.15}{
\begin{picture}(48,98) (349,-205)
    \SetWidth{2}
    \SetColor{Black}
    \Vertex(375,-202){12}
    \Line(375,-202)(399,-159)
    \Line(375,-202)(350,-159)
    \Vertex(350,-159){9}
    \Line(375,-202)(325,-163)
    \Vertex(325,-163){9}
    \Vertex(325,-110){9}
    \Line(325,-163)(325,-110)
    \Vertex(399,-159){9}
    \Vertex(399,-110){9}
    \Line(399,-159)(399,-110)
    \Line(375,-202)(375,-158)
    \Vertex(375,-158){9}
    \Vertex(375,-110){9}
    \Line(375,-158)(375,-110)
    \Vertex(425,-163){9}
    \Vertex(350,-110){9}
    \Line(375,-202)(425,-165)
    \Line(350,-159)(350,-110)
  \end{picture}
}}\,}

\def\matulaifg{\,{\scalebox{0.15}{
  \begin{picture}(68,88) (329,-265)
    \SetWidth{2}
    \SetColor{Black}
    \Vertex(375,-212){9}
    \Vertex(375,-262){12}
    \Line(375,-262)(375,-212)
    \Line(375,-212)(395,-175)
    \Line(375,-212)(354,-174)
    \Vertex(353,-171){9}
    \Vertex(395,-171){9}
    \Line(353,-171)(330,-131)
    \Line(353,-171)(376,-131)
    \Line(353,-171)(353,-129)
    \Vertex(353,-129){9}
    \Vertex(330,-131){9}
    \Vertex(376,-131){9}
  \end{picture}
}}\,}

\def\matulaiga{\,{\scalebox{0.15}{
 \begin{picture}(55,153) (356,-150)
    \SetWidth{2}
    \SetColor{Black}
    \Vertex(375,-147){12}
    \Line(375,-147)(395,-108)
    \Line(375,-147)(354,-108)
    \Vertex(353,-108){9}
    \Vertex(395,-108){9}
    \Vertex(430,-108){9}
    \Line(375,-147)(430,-108)
    \Line(353,-108)(353,-58)
    \Vertex(353,-56){9}
    \Line(353,-56)(333,-8)
    \Vertex(333,-8){9}
    \Vertex(333,35){9}
    \Line(333,-8)(333,35)
    \Line(353,-56)(373,-8)
    \Vertex(373,-8){9}
  \end{picture}
}}\,}

\def\matulaigg{\,{\scalebox{0.15}{
 \begin{picture}(58,153) (359,-150)
    \SetWidth{2}
    \SetColor{Black}
    \Vertex(375,-147){12}
    \Line(375,-147)(395,-108)
    \Line(375,-147)(354,-108)
    \Line(375,-147)(437,-108)
    \Vertex(353,-108){9}
    \Vertex(395,-108){9}
    \Vertex(437,-108){9}
    \Vertex(437,-56){9}
    \Line(437,-108)(437,-56)
    \Vertex(395,-56){9}
    \Line(353,-105)(353,-58)
    \Line(395,-105)(395,-58)
    \Vertex(353,-58){9}
    \Line(353,-52)(353,-5)
    \Vertex(353,-5){9}
    \Line(353,-5)(353,40)
    \Vertex(353,40){9}
    \Line(395,-52)(395,-5)
    \Vertex(395,-5){9}
  \end{picture}
}}\,}

\def\matulaihc{\,{\scalebox{0.15}{
 \begin{picture}(48,153) (349,-150)
    \SetWidth{2}
    \SetColor{Black}
    \Vertex(375,-147){12}
    \Line(375,-147)(395,-108)
    \Line(375,-147)(354,-108)
    \Vertex(353,-108){9}
    \Vertex(395,-108){9}
    \Line(353,-108)(353,-58)
    \Vertex(353,-56){9}
    \Line(353,-56)(333,-8)
    \Vertex(333,-8){9}
    \Vertex(333,35){9}
    \Line(333,-8)(333,35)
    \Line(353,-56)(373,-8)
    \Vertex(373,-8){9}
     \Vertex(373,35){9}
     \Line(373,-8)(373,35)
  \end{picture}
}}\,}

\def\matulaiia{\,{\scalebox{0.15}{
 \begin{picture}(63,150) (364,-255)
    \SetWidth{2}
    \SetColor{Black}
    \Line(374,-149)(395,-111)
    \Line(374,-149)(354,-111)
    \Vertex(353,-111){9}
    \Vertex(395,-111){9}
    \Line(353,-108)(353,-65)
    \Vertex(353,-65){9}
    \Line(400,-202)(374,-153)
    \Vertex(374,-149){9}
    \Vertex(400,-202){9}
    \Vertex(400,-252){12}
    \Line(400,-252)(400,-202)
    \Vertex(426,-149){9}
    \Line(400,-200)(426,-153)
    \Vertex(426,-111){9}
    \Line(426,-149)(426,-111)
  \end{picture}
}}\,}

\def\matulaiig{\,{\scalebox{0.15}{
\begin{picture}(78,98) (379,-216)
    \SetWidth{2}
    \SetColor{Black}
    \Vertex(415,-213){12}
    \Line(415,-213)(390,-165)
    \Line(415,-213)(349,-164)
    \Vertex(349,-163){9}
    \Vertex(390,-163){9}
    \Line(349,-163)(332,-111)
    \Line(349,-163)(366,-111)
    \Vertex(332,-111){9}
    \Vertex(366,-111){9}
    \Line(390,-159)(390,-112)
    \Vertex(390,-111){9}
    \Line(415,-213)(425,-164)
    \Vertex(425,-163){9}
    \Line(415,-213)(460,-164)
    \Vertex(460,-163){9}
    \Vertex(495,-163){9}
    \Line(415,-213)(495,-163)
\end{picture}
}}\,}

\def\couplages{\,{\scalebox{0.6}{
 \begin{picture}(58,253) (359,-50)
    \SetWidth{1}
    \SetColor{Black}
    \GOval(181.98,51.26)(11.96,11.96)(0){0.882}
    \GOval(405,107)(11.96,11.96)(0){0.882}
    \GOval(702,108)(11.96,11.96)(0){0.882}
    \Vertex(52.97,146.95){6.67}
    \GBox(43.57,82.02)(60.66,98.25){0.882}
    \Vertex(52.12,88.85){4.36}
    \Gluon(52.97,146.1)(52.97,90.56){3.41}{8.21}
    \Vertex(574.48,146.95){6.67}
    \Vertex(675.39,147.81){6.67}
    \Vertex(161.47,146.1){6.67}
    \Vertex(181.98,102.52){4.36}
    \Line(181.98,101.67)(181.98,54.68)
    \Vertex(181.98,50.41){4.36}
    \Gluon(162.33,146.1)(181.98,103.38){3.41}{6.5}
    \Vertex(267.42,146.95){6.67}
    \Vertex(267.42,48.7){4.36}
    \GBox(299.03,74.33)(313.55,116.19){0.882}
    \Gluon(268.27,146.95)(268.27,49.55){3.41}{13.71}
    \Line(268.27,47.84)(305.01,106.8)
    \Vertex(306.72,106.8){4.36}
    \Vertex(370.79,146.1){6.67}
    \Vertex(369.94,47.84){4.36}
    \Line(358.83,42.72)(382.76,42.72)
    \Vertex(405.82,106.8){4.36}
    \Gluon(371.65,144.39)(370.79,48.7){3.41}{13.57}
    \Line(370.79,46.99)(405.82,107.65)
    \Vertex(467.34,146.95){6.67}
    \Vertex(469.05,-5.13){4.36}
    \Vertex(497.24,81.16){4.36}
    \GOval(498.09,30.76)(11.96,11.96)(0){0.882}
    \Vertex(499.8,29.9){4.36}
    \Gluon(468.19,146.95)(468.19,-3.42){6.41}{10.64}
    \Line(469.9,-5.13)(499.8,30.76)
    \Vertex(511.76,118.76){4.36}
    \Line(512.62,118.76)(498.09,82.02)
    \DashLine(468.19,146.95)(497.24,82.02){1.71}
    \DashLine(497.24,81.16)(498.95,31.61){1.71}
    \Vertex(598.4,119.61){4.36}
    \Vertex(583.88,94.83){4.36}
    \GOval(565.94,51.26)(11.96,11.96)(0){0.882}
    \Vertex(567.65,50.41){4.36}
    \Vertex(557.4,94.83){4.36}
    \Line(583.88,94.83)(598.4,119.61)
    \Line(568.5,49.55)(558.25,94.83)
    \Gluon(575.34,146.1)(558.25,94.83){6.41}{3.09}
    \Vertex(677.96,69.2){4.36}
    \Vertex(701.88,108.5){4.36}
    \GBox(669.41,14.52)(686.5,30.76){0.882}
    \Vertex(678.81,21.36){4.36}
    \Line(677.1,70.06)(700.17,107.65)
    \Gluon(676.25,147.81)(677.96,69.2){3.41}{10}
    \DashLine(678.81,68.35)(678.81,20.5){1.71}
    \DashLine(577.9,147.81)(583.88,95.69){1.71}
    \DashLine(583.88,95.69)(567.65,50.41){1.71}
    \Text(49,165)[lb]{\Large{\Black{$1$}}}
    \Text(158,165)[lb]{\Large{\Black{$2$}}}
    \Text(265,165)[lb]{\Large{\Black{$3$}}}
    \Text(367,165)[lb]{\Large{\Black{$4$}}}
    \Text(464,165)[lb]{\Large{\Black{$5$}}}
    \Text(571,165)[lb]{\Large{\Black{$6$}}}
    \Text(673,165)[lb]{\Large{\Black{$7$}}}  \end{picture}
}}\,}

\begin{document}

\title[arbres enracin\'es et nombres]{Une remarque sur l'arborification de Matula}

\author{Dominique Manchon}
\address{Laboratoire de Math\'ematiques Blaise Pascal, UMR6620,
	CNRS - Universit\'e Clermont Auvergne,
         3 place Vasar\'ely, CS 60026
         63178 Aubi\`ere, France}       
         \email{Dominique.Manchon@uca.fr}
         \urladdr{https://dominiquemanchon.go.yj.fr/}

\date{3 f\'evrier 2026}
\begin{abstract}
Nous esquissons une application de l'arborification de Matula \`a l'\'etude de la fonction sommatoire des fonctions de M\" obius et de Liouville sur les entiers naturels.
\end{abstract}

\begin{altabstract}
We sketch an application of Matula's arborification to the study of the partial sums of both M\" obius and Liouville function.
\end{altabstract}

\maketitle


\tableofcontents

\section{Introduction}
\label{sect:intro}
Deux ans avant la parution du c\'el\`ebre m\'emoire de B. Riemann \cite{R1859} dans lequel il formule pour la premi\`ere fois l'hypoth\`ese qui porte son nom (les z\'eros non triviaux de la fonction z\^eta ont tous, conjecturalement, une partie r\'eelle \'egale \`a $1/2$), A. Cayley introduisait les arbres enracin\'es dans le but d'\'etudier les champs de vecteurs sur un espace affine \cite{Cay}. La correspondance qu'il introduit est maintenant parfaitement comprise en termes de structures pr\'e-Lie sur l'espace des champs de vecteurs ainsi que sur l'espace vectoriel engendr\'e par les arbres \cite{ChaLiv, DL}, et a trouv\'e des applications inattendues en analyse num\'erique \cite{B72}.\\

Un arbre enracin\'e est un graphe orient\'e connexe, que nous supposerons \`a nombre fini de sommets, tel que chaque sommet admet une et une seule ar\^ete entrante, sauf un sommet particulier, la \textsl{racine}, qui n'admet aucune ar\^ete entrante. Nous dessinerons les arbres enracin\'es avec la racine en bas, les ar\^etes \'etant orient\'ees de bas en haut. Une for\^et est une collection finie d'arbres enracin\'es, avec r\'ep\'etitions possibles.\\

A notre connaissance, le premier \`a avoir introduit les arbres enracin\'es en th\'eorie des nombres est D. W. Ma\-tula en 1968 \cite{M, MC19}, qui a construit un isomorphisme $\Cal A$ de mono\"\i des commutatifs entre l'ensemble $\N^*=\{1,2,3,\ldots,n,\ldots\}$ des entiers naturels non nuls et l'ensemble $\Cal F$ des for\^ets d'arbres enracin\'es, chaque facteur premier correspondant \`a un arbre constituant la for\^et\footnote{Dans sa formulation originale, la correspondance de Matula associe un arbre \`a chaque nombre. Nous avons choisi d'associer plut\^ot la for\^et obtenue en rassemblant toutes les branches constituant l'arbre par greffe sur la racine.}. Ecrit de mani\`ere tr\`es condens\'ee (une quinzaine de lignes tout au plus), l'article de Matula semble \^etre pass\'e inaper\c cu dans un premier temps. Sa correspondance a \'et\'e red\'ecouverte notamment par D. F. Goebel \cite G, P. Cappello  \cite{C88}, ou encore J. Sousselier, ce dernier proposant un expos\'e assez d\'etaill\'e \cite{S99}. On peut aussi citer L. Alexandrov \cite{A98}, et plus r\'ecemment R. G. Batchko \cite{B14}.  Ces auteurs explorent des notions voisines, bien que les arbres ne figurent pas explicitement dans leurs travaux.  Cette bijection a toutefois \'et\'e popularis\'ee par S. B. Elk, I. Gutman et A. Ivi\'c en 1993, dans un article o\`u ces auteurs l'utilisent pour coder les mol\'ecules des alcanes $C_nH_{2n+2}$ en chimie organique \cite{EGI}. L'expression "nombres de Matula" ("Matula numbers") appara\^\i t \`a cette occasion. Notons que de nombreuses caract\'eristiques d'un arbre enracin\'e donn\'e peuvent se traduire sur le nombre de Matula correspondant \cite{D, GN}. Un lien avec la m\'ecanique quantique a \'et\'e sugg\'er\'e dans \cite{N}.\\

\noindent
L'isomorphisme de Matula est d\'efini r\'ecursivement par $\Cal A(1)=\emptyset$ (la for\^et vide) et :
\begin{equation}
\Cal A(p_n)=B_+\big(\Cal A(n)\big),
\end{equation}
o\`u $p_n$ d\'esigne le $n$-i\`eme nombre premier, et o\`u $B_+$ est la bijection canonique des for\^ets vers les arbres qui consiste \`a rajouter une racine commune \`a tous les arbres de la for\^et. Par convention $B_+$ envoie la for\^et vide sur l'arbre \`a un seul sommet $\racine$. Ci-dessous la liste des entiers de $1$ \`a $20$ et les for\^ets correspondantes :
\begin{equation*}
\begin{disarray}{c|cccccccccccccccccccc}
n&1&2&3&4&5&6&7&8
&9&10&11&12&13&14&15&16&17&18&19&20\\
&&&&&&&&&&&&&&&&&&&&\\
\Cal A(n)&\o &\matulab &\matulac &\matulab\matulab&\matulae &\matulac\matulab &\matulag&\matulab\matulab\matulab
&\matulac\matulac &\matulae\matulab &\matulaaa &\matulac\matulab\matulab &\matulaac &\matulag\matulab &\matulae\matulac &\matulab\matulab\matulab\matulab&\matulaag&\matulac\matulac\matulab&\matulaai&\matulae\matulab\matulab\\ 
\end{disarray}
\end{equation*}
\vskip 3mm\noindent
Ci-dessous la liste des arbres enracin\'es \`a $1$, $2$, $3$, $4$ et $5$ sommets, et les nombres premiers qui leur correspondent~:
\begin{equation*}
\begin{disarray}{c||c|c|cc|cccc|ccccccccc}
t&\matulab &\matulac &\matulae &\matulag&\matulaaa&\matulaag&\matulaac&\matulaai &\matulaca
&\matulaei&\matulada&\matulafg&\matulabi&\matulabc&\matuladc &\matulacg &\matulaec\\ 
&&&&&&&&&&&&&&&&&\\
\Cal A^{-1}(t)&2&3&5&7&11&17&13&19
&31&59&41&67&29&23&43&37&53\\
\end{disarray}
\end{equation*}

Il existe des applications en ligne qui fournissent la for\^et associ\'ee \`a un nombre entier, par exemple \cite{Ch} pour les nombres de 1 \`a un milliard.
En appliquant des coupes \'el\'ementaires sur les arbres, on peut donc associer \`a un nombre premier un produit de deux facteurs premiers, et ce de plusieurs mani\`eres possibles suivant le choix de la coupe. Inversement on peut fabriquer un nombre premier \`a partir de deux en greffant le premier arbre sur un sommet du deuxi\`eme, par exemple sa racine. C'est la d\'efinition du \textsl{produit de Butcher} $(s,t)\mapsto s\butcher t$ de deux arbres \cite{B72}. Par exemple~:
\begin{equation*}
\matulac\butcher\matulac=\matulaac,\hskip 15mm \matulag\butcher\matulac=\matuladc
\end{equation*}
L'ensemble des arbres enracin\'es muni du produit de Butcher pr\'esente une forme simplifi\'ee de structure pr\'e-Lie que l'on appelle NAP (pour "non-associative permutative") \cite{L06}. C'est tr\`es pr\'ecis\'ement le magma NAP libre \`a un g\'en\'erateur \cite{DL}, ce qui permet d'\'ecrire tout arbre comme produit de Butcher it\'er\'e de l'arbre $\racine$, et ce de mani\`ere essentiellement unique\footnote{De mani\`ere analogue, le $K$-espace vectoriel engendr\'e par les arbres (o\`u $K$ est un corps) est l'alg\`ebre pr\'e-Lie libre \`a un g\'en\'erateur \cite{ChaLiv, DL}. La structure pr\'e-Lie est donn\'ee par la somme sur toute les mani\`eres de greffer le premier arbre sur un sommet du deuxi\`eme, au lieu de se limiter \`a une greffe sur la racine. On remarque que la structure NAP est ensembliste alors que la structure pr\'e-Lie ne l'est pas, puisqu'elle requiert des combinaisons lin\'eaires formelles d'arbres.}.\\

On transportera sur $\Cal F$ l'ordre total des entiers naturels\footnote{Le bon ordre des entiers naturels fournit le plus petit ordinal infini d\'enombrable $\omega$. Il est tr\`es diff\'erent du bon ordre naturel sur l'ensemble des for\^ets qui l'identifie au plus petit ordinal $\varepsilon_0$ v\'erifiant $\varepsilon_0=\omega^{\varepsilon_0}$, autrement dit $\varepsilon_0=\omega^{\omega^{\omega^{\omega^{...}}}}$. Ce bon ordre sur les for\^ets provient de la forme normale de Cantor \cite{C, Cf} pour chaque ordinal $\lambda<\varepsilon_0$. Voir par exemple \cite{W, W20}.} via l'arborification $\Cal A$. Un premier r\'esultat concernant l'application du produit de Butcher \`a la th\'eorie des nombres est l'in\'egalit\'e suivante~:
\begin{prop}\label{butcher}
Pour tout couple d'arbres $(s,t)$ except\'e $(\matulab,\matulab)$ et $(\matulac,\matulab)$ on a~:
\begin{equation}
s\butcher t > st.
\end{equation}
\end{prop}
\noindent Cette proposition d\'ecoule d'un r\'esultat plus g\'en\'eral, \`a savoir~:
\begin{prop}[Gutman-Ivi\' c \cite{GI}]\label{pan-apn}
$p_{an}>ap_n$ pour tout $a,n\in\N^*,a\ge 2$ except\'e pour $n=1$ et $a=2,3,4$.
\end{prop}
La proposition \ref{butcher} en d\'ecoule imm\'ediatement en se limitant aux $a$ premiers et en appliquant l'arborification. La proposition \ref{pan-apn} est obtenue comme un corollaire des encadrements fins de Massias, Robin et Dusart \cite{MR96, D99} (am\'eliorant les r\'esultats ant\'erieurs de Rosser et Schoenfeld \cite{RS62}) concernant la fonction $n\mapsto p_n$. Nous en redonnons une preuve pour la commodit\'e du lecteur.\\

L'ensemble des arbres enracin\'es est par ailleurs muni d'un produit associatif et commutatif, le \textsl{produit de fusion}, obtenu en fusionnant les racines des deux arbres en une seule. Par exemple~:
\begin{equation}
 \matulae\vee\matulag=\matulacg.
\end{equation}
L'encadrement de Massias, Robin et Dusart permet \'egalement  d'obtenir une s\'erie d'in\'egalit\'es relatives au produit de fusion~:
\begin{prop}\label{fusion}
On a l'in\'egalit\'e $s\vee t< st$ sauf pour $\{s,t\}=\{\matulae,\matulag\}$ ou $\{s,t\}=\{\matulag,\matulag\}$. Autrement dit $p_{mn}< p_mp_n$ sauf si $\{m,n\}=\{3,4\}$ ou  $\{m,n\}=\{4,4\}$.
\end{prop}

La fonction de M\"obius $\mu$ est donn\'ee par $\mu(k)=0$ si $k$ contient au moins un facteur carr\'e, et $\mu(k)=(-1)^{\|k\|}$, o\`u $\|k\|$ est le nombre de facteurs premiers de $k$, si ceux-ci sont tous distincts. Elle est multiplicative. La fonction de Liouville, compl\`etement multiplicative, est d\'efinie par $\lambda(k)=(-1)^{\|k\|}$, o\`u $\|k\|$ est le nombre de facteurs premiers de $k$. Les \'egalit\'es
\[\frac{1}{\zeta(s)}=\sum_{k\ge 1} \frac{\mu(k)}{k^s},\hskip 12mm \frac{\zeta(2s)}{\zeta(s)}=\sum_{k\ge 1} \frac{\lambda(k)}{k^s}\]
illustrent l'importance de ces fonctions pour l'\'etude de la fonction Zeta de Riemann. En application des propositions \ref{butcher} ef \ref{fusion}, nous esquissons une m\'ethode de majoration en valeur absolue des fonctions sommatoires des fonctions de M\"obius et de Liouville~:
\begin{equation}
M(n):=\sum_{k=1}^n\mu(k), \hskip 12mm L(n):= \sum_{k=1}^n\lambda(k)
\end{equation}
en rangeant les nombres entre $1$ et $n$ par paires $\{a,b\}$ avec $\mu(a)+\mu(b)=0$ ou $\lambda(a)+\lambda(b)=0$. Pr\'ecisons \`a ce stade que nous n'obtenons aucun r\'esultat qui soit de nature \`a am\'eliorer notre connaissance de la fonction $\zeta$. Mais, apr\`es \'elimination des paires $\{a,b\}$ ainsi form\'ees, les faibles valeurs de $|M(n)|$ et de $|L(n)|$ sautent aux yeux pour toute valeur de $n$ inf\'erieure \`a quelques centaines. Nous esp\'erons avoir ainsi montr\'e que l'arborification de Matula pourrait \^etre autre chose qu'une simple curiosité.
\section{Quelques in\'egalit\'es concernant les nombres premiers}
\label{sect:pan-apn}
Le but de cette premi\`ere section est de d\'emontrer les propositions \ref{pan-apn} et \ref{fusion}. Nous partirons de l'encadrement suivant (Massias et Robin \cite[Th\'eor\`eme A]{MR96}, am\'elior\'e par P. Dusart \cite{D99})~:
\begin{equation}\label{massias-robin}
n(\log n+\log\log n -1)\le p_n\le n\left(\log n+\log\log n-1+1.8\frac{\log\log n}{\log n}\right)
\end{equation}
pour tout $n\ge 13$. Notons que la borne inf\'erieure dans \eqref{massias-robin}, due \`a P. Dusart, est valide pour tout $n\ge 2$. Notons aussi qu'une version un peu moins fine de cet encadrement est utilis\'ee dans \cite{DO18} pour montrer une propri\'et\'e d'extr\'emalit\'e des nombres de Matula correspondant \`a une famille sp\'ecifique d'arbres et de for\^ets, appel\'es chenilles binaires (binary caterpillars) par l'auteur.
\subsection{D\'emonstration de la Proposition \ref{pan-apn}}
La fonction $x\mapsto \log\log x/\log x$ atteint son maximum en $x=e^e\simeq 15.1542...$ Ce maximum vaut $e^{-1}=0,3678...$. On a donc d'apr\`es \eqref{massias-robin} l'encadrement pour tout $n\ge 13$~:
\begin{equation}\label{massias-robin2}
n(\log n+\log\log n -1)\le p_n\le n(\log n+\log\log n-0.337).
\end{equation}
On a donc pour tout $a\ge 2$ et $n\ge 13$~:
\begin{equation*}
ap_n\le an\log n+an\log\log n-0.337an,
\end{equation*}
alors que~:
\begin{eqnarray*}
p_{an}&\ge& an\log(an)+an\log\log (an)-an\\
&\ge & an\log n + an\log\log n+an(\log a -1).
\end{eqnarray*}
On a donc $p_{an}\ge ap_n$ d\`es que $\log a\ge 0.663$, soit $a\ge 1.9406...$. Tout $a\ge 2$ convient donc. Reste \`a examiner le cas o\`u $n$ est compris entre $1$ et $12$. On fixe donc un tel $n$. Il y a $n+1$ nombres premiers entre $p_{(a-1)n}$ et $p_{an}$, bornes comprises.
\begin{lem}\label{tuple}
Supposons que le couple $(a,n)$ soit diff\'erent de $(2,1)$. Alors pour tout nombre premier $q$, il existe au moins un r\'esidu modulo $q$ absent du $n+1-$uplet $\{p_{(a-1)n},p_{(a-1)n+1},\ldots,p_{an}\}$.
\end{lem}
\begin{proof}
Supposons qu'il existe un nombre premier $q$ qui contredise cette assertion. On a forc\'ement $n+1\ge q$. Le r\'esidu $0$ \'etant pr\'esent dans la suite $\{p_{(a-1)n},p_{(a-1)n+1},\ldots,p_{an}\}$, il existe $j\in\{0,\ldots,n\}$ tel que $q=p_{(a-1)n+j}$. On a donc $n+1\ge p_{(a-1)n+j}$, ce qui est impossible sauf si $a=2$ et $n\in\{1,2\}$. Tous les r\'esidus modulo $2$ apparaissent dans $\{p_1,p_2\}=\{2,3\}$, mais $0$ modulo $2$ n'appara\^\i t pas dans $\{p_2,p_3\}=\{3,5\}$, ce qui r\'eduit l'ensemble des couples \`a rejeter au singleton $\{(2,1)\}$.
\end{proof}
La diff\'erence $p_{an}-p_{(a-1)n}$ est donc sup\'erieure \`a la longueur minimale $c(n)$ d'un $n+1-$uplet de nombres premiers cons\'ecutifs "assez grands" (\cite[A008407]{sloane}). On a $c(n)\ge p_n$ pour $1\le n\le 12$, comme on le voit sur le tableau suivant~:
\begin{equation*}
\begin{disarray}{c|cccccccccccc}
n&1&2&3&4&5&6&7&8
&9&10&11&12\\
\hline\\
p_n&2&3&5&7&11&13&17&19&23&29&31&37\\
c(n)&2&6&8&12&16&20&26&30&32&36&42&48
\end{disarray}
\end{equation*}
On a donc $p_{an}>ap_n$ d\`es que que $a$ est diff\'erent de $2$. Enfin $p_2<2p_1$, $p_3<3p_1$ et $p_4<4p_1$, mais on a clairement $p_5>5p_1$ et par suite $p_a>ap_1$ pour tout $a\ge 5$. Ceci termine la d\'emonstration de la proposition \ref{pan-apn}. La proposition \ref{butcher} en est un corollaire imm\'ediat.
\subsection{D\'emonstration de la proposition \ref{fusion}}
On utilise \`a nouveau l'encadrement \eqref{massias-robin2}~: dans le cas $k\ge 14$ et $l\ge 14$ il suffit donc de montrer l'in\'egalit\'e~:
\begin{equation}\label{ineq1}
(\log k+\log\log k-1)(\log l+\log\log l-1)-(\log k+\log l+\log\log (kl)-0.337)\ge 0.
\end{equation}
Posant $K=\log k$ et $L=\log l$, l'in\'egalit\'e s'\'ecrit $f(K)f(L)-f(K+L)+0.663\ge0$, o\`u $f(x)=x+\log x-1$ est croissante sur $[\log 14,+\infty[$. La fonction $(K,L)\mapsto f(K)f(L)-f(K+L)+0.663$ atteint donc son minimum en $K=L=\log 14$, et on a~:
 \begin{equation}\label{ineq2}
 (\log 14+\log\log 14-1)^2-(2\log 14+\log(2\log 14)-0.337)
 =0.205...\ge 0,
 \end{equation}
et donc $p_kp_l\ge p_{kl}$ pour $k,l\ge 14$. Notons que l'in\'egalit\'e \eqref{ineq2} devient fausse si on remplace 14 par 13.
Dans le cas $k\le 13$ et $l\ge 13$ on a~
\begin{eqnarray*}
p_{kl}&\le& kl\big(\log k+\log l +\log(\log k+\log l)-0.337\big) \hbox{ et}\\
p_kp_l&\ge&p_kl(\log l+\log\log l-1),
\end{eqnarray*}
d'o\`u finalement~:
\begin{eqnarray*}
p_kp_l-p_{kl}&\ge& l\big((p_k-k)\log l+p_k\log\log l - k\log\log (kl)-k\log k-p_k+0.337k\big)\\
&\ge &l\Big((p_k-k)\log l+p_k\log\log l - k\log\big((\log l)(1+\frac {\log k}{\log l})\big)-k\log k-p_k+0.337k\Big)\\
&\ge &l\Big((p_k-k)\log l+p_k\log\log l - k\log\log l-k\frac {\log k}{\log l}-k\log k-p_k+0.337k\Big)\\
	&\ge &l\Big((p_k-k)(\log l+\log\log l-1)-k(\frac{\log k}{\log l}+\log k+0.663)\Big).
\end{eqnarray*}
Le cas $k=1$ \'etant trivial, on regarde sur $[13,+\infty[$ la fonction~:
\begin{equation*}
l\mapsto f_k(l):=(p_k-k)(\log l+\log\log l-1)-k(\frac{\log k}{\log l}+\log k+0.663)
\end{equation*}
pour $2\ge k\ge 13$. La fonction $f_k$ est croissante. Le calcul num\'erique explicite donne $f_2(21)<0<f_2(22)$, $f_3(21)<0<f_3(22)$ et $f_4(23)<0<f_4(24)$. On a ensuite $f_5(13)>0$, $f_6(14)<0<f_6(15)$, $f_8(13)<0<f_8(14)$, et finalement $f_k(13)>0$ pour $k=7,9,10,11,12$ et $13$. On a repr\'esent\'e en appendice les valeurs de $p_kp_l$ et de $p_{kl}$ (sous la forme du rapport $\frac{p_kp_l}{p_{kl}}$) pour les cas qui restent \`a examiner, \`a savoir les paires $\{k,l\}$ avec $k$ et $l$ entre $1$ et $13$, et les paires avec $l\ge 14$ et $f_k(l)<0$. On s'est limit\'e par sym\'etrie au cas $k\le l$~: toutes les paires repr\'esent\'ees donnent un rapport inf\'erieur \`a $1$ sauf les deux paires exceptionnelles $\{3,4\}$ et $\{4,4\}$, ce qui termine la preuve de la proposition \ref{fusion}.
\subsection{D\'emonstration d'une conjecture de J. Sousselier}
Pour tout entier $n\ge 2$, Jean Sousselier conjecture l'in\'egalit\'e suivante ("conjecture pr\'ealable", \cite{S99})~:
\begin{equation}\label{js}
\frac{p_n}{n}\le\frac{P_{p_n}}{p_n}.
\end{equation}
L'in\'egalit\'e \eqref{js} est imm\'ediatement v\'erifi\'ee pour $n=2,3,4,5$. Pour $n\ge 6$ il suffit d'utiliser l'encadrement de Rosser et Schoenfeld \cite{RS62}~:
\begin{equation}
\log n\le\frac{p_n}{n}\le\log n+\log\log n.
\end{equation}
On en d\'eduit imm\'ediatement~:
\begin{equation*}
\log n\le\frac{p_n}{n}\le\log n+\log\log n=\log(n\log n)\le\log p_n\le\frac{p_{p_n}}{p_n},
\end{equation*}
ce qui d\'emontre \eqref{js}.
\subsection{Encore une in\'egalit\'e sur les nombres premiers}
\begin{prop}\label{trois-n}
Pour tout entier $n\ge 12$ on a~:
\begin{equation}
p_n>3n.
\end{equation}
\end{prop}
\begin{proof}
L'in\'egalit\'e $p_n\ge n\log n$ fournit la d\'emonstration pour $n\ge 21$ (en utilisant $e^3=20.08...$). L'examen des vingt premiers nombres premiers montre que l'in\'egalit\'e est en fait v\'erifi\'ee d\`es $n=12$.
\end{proof}
En termes d'arbres, on peut en d\'eduire qu'un arbre $t$ \`a une seule branche est toujours sup\'erieur \`a la for\^et $\arbrea\, t'$ o\`u $t'$ est d\'efini par $t=t'\butcher \matulab$, sauf si $t=\matulac,\matulae,\matulaaa,\matulaag$ ou $\matulaca$. Par exemple $\matulaei >\matulac\matulaag$, soit $59>51$.
\section{Entiers naturels et arbres}
\subsection{La structure non-associative permutative sur les arbres}
Un \textsl{magma non-associatif permutatif} (ou "magma NAP") \cite{DL, L06} est un ensemble $E$ muni d'une loi interne binaire $\RHD$ telle que~:
\begin{equation}\label{nap}
x\RHD(y\RHD z)=y\RHD(x\RHD z).
\end{equation}
Le terme "commutatif \`a gauche" est parfois utilis\'e \`a la place de NAP\footnote{plus exactement "commutatif \`a droite", car les auteurs ci-dessus utilisent la variante \`a droite not\'ee $\LHD$, dans laquelle la relation \eqref{nap} est remplac\'ee par $(x\LHD y)\LHD z=(x\LHD z)\LHD y$.}. En-dehors des exemples \'evidents fournis par les mono\"\i des ab\'eliens, le prototype est fourni par l'ensemble $\Cal T$ des arbres enracin\'es munis du produit de Butcher. Plus pr\'ecis\'ement \cite{DL}, $(\Cal T,\,\butcher)$ est le mono\"\i de NAP libre \`a un g\'en\'erateur. La relation \eqref{nap} s'exprime par le fait que l'ordre de greffage des branches sur la racine n'a pas d'importance.
\subsection{Retour sur l'arborification de Matula-Goebel-Cappello}
L'utilisation de l'arborification $\Cal A:\N^*\to\Cal F$ d\'efinie dans l'introduction permet de transporter le produit de Butcher sur les entiers positifs non nuls, munissant ainsi l'ensemble $\Cal P$ des nombres premiers d'une structure de mono\"\i de NAP. Le produit de Butcher de deux nombres premiers $q=p_m$ et $r=p_n$ est donn\'e par~:
\begin{equation}
p_m\butcher p_n=p_{{p_m}n}.
\end{equation}
La libert\'e et la mono-g\'en\'eration du mono\"\i de NAP $\Cal P$ s'expriment par le fait que tout nombre premier admet une expression comme produit de Butcher it\'er\'e et convenablement parenth\'es\'e de $p_1=2$, cette expression \'etant unique modulo les relations \eqref{nap}. Le produit de Butcher est compatible avec l'ordre total, \`a savoir que si $q\le q'$ et $r\le r'$, alors $q\butcher r\le q'\butcher r'$. En revanche, comme P. Cappello l'avait remarqu\'e \cite{C88}, le nombre de sommets n'est pas une fonction croissante de l'arbre~: en effet $\Cal A(53)=\matulaec$ et $\Cal A(59)=\matulaei$ ont cinq sommets, alors que $\Cal A(47)=\matuladg$ en a six. L'in\'egalit\'e $p_n> 3n$ (Proposition \ref{trois-n}) fournit une exemple de ce ph\'enom\`ene pour tout $n$ premier au moins \'egal \`a $12$~: en effet, $\mathcal A(3n)$ poss\`ede un sommet de plus que $\mathcal A(p_n)$.
\subsection{Sommets, ar\^etes et feuilles}
La \textsl{fonction de Gutman-Ivi\'c-Matula} $n\mapsto v(n)$, qui associe \`a un entier $n$ le nombre total de sommets de sa for\^et $\Cal A(n)$, est compl\`etement additive. Elle a \'et\'e \'etudi\'ee en d\'etail par R. de la Bret\`eche et G. Tenenbaum \cite{BT}. On notera $a(n)$ le nombre total d'ar\^etes de la for\^et $\Cal A(n)$, et $f(n)$ le nombre total de feuilles de la for\^et $\Cal A(n)$. Les fonctions $a$ et $f$ sont \'egalement totalement additives, et le nombre de facteurs de $n$ est donn\'e par~:
\begin{equation}
\|n\|=v(n)-a(n).
\end{equation}
$1$ est le seul nombre sans feuilles. Les nombres \`a une seule feuille constituent la suite des nombres hyper-premiers de Wilson ("Wilson's primeth sequence", \cite[A007097]{sloane})~:
\vskip 2mm
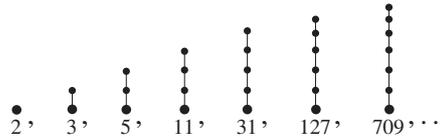
$$\genfrac{}{}{0pt}{1}{\matulab}{2}, \hskip 3mm\genfrac{}{}{0pt}{1}{\matulac}{3}, \hskip 3mm\genfrac{}{}{0pt}{1}{\matulae}{5}, \hskip 3mm\genfrac{}{}{0pt}{1}{\matulaaa}{11}, \hskip 3mm\genfrac{}{}{0pt}{1}{\matulaca}{31}, \hskip 3mm\genfrac{}{}{0pt}{1}{\matulaabg}{127}, \hskip 3mm\genfrac{}{}{0pt}{1}{\matulagoi}{709},\ldots$$
Les premiers nombres \`a deux feuilles sont~:
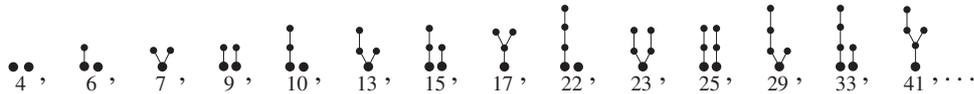
$$\hskip 3mm\genfrac{}{}{0pt}{1}{\matulab\matulab}{4}, \hskip 3mm\genfrac{}{}{0pt}{1}{\matulac\matulab}{6}, \hskip 3mm\genfrac{}{}{0pt}{1}{\matulag}{7}, \hskip 3mm\genfrac{}{}{0pt}{1}{\matulac\matulac}{9}, \hskip 3mm\genfrac{}{}{0pt}{1}{\matulae\matulab}{10}, \hskip 3mm\genfrac{}{}{0pt}{1}{\matulaac}{13}, \hskip 3mm\genfrac{}{}{0pt}{1}{\matulae\matulac}{15}, \hskip 3mm\genfrac{}{}{0pt}{1}{\matulaag}{17}, \hskip 3mm\genfrac{}{}{0pt}{1}{\matulaaa\matulab}{22}, \hskip 3mm\genfrac{}{}{0pt}{1}{\matulabc}{23}, \hskip 3mm\genfrac{}{}{0pt}{1}{\matulae\matulae}{25}, \hskip 3mm\genfrac{}{}{0pt}{1}{\matulabi}{29}, \hskip 3mm\genfrac{}{}{0pt}{1}{\matulaaa\matulac}{33}, \hskip 3mm\genfrac{}{}{0pt}{1}{\matulada}{41},\ldots $$
Pour un tour d'horizon r\'ecent des diverses fonctions que l'on peut d\'efinir sur les entiers via leur arborification, voir \cite{D}.
\subsection{Une graduation}
On d\'efinit le degr\'e $\delta(n)$ d'un nombre naturel (non nul) $n$ par $\delta(n)=v(n)+a(n)$. Autrement dit~:
\begin{equation}
\delta(n)=2v(n)-\|n\|.
\end{equation}
Le degr\'e est compl\`etement additif. Il est facile de voir que pour tout entier $m$, l'ensemble des entiers $n$ de degr\'e $m$ est fini. Ci-dessous, la liste des entiers de de degr\'e $\le 5$ avec leur for\^et correspondante~:\\

Degr\'e $0$ : $1$. \hskip 8mm Degr\'e $1$ : $\genfrac{}{}{0pt}{1}{\matulab}{2}$. \hskip 8mm Degr\'e $2$ : $\genfrac{}{}{0pt}{1}{\matulab\matulab}{4}$. 

Degr\'e $3$ : $\genfrac{}{}{0pt}{1}{\matulac}{3}, \genfrac{}{}{0pt}{1}{\matulab\matulab\matulab}{8}$.\hskip 8mm Degr\'e $4$ : $\genfrac{}{}{0pt}{1}{\matulac\matulab}{6},\genfrac{}{}{0pt}{1}{\matulab\matulab\matulab\matulab}{16}$.\hskip 8mm
Degr\'e $5$ : $\genfrac{}{}{0pt}{1}{\matulae}{5}, \genfrac{}{}{0pt}{1}{\matulag}{7}, \genfrac{}{}{0pt}{1}{\matulac\matulab\matulab}{12}, \genfrac{}{}{0pt}{1}{\matulab\matulab\matulab\matulab\matulab}{32}$.\\

\noindent En particulier, le degr\'e $\delta$ a m\^eme parit\'e que le nombre de facteurs. Pour donner quelques autres exemples, $\genfrac{}{}{0pt}{1}{\matulaac}{13}$ est de degr\'e $7$, $\genfrac{}{}{0pt}{1}{\matuladg}{47}$ est de degr\'e $11$,  $\genfrac{}{}{0pt}{1}{\matulaai\matulac}{57}$ est de degr\'e $10$ et $\genfrac{}{}{0pt}{1}{\matulaec\matulag\matulag}{2597}$ est de degr\'e $19$. On remarque que le produit de fusion et la coupe d'une branche abaissent tous les deux le degr\'e d'une unit\'e, puisqu'il en r\'esulte la suppression d'un sommet et d'une ar\^ete respectivement.
\section{Etude des fonctions sommatoires des fonctions de M\" obius et de Liouville}
Le propos de cette section est d'obtenir une majoration de $|M(n)|$ pour tout entier $n$, si possible plus pr\'ecise que la majoration triviale $|M(n)|\le n/4$ obtenue en appariant chaque nombre pair sans facteur carré avec sa moiti\'e, ou m\^eme que la majoration non triviale $|M(n)|=o(n)$, \'equivalente au th\'eor\`eme des nombres premiers \cite{S48, E49}. Le principe est le suivant : r\'ealiser le plus possible de couplages $(k,l)$ de nombres sans facteur carr\'e \`a l'int\'erieur de $\{1,\ldots n\}$ tels que $\mu(k)+\mu(l)=0$. La majoration cherch\'ee est alors obtenue en comptant les nombres ayant \'echapp\'e au processus de couplage. Une fa\c con de chercher un partenaire $l$ d'un nombre $k$ (avec $l<k$) consiste \`a couper l'une des branches de la for\^et de $k$ ou, si l'on pr\'ef\`ere, \`a fusionner deux arbres de la for\^et (en \'evitant les quelques exceptions mentionn\'ees plus haut), en prenant soin de ne pas faire appara\^\i tre de facteurs carr\'es. La deuxi\`eme op\'eration n'est \'evidemment possible que si $k$ poss\`ede au moins deux facteurs premiers.\\

Pour la fonction $L$ sommatoire de la fonction de Liouville, le principe est exactement le m\^eme \`a ceci pr\`es qu'on ne se pr\'eoccupe pas des facteurs carr\'es.\\

On se heurte rapidement \`a de nombreuses difficult\'es. La plus \'evidente r\'eside dans les multiples choix possibles de branches \`a couper ou d'arbres \`a fusionner. Un autre probl\`eme appara\^\i t d'embl\'ee lorsqu'on veut r\'ealiser plusieurs couplages~: il peut se faire que tous les partenaires $l$ possibles d'un nombre $k$ rest\'e seul soient d\'ej\`a membres d'une paire form\'ee pr\'ec\'edemment. Les exemples donn\'es ci-dessous pour des petites valeurs de $n$ (Appendice C) semble toutefois indiquer que la m\'ethode fonctionne bien, reste \`a expliquer pourquoi...
\ignore{
\subsection{Description des couplages possibles}
On \'ecrit $\mathbb N^*$ comme r\'eunion disjointe de deux parties
\begin{equation}
\mathbb N^*=\mathcal D\sqcup\mathcal Q,
\end{equation}
o\`u $\mathcal D$ d\'esigne l'ensemble des entiers non nuls dont tous les facteurs premiers sont distincts, et o\`u $\mathcal Q$ d\'esigne l'ensemble des entiers non nuls contenant au moins un facteur carr\'e non trivial. nous aurons besoin \'egalement du sous-ensemble $\mathcal M$ constitu\'e des entiers dont \textsl{tous} les facteurs sont multiples, i.e. apparaissent au moins deux fois.\\

La r\`egle du jeu est la suivante~: soit $k\in\mathcal D$. Un entier $l<k$, tel que $\mu(k)+\mu(l)=0$, pourra \^etre obtenu soit en coupant l'une des branches de la for\^et de $k$ en \'evitant de couper la branche la plus grande, soit en fusionnant le plus grand arbre de la for\^et de $k$ avec un autre facteur. Les deux exceptions $\genfrac{}{}{0pt}{1}{\matulag\matulae}{35}\mapsto\genfrac{}{}{0pt}{1}{\matulacg}{37}$ et $\genfrac{}{}{0pt}{1}{\matulag\matulag}{49}\mapsto\genfrac{}{}{0pt}{1}{\matulaec}{53}$ doivent \^etre \'evit\'ees. La coupe de branches, en plus de l'\'elimination des deux exceptions $\genfrac{}{}{0pt}{1}{\matulac}{3}\mapsto\genfrac{}{}{0pt}{1}{\matulab\matulab}{4}$ et $\genfrac{}{}{0pt}{1}{\matulae}{5}\mapsto\genfrac{}{}{0pt}{1}{\matulac\matulab}{6}$, doit ob\'eir \'egalement aux restrictions suivantes~:
\begin{enumerate}
\item Pour un arbre dont la fertilit\'e de la racine est sup\'erieure ou \'egale \`a $2$, seules les coupes \`a la racine sont autoris\'ees (la fertilit\'e d'un sommet est par d\'efinition le nombre d'ar\^etes issues de ce sommet).

\item Lorsque la m\^eme branche appara\^\i t plusieurs fois dans le m\^eme facteur, sa coupe n'est pas autoris\'ee. 

\item Pour un arbre avec racine de fertilit\'e $1$ (correspondant \`a un nombre premier de rang premier), on autorise \'egalement les coupes au sommet imm\'ediatement au-dessus. Si celui-ci est \'egalement de fertilit\'e $1$ on autorise aussi les coupes au sommet situ\'e encore un \'etage au-dessus, et ainsi de suite jusqu'\`a rencontrer un sommet de fertilit\'e $\ge 2$. Par exemple, $\genfrac{}{}{0pt}{1}{\matulaagi}{179}$ peut \^etre coupl\'e \`a $\genfrac{}{}{0pt}{1}{\matulada\matulab}{82}$, $\genfrac{}{}{0pt}{1}{\matulaac\matulac}{39}$, $\genfrac{}{}{0pt}{1}{\matulaca\matulab}{62}$ et \'eventuellement $\genfrac{}{}{0pt}{1}{\matulaaa\matulac}{33}$.
\end{enumerate}
\ignore{Par ailleurs on autorise le couplage exotique $\genfrac{}{}{0pt}{1}{\matulag}{7}\mapsto \genfrac{}{}{0pt}{1}{\matulab\matulab}{4}$, qui diminue le degr\'e de $3$, en plus du couplage $\genfrac{}{}{0pt}{1}{\matulag}{7}\mapsto \genfrac{}{}{0pt}{1}{\matulac\matulab}{6}$.}
\subsection{Impasses et nombres univoques}
Un nombre $k\in\mathcal D$ est une \textsl{impasse} si aucun nombre $l<k$ ne peut lui \^etre appari\'e en suivant les r\`egles du jeu ci-dessus. On dit que $k$ est \textsl{univoque} si un et un seul nombre $l<k$ peut lui \^etre appari\'e de cette mani\`ere. On note $\Cal C$ et $\Cal U$ l'ensemble des impasses et des nombres univoques respectivement, puis, pour tout $n\in\N^*$, $\Cal C_n:=\Cal C\cap\{1,\ldots,n\}$ et $\Cal U_n:=\Cal U\cap\{1,\ldots,n\}$. L'ensemble $\Cal C$ est constitu\'e~:
\begin{itemize}
\item des nombres $2$, $3$, $5$ et $35$,
\item de $p(\mathcal M)$, c'est-\`a-dire des nombres premiers dont le rang n'a que des facteurs multiples~: en particulier tous les nombres premiers de rang carr\'e sont dans $\Cal C$,
\item des nombres du type $p_{p_nn}$ o\`u $n$ appartient \`a $\mathcal M$, par exemple $\genfrac{}{}{0pt}{1}{\matulaaog}{107}$, correspondant \`a $n=4$.
\end{itemize}
\vskip 1mm
L'ensemble $\Cal U$ quant \`a lui est constitu\'e~:
\begin{itemize}
\item de tous les nombres premiers de rang  $p^r$ avec $p$ premier et $r\ge 2$,
\item des produits d'un premier de rang $p^r$ comme ci-dessus avec un nombre dont tous les facteurs premiers sont multiples.
\item des nombres $\genfrac{}{}{0pt}{1}{\matulac\matulab}{6}$, $\genfrac{}{}{0pt}{1}{\matulae\matulab}{10}$, $\genfrac{}{}{0pt}{1}{\matulae\matulac}{15}$ et $\genfrac{}{}{0pt}{1}{\matulag\matulae}{35}$.
\item des nombres $6m$, $10m$ et $15m$ et $35m$, o\`u $m$ est un nombre dont tous les facteurs premiers sont multiples.
\end{itemize}
\vskip 1mm
\subsection{Construction des couplages}
Le but de ce paragraphe est d'\'etablir le th\'eor\`eme crucial suivant~:
\begin{thm}
Pour tout $n\in\N^*$ il existe une famille de couplages $k\mapsto l$ \redtext{avec $k<l$ ou non?} telle que l'ensemble des $m\in\{1,\ldots, n\}$ rest\'es seuls soit inclus dans $\Cal U_n\amalg\Cal C_n$.
\end{thm}
\begin{proof}
On proc\`ede par r\'ecurrence sur $n$. Le r\'esultat est trivial pour les tr\`es petites valeurs de $n$, et facile \`a \'etablir, en suivant les r\`egles du jeu \'enonc\'ees ci-dessus, pour des valeurs un peu plus grandes. Ci-dessous, une famille de couplages possibles pour $n=41$~:

\noindent $(\genfrac{}{}{0pt}{1}{\matulac\matulab}{6},\genfrac{}{}{0pt}{1}{\matulac}{3})$, \hskip 5mm $(\genfrac{}{}{0pt}{1}{\matulae\matulab}{10},\genfrac{}{}{0pt}{1}{\matulae}{5})$, \hskip 5mm $(\genfrac{}{}{0pt}{1}{\matulaaa}{11},\genfrac{}{}{0pt}{1}{\matulac\matulac}{9})$, \hskip 5mm  $(\genfrac{}{}{0pt}{1}{\matulag\matulab}{14},\genfrac{}{}{0pt}{1}{\matulag}{7})$, \hskip 5mm  $(\genfrac{}{}{0pt}{1}{\matulag\matulac}{21},\genfrac{}{}{0pt}{1}{\matulaai}{19})$,  \hskip 5mm  $(\genfrac{}{}{0pt}{1}{\matulaac\matulab}{26},\genfrac{}{}{0pt}{1}{\matulaac}{13})$,  \hskip 5mm $(\genfrac{}{}{0pt}{1}{\matulaca}{31},\genfrac{}{}{0pt}{1}{\matulae\matulac}{15})$, \hskip 5mm $(\genfrac{}{}{0pt}{1}{\matulaaa\matulac}{33},\genfrac{}{}{0pt}{1}{\matulabi}{29})$, \\

\noindent  $(\genfrac{}{}{0pt}{1}{\matulaag\matulab}{34},\genfrac{}{}{0pt}{1}{\matulaag}{17})$,\hskip 5mm $(\genfrac{}{}{0pt}{1}{\matulag\matulae}{35},\genfrac{}{}{0pt}{1}{\matulae\matulac\matulab}{30})$,  \hskip 5mm $(\genfrac{}{}{0pt}{1}{\matulaai\matulab}{38},\genfrac{}{}{0pt}{1}{\matulag\matulab\matulab}{28})$, \hskip 5mm  $(\genfrac{}{}{0pt}{1}{\matulaac\matulac}{39},\genfrac{}{}{0pt}{1}{\matulacg}{37})$,\hskip 5mm $(\genfrac{}{}{0pt}{1}{\matulada}{41},\genfrac{}{}{0pt}{1}{\matulaaa\matulab}{22})$.
\vskip 2mm
\noindent Les nombres rest\'es seuls sont~:

\noindent $1$, $\genfrac{}{}{0pt}{1}{\matulab}{2}$, \hskip 2mm $\genfrac{}{}{0pt}{1}{\matulab\matulab}{4}$,  \hskip 2mm  $\genfrac{}{}{0pt}{1}{\matulab\matulab\matulab}{8}$,  \hskip 2mm $\genfrac{}{}{0pt}{1}{\matulac\matulab\matulab}{12}$, $\genfrac{}{}{0pt}{1}{\matulab\matulab\matulab\matulab}{16}$,  \hskip 2mm $\genfrac{}{}{0pt}{1}{\matulac\matulac\matulab}{18}$, \hskip 2mm $\genfrac{}{}{0pt}{1}{\matulae\matulab\matulab}{20}$,  \hskip 2mm $\genfrac{}{}{0pt}{1}{\matulabc}{23}$,  \hskip 2mm $\genfrac{}{}{0pt}{1}{\matulac\matulab\matulab\matulab}{24}$, \hskip 2mm $\genfrac{}{}{0pt}{1}{\matulae\matulae}{25}$,  \hskip 2mm $\genfrac{}{}{0pt}{1}{\matulac\matulac\matulac}{27}$,  \hskip 2mm $\genfrac{}{}{0pt}{1}{\matulab\matulab\matulab\matulab\matulab}{32}$, \hskip 2mm  $\genfrac{}{}{0pt}{1}{\matulac\matulac\matulab\matulab}{36}$,  \hskip 2mm $\genfrac{}{}{0pt}{1}{\matulae\matulab\matulab\matulab}{40}$.
\vskip 2mm\noindent
Ils sont tous dans $\Cal C$ except\'e $23$ qui est dans $\Cal U$.\\

Supposons donc le th\'eor\`eme vrai pour $n-1$. La famille de couplages pour $n$ est toute trouv\'ee si $n$ appartient \`a $\Cal C$ ou \`a $\Cal U$~: il suffit de laisser $n$ tout seul et de garder la famille de couplages d\'ej\`a trouv\'ee pour $n-1$. Dans le cas contraire on a au moins deux possibilit\'es de coupler $n$ avec un $k<n$ tel que $\lambda(k)+\lambda(n)=0$. On en choisit un. Si $k$ \'etait tout seul dans la famille de couplages pr\'ec\'edente, le processus est termin\'e. Mais en g\'en\'eral $k$ \'etait d\'ej\`a coupl\'e \`a un autre $l<n$, avec $l<k$ ou $k<l$. C'est donc $l$ qui se retrouve maintenant seul. Si $l\in\Cal C$ ou $l\in\Cal U$ le processus s'arr\^ete, sinon on essaie de l'apparier avec un certain $m<l$, et le raisonnement peut continuer.\\

On esp\`ere ainsi, au bout d'un nombre fini d'\'etapes, obtenir un nombre solitaire $m\in\{1,\ldots,n\}$ qui soit dans $\Cal C$ ou dans $\Cal U$. C'est sans doute souvent possible, mais rien n'exclut la possibilit\'e d'un cycle infernal, qui verrait d\'efiler ind\'efiniment des nombres dans $\{1,\ldots ,n\}\backslash(\Cal C_n\amalg\Cal U_n)$. \\

Pour contourner ce probl\`eme, il faut assouplir un peu la r\`egle de formation des paires, en permettant par exemple des \'echanges entre deux paires bien choisies. Un couplage issu de ces r\`egles du jeu \'etendues (encore \`a d\'efinir) sera dit \textsl{graphique} s'il est obtenu par fusion ou coupe d'une branche en suivant les r\`egles du jeu plus strictes du paragraphe pr\'ec\'edent, \textsl{non graphique} sinon. Hormis le cas trivial \'evoqu\'e ci-dessus (\`a savoir, $n\in\Cal C\amalg\Cal U)$, on distingue sept cas diff\'erents qui sont visualis\'es ci-dessous~:
$$\couplages$$
\vskip -6mm
\hskip0.9cm
\begin{minipage}{5.50in}
{\sl Le sommet le plus gros d\'esigne l'entier $n$, les sommets plus petits d\'esignent des entiers $k$ plus petits que $n$. Un sommet entour\'e d'un carr\'e gris d\'esigne un nombre laiss\'e seul dans la famille de couplages relative \`a $n-1$, un sommet entour\'e d'un rectangle gris d\'esigne un nombre dans $\Cal U$, le sommet soulign\'e du cas num\'ero $4$ d\'esigne un nombre dans $\Cal C$. Une ar\^ete pleine d\'esigne un couplage de la famille relative \`a $n-1$, qui peut \^etre graphique ou non. Une ar\^ete en pointill\'e d\'esigne un couplage graphique possible (mais non effectu\'e), et enfin l'ar\^ete boucl\'ee d\'esigne le couplage qui sera effectivement appliqu\'e. Des petites boucles serr\'ees d\'esignent un couplage graphique. Le sommet entour\'e d'un cercle gris est celui qui sera laiss\'e seul apr\`es application de ce couplage.}
\end{minipage}
\vskip 10mm
\noindent\textbf{Premier cas :} l'entier $n$ est coupl\'e graphiquement \`a un entier $k<n$ qui \'etait seul, ce qui r\'esout le probl\`eme.\\

\noindent\textbf{Deuxi\`eme cas :} l'entier $n$ est coupl\'e graphiquement \`a un entier $k<n$, qui \'etait lui-m\^eme coupl\'e \`a un entier $l<k$. L'entier $l$ se retrouve donc seul, et v\'erifie $\delta(n)-\delta(l)\ge 2$. On recommence alors le processus de couplage avec $l$ \`a la place de $n$.\\

\noindent\textbf{Troisi\`eme cas :} l'entier $n$ est coupl\'e graphiquement \`a un entier $k<n$, qui \'etait lui-m\^eme coupl\'e \`a un entier $l>k$ avec $l\in\Cal U$. L'entier $l$ se retrouve donc seul et peut le rester, car il est suppos\'e univoque, ce qui r\'esout le probl\`eme.\\

\noindent\textbf{Quatri\`eme cas :} l'entier $n$ est coupl\'e graphiquement \`a un entier $k<n$ avec $k\in\Cal C$, qui \'etait lui-m\^eme coupl\'e \`a un entier $l>k$. L'entier $k$ \'etant une impasse, on ne peut pas appliquer les proc\'ed\'es d\'ecrits aux cas $4$ et $5$ ci-dessous. Il faut donc recommencer le processus de couplage avec $l$ \`a la place de $n$.\\

\noindent\textbf{Cinqui\`eme cas :} l'entier $n$ est d'abord coupl\'e graphiquement \`a un entier $k<n$, qui \'etait lui-m\^eme coupl\'e \`a un entier $l>k$ a priori quelconque, et on suppose que $k$ peut par ailleurs \^etre coupl\'e \`a un entier $m<k$ par fusion ou coupe d'une branche, lui-m\^eme coupl\'e \`a $r<m$. On couple alors $n$ avec $r$ et on laisse inchang\'e le couplage $(l,k)$. L'entier $m$ reste seul et on a $\delta(n)-\delta(m)=2$. On recommence alors le processus de couplage avec $m$ \`a la place de $n$. On remarque que le couplage $(n,r)$ n'est pas graphique a priori.\\

\noindent\textbf{Sixi\`eme cas :} l'entier $n$ est d'abord coupl\'e graphiquement \`a un entier $k<n$, qui \'etait lui-m\^eme coupl\'e \`a un entier $l>k$ a priori quelconque, et on suppose que $k$ peut par ailleurs \^etre coupl\'e \`a un entier $m<k$ par fusion ou coupe d'une branche, lui-m\^eme coupl\'e \`a $r>m$.  On couple alors $n$ avec $r$ et on laisse inchang\'e le couplage $(l,k)$. L'entier $m$ reste seul et on a $\delta(n)-\delta(m)=2$. On recommence alors le processus de couplage avec $m$ \`a la place de $n$. Le couplage $(n,r)$ n'est pas forc\'ement graphique, et il est important de noter qu'il n'est pas a priori possible de comparer $\delta(n)$ (le degr\'e de $n$) et $\delta(r)$ (le degr\'e du nouveau partenaire de $n$).\\

\noindent\textbf{Septi\`eme cas :} l'entier $n$ est d'abord coupl\'e graphiquement \`a un entier $k<n$, qui \'etait lui-m\^eme coupl\'e \`a un entier $l>k$ a priori quelconque, et on suppose que $k$ peut par ailleurs \^etre coupl\'e graphiquement \`a un entier $m<k$ qui \'etait seul, par fusion ou coupe d'une branche. Il est alors impossible de proc\'eder comme dans les cas $5$ et $6$~: on effectue le couplage graphique $(n,k)$, et il faut recommencer le processus de couplage avec $l$ \`a la place de $n$.
\end{proof}
\redtext{Ca prend encore l'eau, \`a cause de couplages non-graphiques introduits aux cas $5$ et $6$. Il se peut alors que, par exemple, le nombre rest\'e seul au cas $2$ soit de degr\'e $\ge \delta(n)$...}
\vfill
\eject
\ignore{
\subsection{Le principe de compensation}
\redtext{Ca ne peut pas marcher comme \c ca. A revoir...}
\begin{prop}\label{ppdec}
Soit $\Phi$ une application de $\{1,\ldots ,n\}$ dans $\N^*$ telle que $\lambda\big(\Phi(k)\big)=-\lambda(k)$ pour tout $1\le k\le n$. Soit $D$ une partie de $\{1,\ldots ,n\}$ telle que $\Phi\restr{D}$ soit injective et telle que $\Phi(k)<k$ pour tout $k\in D$. Alors $|\Lambda(n)|\le n-|D|$.
\end{prop}
On appellera \textsl{application compensatoire} toute application $\Phi$ qui change le signe de la fonction de Liouville comme ci-dessus.
\begin{proof}Pour toute partie $P$ de l'intervalle $\{1,\ldots ,n\}$ on pose~:
\begin{equation}
\Lambda_P:=\sum_{k\in P}\lambda(k).
\end{equation}
Soit $C$ le compl\'ementaire de $D$ dans $\{1,\ldots ,n\}$. On note $Q$ le compl\'ementaire de $\Phi(D)$ dans $\{1,\ldots ,n\}$, qui a manifestement le m\^eme cardinal que $C$. On a alors~:
\begin{eqnarray*}
\Lambda(n)&=&\Lambda_D+\Lambda_C\\
&=&\Lambda_{\Phi(D)}+\Lambda_Q.
\end{eqnarray*}
En utilisant l'\'egalit\'e $\Lambda_D+\Lambda_{\Phi(D)}=0$ et en sommant on a donc $2\Lambda(n)=\Lambda_C+\Lambda_Q$, d'o\`u $|\Lambda(n)|\le |C|$, ce qui d\'emontre la proposition \ref{ppdec}. 
\end{proof}
Un exemple d'application $\Phi$ satisfaisant les hypoth\`eses de la proposition \ref{ppdec} est $k\mapsto pk$ o\`u $p$ est n'importe quel nombre premier. Elle a le d\'efaut de n'aboutir \`a aucune majoration non-triviale, puisque dans ce cas $D=\emptyset$. Un autre exemple est donn\'e par $\Phi(2k)=k$ et $\Phi(2k+1)=p(2k+1)$ o\`u $p$ est un nombre premier plus grand que $n$. L'ensemble $D$ maximal possible (resp. son compl\'ementaire $C$) est constitu\'e des entiers pairs (resp. impairs) de $\{1,\ldots ,n\}$, ce qui aboutit \`a la majoration $|\Lambda(n)|\le [(n+1)/2]$. Il s'agit maintenant de construire une applicationtion compensatoire $\Phi$ avec un ensemble $D$ le plus gros possible.
\begin{rmk}
Le principe de compensation est valable en rempla\c cant la fonction de Liouville $\lambda$ par n'importe quelle application $f$ d\'efinie sur $\N^*$ \`a valeurs dans $\{-1,1\}$, ou m\^eme dans $\{-1,0,1\}$, en rempla\c cant alors $\Lambda$ par la fonction sommatoire $F$ de $f$.
\end{rmk}
\subsection{Construction d'une application compensatoire}
Nous la construisons d'abord sur les nombres premiers avant de l'\'etendre aux nombres compos\'es, en esquissant d'abord deux tentatives provisoires pour mieux faire comprendre la d\'emarche. Nous passerons librement de la repr\'esentation traditionnelle des nombres \`a la repr\'esentation par les for\^ets et vice-versa en appliquant, souvent tacitement, l'arborification de Cappello $\Cal A$ ou son inverse.\\
\begin{itemize}
\item
En vue d'appliquer la proposition \ref{butcher}, une premi\`ere tentative consiste \`a couper la plus petite branche de l'arbre, c'est-\`a-dire que l'on pose $\Phi\big(B+(t_1\cdots t_k)\big)=t_1\Phi\big(B+(t_1\cdots t_k)\big)$ o\`u les branches $t_j$ sont class\'ees par ordre croissant. Elle donne le m\^eme r\'esultat sur $\racine\butcher t$ et sur $t\butcher\racine$ pour tout arbre $t$, \`a savoir $\racine\, t$. Par exemple $\Phi(\arbrecb)=\Phi(\arbrecd)=\racine\arbrebb$. La restriction aux arbres \`a au moins deux branches est en revanche injective.\\

\item
Nous corrigerons cette premi\`ere tentative en posant comme ci-dessus $\Phi\big(B+(t_1\cdots t_k)\big)=t_1\Phi\big(B+(t_1\cdots t_k)\big)$ si $k\ge 2$, et $\Phi\big(B_+(t)\big)=\Phi(t\butcher\racine)=\racine\racine\racine\, t$. Autrement dit en \'evitant le langage des arbres, $\Phi(p_m)=qp_{m/q}$ si $m$ n'est pas premier et si $q$ est le plus petit facteur premier de $m$, et $\Phi(p_m)=8m$ si $m$ est premier. On voit facilement que $\Phi$ est alors injective. On essaie alors d'\'etendre l'application $\Phi$ aux for\^ets, c'est-\`a-dire \`a tous les nombres premiers ou non. On posera $\Phi(t_1\cdots t_k):=t_1\cdots t_{k-1}\Phi(t_k)$ o\`u $t_k$ est le plus grand arbre de la for\^et. L'injectivit\'e n'est alors plus assur\'ee, puisque l'on a par exemple~:
\begin{equation*}
\Phi(\racine\racine \arbrebb)=\Phi(\arbrecb)=\racine\racine\racine\arbrea.
\end{equation*}\\

\item
On corrige encore une fois en gardant la m\^eme formule pour les arbres \`a une seule branche (c'est-\`a-dire pour les nombres premiers de rang premier), mais en choisissant, pour les arbres \`a plusieurs branches, de couper la plus petite branche diff\'erente de $\racine$. Il ne restera plus qu'\`a traiter le cas des couronnes $\arbrea,\arbrebb,\arbrecd,\arbredh, (\ldots)$.
\end{itemize}.
\vskip 2mm
Nous adoptons donc la d\'efinition formelle suivante pour $\Phi$, en notant $\Phi^\prec$ la version arborifi\'ee donn\'ee par~:
\begin{equation}
\Phi^\prec=\Cal A\circ\Phi\circ\Cal A^{-1}.
\end{equation}
\begin{defn}
L'application $\Phi^\prec:\Cal A(\{1,\ldots,n\})\to \Cal F$ est d\'efinie de la mani\`ere suivante~:
\begin{enumerate}
\item Pour tout arbre $t=B_+(t_1,\ldots,t_k)$ avec $k\ge 2$, $t_1\le\cdots\le t_k$ et $t_k\not =\racine$, on pose~:
\begin{equation}
\Phi^\prec(t)=t_jB_+(t_1,\ldots,\widehat{t_j},\ldots,t_k),
\end{equation}
o\`u $j$ est le plus petit indice tel que $t_j\not =\racine$.
\item Pour toute couronne $c_k:=B_+(\racine^k)$ avec $k\ge 2$ on pose $\Phi(c_k)=\racine\, c_{k-1}$. Par exemple, $\Phi(\arbredh)=\racine\arbrecd$.
\item Pour tout arbre $t=B_+(u)$ \`a une seule branche, except\'e $t=\arbrea$, on pose~:
\begin{equation}
\Phi(t)=\racine\racine\racine\Phi^\prec(u).
\end{equation}
\item Pour toute for\^et $f=t_1\cdots t_k \in \Cal A(\{1,\ldots,n\})$ avec $k\ge 2$ et $t_1\le\cdots\le t_k$, on pose~:
\begin{equation}
\Phi^\prec(f):=t_1\cdots t_{k-1}\Phi^\prec(t_k).
\end{equation}
\end{enumerate}
\end{defn}
}
}
\section*{Appendice A~: tableau de valeurs pour le rapport $p_kp_l/p_{kl}$}
\begin{equation*}
\begin{disarray}{c|ccccccccccccc}
l\backslash k&2&3&4&5&6&7&8
&9&10&11&12&13\\
\hline\\
2&\frac{9}{7}\\
&&&\\
3&\frac{15}{13}&\frac{25}{23}\\
&&&&&&&&&&&&&\\
4&\frac{21}{19}&\mathbf{\frac{35}{37}}&\mathbf{\frac{49}{53}}\\
&&&&&&&&&&&&&\\
5&\frac{33}{29}&\frac{55}{47}&\frac{77}{71}&\frac{121}{97}\\
&&&&&&&&&&&&&\\
6&\frac{39}{37}&\frac{65}{61}&\frac{91}{89}&\frac{143}{113}&\frac{169}{151}\\
&&&&&&&&&&&&&\\
7&\frac{51}{43}&\frac{85}{73}&\frac{119}{107}&\frac{187}{149}&\frac{221}{181}&\frac{289}{227}\\
&&&&&&&&&&&&&\\
8&\frac{57}{53}&\frac{95}{89}&\frac{133}{131}&\frac{209}{173}&\frac{247}{223}&\frac{323}{263}&\frac{361}{311}\\
&&&&&&&&&&&&&\\
9&\frac {69}{61}&\frac{115}{103}&\frac{161}{151}&\frac{253}{197}&\frac{299}{251}&\frac{391}{307}&\frac{437}{359}&\frac{529}{419}\\
&&&&&&&&&&&&&\\
10&\frac{87}{71}&\frac{145}{113}&\frac{203}{173}&\frac{319}{229}&\frac{377}{281}&\frac{493}{349}&\frac{551}{409}&\frac{621}{463}&\frac{841}{541}\\
&&&&&&&&&&&&&\\
11&\frac{93}{79}&\frac{165}{137}&\frac{217}{193}&\frac{341}{257}&\frac{403}{317}&\frac{527}{389}
&\frac{589}{457}&\frac{713}{523}&\frac{899}{601}&\frac{961}{661}\\
&&&&&&&&&&&&&\\
12&\frac{111}{89}&\frac{185}{151}&\frac{259}{223}&\frac{407}{281}&\frac{481}{359}&\frac{629}{433}
&\frac{703}{503}&\frac{851}{593}&\frac{1073}{659}&\frac{1147}{739}&\frac{1369}{827}\\
&&&&&&&&&&&&&\\
13&\frac{123}{101}&\frac{205}{167}&\frac{287}{239}&\frac{451}{313}&\frac{533}{397}&\frac{697}{467}&\frac{779}{557}&\frac{943}{643}&\frac{1189}{733}&\frac{1271}{823}&\frac{1517}{911}
&\frac{1681}{1009}\\
&&&&&\\
14&\frac {129}{107}&\frac{215}{181}&\frac{301}{263}&&\frac{559}{433}\\
&&&\\
15&\frac{141}{113}&\frac{235}{197}&\frac{329}{281}\\
&&&\\
16&\frac{159}{131}&\frac{265}{223}&\frac{371}{311}\\
&&&\\
17&\frac{177}{139}&\frac{295}{233}&\frac{413}{337}\\
&&&\\
18&\frac{183}{151}&\frac{305}{251}&\frac{427}{359}\\
&&&\\
19&\frac{201}{163}&\frac{335}{269}&\frac{469}{383}\\
&&&\\
20&\frac{213}{173}&\frac{355}{281}&\frac{497}{409}\\
&&&\\
21&\frac{219}{181}&\frac{365}{307}&\frac{511}{433}\\
&&&\\
22&&&\frac{553}{457}\\
&&&\\
23&&&\frac{581}{479}
\end{disarray}
\end{equation*}
\eject
\section*{Appendice B~: arborification des nombres entiers de 1 \`a 1000}
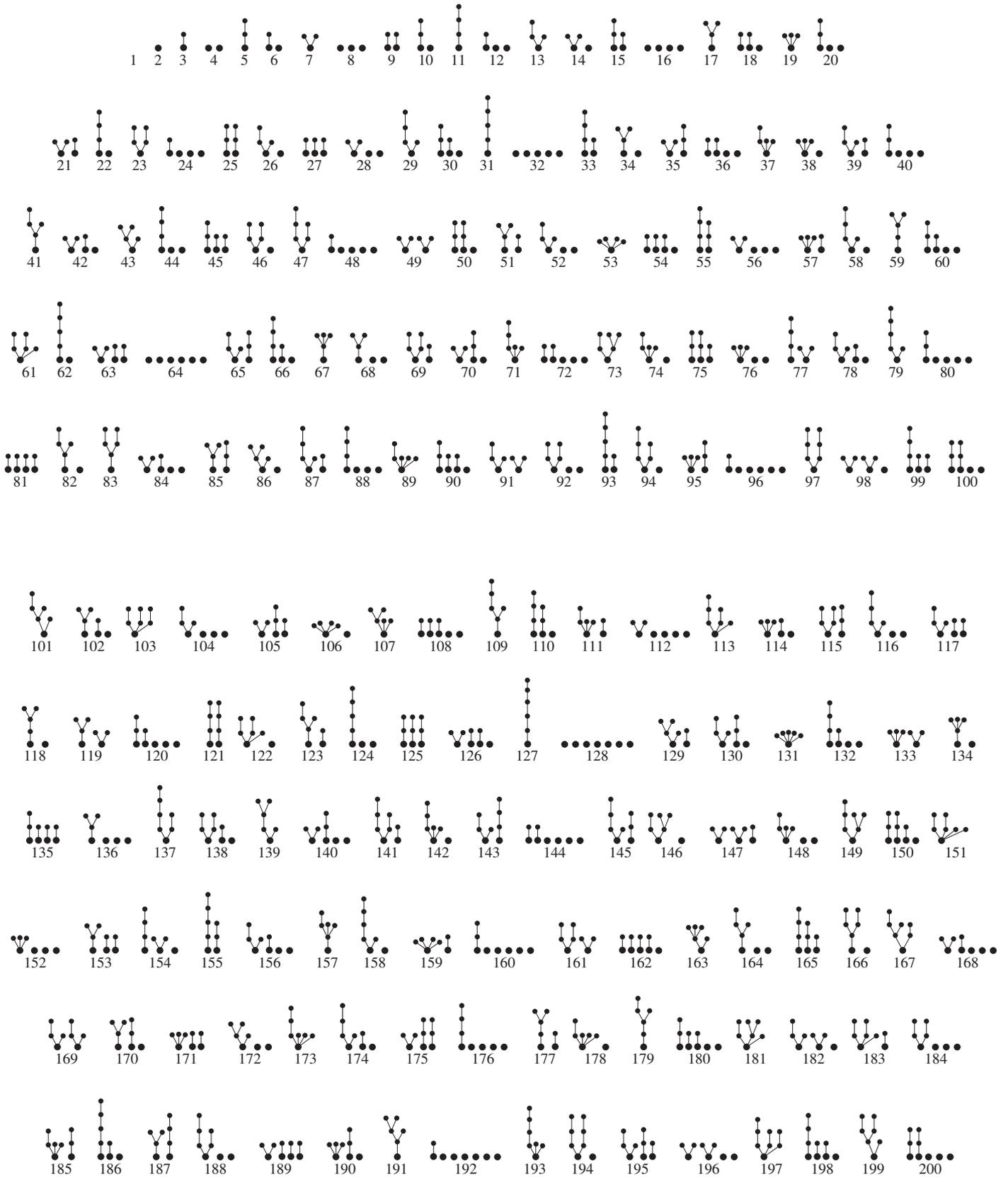
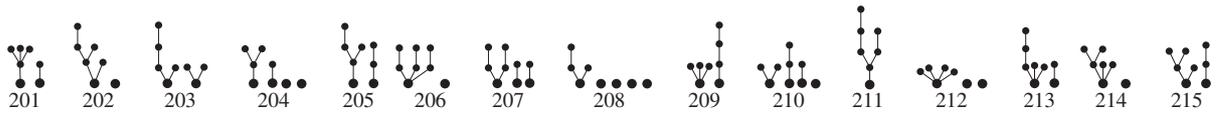
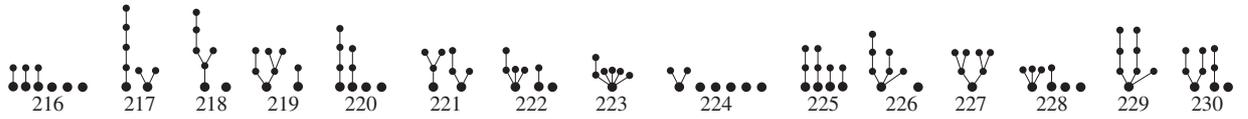
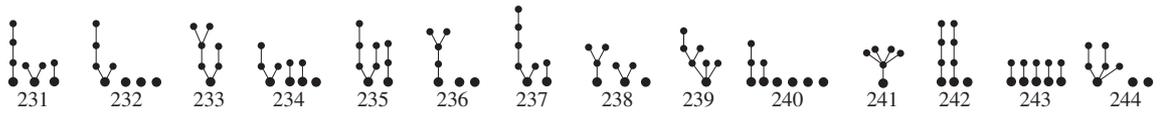
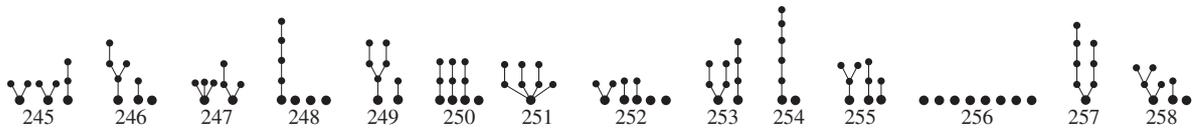
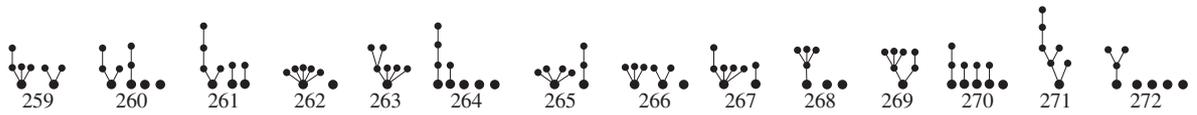
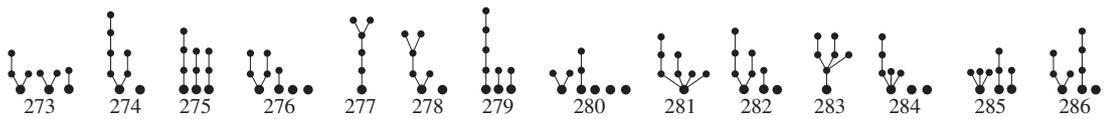
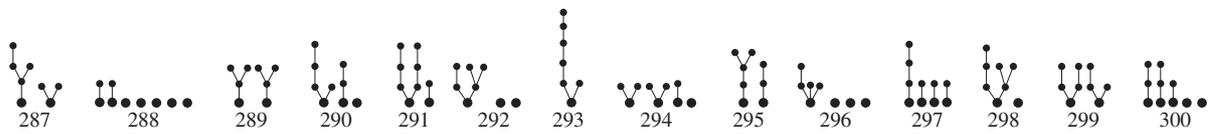
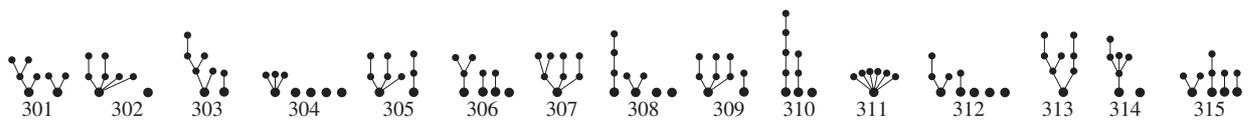
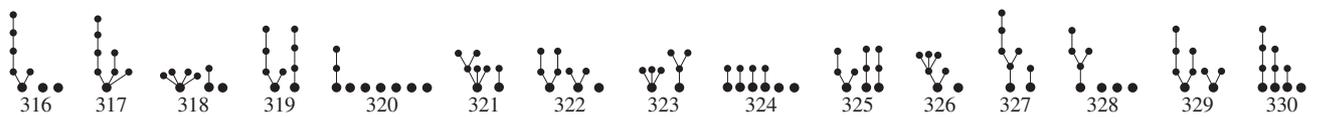
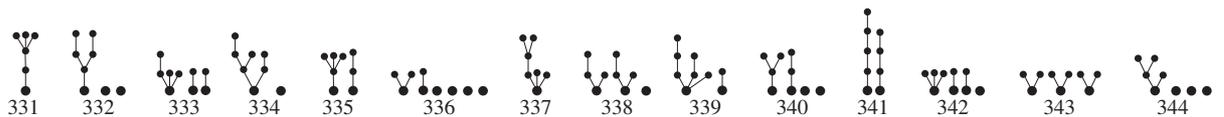
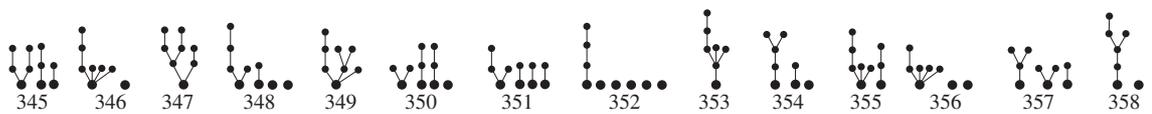
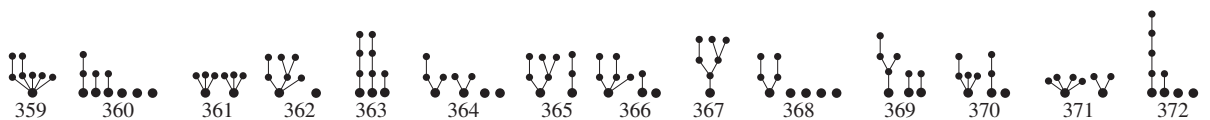
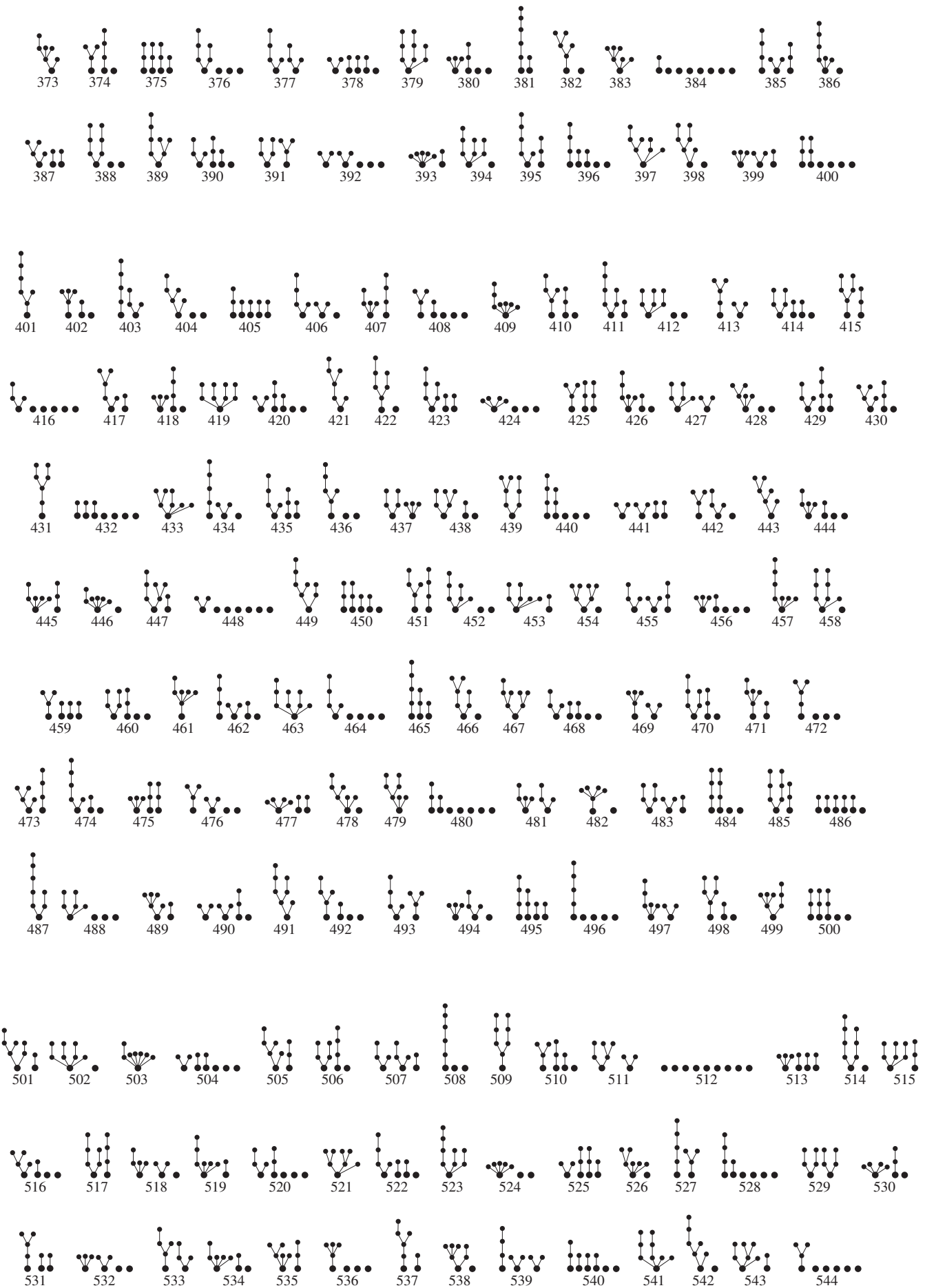
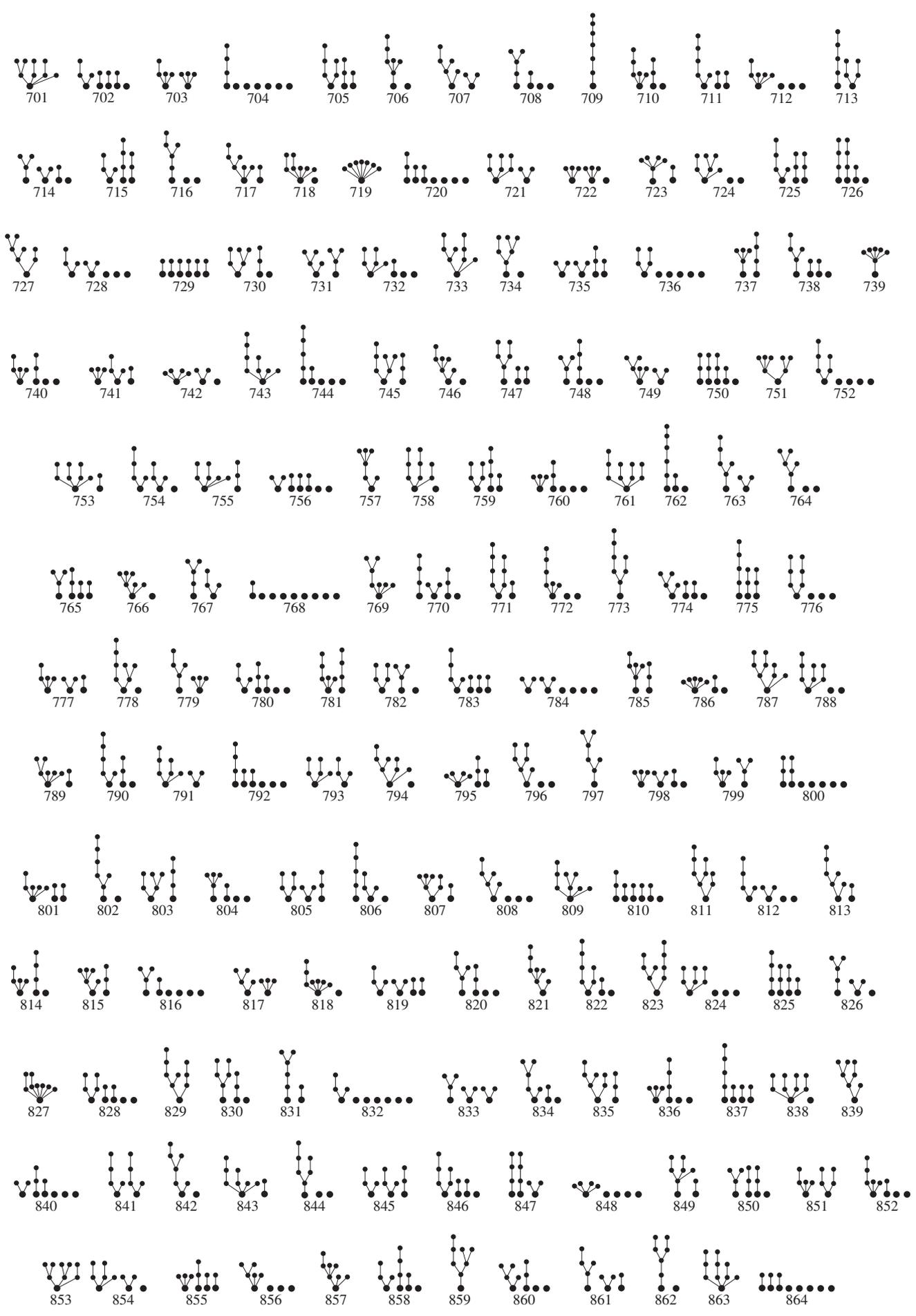
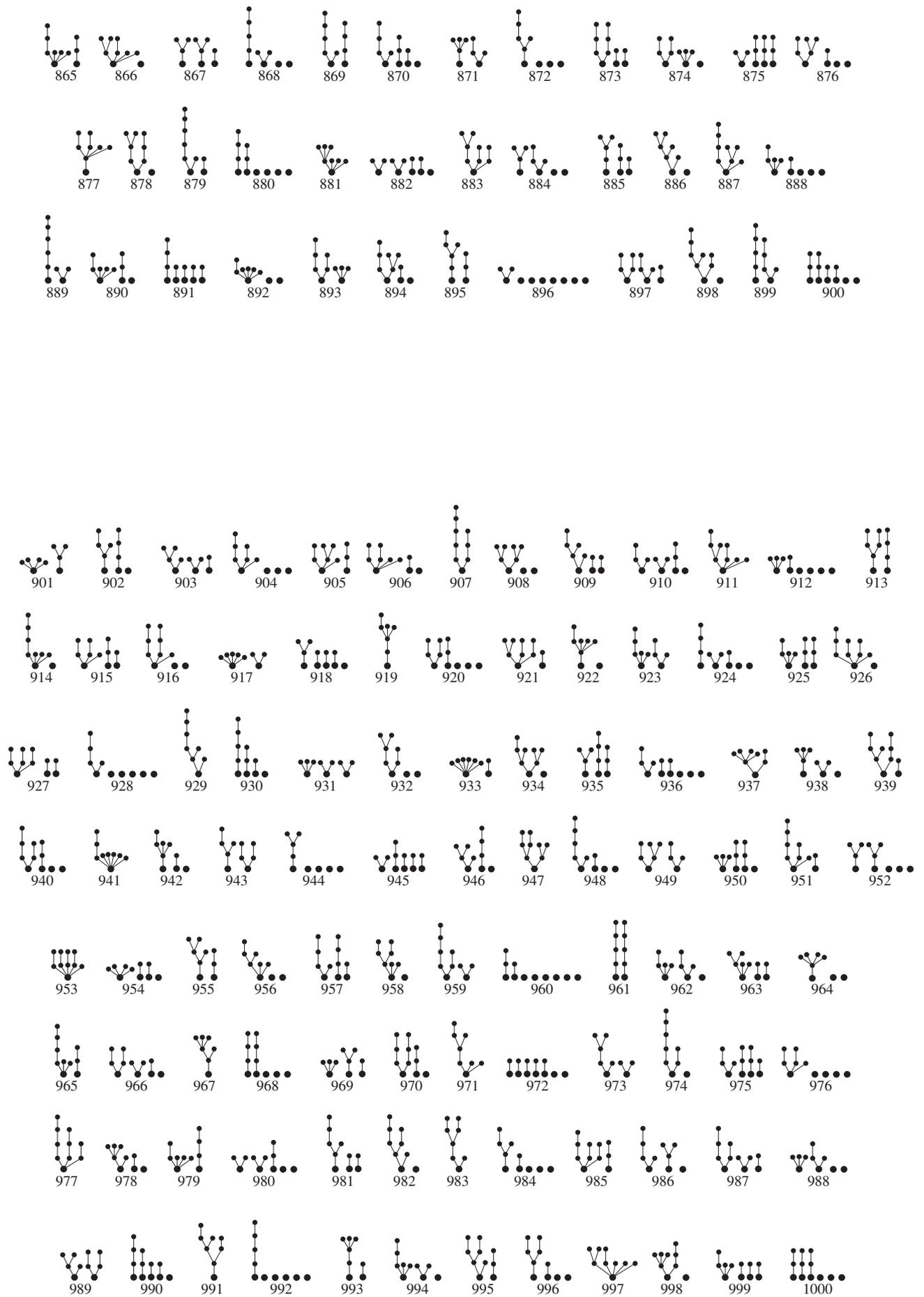
\def\Ma{\hskip 2mm\genfrac{}{}{0pt}{1}{}{1}}
\def\Mb{\hskip 2mm\genfrac{}{}{0pt}{1}{\matulab}{2}}
\def\Mc{\hskip 2mm\genfrac{}{}{0pt}{1}{\matulac}{3}}
\def\Md{\hskip 2mm\genfrac{}{}{0pt}{1}{\matulab\matulab}{4}}
\def\Me{\hskip 2mm\genfrac{}{}{0pt}{1}{\matulae}{5}}
\def\Mf{\hskip 2mm\genfrac{}{}{0pt}{1}{\matulac\matulab}{6}}
\def\Mg{\hskip 2mm\genfrac{}{}{0pt}{1}{\matulag}{7}}
\def\Mh{\hskip 2mm\genfrac{}{}{0pt}{1}{\matulab\matulab\matulab}{8}}
\def\Mi{\hskip 2mm\genfrac{}{}{0pt}{1}{\matulac\matulac}{9}}
\def\Mao{\hskip 2mm\genfrac{}{}{0pt}{1}{\matulae\matulab}{10}}
\def\Maa{\hskip 2mm\genfrac{}{}{0pt}{1}{\matulaaa}{11}}
\def\Mab{\hskip 2mm\genfrac{}{}{0pt}{1}{\matulac\matulab\matulab}{12}}
\def\Mac{\hskip 2mm\genfrac{}{}{0pt}{1}{\matulaac}{13}}
\def\Mad{\hskip 2mm\genfrac{}{}{0pt}{1}{\matulag\matulab}{14}}
\def\Mae{\hskip 2mm\genfrac{}{}{0pt}{1}{\matulae\matulac}{15}}
\def\Maf{\hskip 2mm\genfrac{}{}{0pt}{1}{\matulab\matulab\matulab\matulab}{16}}
\def\Mag{\hskip 2mm\genfrac{}{}{0pt}{1}{\matulaag}{17}}
\def\Mah{\hskip 2mm\genfrac{}{}{0pt}{1}{\matulac\matulac\matulab}{18}}
\def\Mai{\hskip 2mm\genfrac{}{}{0pt}{1}{\matulaai}{19}}
\def\Mbo{\hskip 2mm\genfrac{}{}{0pt}{1}{\matulae\matulab\matulab}{20}}

\def\Mba{\hskip 2mm\genfrac{}{}{0pt}{1}{\matulag\matulac}{21}}
\def\Mbb{\hskip 2mm\genfrac{}{}{0pt}{1}{\matulaaa\matulab}{22}}
\def\Mbc{\hskip 2mm\genfrac{}{}{0pt}{1}{\matulabc}{23}}
\def\Mbd{\hskip 2mm\genfrac{}{}{0pt}{1}{\matulac\matulab\matulab\matulab}{24}}
\def\Mbe{\hskip 2mm\genfrac{}{}{0pt}{1}{\matulae\matulae}{25}}
\def\Mbf{\hskip 2mm\genfrac{}{}{0pt}{1}{\matulaac\matulab}{26}}
\def\Mbg{\hskip 2mm\genfrac{}{}{0pt}{1}{\matulac\matulac\matulac}{27}}
\def\Mbh{\hskip 2mm\genfrac{}{}{0pt}{1}{\matulag\matulab\matulab}{28}}
\def\Mbi{\hskip 2mm\genfrac{}{}{0pt}{1}{\matulabi}{29}}
\def\Mco{\hskip 2mm\genfrac{}{}{0pt}{1}{\matulae\matulac\matulab}{30}}
\def\Mca{\hskip 2mm\genfrac{}{}{0pt}{1}{\matulaca}{31}}
\def\Mcb{\hskip 2mm\genfrac{}{}{0pt}{1}{\matulab\matulab\matulab\matulab\matulab}{32}}
\def\Mcc{\hskip 2mm\genfrac{}{}{0pt}{1}{\matulaaa\matulac}{33}}
\def\Mcd{\hskip 2mm\genfrac{}{}{0pt}{1}{\matulaag\matulab}{34}}
\def\Mce{\hskip 2mm\genfrac{}{}{0pt}{1}{\matulag\matulae}{35}}
\def\Mcf{\hskip 2mm\genfrac{}{}{0pt}{1}{\matulac\matulac\matulab\matulab}{36}}
\def\Mcg{\hskip 2mm\genfrac{}{}{0pt}{1}{\matulacg}{37}}
\def\Mch{\hskip 2mm\genfrac{}{}{0pt}{1}{\matulaai\matulab}{38}}
\def\Mci{\hskip 2mm\genfrac{}{}{0pt}{1}{\matulaac\matulac}{39}}
\def\Mdo{\hskip 2mm\genfrac{}{}{0pt}{1}{\matulae\matulab\matulab\matulab}{40}}

\def\Mda{\hskip 2mm\genfrac{}{}{0pt}{1}{\matulada}{41}}
\def\Mdb{\hskip 2mm\genfrac{}{}{0pt}{1}{\matulag\matulac\matulab}{42}}
\def\Mdc{\hskip 2mm\genfrac{}{}{0pt}{1}{\matuladc}{43}}
\def\Mdd{\hskip 2mm\genfrac{}{}{0pt}{1}{\matulaaa\matulab\matulab}{44}}
\def\Mde{\hskip 2mm\genfrac{}{}{0pt}{1}{\matulae\matulac\matulac}{45}}
\def\Mdf{\hskip 2mm\genfrac{}{}{0pt}{1}{\matulabc\matulab}{46}}
\def\Mdg{\hskip 2mm\genfrac{}{}{0pt}{1}{\matuladg}{47}}
\def\Mdh{\hskip 2mm\genfrac{}{}{0pt}{1}{\matulac\matulab\matulab\matulab\matulab}{48}}
\def\Mdi{\hskip 2mm\genfrac{}{}{0pt}{1}{\matulag\matulag}{49}}
\def\Meo{\hskip 2mm\genfrac{}{}{0pt}{1}{\matulae\matulae\matulab}{50}}
\def\Mea{\hskip 2mm\genfrac{}{}{0pt}{1}{\matulaag\matulac}{51}}
\def\Meb{\hskip 2mm\genfrac{}{}{0pt}{1}{\matulaac\matulab\matulab}{52}}
\def\Mec{\hskip 2mm\genfrac{}{}{0pt}{1}{\matulaec}{53}}
\def\Med{\hskip 2mm\genfrac{}{}{0pt}{1}{\matulac\matulac\matulac\matulab}{54}} 
\def\Mee{\hskip 2mm\genfrac{}{}{0pt}{1}{\matulaaa\matulae}{55}}
\def\Mef{\hskip 2mm\genfrac{}{}{0pt}{1}{\matulag\matulab\matulab\matulab}{56}} 
\def\Meg{\hskip 2mm\genfrac{}{}{0pt}{1}{\matulaai\matulac}{57}} 
\def\Meh{\hskip 2mm\genfrac{}{}{0pt}{1}{\matulabi\matulab}{58}} 
\def\Mei{\hskip 2mm\genfrac{}{}{0pt}{1}{\matulaei}{59}}
\def\Mfo{\hskip 2mm\genfrac{}{}{0pt}{1}{\matulae\matulac\matulab\matulab}{60}}

\def\Mfa{\hskip 2mm\genfrac{}{}{0pt}{1}{\matulafa}{61}}
\def\Mfb{\hskip 2mm\genfrac{}{}{0pt}{1}{\matulaca\matulab}{62}}
\def\Mfc{\hskip 2mm\genfrac{}{}{0pt}{1}{\matulag\matulac\matulac}{63}}
\def\Mfd{\hskip 2mm\genfrac{}{}{0pt}{1}{\matulab\matulab\matulab\matulab\matulab\matulab}{64}}
\def\Mfe{\hskip 2mm\genfrac{}{}{0pt}{1}{\matulaac\matulae}{65}}
\def\Mff{\hskip 2mm\genfrac{}{}{0pt}{1}{\matulaaa\matulac\matulab}{66}}
\def\Mfg{\hskip 2mm\genfrac{}{}{0pt}{1}{\matulafg}{67}}
\def\Mfh{\hskip 2mm\genfrac{}{}{0pt}{1}{\matulaag\matulab\matulab}{68}}
\def\Mfi{\hskip 2mm\genfrac{}{}{0pt}{1}{\matulabc\matulac}{69}}
\def\Mgo{\hskip 2mm\genfrac{}{}{0pt}{1}{\matulag\matulae\matulab}{70}}
\def\Mga{\hskip 2mm\genfrac{}{}{0pt}{1}{\matulaga}{71}}
\def\Mgb{\hskip 2mm\genfrac{}{}{0pt}{1}{\matulac\matulac\matulab\matulab\matulab}{72}}
\def\Mgc{\hskip 2mm\genfrac{}{}{0pt}{1}{\matulagc}{73}}
\def\Mgd{\hskip 2mm\genfrac{}{}{0pt}{1}{\matulacg\matulab}{74}}
\def\Mge{\hskip 2mm\genfrac{}{}{0pt}{1}{\matulae\matulae\matulac}{75}}
\def\Mgf{\hskip 2mm\genfrac{}{}{0pt}{1}{\matulaai\matulab\matulab}{76}}
\def\Mgg{\hskip 2mm\genfrac{}{}{0pt}{1}{\matulaaa\matulag}{77}}
\def\Mgh{\hskip 2mm\genfrac{}{}{0pt}{1}{\matulaac\matulac\matulab}{78}}
\def\Mgi{\hskip 2mm\genfrac{}{}{0pt}{1}{\matulagi}{79}}
\def\Mho{\hskip 2mm\genfrac{}{}{0pt}{1}{\matulae\matulab\matulab\matulab\matulab}{80}}
\def\Mha{\hskip 2mm\genfrac{}{}{0pt}{1}{\matulac\matulac\matulac\matulac}{81}}
\def\Mhb{\hskip 2mm\genfrac{}{}{0pt}{1}{\matulada\matulab}{82}}
\def\Mhc{\hskip 2mm\genfrac{}{}{0pt}{1}{\matulahc}{83}}
\def\Mhd{\hskip 2mm\genfrac{}{}{0pt}{1}{\matulag\matulac\matulab\matulab}{84}}
\def\Mhe{\hskip 2mm\genfrac{}{}{0pt}{1}{\matulaag\matulae}{85}}
\def\Mhf{\hskip 2mm\genfrac{}{}{0pt}{1}{\matuladc\matulab}{86}}
\def\Mhg{\hskip 2mm\genfrac{}{}{0pt}{1}{\matulabi\matulac}{87}}
\def\Mhh{\hskip 2mm\genfrac{}{}{0pt}{1}{\matulaaa\matulab\matulab\matulab}{88}}
\def\Mhi{\hskip 2mm\genfrac{}{}{0pt}{1}{\matulahi}{89}}
\def\Mio{\hskip 2mm\genfrac{}{}{0pt}{1}{\matulae\matulac\matulac\matulab}{90}}
\def\Mia{\hskip 2mm\genfrac{}{}{0pt}{1}{\matulaac\matulag}{91}}
\def\Mib{\hskip 2mm\genfrac{}{}{0pt}{1}{\matulabc\matulab\matulab}{92}}
\def\Mic{\hskip 2mm\genfrac{}{}{0pt}{1}{\matulaca\matulac}{93}}
\def\Mid{\hskip 2mm\genfrac{}{}{0pt}{1}{\matuladg\matulab}{94}}
\def\Mie{\hskip 2mm\genfrac{}{}{0pt}{1}{\matulaai\matulae}{95}}
\def\Mif{\hskip 2mm\genfrac{}{}{0pt}{1}{\matulac\matulab\matulab\matulab\matulab\matulab}{96}}
\def\Mig{\hskip 2mm\genfrac{}{}{0pt}{1}{\matulaig}{97}}
\def\Mih{\hskip 2mm\genfrac{}{}{0pt}{1}{\matulag\matulag\matulab}{98}}
\def\Mii{\hskip 2mm\genfrac{}{}{0pt}{1}{\matulaaa\matulac\matulac}{99}}
\def\Maoo{\hskip 2mm\genfrac{}{}{0pt}{1}{\matulae\matulae\matulab\matulab}{100}}
\def\Maoa{\hskip 3mm\genfrac{}{}{0pt}{1}{\matulaaoa}{101}}
\def\Maob{\hskip 3mm\genfrac{}{}{0pt}{1}{\matulaag\matulac\matulab}{102}}
\def\Maoc{\hskip 3mm\genfrac{}{}{0pt}{1}{\matulaaoc}{103}}
\def\Maod{\hskip 3mm\genfrac{}{}{0pt}{1}{\matulaac\matulab\matulab\matulab}{104}}
\def\Maoe{\hskip 3mm\genfrac{}{}{0pt}{1}{\matulag\matulae\matulac}{105}}
\def\Maof{\hskip 3mm\genfrac{}{}{0pt}{1}{\matulaec\matulab}{106}}
\def\Maog{\hskip 3mm\genfrac{}{}{0pt}{1}{\matulaaog}{107}}
\def\Maoh{\hskip 3mm\genfrac{}{}{0pt}{1}{\matulac\matulac\matulac\matulab\matulab}{108}}
\def\Maoi{\hskip 3mm\genfrac{}{}{0pt}{1}{\matulaaoi}{109}}
\def\Maao{\hskip 3mm\genfrac{}{}{0pt}{1}{\matulaaa\matulae\matulab}{110}}
\def\Maaa{\hskip 3mm\genfrac{}{}{0pt}{1}{\matulacg\matulac}{111}}
\def\Maab{\hskip 3mm\genfrac{}{}{0pt}{1}{\matulag\matulab\matulab\matulab\matulab}{112}}
\def\Maac{\hskip 3mm\genfrac{}{}{0pt}{1}{\matulaaac}{113} }
\def\Maad{\hskip 3mm\genfrac{}{}{0pt}{1}{\matulaai\matulac\matulab}{114}}
\def\Maae{\hskip 3mm\genfrac{}{}{0pt}{1}{\matulabc\matulae}{115}}
\def\Maaf{\hskip 3mm\genfrac{}{}{0pt}{1}{\matulabi\matulab\matulab}{116}}
\def\Maag{\hskip 3mm\genfrac{}{}{0pt}{1}{\matulaac\matulac\matulac}{117}}

\def\Maah{\hskip 3mm\genfrac{}{}{0pt}{1}{\matulaei\matulab}{118}}
\def\Maai{\hskip 3mm\genfrac{}{}{0pt}{1}{\matulaag\matulag}{119}}
\def\Mabo{\hskip 3mm\genfrac{}{}{0pt}{1}{\matulae\matulac\matulab\matulab\matulab}{120}}
\def\Maba{\hskip 3mm\genfrac{}{}{0pt}{1}{\matulaaa\matulaaa}{121}}
\def\Mabb{\hskip 3mm\genfrac{}{}{0pt}{1}{\matulafa\matulab}{122}}
\def\Mabc{\hskip 3mm\genfrac{}{}{0pt}{1}{\matulada\matulac}{123}}
\def\Mabd{\hskip 3mm\genfrac{}{}{0pt}{1}{\matulaca\matulab\matulab}{124}}
\def\Mabe{\hskip 3mm\genfrac{}{}{0pt}{1}{\matulae\matulae\matulae}{125}}
\def\Mabf{\hskip 3mm\genfrac{}{}{0pt}{1}{\matulag\matulac\matulac\matulab}{126}}
\def\Mabg{\hskip 3mm\genfrac{}{}{0pt}{1}{\matulaabg}{127}}
\def\Mabh{\hskip 3mm\genfrac{}{}{0pt}{1}{\matulab\matulab\matulab\matulab\matulab\matulab\matulab}{128}}
\def\Mabi{\hskip 3mm\genfrac{}{}{0pt}{1}{\matuladc\matulac}{129}}
\def\Maco{\hskip 3mm\genfrac{}{}{0pt}{1}{\matulaac\matulae\matulab}{130}}
\def\Maca{\hskip 3mm\genfrac{}{}{0pt}{1}{\matulaaca}{131}}
\def\Macb{\hskip 3mm\genfrac{}{}{0pt}{1}{\matulaaa\matulac\matulab\matulab}{132}}
\def\Macc{\hskip 3mm\genfrac{}{}{0pt}{1}{\matulaai\matulag}{133}}
\def\Macd{\hskip 3mm\genfrac{}{}{0pt}{1}{\matulafg\matulab}{134}}

\def\Mace{\hskip 3mm\genfrac{}{}{0pt}{1}{\matulae\matulac\matulac\matulac}{135}}
\def\Macf{\hskip 3mm\genfrac{}{}{0pt}{1}{\matulaag\matulab\matulab\matulab}{136}}
\def\Macg{\hskip 3mm\genfrac{}{}{0pt}{1}{\matulaacg}{137}}
\def\Mach{\hskip 3mm\genfrac{}{}{0pt}{1}{\matulabc\matulac\matulab}{138}}
\def\Maci{\hskip 3mm\genfrac{}{}{0pt}{1}{\matulaaci}{139}}
\def\Mado{\hskip 3mm\genfrac{}{}{0pt}{1}{\matulag\matulae\matulab\matulab}{140}}
\def\Mada{\hskip 3mm\genfrac{}{}{0pt}{1}{\matuladg\matulac}{141}}
\def\Madb{\hskip 3mm\genfrac{}{}{0pt}{1}{\matulaga\matulab}{142}}
\def\Madc{\hskip 3mm\genfrac{}{}{0pt}{1}{\matulaac\matulaaa}{143}}
\def\Madd{\hskip 3mm\genfrac{}{}{0pt}{1}{\matulac\matulac\matulab\matulab\matulab\matulab}{144}}
\def\Made{\hskip 3mm\genfrac{}{}{0pt}{1}{\matulabi\matulae}{145}}
\def\Madf{\hskip 3mm\genfrac{}{}{0pt}{1}{\matulagc\matulab}{146}}
\def\Madg{\hskip 3mm\genfrac{}{}{0pt}{1}{\matulag\matulag\matulac}{147}}
\def\Madh{\hskip 3mm\genfrac{}{}{0pt}{1}{\matulacg\matulab\matulab}{148}}
\def\Madi{\hskip 3mm\genfrac{}{}{0pt}{1}{\matulaadi}{149}}
\def\Maeo{\hskip 3mm\genfrac{}{}{0pt}{1}{\matulae\matulae\matulac\matulab}{150}}
\def\Maea{\hskip 3mm\genfrac{}{}{0pt}{1}{\matulaaea}{151}}

\def\Maeb{\hskip 3mm\genfrac{}{}{0pt}{1}{\matulaai\matulab\matulab\matulab}{152}}
\def\Maec{\hskip 3mm\genfrac{}{}{0pt}{1}{\matulaag\matulac\matulac}{153}}
\def\Maed{\hskip 3mm\genfrac{}{}{0pt}{1}{\matulaaa\matulag\matulab}{154}}
\def\Maee{\hskip 3mm\genfrac{}{}{0pt}{1}{\matulaca\matulae}{155}}
\def\Maef{\hskip 3mm\genfrac{}{}{0pt}{1}{\matulaac\matulac\matulab\matulab}{156}}
\def\Maeg{\hskip 3mm\genfrac{}{}{0pt}{1}{\matulaaeg}{157}}
\def\Maeh{\hskip 3mm\genfrac{}{}{0pt}{1}{\matulagi\matulab}{158}}
\def\Maei{\hskip 3mm\genfrac{}{}{0pt}{1}{\matulaec\matulac}{159}}
\def\Mafo{\hskip 3mm\genfrac{}{}{0pt}{1}{\matulae\matulab\matulab\matulab\matulab\matulab}{160}}
\def\Mafa{\hskip 3mm\genfrac{}{}{0pt}{1}{\matulabc\matulag}{161}}
\def\Mafb{\hskip 3mm\genfrac{}{}{0pt}{1}{\matulac\matulac\matulac\matulac\matulab}{162}}
\def\Mafc{\hskip 3mm\genfrac{}{}{0pt}{1}{\matulaafc}{163}}
\def\Mafd{\hskip 3mm\genfrac{}{}{0pt}{1}{\matulada\matulab\matulab}{164}}
\def\Mafe{\hskip 3mm\genfrac{}{}{0pt}{1}{\matulaaa\matulae\matulac}{165}}
\def\Maff{\hskip 3mm\genfrac{}{}{0pt}{1}{\matulahc\matulab}{166}}
\def\Mafg{\hskip 3mm\genfrac{}{}{0pt}{1}{\matulaafg}{167}}
\def\Mafh{\hskip 3mm\genfrac{}{}{0pt}{1}{\matulag\matulac\matulab\matulab\matulab}{168}}

\def\Mafi{\hskip 3mm\genfrac{}{}{0pt}{1}{\matulaac\matulaac}{169}}
\def\Mago{\hskip 3mm\genfrac{}{}{0pt}{1}{\matulaag\matulae\matulab}{170}}
\def\Maga{\hskip 3mm\genfrac{}{}{0pt}{1}{\matulaai\matulac\matulac}{171}}
\def\Magb{\hskip 3mm\genfrac{}{}{0pt}{1}{\matuladc\matulab\matulab}{172}}
\def\Magc{\hskip 3mm\genfrac{}{}{0pt}{1}{\matulaagc}{173}}
\def\Magd{\hskip 3mm\genfrac{}{}{0pt}{1}{\matulabi\matulac\matulab}{174}}
\def\Mage{\hskip 3mm\genfrac{}{}{0pt}{1}{\matulag\matulae\matulae}{175}}
\def\Magf{\hskip 3mm\genfrac{}{}{0pt}{1}{\matulaaa\matulab\matulab\matulab\matulab}{176}}
\def\Magg{\hskip 3mm\genfrac{}{}{0pt}{1}{\matulaei\matulac}{177}}
\def\Magh{\hskip 3mm\genfrac{}{}{0pt}{1}{\matulahi\matulab}{178}}
\def\Magi{\hskip 3mm\genfrac{}{}{0pt}{1}{\matulaagi}{179}}
\def\Maho{\hskip 3mm\genfrac{}{}{0pt}{1}{\matulae\matulac\matulac\matulab\matulab}{180}}
\def\Maha{\hskip 3mm\genfrac{}{}{0pt}{1}{\matulaaha}{181}}
\def\Mahb{\hskip 3mm\genfrac{}{}{0pt}{1}{\matulaac\matulag\matulab}{182}}
\def\Mahc{\hskip 3mm\genfrac{}{}{0pt}{1}{\matulafa\matulac}{183}}
\def\Mahd{\hskip 3mm\genfrac{}{}{0pt}{1}{\matulabc\matulab\matulab\matulab}{184}}

\def\Mahe{\hskip 3mm\genfrac{}{}{0pt}{1}{\matulacg\matulae}{185}}
\def\Mahf{\hskip 3mm\genfrac{}{}{0pt}{1}{\matulaca\matulac\matulab}{186}}
\def\Mahg{\hskip 3mm\genfrac{}{}{0pt}{1}{\matulaag\matulaaa}{187}}
\def\Mahh{\hskip 3mm\genfrac{}{}{0pt}{1}{\matuladg\matulab\matulab}{188}}
\def\Mahi{\hskip 3mm\genfrac{}{}{0pt}{1}{\matulag\matulac\matulac\matulac}{189}}
\def\Maio{\hskip 3mm\genfrac{}{}{0pt}{1}{\matulaai\matulae\matulab}{190}}
\def\Maia{\hskip 3mm\genfrac{}{}{0pt}{1}{\matulaaia}{191}} 
\def\Maib{\hskip 3mm\genfrac{}{}{0pt}{1}{\matulac\matulab\matulab\matulab\matulab\matulab\matulab}{192}}
\def\Maic{\hskip 3mm\genfrac{}{}{0pt}{1}{\matulaaic}{193}} 
\def\Maid{\hskip 3mm\genfrac{}{}{0pt}{1}{\matulaig\matulab}{194}}
\def\Maie{\hskip 3mm\genfrac{}{}{0pt}{1}{\matulaac\matulae\matulac}{195}}
\def\Maif{\hskip 3mm\genfrac{}{}{0pt}{1}{\matulag\matulag\matulab\matulab}{196}}
\def\Maig{\hskip 3mm\genfrac{}{}{0pt}{1}{\matulaaig}{197}} 
\def\Maih{\hskip 3mm\genfrac{}{}{0pt}{1}{\matulaaa\matulac\matulac\matulab}{198}}
\def\Maii{\hskip 3mm\genfrac{}{}{0pt}{1}{\matulaaii}{199}}
\def\Mboo{\hskip 3mm\genfrac{}{}{0pt}{1}{\matulae\matulae\matulab\matulab\matulab}{200}}

\def\Mboa{\hskip 3mm\genfrac{}{}{0pt}{1}{\matulafg\matulac}{201}}
\def\Mbob{\hskip 3mm\genfrac{}{}{0pt}{1}{\matulaaoa\matulab}{202}}
\def\Mboc{\hskip 3mm\genfrac{}{}{0pt}{1}{\matulabi\matulag}{203}}
\def\Mbod{\hskip 3mm\genfrac{}{}{0pt}{1}{\matulaag\matulac\matulab\matulab}{204}}
\def\Mboe{\hskip 3mm\genfrac{}{}{0pt}{1}{\matulada\matulae}{205}}
\def\Mbof{\hskip 3mm\genfrac{}{}{0pt}{1}{\matulaaoc\matulab}{206}}
\def\Mbog{\hskip 3mm\genfrac{}{}{0pt}{1}{\matulabc\matulac\matulac}{207}}
\def\Mboh{\hskip 3mm\genfrac{}{}{0pt}{1}{\matulaac\matulab\matulab\matulab\matulab}{208}}
\def\Mboi{\hskip 3mm\genfrac{}{}{0pt}{1}{\matulaai\matulaaa}{209}}
\def\Mbao{\hskip 3mm\genfrac{}{}{0pt}{1}{\matulag\matulae\matulac\matulab}{210}}
\def\Mbaa{\hskip 3mm\genfrac{}{}{0pt}{1}{\matulabaa}{211}} 
\def\Mbab{\hskip 3mm\genfrac{}{}{0pt}{1}{\matulaec\matulab\matulab}{212}}
\def\Mbac{\hskip 3mm\genfrac{}{}{0pt}{1}{\matulaga\matulac}{213}}
\def\Mbad{\hskip 3mm\genfrac{}{}{0pt}{1}{\matulaaog\matulab}{214}}
\def\Mbae{\hskip 3mm\genfrac{}{}{0pt}{1}{\matuladc\matulae}{215}}

\def\Mbaf{\hskip 3mm\genfrac{}{}{0pt}{1}{\matulac\matulac\matulac\matulab\matulab\matulab}{216}}
\def\Mbag{\hskip 3mm\genfrac{}{}{0pt}{1}{\matulaca\matulag}{217}} 
\def\Mbah{\hskip 3mm\genfrac{}{}{0pt}{1}{\matulaaoi\matulab}{218}}
\def\Mbai{\hskip 3mm\genfrac{}{}{0pt}{1}{\matulagc\matulac}{219}}
\def\Mbbo{\hskip 3mm\genfrac{}{}{0pt}{1}{\matulaaa\matulae\matulab\matulab}{220}}
\def\Mbba{\hskip 3mm\genfrac{}{}{0pt}{1}{\matulaag\matulaac}{221}}
\def\Mbbb{\hskip 3mm\genfrac{}{}{0pt}{1}{\matulacg\matulac\matulab}{222}}
\def\Mbbc{\hskip 3mm\genfrac{}{}{0pt}{1}{\matulabbc}{223}}
\def\Mbbd{\hskip 3mm\genfrac{}{}{0pt}{1}{\matulag\matulab\matulab\matulab\matulab\matulab}{224}}
\def\Mbbe{\hskip 3mm\genfrac{}{}{0pt}{1}{\matulae\matulae\matulac\matulac}{225}}
\def\Mbbf{\hskip 3mm\genfrac{}{}{0pt}{1}{\matulaaac\matulab}{226}}
\def\Mbbg{\hskip 3mm\genfrac{}{}{0pt}{1}{\matulabbg}{227}}
\def\Mbbh{\hskip 3mm\genfrac{}{}{0pt}{1}{\matulaai\matulac\matulab\matulab}{228}}
\def\Mbbi{\hskip 3mm\genfrac{}{}{0pt}{1}{\matulabbi}{229}} 
\def\Mbco{\hskip 3mm\genfrac{}{}{0pt}{1}{\matulabc\matulae\matulab}{230}}

\def\Mbca{\hskip 3mm\genfrac{}{}{0pt}{1}{\matulaaa\matulag\matulac}{231}}
\def\Mbcb{\hskip 3mm\genfrac{}{}{0pt}{1}{\matulabi\matulab\matulab\matulab}{232}}
\def\Mbcc{\hskip 3mm\genfrac{}{}{0pt}{1}{\matulabcc}{233}}
\def\Mbcd{\hskip 3mm\genfrac{}{}{0pt}{1}{\matulaac\matulac\matulac\matulab}{234}}
\def\Mbce{\hskip 3mm\genfrac{}{}{0pt}{1}{\matuladg\matulae}{235}}
\def\Mbcf{\hskip 3mm\genfrac{}{}{0pt}{1}{\matulaei\matulab\matulab}{236}}
\def\Mbcg{\hskip 3mm\genfrac{}{}{0pt}{1}{\matulagi\matulac}{237}}
\def\Mbch{\hskip 3mm\genfrac{}{}{0pt}{1}{\matulaag\matulag\matulab}{238}}
\def\Mbci{\hskip 3mm\genfrac{}{}{0pt}{1}{\matulabci}{239}}
\def\Mbdo{\hskip 3mm\genfrac{}{}{0pt}{1}{\matulae\matulac\matulab\matulab\matulab\matulab}{240}}
\def\Mbda{\hskip 3mm\genfrac{}{}{0pt}{1}{\matulabda}{241}} 
\def\Mbdb{\hskip 3mm\genfrac{}{}{0pt}{1}{\matulaaa\matulaaa\matulab}{242}}
\def\Mbdc{\hskip 3mm\genfrac{}{}{0pt}{1}{\matulac\matulac\matulac\matulac\matulac}{243}}
\def\Mbdd{\hskip 3mm\genfrac{}{}{0pt}{1}{\matulafa\matulab\matulab}{244}}

\def\Mbde{\hskip 3mm\genfrac{}{}{0pt}{1}{\matulag\matulag\matulae}{245}}
\def\Mbdf{\hskip 3mm\genfrac{}{}{0pt}{1}{\matulada\matulac\matulab}{246}}
\def\Mbdg{\hskip 3mm\genfrac{}{}{0pt}{1}{\matulaai\matulaac}{247}}
\def\Mbdh{\hskip 3mm\genfrac{}{}{0pt}{1}{\matulaca\matulab\matulab\matulab}{248}}
\def\Mbdi{\hskip 3mm\genfrac{}{}{0pt}{1}{\matulahc\matulac}{249}}
\def\Mbeo{\hskip 3mm\genfrac{}{}{0pt}{1}{\matulae\matulae\matulae\matulab}{250}}
\def\Mbea{\hskip 3mm\genfrac{}{}{0pt}{1}{\matulabea}{251}} 
\def\Mbeb{\hskip 3mm\genfrac{}{}{0pt}{1}{\matulag\matulac\matulac\matulab\matulab}{252}}
\def\Mbec{\hskip 3mm\genfrac{}{}{0pt}{1}{\matulabc\matulaaa}{253}}
\def\Mbed{\hskip 3mm\genfrac{}{}{0pt}{1}{\matulaabg\matulab}{254}}
\def\Mbee{\hskip 3mm\genfrac{}{}{0pt}{1}{\matulaag\matulae\matulac}{255}}
\def\Mbef{\hskip 3mm\genfrac{}{}{0pt}{1}{\matulab\matulab\matulab\matulab\matulab\matulab\matulab\matulab}{256}}
\def\Mbeg{\hskip 3mm\genfrac{}{}{0pt}{1}{\matulabeg}{257}} 
\def\Mbeh{\hskip 3mm\genfrac{}{}{0pt}{1}{\matuladc\matulac\matulab}{258}}

\def\Mbei{\hskip 3mm\genfrac{}{}{0pt}{1}{\matulacg\matulag}{259}}
\def\Mbfo{\hskip 3mm\genfrac{}{}{0pt}{1}{\matulaac\matulae\matulab\matulab}{260}}
\def\Mbfa{\hskip 3mm\genfrac{}{}{0pt}{1}{\matulabi\matulac\matulac}{261}}
\def\Mbfb{\hskip 3mm\genfrac{}{}{0pt}{1}{\matulaaca\matulab}{262}}
\def\Mbfc{\hskip 3mm\genfrac{}{}{0pt}{1}{\matulabfc}{263}}
\def\Mbfd{\hskip 3mm\genfrac{}{}{0pt}{1}{\matulaaa\matulac\matulab\matulab\matulab}{264}}
\def\Mbfe{\hskip 3mm\genfrac{}{}{0pt}{1}{\matulaec\matulae}{265}}
\def\Mbff{\hskip 3mm\genfrac{}{}{0pt}{1}{\matulaai\matulag\matulab}{266}}
\def\Mbfg{\hskip 3mm\genfrac{}{}{0pt}{1}{\matulahi\matulac}{267}}
\def\Mbfh{\hskip 3mm\genfrac{}{}{0pt}{1}{\matulafg\matulab\matulab}{268}}
\def\Mbfi{\hskip 3mm\genfrac{}{}{0pt}{1}{\matulabfi}{269}} 
\def\Mbgo{\hskip 3mm\genfrac{}{}{0pt}{1}{\matulae\matulac\matulac\matulac\matulab}{270}}
\def\Mbga{\hskip 3mm\genfrac{}{}{0pt}{1}{\matulabga}{271}} 
\def\Mbgb{\hskip 3mm\genfrac{}{}{0pt}{1}{\matulaag\matulab\matulab\matulab\matulab}{272}}

\def\Mbgc{\hskip 3mm\genfrac{}{}{0pt}{1}{\matulaac\matulag\matulac}{273}}
\def\Mbgd{\hskip 3mm\genfrac{}{}{0pt}{1}{\matulaacg\matulab}{274}}
\def\Mbge{\hskip 3mm\genfrac{}{}{0pt}{1}{\matulaaa\matulae\matulae}{275}}
\def\Mbgf{\hskip 3mm\genfrac{}{}{0pt}{1}{\matulabc\matulac\matulab\matulab}{276}}
\def\Mbgg{\hskip 3mm\genfrac{}{}{0pt}{1}{\matulabgg}{277}}
\def\Mbgh{\hskip 3mm\genfrac{}{}{0pt}{1}{\matulaaci\matulab}{278}}
\def\Mbgi{\hskip 3mm\genfrac{}{}{0pt}{1}{\matulaca\matulac\matulac}{279}}
\def\Mbho{\hskip 3mm\genfrac{}{}{0pt}{1}{\matulag\matulae\matulab\matulab\matulab}{280}}
\def\Mbha{\hskip 3mm\genfrac{}{}{0pt}{1}{\matulabha}{281}} 
\def\Mbhb{\hskip 3mm\genfrac{}{}{0pt}{1}{\matuladg\matulac\matulab}{282}}
\def\Mbhc{\hskip 3mm\genfrac{}{}{0pt}{1}{\matulabhc}{283}} 
\def\Mbhd{\hskip 3mm\genfrac{}{}{0pt}{1}{\matulaga\matulab\matulab}{284}}
\def\Mbhe{\hskip 3mm\genfrac{}{}{0pt}{1}{\matulaai\matulae\matulac}{285}}
\def\Mbhf{\hskip 3mm\genfrac{}{}{0pt}{1}{\matulaac\matulaaa\matulab}{286}}

\def\Mbhg{\hskip 3mm\genfrac{}{}{0pt}{1}{\matulada\matulag}{287}}
\def\Mbhh{\hskip 3mm\genfrac{}{}{0pt}{1}{\matulac\matulac\matulab\matulab\matulab\matulab\matulab}{288}}
\def\Mbhi{\hskip 3mm\genfrac{}{}{0pt}{1}{\matulaag\matulaag}{289}}
\def\Mbio{\hskip 3mm\genfrac{}{}{0pt}{1}{\matulabi\matulae\matulab}{290}}
\def\Mbia{\hskip 3mm\genfrac{}{}{0pt}{1}{\matulaig\matulac}{291}}
\def\Mbib{\hskip 3mm\genfrac{}{}{0pt}{1}{\matulagc\matulab\matulab}{292}}
\def\Mbic{\hskip 3mm\genfrac{}{}{0pt}{1}{\matulabic}{293}}
\def\Mbid{\hskip 3mm\genfrac{}{}{0pt}{1}{\matulag\matulag\matulac\matulab}{294}}
\def\Mbie{\hskip 3mm\genfrac{}{}{0pt}{1}{\matulaei\matulae}{295}}
\def\Mbif{\hskip 3mm\genfrac{}{}{0pt}{1}{\matulacg\matulab\matulab\matulab}{296}}
\def\Mbig{\hskip 3mm\genfrac{}{}{0pt}{1}{\matulaaa\matulac\matulac\matulac}{297}}
\def\Mbih{\hskip 3mm\genfrac{}{}{0pt}{1}{\matulaadi\matulab}{298}}
\def\Mbii{\hskip 3mm\genfrac{}{}{0pt}{1}{\matulabc\matulaac}{299}}
\def\Mcoo{\hskip 3mm\genfrac{}{}{0pt}{1}{\matulae\matulae\matulac\matulab\matulab}{300}}

\def\Mcoa{\hskip 3mm\genfrac{}{}{0pt}{1}{\matuladc\matulag}{301}}
\def\Mcob{\hskip 3mm\genfrac{}{}{0pt}{1}{\matulaaea\matulab}{302}}
\def\Mcoc{\hskip 3mm\genfrac{}{}{0pt}{1}{\matulaaoa\matulac}{303}}
\def\Mcod{\hskip 3mm\genfrac{}{}{0pt}{1}{\matulaai\matulab\matulab\matulab\matulab}{304}}
\def\Mcoe{\hskip 3mm\genfrac{}{}{0pt}{1}{\matulafa\matulae}{305}}
\def\Mcof{\hskip 3mm\genfrac{}{}{0pt}{1}{\matulaag\matulac\matulac\matulab}{306}}
\def\Mcog{\hskip 3mm\genfrac{}{}{0pt}{1}{\matulacog}{307}}
\def\Mcoh{\hskip 3mm\genfrac{}{}{0pt}{1}{\matulaaa\matulag\matulab\matulab}{308}}
\def\Mcoi{\hskip 3mm\genfrac{}{}{0pt}{1}{\matulaaoc\matulac}{309}}
\def\Mcao{\hskip 3mm\genfrac{}{}{0pt}{1}{\matulaca\matulae\matulab}{310}}
\def\Mcaa{\hskip 3mm\genfrac{}{}{0pt}{1}{\matulacaa}{311}}
\def\Mcab{\hskip 3mm\genfrac{}{}{0pt}{1}{\matulaac\matulac\matulab\matulab\matulab}{312}}
\def\Mcac{\hskip 3mm\genfrac{}{}{0pt}{1}{\matulacac}{313}} 
\def\Mcad{\hskip 3mm\genfrac{}{}{0pt}{1}{\matulaaeg\matulab}{314}}
\def\Mcae{\hskip 3mm\genfrac{}{}{0pt}{1}{\matulag\matulae\matulac\matulac}{315}}

\def\Mcaf{\hskip 3mm\genfrac{}{}{0pt}{1}{\matulagi\matulab\matulab}{316}}
\def\Mcag{\hskip 3mm\genfrac{}{}{0pt}{1}{\matulacag}{317}}
\def\Mcah{\hskip 3mm\genfrac{}{}{0pt}{1}{\matulaec\matulac\matulab}{318}}
\def\Mcai{\hskip 3mm\genfrac{}{}{0pt}{1}{\matulabi\matulaaa}{319}}
\def\Mcbo{\hskip 3mm\genfrac{}{}{0pt}{1}{\matulae\matulab\matulab\matulab\matulab\matulab\matulab}{320}}
\def\Mcba{\hskip 3mm\genfrac{}{}{0pt}{1}{\matulaaog\matulac}{321}}
\def\Mcbb{\hskip 3mm\genfrac{}{}{0pt}{1}{\matulabc\matulag\matulab}{322}}
\def\Mcbc{\hskip 3mm\genfrac{}{}{0pt}{1}{\matulaai\matulaag}{323}}
\def\Mcbd{\hskip 3mm\genfrac{}{}{0pt}{1}{\matulac\matulac\matulac\matulac\matulab\matulab}{324}}
\def\Mcbe{\hskip 3mm\genfrac{}{}{0pt}{1}{\matulaac\matulae\matulae}{325}}
\def\Mcbf{\hskip 3mm\genfrac{}{}{0pt}{1}{\matulaafc\matulab}{326}}
\def\Mcbg{\hskip 3mm\genfrac{}{}{0pt}{1}{\matulaaoi\matulac}{327}}
\def\Mcbh{\hskip 3mm\genfrac{}{}{0pt}{1}{\matulada\matulab\matulab\matulab}{328}}
\def\Mcbi{\hskip 3mm\genfrac{}{}{0pt}{1}{\matuladg\matulag}{329}}
\def\Mcco{\hskip 3mm\genfrac{}{}{0pt}{1}{\matulaaa\matulae\matulac\matulab}{330}}

\def\Mcca{\hskip 3mm\genfrac{}{}{0pt}{1}{\matulacca}{331}}
\def\Mccb{\hskip 3mm\genfrac{}{}{0pt}{1}{\matulahc\matulab\matulab}{332}}
\def\Mccc{\hskip 3mm\genfrac{}{}{0pt}{1}{\matulacg\matulac\matulac}{333}}
\def\Mccd{\hskip 3mm\genfrac{}{}{0pt}{1}{\matulaafg\matulab}{334}}
\def\Mcce{\hskip 3mm\genfrac{}{}{0pt}{1}{\matulafg\matulae}{335}}
\def\Mccf{\hskip 3mm\genfrac{}{}{0pt}{1}{\matulag\matulac\matulab\matulab\matulab\matulab}{336}}
\def\Mccg{\hskip 3mm\genfrac{}{}{0pt}{1}{\matulaccg}{337}} 
\def\Mcch{\hskip 3mm\genfrac{}{}{0pt}{1}{\matulaac\matulaac\matulab}{338}}
\def\Mcci{\hskip 3mm\genfrac{}{}{0pt}{1}{\matulaaac\matulac}{339}}
\def\Mcdo{\hskip 3mm\genfrac{}{}{0pt}{1}{\matulaag\matulae\matulab\matulab}{340}}
\def\Mcda{\hskip 3mm\genfrac{}{}{0pt}{1}{\matulaca\matulaaa}{341}}
\def\Mcdb{\hskip 3mm\genfrac{}{}{0pt}{1}{\matulaai\matulac\matulac\matulab}{342}}
\def\Mcdc{\hskip 3mm\genfrac{}{}{0pt}{1}{\matulag\matulag\matulag}{343}}
\def\Mcdd{\hskip 3mm\genfrac{}{}{0pt}{1}{\matuladc\matulab\matulab\matulab}{344}}

\def\Mcde{\hskip 3mm\genfrac{}{}{0pt}{1}{\matulabc\matulae\matulac}{345}}
\def\Mcdf{\hskip 3mm\genfrac{}{}{0pt}{1}{\matulaagc\matulab}{346}}
\def\Mcdg{\hskip 3mm\genfrac{}{}{0pt}{1}{\matulacdg}{347}} 
\def\Mcdh{\hskip 3mm\genfrac{}{}{0pt}{1}{\matulabi\matulac\matulab\matulab}{348}}
\def\Mcdi{\hskip 3mm\genfrac{}{}{0pt}{1}{\matulacdi}{349}} 
\def\Mceo{\hskip 3mm\genfrac{}{}{0pt}{1}{\matulag\matulae\matulae\matulab}{350}}
\def\Mcea{\hskip 3mm\genfrac{}{}{0pt}{1}{\matulaac\matulac\matulac\matulac}{351}}
\def\Mceb{\hskip 3mm\genfrac{}{}{0pt}{1}{\matulaaa\matulab\matulab\matulab\matulab\matulab}{352}}
\def\Mcec{\hskip 3mm\genfrac{}{}{0pt}{1}{\matulacec}{353}}
\def\Mced{\hskip 3mm\genfrac{}{}{0pt}{1}{\matulaei\matulac\matulab}{354}}
\def\Mcee{\hskip 3mm\genfrac{}{}{0pt}{1}{\matulaga\matulae}{355}}
\def\Mcef{\hskip 3mm\genfrac{}{}{0pt}{1}{\matulahi\matulab\matulab}{356}}
\def\Mceg{\hskip 3mm\genfrac{}{}{0pt}{1}{\matulaag\matulag\matulac}{357}}
\def\Mceh{\hskip 3mm\genfrac{}{}{0pt}{1}{\matulaagi\matulab}{358}}

\def\Mcei{\hskip 3mm\genfrac{}{}{0pt}{1}{\matulacei}{359}}
\def\Mcfo{\hskip 3mm\genfrac{}{}{0pt}{1}{\matulae\matulac\matulac\matulab\matulab\matulab}{360}}
\def\Mcfa{\hskip 3mm\genfrac{}{}{0pt}{1}{\matulaai\matulaai}{361}}
\def\Mcfb{\hskip 3mm\genfrac{}{}{0pt}{1}{\matulaaha\matulab}{362}}
\def\Mcfc{\hskip 3mm\genfrac{}{}{0pt}{1}{\matulaaa\matulaaa\matulac}{363}}
\def\Mcfd{\hskip 3mm\genfrac{}{}{0pt}{1}{\matulaac\matulag\matulab\matulab}{364}}
\def\Mcfe{\hskip 3mm\genfrac{}{}{0pt}{1}{\matulagc\matulae}{365}}
\def\Mcff{\hskip 3mm\genfrac{}{}{0pt}{1}{\matulafa\matulac\matulab}{366}}
\def\Mcfg{\hskip 3mm\genfrac{}{}{0pt}{1}{\matulacfg}{367}}
\def\Mcfh{\hskip 3mm\genfrac{}{}{0pt}{1}{\matulabc\matulab\matulab\matulab\matulab}{368}}
\def\Mcfi{\hskip 3mm\genfrac{}{}{0pt}{1}{\matulada\matulac\matulac}{369}}
\def\Mcgo{\hskip 3mm\genfrac{}{}{0pt}{1}{\matulacg\matulae\matulab}{370}}
\def\Mcga{\hskip 3mm\genfrac{}{}{0pt}{1}{\matulaec\matulag}{371}}
\def\Mcgb{\hskip 3mm\genfrac{}{}{0pt}{1}{\matulaca\matulac\matulab\matulab}{372}}

\def\Mcgc{\hskip 3mm\genfrac{}{}{0pt}{1}{\matulacgc}{373}}
\def\Mcgd{\hskip 3mm\genfrac{}{}{0pt}{1}{\matulaag\matulaaa\matulab}{374}}
\def\Mcge{\hskip 3mm\genfrac{}{}{0pt}{1}{\matulae\matulae\matulae\matulac}{375}}
\def\Mcgf{\hskip 3mm\genfrac{}{}{0pt}{1}{\matuladg\matulab\matulab\matulab}{376}}
\def\Mcgg{\hskip 3mm\genfrac{}{}{0pt}{1}{\matulabi\matulaac}{377}}
\def\Mcgh{\hskip 3mm\genfrac{}{}{0pt}{1}{\matulag\matulac\matulac\matulac\matulab}{378}}
\def\Mcgi{\hskip 3mm\genfrac{}{}{0pt}{1}{\matulacgi}{379}}
\def\Mcho{\hskip 3mm\genfrac{}{}{0pt}{1}{\matulaai\matulae\matulab\matulab}{380}}
\def\Mcha{\hskip 3mm\genfrac{}{}{0pt}{1}{\matulaabg\matulac}{381}}
\def\Mchb{\hskip 3mm\genfrac{}{}{0pt}{1}{\matulaaia\matulab}{382}}
\def\Mchc{\hskip 3mm\genfrac{}{}{0pt}{1}{\matulachc}{383}}
\def\Mchd{\hskip 3mm\genfrac{}{}{0pt}{1}{\matulac\matulab\matulab\matulab\matulab\matulab\matulab\matulab}{384}}
\def\Mche{\hskip 3mm\genfrac{}{}{0pt}{1}{\matulaaa\matulag\matulae}{385}}
\def\Mchf{\hskip 3mm\genfrac{}{}{0pt}{1}{\matulaaic\matulab}{386}}

\def\Mchg{\hskip 3mm\genfrac{}{}{0pt}{1}{\matuladc\matulac\matulac}{387}}
\def\Mchh{\hskip 3mm\genfrac{}{}{0pt}{1}{\matulaig\matulab\matulab}{388}}
\def\Mchi{\hskip 3mm\genfrac{}{}{0pt}{1}{\matulachi}{389}}
\def\Mcio{\hskip 3mm\genfrac{}{}{0pt}{1}{\matulaac\matulae\matulac\matulab}{390}}
\def\Mcia{\hskip 3mm\genfrac{}{}{0pt}{1}{\matulabc\matulaag}{391}}
\def\Mcib{\hskip 3mm\genfrac{}{}{0pt}{1}{\matulag\matulag\matulab\matulab\matulab}{392}}
\def\Mcic{\hskip 3mm\genfrac{}{}{0pt}{1}{\matulaaca\matulac}{393}}
\def\Mcid{\hskip 3mm\genfrac{}{}{0pt}{1}{\matulaaig\matulab}{394}}
\def\Mcie{\hskip 3mm\genfrac{}{}{0pt}{1}{\matulagi\matulae}{395}}
\def\Mcif{\hskip 3mm\genfrac{}{}{0pt}{1}{\matulaaa\matulac\matulac\matulab\matulab}{396}}
\def\Mcig{\hskip 3mm\genfrac{}{}{0pt}{1}{\matulacig}{397}}
\def\Mcih{\hskip 3mm\genfrac{}{}{0pt}{1}{\matulaaii\matulab}{398}}
\def\Mcii{\hskip 3mm\genfrac{}{}{0pt}{1}{\matulaai\matulag\matulac}{399}}
\def\Mdoo{\hskip 3mm\genfrac{}{}{0pt}{1}{\matulae\matulae\matulab\matulab\matulab\matulab}{400}}

\def\Mdoa{\hskip 3mm\genfrac{}{}{0pt}{1}{\matuladoa}{401}} 
\def\Mdob{\hskip 3mm\genfrac{}{}{0pt}{1}{\matulafg\matulac\matulab}{402}}
\def\Mdoc{\hskip 3mm\genfrac{}{}{0pt}{1}{\matulaca\matulaac}{403}}
\def\Mdod{\hskip 3mm\genfrac{}{}{0pt}{1}{\matulaaoa\matulab\matulab}{404}}
\def\Mdoe{\hskip 3mm\genfrac{}{}{0pt}{1}{\matulae\matulac\matulac\matulac\matulac}{405}}
\def\Mdof{\hskip 3mm\genfrac{}{}{0pt}{1}{\matulabi\matulag\matulab}{406}}
\def\Mdog{\hskip 3mm\genfrac{}{}{0pt}{1}{\matulacg\matulaaa}{407}}
\def\Mdoh{\hskip 3mm\genfrac{}{}{0pt}{1}{\matulaag\matulac\matulab\matulab\matulab}{408}}
\def\Mdoi{\hskip 3mm\genfrac{}{}{0pt}{1}{\matuladoi}{409}} 
\def\Mdao{\hskip 3mm\genfrac{}{}{0pt}{1}{\matulada\matulae\matulab}{410}}
\def\Mdaa{\hskip 3mm\genfrac{}{}{0pt}{1}{\matulaacg\matulac}{411}}
\def\Mdab{\hskip 3mm\genfrac{}{}{0pt}{1}{\matulaaoc\matulab\matulab}{412}}
\def\Mdac{\hskip 3mm\genfrac{}{}{0pt}{1}{\matulaei\matulag}{413}}
\def\Mdad{\hskip 3mm\genfrac{}{}{0pt}{1}{\matulabc\matulac\matulac\matulab}{414}}
\def\Mdae{\hskip 3mm\genfrac{}{}{0pt}{1}{\matulahc\matulae}{415}}

\def\Mdaf{\hskip 3mm\genfrac{}{}{0pt}{1}{\matulaac\matulab\matulab\matulab\matulab\matulab}{416}}
\def\Mdag{\hskip 3mm\genfrac{}{}{0pt}{1}{\matulaaci\matulac}{417}}
\def\Mdah{\hskip 3mm\genfrac{}{}{0pt}{1}{\matulaai\matulaaa\matulab}{418}}
\def\Mdai{\hskip 3mm\genfrac{}{}{0pt}{1}{\matuladai}{419}}
\def\Mdbo{\hskip 3mm\genfrac{}{}{0pt}{1}{\matulag\matulae\matulac\matulab\matulab}{420}}
\def\Mdba{\hskip 3mm\genfrac{}{}{0pt}{1}{\matuladba}{421}} 
\def\Mdbb{\hskip 3mm\genfrac{}{}{0pt}{1}{\matulabaa\matulab}{422}}
\def\Mdbc{\hskip 3mm\genfrac{}{}{0pt}{1}{\matuladg\matulac\matulac}{423}}
\def\Mdbd{\hskip 3mm\genfrac{}{}{0pt}{1}{\matulaec\matulab\matulab\matulab}{424}}
\def\Mdbe{\hskip 3mm\genfrac{}{}{0pt}{1}{\matulaag\matulae\matulae}{425}}
\def\Mdbf{\hskip 3mm\genfrac{}{}{0pt}{1}{\matulaga\matulac\matulab}{426}}
\def\Mdbg{\hskip 3mm\genfrac{}{}{0pt}{1}{\matulafa\matulag}{427}}
\def\Mdbh{\hskip 3mm\genfrac{}{}{0pt}{1}{\matulaaog\matulab\matulab}{428}}
\def\Mdbi{\hskip 3mm\genfrac{}{}{0pt}{1}{\matulaac\matulaaa\matulac}{429}}
\def\Mdco{\hskip 3mm\genfrac{}{}{0pt}{1}{\matuladc\matulae\matulab}{430}}

\def\Mdca{\hskip 3mm\genfrac{}{}{0pt}{1}{\matuladca}{431}}
\def\Mdcb{\hskip 3mm\genfrac{}{}{0pt}{1}{\matulac\matulac\matulac\matulab\matulab\matulab\matulab}{432}}
\def\Mdcc{\hskip 3mm\genfrac{}{}{0pt}{1}{\matuladcc}{433}}
\def\Mdcd{\hskip 3mm\genfrac{}{}{0pt}{1}{\matulaca\matulag\matulab}{434}}
\def\Mdce{\hskip 3mm\genfrac{}{}{0pt}{1}{\matulabi\matulae\matulac}{435}}
\def\Mdcf{\hskip 3mm\genfrac{}{}{0pt}{1}{\matulaaoi\matulab\matulab}{436}}
\def\Mdcg{\hskip 3mm\genfrac{}{}{0pt}{1}{\matulabc\matulaai}{437}}
\def\Mdch{\hskip 3mm\genfrac{}{}{0pt}{1}{\matulagc\matulac\matulab}{438}}
\def\Mdci{\hskip 3mm\genfrac{}{}{0pt}{1}{\matuladci}{439}}
\def\Mddo{\hskip 3mm\genfrac{}{}{0pt}{1}{\matulaaa\matulae\matulab\matulab\matulab}{440}}
\def\Mdda{\hskip 3mm\genfrac{}{}{0pt}{1}{\matulag\matulag\matulac\matulac}{441}}
\def\Mddb{\hskip 3mm\genfrac{}{}{0pt}{1}{\matulaag\matulaac\matulab}{442}}
\def\Mddc{\hskip 3mm\genfrac{}{}{0pt}{1}{\matuladdc}{443}} 
\def\Mddd{\hskip 3mm\genfrac{}{}{0pt}{1}{\matulacg\matulac\matulab\matulab}{444}}

\def\Mdde{\hskip 3mm\genfrac{}{}{0pt}{1}{\matulahi\matulae}{445}}
\def\Mddf{\hskip 3mm\genfrac{}{}{0pt}{1}{\matulabbc\matulab}{446}}
\def\Mddg{\hskip 3mm\genfrac{}{}{0pt}{1}{\matulaadi\matulac}{447}}
\def\Mddh{\hskip 3mm\genfrac{}{}{0pt}{1}{\matulag\matulab\matulab\matulab\matulab\matulab\matulab}{448}}
\def\Mddi{\hskip 3mm\genfrac{}{}{0pt}{1}{\matuladdi}{449}}
\def\Mdeo{\hskip 3mm\genfrac{}{}{0pt}{1}{\matulae\matulae\matulac\matulac\matulab}{450}}
\def\Mdea{\hskip 3mm\genfrac{}{}{0pt}{1}{\matulada\matulaaa}{451}}
\def\Mdeb{\hskip 3mm\genfrac{}{}{0pt}{1}{\matulaaac\matulab\matulab}{452}}
\def\Mdec{\hskip 3mm\genfrac{}{}{0pt}{1}{\matulaaea\matulac}{453}}
\def\Mded{\hskip 3mm\genfrac{}{}{0pt}{1}{\matulabbg\matulab}{454}}
\def\Mdee{\hskip 3mm\genfrac{}{}{0pt}{1}{\matulaac\matulag\matulae}{455}}
\def\Mdef{\hskip 3mm\genfrac{}{}{0pt}{1}{\matulaai\matulac\matulab\matulab\matulab}{456}}
\def\Mdeg{\hskip 3mm\genfrac{}{}{0pt}{1}{\matuladeg}{457}}
\def\Mdeh{\hskip 3mm\genfrac{}{}{0pt}{1}{\matulabbi\matulab}{458}}

\def\Mdei{\hskip 3mm\genfrac{}{}{0pt}{1}{\matulaag\matulac\matulac\matulac}{459}}
\def\Mdfo{\hskip 3mm\genfrac{}{}{0pt}{1}{\matulabc\matulae\matulab\matulab}{460}}
\def\Mdfa{\hskip 3mm\genfrac{}{}{0pt}{1}{\matuladfa}{461}}
\def\Mdfb{\hskip 3mm\genfrac{}{}{0pt}{1}{\matulaaa\matulag\matulac\matulab}{462}}
\def\Mdfc{\hskip 3mm\genfrac{}{}{0pt}{1}{\matuladfc}{463}}
\def\Mdfd{\hskip 3mm\genfrac{}{}{0pt}{1}{\matulabi\matulab\matulab\matulab\matulab}{464}}
\def\Mdfe{\hskip 3mm\genfrac{}{}{0pt}{1}{\matulaca\matulae\matulac}{465}}
\def\Mdff{\hskip 3mm\genfrac{}{}{0pt}{1}{\matulabcc\matulab}{466}}
\def\Mdfg{\hskip 3mm\genfrac{}{}{0pt}{1}{\matuladfg}{467}}
\def\Mdfh{\hskip 3mm\genfrac{}{}{0pt}{1}{\matulaac\matulac\matulac\matulab\matulab}{468}}
\def\Mdfi{\hskip 3mm\genfrac{}{}{0pt}{1}{\matulafg\matulag}{469}}
\def\Mdgo{\hskip 3mm\genfrac{}{}{0pt}{1}{\matuladg\matulae\matulab}{470}}
\def\Mdga{\hskip 3mm\genfrac{}{}{0pt}{1}{\matulaaeg\matulac}{471}}
\def\Mdgb{\hskip 3mm\genfrac{}{}{0pt}{1}{\matulaei\matulab\matulab\matulab}{472}}

\def\Mdgc{\hskip 3mm\genfrac{}{}{0pt}{1}{\matuladc\matulaaa}{473}}
\def\Mdgd{\hskip 3mm\genfrac{}{}{0pt}{1}{\matulagi\matulac\matulab}{474}}
\def\Mdge{\hskip 3mm\genfrac{}{}{0pt}{1}{\matulaai\matulae\matulae}{475}}
\def\Mdgf{\hskip 3mm\genfrac{}{}{0pt}{1}{\matulaag\matulag\matulab\matulab}{476}}
\def\Mdgg{\hskip 3mm\genfrac{}{}{0pt}{1}{\matulaec\matulac\matulac}{477}}
\def\Mdgh{\hskip 3mm\genfrac{}{}{0pt}{1}{\matulabci\matulab}{478}}
\def\Mdgi{\hskip 3mm\genfrac{}{}{0pt}{1}{\matuladgi}{479}}
\def\Mdho{\hskip 3mm\genfrac{}{}{0pt}{1}{\matulae\matulac\matulab\matulab\matulab\matulab\matulab}{480}}
\def\Mdha{\hskip 3mm\genfrac{}{}{0pt}{1}{\matulacg\matulaac}{481}}
\def\Mdhb{\hskip 3mm\genfrac{}{}{0pt}{1}{\matulabda\matulab}{482}}
\def\Mdhc{\hskip 3mm\genfrac{}{}{0pt}{1}{\matulabc\matulag\matulac}{483}}
\def\Mdhd{\hskip 3mm\genfrac{}{}{0pt}{1}{\matulaaa\matulaaa\matulab\matulab}{484}}
\def\Mdhe{\hskip 3mm\genfrac{}{}{0pt}{1}{\matulaig\matulae}{485}}
\def\Mdhf{\hskip 3mm\genfrac{}{}{0pt}{1}{\matulac\matulac\matulac\matulac\matulac\matulab}{486}}

\def\Mdhg{\hskip 3mm\genfrac{}{}{0pt}{1}{\matuladhg}{487}}
\def\Mdhh{\hskip 3mm\genfrac{}{}{0pt}{1}{\matulafa\matulab\matulab\matulab}{488}}
\def\Mdhi{\hskip 3mm\genfrac{}{}{0pt}{1}{\matulaafc\matulac}{489}}
\def\Mdio{\hskip 3mm\genfrac{}{}{0pt}{1}{\matulag\matulag\matulae\matulab}{490}}
\def\Mdia{\hskip 3mm\genfrac{}{}{0pt}{1}{\matuladia}{491}}
\def\Mdib{\hskip 3mm\genfrac{}{}{0pt}{1}{\matulada\matulac\matulab\matulab}{492}}
\def\Mdic{\hskip 3mm\genfrac{}{}{0pt}{1}{\matulabi\matulaag}{493}}
\def\Mdid{\hskip 3mm\genfrac{}{}{0pt}{1}{\matulaai\matulaac\matulab}{494}}
\def\Mdie{\hskip 3mm\genfrac{}{}{0pt}{1}{\matulaaa\matulae\matulac\matulac}{495}}
\def\Mdif{\hskip 3mm\genfrac{}{}{0pt}{1}{\matulaca\matulab\matulab\matulab\matulab}{496}}
\def\Mdig{\hskip 3mm\genfrac{}{}{0pt}{1}{\matulaga\matulag}{497}}
\def\Mdih{\hskip 3mm\genfrac{}{}{0pt}{1}{\matulahc\matulac\matulab}{498}}
\def\Mdii{\hskip 3mm\genfrac{}{}{0pt}{1}{\matuladii}{499}}
\def\Meoo{\hskip 3mm\genfrac{}{}{0pt}{1}{\matulae\matulae\matulae\matulab\matulab}{500}}

\def\Meoa{\hskip 3mm\genfrac{}{}{0pt}{1}{\matulaafg\matulac}{501}}
\def\Meob{\hskip 3mm\genfrac{}{}{0pt}{1}{\matulabea\matulab}{502}}
\def\Meoc{\hskip 3mm\genfrac{}{}{0pt}{1}{\matulaeoc}{503}}
\def\Meod{\hskip 3mm\genfrac{}{}{0pt}{1}{\matulag\matulac\matulac\matulab\matulab\matulab}{504}}
\def\Meoe{\hskip 3mm\genfrac{}{}{0pt}{1}{\matulaaoa\matulae}{505}}
\def\Meof{\hskip 3mm\genfrac{}{}{0pt}{1}{\matulabc\matulaaa\matulab}{506}}
\def\Meog{\hskip 3mm\genfrac{}{}{0pt}{1}{\matulaac\matulaac\matulac}{507}}
\def\Meoh{\hskip 3mm\genfrac{}{}{0pt}{1}{\matulaabg\matulab\matulab}{508}}
\def\Meoi{\hskip 3mm\genfrac{}{}{0pt}{1}{\matulaeoi}{509}}
\def\Meao{\hskip 3mm\genfrac{}{}{0pt}{1}{\matulaag\matulae\matulac\matulab}{510}}
\def\Meaa{\hskip 3mm\genfrac{}{}{0pt}{1}{\matulagc\matulag}{511}}
\def\Meab{\hskip 3mm\genfrac{}{}{0pt}{1}
{\matulab\matulab\matulab\matulab\matulab\matulab\matulab\matulab\matulab}{512}}
\def\Meac{\hskip 3mm\genfrac{}{}{0pt}{1}{\matulaai\matulac\matulac\matulac}{513}}
\def\Mead{\hskip 3mm\genfrac{}{}{0pt}{1}{\matulabeg\matulab}{514}}
\def\Meae{\hskip 3mm\genfrac{}{}{0pt}{1}{\matulaaoc\matulae}{515}}

\def\Meaf{\hskip 3mm\genfrac{}{}{0pt}{1}{\matuladc\matulac\matulab\matulab}{516}}
\def\Meag{\hskip 3mm\genfrac{}{}{0pt}{1}{\matuladg\matulaaa}{517}}
\def\Meah{\hskip 3mm\genfrac{}{}{0pt}{1}{\matulacg\matulag\matulab}{518}}
\def\Meai{\hskip 3mm\genfrac{}{}{0pt}{1}{\matulaagc\matulac}{519}}
\def\Mebo{\hskip 3mm\genfrac{}{}{0pt}{1}{\matulaac\matulae\matulab\matulab\matulab}{520}}
\def\Meba{\hskip 3mm\genfrac{}{}{0pt}{1}{\matulaeba}{521}}
\def\Mebb{\hskip 3mm\genfrac{}{}{0pt}{1}{\matulabi\matulac\matulac\matulab}{522}}
\def\Mebc{\hskip 3mm\genfrac{}{}{0pt}{1}{\matulaebc}{523}}
\def\Mebd{\hskip 3mm\genfrac{}{}{0pt}{1}{\matulaaca\matulab\matulab}{524}}
\def\Mebe{\hskip 3mm\genfrac{}{}{0pt}{1}{\matulag\matulae\matulae\matulac}{525}}
\def\Mebf{\hskip 3mm\genfrac{}{}{0pt}{1}{\matulabfc\matulab}{526}}
\def\Mebg{\hskip 3mm\genfrac{}{}{0pt}{1}{\matulaca\matulaag}{527}}
\def\Mebh{\hskip 3mm\genfrac{}{}{0pt}{1}{\matulaaa\matulac\matulab\matulab\matulab\matulab}{528}}
\def\Mebi{\hskip 3mm\genfrac{}{}{0pt}{1}{\matulabc\matulabc}{529}}
\def\Meco{\hskip 3mm\genfrac{}{}{0pt}{1}{\matulaec\matulae\matulab}{530}}

\def\Meca{\hskip 3mm\genfrac{}{}{0pt}{1}{\matulaei\matulac\matulac}{531}}
\def\Mecb{\hskip 3mm\genfrac{}{}{0pt}{1}{\matulaai\matulag\matulab\matulab}{532}}
\def\Mecc{\hskip 3mm\genfrac{}{}{0pt}{1}{\matulada\matulaac}{533}}
\def\Mecd{\hskip 3mm\genfrac{}{}{0pt}{1}{\matulahi\matulac\matulab}{534}}
\def\Mece{\hskip 3mm\genfrac{}{}{0pt}{1}{\matulaaog\matulae}{535}}
\def\Mecf{\hskip 3mm\genfrac{}{}{0pt}{1}{\matulafg\matulab\matulab\matulab}{536}}
\def\Mecg{\hskip 3mm\genfrac{}{}{0pt}{1}{\matulaagi\matulac}{537}}
\def\Mech{\hskip 3mm\genfrac{}{}{0pt}{1}{\matulabfi\matulab}{538}} 
\def\Meci{\hskip 3mm\genfrac{}{}{0pt}{1}{\matulaaa\matulag\matulag}{539}}
\def\Medo{\hskip 3mm\genfrac{}{}{0pt}{1}{\matulae\matulac\matulac\matulac\matulab\matulab}{540}}
\def\Meda{\hskip 3mm\genfrac{}{}{0pt}{1}{\matulaeda}{541}}
\def\Medb{\hskip 3mm\genfrac{}{}{0pt}{1}{\matulabga\matulab}{542}} 
\def\Medc{\hskip 3mm\genfrac{}{}{0pt}{1}{\matulaaha\matulac}{543}}
\def\Medd{\hskip 3mm\genfrac{}{}{0pt}{1}{\matulaag\matulab\matulab\matulab\matulab\matulab}{544}}

\def\Mede{\hskip 3mm\genfrac{}{}{0pt}{1}{\matulaaoi\matulae}{545}}
\def\Medf{\hskip 3mm\genfrac{}{}{0pt}{1}{\matulaac\matulag\matulac\matulab}{546}}
\def\Medg{\hskip 3mm\genfrac{}{}{0pt}{1}{\matulaedg}{547}}
\def\Medh{\hskip 3mm\genfrac{}{}{0pt}{1}{\matulaacg\matulab\matulab}{548}}
\def\Medi{\hskip 3mm\genfrac{}{}{0pt}{1}{\matulafa\matulac\matulac}{549}}
\def\Meeo{\hskip 3mm\genfrac{}{}{0pt}{1}{\matulaaa\matulae\matulae\matulab}{550}}
\def\Meea{\hskip 3mm\genfrac{}{}{0pt}{1}{\matulabi\matulaai}{551}}
\def\Meeb{\hskip 3mm\genfrac{}{}{0pt}{1}{\matulabc\matulac\matulab\matulab\matulab}{552}}
\def\Meec{\hskip 3mm\genfrac{}{}{0pt}{1}{\matulagi\matulag}{553}}
\def\Meed{\hskip 3mm\genfrac{}{}{0pt}{1}{\matulabgg\matulab}{554}}
\def\Meee{\hskip 3mm\genfrac{}{}{0pt}{1}{\matulacg\matulae\matulac}{555}}
\def\Meef{\hskip 3mm\genfrac{}{}{0pt}{1}{\matulaaci\matulab\matulab}{556}}
\def\Meeg{\hskip 3mm\genfrac{}{}{0pt}{1}{\matulaeeg}{557}}
\def\Meeh{\hskip 3mm\genfrac{}{}{0pt}{1}{\matulaca\matulac\matulac\matulab}{558}}

\def\Meei{\hskip 3mm\genfrac{}{}{0pt}{1}{\matuladc\matulaac}{559}}
\def\Mefo{\hskip 3mm\genfrac{}{}{0pt}{1}{\matulag\matulae\matulab\matulab\matulab\matulab}{560}}
\def\Mefa{\hskip 3mm\genfrac{}{}{0pt}{1}{\matulaag\matulaaa\matulac}{561}}
\def\Mefb{\hskip 3mm\genfrac{}{}{0pt}{1}{\matulabha\matulab}{562}} 
\def\Mefc{\hskip 3mm\genfrac{}{}{0pt}{1}{\matulaefc}{563}}
\def\Mefd{\hskip 3mm\genfrac{}{}{0pt}{1}{\matuladg\matulac\matulab\matulab}{564}}
\def\Mefe{\hskip 3mm\genfrac{}{}{0pt}{1}{\matulaaac\matulae}{565}}
\def\Meff{\hskip 3mm\genfrac{}{}{0pt}{1}{\matulabhc\matulab}{566}}
\def\Mefg{\hskip 3mm\genfrac{}{}{0pt}{1}{\matulag\matulac\matulac\matulac\matulac}{567}}
\def\Mefh{\hskip 3mm\genfrac{}{}{0pt}{1}{\matulaga\matulab\matulab\matulab}{568}}
\def\Mefi{\hskip 3mm\genfrac{}{}{0pt}{1}{\matulaefi}{569}}
\def\Mego{\hskip 3mm\genfrac{}{}{0pt}{1}{\matulaai\matulae\matulac\matulab}{570}}
\def\Mega{\hskip 3mm\genfrac{}{}{0pt}{1}{\matulaega}{571}}
\def\Megb{\hskip 3mm\genfrac{}{}{0pt}{1}{\matulaac\matulaaa\matulab\matulab}{572}}

\def\Megc{\hskip 3mm\genfrac{}{}{0pt}{1}{\matulaaia\matulac}{573}}
\def\Megd{\hskip 3mm\genfrac{}{}{0pt}{1}{\matulada\matulag\matulab}{574}}
\def\Mege{\hskip 3mm\genfrac{}{}{0pt}{1}{\matulabc\matulae\matulae}{575}}
\def\Megf{\hskip 3mm\genfrac{}{}{0pt}{1}{\matulac\matulac\matulab\matulab\matulab\matulab\matulab\matulab}{576}}
\def\Megg{\hskip 3mm\genfrac{}{}{0pt}{1}{\matulaegg}{577}}
\def\Megh{\hskip 3mm\genfrac{}{}{0pt}{1}{\matulaag\matulaag\matulab}{578}}
\def\Megi{\hskip 3mm\genfrac{}{}{0pt}{1}{\matulaaic\matulac}{579}}
\def\Meho{\hskip 3mm\genfrac{}{}{0pt}{1}{\matulabi\matulae\matulab\matulab}{580}}
\def\Meha{\hskip 3mm\genfrac{}{}{0pt}{1}{\matulahc\matulag}{581}}
\def\Mehb{\hskip 3mm\genfrac{}{}{0pt}{1}{\matulaig\matulac\matulab}{582}}
\def\Mehc{\hskip 3mm\genfrac{}{}{0pt}{1}{\matulaec\matulaaa}{583}}
\def\Mehd{\hskip 3mm\genfrac{}{}{0pt}{1}{\matulagc\!\!\matulab\matulab\matulab}{584}}
\def\Mehe{\hskip 3mm\genfrac{}{}{0pt}{1}{\matulaac\matulae\matulac\matulac}{585}}
\def\Mehf{\hskip 3mm\genfrac{}{}{0pt}{1}{\matulabic\matulab}{586}}

\def\Mehg{\hskip 3mm\genfrac{}{}{0pt}{1}{\matulaehg}{587}}
\def\Mehh{\hskip 3mm\genfrac{}{}{0pt}{1}{\matulag\matulag\matulac\matulab\matulab}{588}}
\def\Mehi{\hskip 3mm\genfrac{}{}{0pt}{1}{\matulaca\matulaai}{589}}
\def\Meio{\hskip 3mm\genfrac{}{}{0pt}{1}{\matulaei\matulae\matulab}{590}}
\def\Meia{\hskip 3mm\genfrac{}{}{0pt}{1}{\matulaaig\matulac}{591}}
\def\Meib{\hskip 3mm\genfrac{}{}{0pt}{1}{\matulacg\matulab\matulab\matulab\matulab}{592}}
\def\Meic{\hskip 3mm\genfrac{}{}{0pt}{1}{\matulaeic}{593}}
\def\Meid{\hskip 3mm\genfrac{}{}{0pt}{1}{\matulaaa\matulac\matulac\matulac\matulab}{594}}
\def\Meie{\hskip 3mm\genfrac{}{}{0pt}{1}{\matulaag\matulag\matulae}{595}}
\def\Meif{\hskip 3mm\genfrac{}{}{0pt}{1}{\matulaadi\matulab\matulab}{596}}
\def\Meig{\hskip 3mm\genfrac{}{}{0pt}{1}{\matulaaii\matulac}{597}}
\def\Meih{\hskip 3mm\genfrac{}{}{0pt}{1}{\matulabc\matulaac\matulab}{598}}
\def\Meii{\hskip 3mm\genfrac{}{}{0pt}{1}{\matulaeii}{599}}
\def\Mfoo{\hskip 3mm\genfrac{}{}{0pt}{1}{\matulae\matulae\matulac\matulab\matulab\matulab}{600}}

\def\Mfoa{\hskip 3mm\genfrac{}{}{0pt}{1}{\matulafoa}{601}}
\def\Mfob{\hskip 3mm\genfrac{}{}{0pt}{1}{\matuladc\matulag\matulab}{602}}
\def\Mfoc{\hskip 3mm\genfrac{}{}{0pt}{1}{\matulafg\matulac\matulac}{603}}
\def\Mfod{\hskip 3mm\genfrac{}{}{0pt}{1}{\matulaaea\matulab\matulab}{604}}
\def\Mfoe{\hskip 3mm\genfrac{}{}{0pt}{1}{\matulaaa\matulaaa\matulae}{605}}
\def\Mfof{\hskip 3mm\genfrac{}{}{0pt}{1}{\matulaaoa\matulac\matulab}{606}}
\def\Mfog{\hskip 3mm\genfrac{}{}{0pt}{1}{\matulafog}{607}}
\def\Mfoh{\hskip 3mm\genfrac{}{}{0pt}{1}{\matulaai\matulab\matulab\matulab\matulab\matulab}{608}}
\def\Mfoi{\hskip 3mm\genfrac{}{}{0pt}{1}{\matulabi\matulag\matulac}{609}}
\def\Mfao{\hskip 3mm\genfrac{}{}{0pt}{1}{\matulafa\matulae\matulab}{610}}
\def\Mfaa{\hskip 3mm\genfrac{}{}{0pt}{1}{\matuladg\matulaac}{611}}
\def\Mfab{\hskip 3mm\genfrac{}{}{0pt}{1}{\matulaag\matulac\matulac\matulab\matulab}{612}}
\def\Mfac{\hskip 3mm\genfrac{}{}{0pt}{1}{\matulafac}{613}}
\def\Mfad{\hskip 3mm\genfrac{}{}{0pt}{1}{\matulacog\matulab}{614}}
\def\Mfae{\hskip 3mm\genfrac{}{}{0pt}{1}{\matulada\matulae\matulac}{615}}

\def\Mfaf{\hskip 3mm\genfrac{}{}{0pt}{1}{\matulaaa\matulag\matulab\matulab\matulab}{616}}
\def\Mfag{\hskip 3mm\genfrac{}{}{0pt}{1}{\matulafag}{617}} \ \ 
\def\Mfah{\hskip 3mm\genfrac{}{}{0pt}{1}{\matulaaoc\matulac\matulab}{618}}
\def\Mfai{\hskip 3mm\genfrac{}{}{0pt}{1}{\matulafai}{619}}
\def\Mfbo{\hskip 3mm\genfrac{}{}{0pt}{1}{\matulaca\matulae\matulab\matulab}{620}}
\def\Mfba{\hskip 3mm\genfrac{}{}{0pt}{1}{\matulabc\matulac\matulac\matulac}{621}}
\def\Mfbb{\hskip 3mm\genfrac{}{}{0pt}{1}{\matulacaa\matulab}{622}}
\def\Mfbc{\hskip 3mm\genfrac{}{}{0pt}{1}{\matulahi\matulag}{623}}
\def\Mfbd{\hskip 3mm\genfrac{}{}{0pt}{1}{\matulaac\matulac\matulab\matulab\matulab\matulab}{624}}
\def\Mfbe{\hskip 3mm\genfrac{}{}{0pt}{1}{\matulae\matulae\matulae\matulae}{625}}
\def\Mfbf{\hskip 3mm\genfrac{}{}{0pt}{1}{\matulacac\matulab}{626}}
\def\Mfbg{\hskip 3mm\genfrac{}{}{0pt}{1}{\matulaai\matulaaa\matulac}{627}}
\def\Mfbh{\hskip 3mm\genfrac{}{}{0pt}{1}{\matulaaeg\matulab\matulab}{628}}
\def\Mfbi{\hskip 3mm\genfrac{}{}{0pt}{1}{\matulacg\matulaag}{629}}
\def\Mfco{\hskip 3mm\genfrac{}{}{0pt}{1}{\matulag\matulae\matulac\matulac\matulab}{630}}

\def\Mfca{\hskip 3mm\genfrac{}{}{0pt}{1}{\matulafca}{631}}
\def\Mfcb{\hskip 3mm\genfrac{}{}{0pt}{1}{\matulagi\matulab\matulab\matulab}{632}}
\def\Mfcc{\hskip 3mm\genfrac{}{}{0pt}{1}{\matulabaa\matulac}{633}}
\def\Mfcd{\hskip 3mm\genfrac{}{}{0pt}{1}{\matulacag\matulab}{634}}
\def\Mfce{\hskip 3mm\genfrac{}{}{0pt}{1}{\matulaabg\matulae}{635}}
\def\Mfcf{\hskip 3mm\genfrac{}{}{0pt}{1}{\matulaec\matulac\matulab\matulab}{636}}
\def\Mfcg{\hskip 3mm\genfrac{}{}{0pt}{1}{\matulaac\matulag\matulag}{637}}
\def\Mfch{\hskip 3mm\genfrac{}{}{0pt}{1}{\matulabi\matulaaa\matulab}{638}}
\def\Mfci{\hskip 3mm\genfrac{}{}{0pt}{1}{\matulaga\matulac\matulac}{639}}
\def\Mfdo{\hskip 3mm\genfrac{}{}{0pt}{1}{\matulae\matulab\matulab\matulab\matulab\matulab\matulab\matulab}{640}}
\def\Mfda{\hskip 3mm\genfrac{}{}{0pt}{1}{\matulafda}{641}}
\def\Mfdb{\hskip 3mm\genfrac{}{}{0pt}{1}{\matulaaog\matulac\matulab}{642}}
\def\Mfdc{\hskip 3mm\genfrac{}{}{0pt}{1}{\matulafdc}{643}}
\def\Mfdd{\hskip 3mm\genfrac{}{}{0pt}{1}{\matulabc\matulag\matulab\matulab}{644}}

\def\Mfde{\hskip 3mm\genfrac{}{}{0pt}{1}{\matuladc\matulae\matulac}{645}}
\def\Mfdf{\hskip 3mm\genfrac{}{}{0pt}{1}{\matulaai\matulaag\matulab}{646}}
\def\Mfdg{\hskip 3mm\genfrac{}{}{0pt}{1}{\matulafdg}{647}}
\def\Mfdh{\hskip 3mm\genfrac{}{}{0pt}{1}{\matulac\matulac\matulac\matulac\matulab\matulab\matulab}{648}}
\def\Mfdi{\hskip 3mm\genfrac{}{}{0pt}{1}{\matulaei\matulaaa}{649}}
\def\Mfeo{\hskip 3mm\genfrac{}{}{0pt}{1}{\matulaac\matulae\matulae\matulab}{650}}
\def\Mfea{\hskip 3mm\genfrac{}{}{0pt}{1}{\matulaca\matulag\matulac}{651}} 
\def\Mfeb{\hskip 3mm\genfrac{}{}{0pt}{1}{\matulaafc\matulab\matulab}{652}} 
\def\Mfec{\hskip 3mm\genfrac{}{}{0pt}{1}{\matulafec}{653}}
\def\Mfed{\hskip 3mm\genfrac{}{}{0pt}{1}{\matulaaoi\matulac\matulab}{654}} 
\def\Mfee{\hskip 3mm\genfrac{}{}{0pt}{1}{\matulaaca\matulae}{655}}
\def\Mfef{\hskip 3mm\genfrac{}{}{0pt}{1}{\matulada\matulab\matulab\matulab\matulab}{656}}
\def\Mfeg{\hskip 3mm\genfrac{}{}{0pt}{1}{\matulagc\matulac\matulac}{657}}
\def\Mfeh{\hskip 3mm\genfrac{}{}{0pt}{1}{\matuladg\matulag\matulab}{658}}

\def\Mfei{\hskip 3mm\genfrac{}{}{0pt}{1}{\matulafei}{659}}
\def\Mffo{\hskip 3mm\genfrac{}{}{0pt}{1}{\matulaaa\matulae\matulac\matulab\matulab}{660}}
\def\Mffa{\hskip 3mm\genfrac{}{}{0pt}{1}{\matulaffa}{661}}
\def\Mffb{\hskip 3mm\genfrac{}{}{0pt}{1}{\matulacca\matulab}{662}}
\def\Mffc{\hskip 3mm\genfrac{}{}{0pt}{1}{\matulaag\matulaac\matulac}{663}}
\def\Mffd{\hskip 3mm\genfrac{}{}{0pt}{1}{\matulahc\matulab\matulab\matulab}{664}}
\def\Mffe{\hskip 3mm\genfrac{}{}{0pt}{1}{\matulaai\matulag\matulae}{665}}
\def\Mfff{\hskip 3mm\genfrac{}{}{0pt}{1}{\matulacg\matulac\matulac\matulab}{666}}
\def\Mffg{\hskip 3mm\genfrac{}{}{0pt}{1}{\matulabi\matulabc}{667}}
\def\Mffh{\hskip 3mm\genfrac{}{}{0pt}{1}{\matulaafg\matulab\matulab}{668}}
\def\Mffi{\hskip 3mm\genfrac{}{}{0pt}{1}{\matulabbc\matulac}{669}}
\def\Mfgo{\hskip 3mm\genfrac{}{}{0pt}{1}{\matulafg\matulae\matulab}{670}}
\def\Mfga{\hskip 3mm\genfrac{}{}{0pt}{1}{\matulafa\matulaaa}{671}}
\def\Mfgb{\hskip 3mm\genfrac{}{}{0pt}{1}{\matulag\matulac\matulab\matulab\matulab\matulab\matulab}{672}}

\def\Mfgc{\hskip 3mm\genfrac{}{}{0pt}{1}{\matulafgc}{673}}
\def\Mfgd{\hskip 3mm\genfrac{}{}{0pt}{1}{\matulaccg\matulab}{674}}
\def\Mfge{\hskip 3mm\genfrac{}{}{0pt}{1}{\matulae\matulae\matulac\matulac\matulac}{675}}
\def\Mfgf{\hskip 3mm\genfrac{}{}{0pt}{1}{\matulaac\matulaac\matulab\matulab}{676}}
\def\Mfgg{\hskip 3mm\genfrac{}{}{0pt}{1}{\matulafgg}{677}}
\def\Mfgh{\hskip 3mm\genfrac{}{}{0pt}{1}{\matulaaac\matulac\matulab}{678}}
\def\Mfgi{\hskip 3mm\genfrac{}{}{0pt}{1}{\matulaig\matulag}{679}}
\def\Mfho{\hskip 3mm\genfrac{}{}{0pt}{1}{\matulaag\matulae\matulab\matulab\matulab}{680}}
\def\Mfha{\hskip 3mm\genfrac{}{}{0pt}{1}{\matulabbg\,\matulac}{681}}
\def\Mfhb{\hskip 3mm\genfrac{}{}{0pt}{1}{\matulaca\matulaaa\matulab}{682}}
\def\Mfhc{\hskip 3mm\genfrac{}{}{0pt}{1}{\matulafhc}{683}}
\def\Mfhd{\hskip 3mm\genfrac{}{}{0pt}{1}{\matulaai\matulac\matulac\matulab\matulab}{684}}
\def\Mfhe{\hskip 3mm\genfrac{}{}{0pt}{1}{\matulaacg\matulae}{685}}
\def\Mfhf{\hskip 3mm\genfrac{}{}{0pt}{1}{\matulag\matulag\matulag\matulab}{686}}

\def\Mfhg{\hskip 3mm\genfrac{}{}{0pt}{1}{\matulabbi\ \matulac}{687}} 
\def\Mfhh{\hskip 3mm\genfrac{}{}{0pt}{1}{\matuladc\matulab\matulab\matulab\matulab}{688}}
\def\Mfhi{\hskip 3mm\genfrac{}{}{0pt}{1}{\matulaec\matulaac}{689}}
\def\Mfio{\hskip 3mm\genfrac{}{}{0pt}{1}{\matulabc\matulae\matulac\matulab}{690}}
\def\Mfia{\hskip 3mm\genfrac{}{}{0pt}{1}{\matulafia}{691}}
\def\Mfib{\hskip 3mm\genfrac{}{}{0pt}{1}{\matulaagc\matulab\matulab}{692}}
\def\Mfic{\hskip 3mm\genfrac{}{}{0pt}{1}{\matulaaa\matulag\matulac\matulac}{693}}
\def\Mfid{\hskip 3mm\genfrac{}{}{0pt}{1}{\matulacdg\matulab}{694}}
\def\Mfie{\hskip 3mm\genfrac{}{}{0pt}{1}{\matulaaci\matulae}{695}}
\def\Mfif{\hskip 3mm\genfrac{}{}{0pt}{1}{\matulabi\matulac\matulab\matulab\matulab}{696}}
\def\Mfig{\hskip 3mm\genfrac{}{}{0pt}{1}{\matulada\matulaag}{697}}
\def\Mfih{\hskip 3mm\genfrac{}{}{0pt}{1}{\matulacdi\matulab}{698}} 
\def\Mfii{\hskip 3mm\genfrac{}{}{0pt}{1}{\matulabcc\matulac}{699}}
\def\Mgoo{\hskip 3mm\genfrac{}{}{0pt}{1}{\matulag\matulae\matulae\matulab\matulab}{700}}

\def\Mgoa{\hskip 4mm\genfrac{}{}{0pt}{1}{\matulagoa}{701}} 
\def\Mgob{\hskip 4mm\genfrac{}{}{0pt}{1}{\matulaac\matulac\matulac\matulac\matulab}{702}}
\def\Mgoc{\hskip 4mm\genfrac{}{}{0pt}{1}{\matulacg\matulaai}{703}}
\def\Mgod{\hskip 4mm\genfrac{}{}{0pt}{1}{\matulaaa\matulab\matulab\matulab\matulab\matulab\matulab}{704}}
\def\Mgoe{\hskip 4mm\genfrac{}{}{0pt}{1}{\matuladg\matulae\matulac}{705}}
\def\Mgof{\hskip 4mm\genfrac{}{}{0pt}{1}{\matulacec\matulab}{706}}
\def\Mgog{\hskip 4mm\genfrac{}{}{0pt}{1}{\matulaaoa\matulag}{707}}
\def\Mgoh{\hskip 4mm\genfrac{}{}{0pt}{1}{\matulaei\matulac\matulab\matulab}{708}}
\def\Mgoi{\hskip 4mm\genfrac{}{}{0pt}{1}{\matulagoi}{709}}
\def\Mgao{\hskip 4mm\genfrac{}{}{0pt}{1}{\matulaga\matulae\matulab}{710}}
\def\Mgaa{\hskip 4mm\genfrac{}{}{0pt}{1}{\matulagi\matulac\matulac}{711}}
\def\Mgab{\hskip 4mm\genfrac{}{}{0pt}{1}{\matulahi\matulab\matulab\matulab}{712}}
\def\Mgac{\hskip 4mm\genfrac{}{}{0pt}{1}{\matulaca\matulabc}{713}}

\def\Mgad{\hskip 4mm\genfrac{}{}{0pt}{1}{\matulaag\matulag\matulac\matulab}{714}}
\def\Mgae{\hskip 4mm\genfrac{}{}{0pt}{1}{\matulaac\matulaaa\matulae}{715}}
\def\Mgaf{\hskip 4mm\genfrac{}{}{0pt}{1}{\matulaagi\matulab\matulab}{716}}
\def\Mgag{\hskip 4mm\genfrac{}{}{0pt}{1}{\matulabci\ \matulac}{717}}
\def\Mgah{\hskip 4mm\genfrac{}{}{0pt}{1}{\matulacei\matulab}{718}}
\def\Mgai{\hskip 4mm\genfrac{}{}{0pt}{1}{\matulagai}{719}}
\def\Mgbo{\hskip 4mm\genfrac{}{}{0pt}{1}{\matulae\matulac\matulac\matulab\matulab\matulab\matulab}{720}}
\def\Mgba{\hskip 4mm\genfrac{}{}{0pt}{1}{\matulaaoc\matulag}{721}}
\def\Mgbb{\hskip 4mm\genfrac{}{}{0pt}{1}{\matulaai\matulaai\matulab}{722}}
\def\Mgbc{\hskip 4mm\genfrac{}{}{0pt}{1}{\matulabda\matulac}{723}}
\def\Mgbd{\hskip 4mm\genfrac{}{}{0pt}{1}{\matulaaha\matulab\matulab}{724}}
\def\Mgbe{\hskip 4mm\genfrac{}{}{0pt}{1}{\matulabi\matulae\matulae}{725}}
\def\Mgbf{\hskip 4mm\genfrac{}{}{0pt}{1}{\matulaaa\matulaaa\matulac\matulab}{726}}

\def\Mgbg{\hskip 4mm\genfrac{}{}{0pt}{1}{\matulagbg}{727}}
\def\Mgbh{\hskip 4mm\genfrac{}{}{0pt}{1}{\matulaac\matulag\matulab\matulab\matulab}{728}}
\def\Mgbi{\hskip 4mm\genfrac{}{}{0pt}{1}{\matulac\matulac\matulac\matulac\matulac\matulac}{729}}
\def\Mgco{\hskip 4mm\genfrac{}{}{0pt}{1}{\matulagc\matulae\matulab}{730}}
\def\Mgca{\hskip 4mm\genfrac{}{}{0pt}{1}{\matuladc\matulaag}{731}}
\def\Mgcb{\hskip 4mm\genfrac{}{}{0pt}{1}{\matulafa\matulac\matulab\matulab}{732}}
\def\Mgcc{\hskip 4mm\genfrac{}{}{0pt}{1}{\matulagcc}{733}}
\def\Mgcd{\hskip 4mm\genfrac{}{}{0pt}{1}{\matulacfg\matulab}{734}}
\def\Mgce{\hskip 4mm\genfrac{}{}{0pt}{1}{\matulag\matulag\matulae\matulac}{735}}
\def\Mgcf{\hskip 4mm\genfrac{}{}{0pt}{1}{\matulabc\matulab\matulab\matulab\matulab\matulab}{736}}
\def\Mgcg{\hskip 4mm\genfrac{}{}{0pt}{1}{\matulafg\matulaaa}{737}}
\def\Mgch{\hskip 4mm\genfrac{}{}{0pt}{1}{\matulada\matulac\matulac\matulab}{738}}
\def\Mgci{\hskip 4mm\genfrac{}{}{0pt}{1}{\matulagci}{739}}

\def\Mgdo{\hskip 4mm\genfrac{}{}{0pt}{1}{\matulacg\matulae\matulab\matulab}{740}}
\def\Mgda{\hskip 4mm\genfrac{}{}{0pt}{1}{\matulaai\matulaac\matulac}{741}}
\def\Mgdb{\hskip 4mm\genfrac{}{}{0pt}{1}{\matulaec\matulag\matulab}{742}}
\def\Mgdc{\hskip 4mm\genfrac{}{}{0pt}{1}{\matulagdc}{743}}
\def\Mgdd{\hskip 4mm\genfrac{}{}{0pt}{1}{\matulaca\matulac\matulab\matulab\matulab}{744}}
\def\Mgde{\hskip 4mm\genfrac{}{}{0pt}{1}{\matulaadi\ \matulae}{745}}
\def\Mgdf{\hskip 4mm\genfrac{}{}{0pt}{1}{\matulacgc\matulab}{746}}
\def\Mgdg{\hskip 4mm\genfrac{}{}{0pt}{1}{\matulahc\matulac\matulac}{747}}
\def\Mgdh{\hskip 4mm\genfrac{}{}{0pt}{1}{\matulaag\matulaaa\matulab\matulab}{748}}
\def\Mgdi{\hskip 4mm\genfrac{}{}{0pt}{1}{\matulaaog\matulag}{749}}
\def\Mgeo{\hskip 4mm\genfrac{}{}{0pt}{1}{\matulae\matulae\matulae\matulac\matulab}{750}}
\def\Mgea{\hskip 4mm\genfrac{}{}{0pt}{1}{\matulagea}{751}} 
\def\Mgeb{\hskip 4mm\genfrac{}{}{0pt}{1}{\matuladg\matulab\matulab\matulab\matulab}{752}}

\def\Mgec{\hskip 4mm\genfrac{}{}{0pt}{1}{\matulabea\matulac}{753}}
\def\Mged{\hskip 4mm\genfrac{}{}{0pt}{1}{\matulabi\matulaac\matulab}{754}}
\def\Mgee{\hskip 4mm\genfrac{}{}{0pt}{1}{\matulaaea\matulae}{755}}
\def\Mgef{\hskip 4mm\genfrac{}{}{0pt}{1}{\matulag\matulac\matulac\matulac\matulab\matulab}{756}}
\def\Mgeg{\hskip 4mm\genfrac{}{}{0pt}{1}{\matulageg}{757}}
\def\Mgeh{\hskip 4mm\genfrac{}{}{0pt}{1}{\matulacgi\matulab}{758}}
\def\Mgei{\hskip 4mm\genfrac{}{}{0pt}{1}{\matulabc\matulaaa\matulac}{759}}
\def\Mgfo{\hskip 4mm\genfrac{}{}{0pt}{1}{\matulaai\matulae\matulab\matulab\matulab}{760}}
\def\Mgfa{\hskip 4mm\genfrac{}{}{0pt}{1}{\matulagfa}{761}} 
\def\Mgfb{\hskip 4mm\genfrac{}{}{0pt}{1}{\matulaabg\matulac\matulab}{762}}
\def\Mgfc{\hskip 4mm\genfrac{}{}{0pt}{1}{\matulaaoi\matulag}{763}}
\def\Mgfd{\hskip 4mm\genfrac{}{}{0pt}{1}{\matulaaia\matulab\matulab}{764}}

\def\Mgfe{\hskip 4mm\genfrac{}{}{0pt}{1}{\matulaag\matulae\matulac\matulac}{765}}
\def\Mgff{\hskip 4mm\genfrac{}{}{0pt}{1}{\matulachc\matulab}{766}}
\def\Mgfg{\hskip 4mm\genfrac{}{}{0pt}{1}{\matulaei\matulaac}{767}}
\def\Mgfh{\hskip 4mm\genfrac{}{}{0pt}{1}{\matulac\matulab\matulab\matulab\matulab\matulab\matulab\matulab\matulab}{768}}
\def\Mgfi{\hskip 4mm\genfrac{}{}{0pt}{1}{\matulagfi}{769}} 
\def\Mggo{\hskip 4mm\genfrac{}{}{0pt}{1}{\matulaaa\matulag\matulae\matulab}{770}}
\def\Mgga{\hskip 4mm\genfrac{}{}{0pt}{1}{\matulabeg\matulac}{771}}
\def\Mggb{\hskip 4mm\genfrac{}{}{0pt}{1}{\matulaaic\matulab\matulab}{772}}
\def\Mggc{\hskip 4mm\genfrac{}{}{0pt}{1}{\matulaggc}{773}}
\def\Mggd{\hskip 4mm\genfrac{}{}{0pt}{1}{\matuladc\matulac\matulac\matulab}{774}}
\def\Mgge{\hskip 4mm\genfrac{}{}{0pt}{1}{\matulaca\matulae\matulae}{775}}
\def\Mggf{\hskip 4mm\genfrac{}{}{0pt}{1}{\matulaig\matulab\matulab\matulab}{776}}

\def\Mggg{\hskip 4mm\genfrac{}{}{0pt}{1}{\matulacg\matulag\matulac}{777}}
\def\Mggh{\hskip 4mm\genfrac{}{}{0pt}{1}{\matulachi\matulab}{778}}
\def\Mggi{\hskip 4mm\genfrac{}{}{0pt}{1}{\matulada\matulaai}{779}}
\def\Mgho{\hskip 4mm\genfrac{}{}{0pt}{1}{\matulaac\matulae\matulac\matulab\matulab}{780}}
\def\Mgha{\hskip 4mm\genfrac{}{}{0pt}{1}{\matulaga\matulaaa}{781}}
\def\Mghb{\hskip 4mm\genfrac{}{}{0pt}{1}{\matulabc\matulaag\matulab}{782}}
\def\Mghc{\hskip 4mm\genfrac{}{}{0pt}{1}{\matulabi\matulac\matulac\matulac}{783}}
\def\Mghd{\hskip 4mm\genfrac{}{}{0pt}{1}{\matulag\matulag\matulab\matulab\matulab\matulab}{784}}
\def\Mghe{\hskip 4mm\genfrac{}{}{0pt}{1}{\matulaaeg\matulae}{785}}
\def\Mghf{\hskip 4mm\genfrac{}{}{0pt}{1}{\matulaaca\matulac\matulab}{786}}
\def\Mghg{\hskip 4mm\genfrac{}{}{0pt}{1}{\matulaghg}{787}} 
\def\Mghh{\hskip 4mm\genfrac{}{}{0pt}{1}{\matulaaig\matulab\matulab}{788}}

\def\Mghi{\hskip 4mm\genfrac{}{}{0pt}{1}{\matulabfc\ \,\matulac}{789}}
\def\Mgio{\hskip 4mm\genfrac{}{}{0pt}{1}{\matulagi\matulae\matulab}{790}}
\def\Mgia{\hskip 4mm\genfrac{}{}{0pt}{1}{\matulaaac\matulag}{791}}
\def\Mgib{\hskip 4mm\genfrac{}{}{0pt}{1}{\matulaaa\matulac\matulac\matulab\matulab\matulab}{792}}
\def\Mgic{\hskip 4mm\genfrac{}{}{0pt}{1}{\matulafa\matulaac}{793}}
\def\Mgid{\hskip 4mm\genfrac{}{}{0pt}{1}{\matulacig\matulab}{794}}
\def\Mgie{\hskip 4mm\genfrac{}{}{0pt}{1}{\matulaec\matulae\matulac}{795}}
\def\Mgif{\hskip 4mm\genfrac{}{}{0pt}{1}{\matulaaii\matulab\matulab}{796}}
\def\Mgig{\hskip 4mm\genfrac{}{}{0pt}{1}{\matulagig}{797}}
\def\Mgih{\hskip 4mm\genfrac{}{}{0pt}{1}{\matulaai\matulag\matulac\matulab}{798}}
\def\Mgii{\hskip 4mm\genfrac{}{}{0pt}{1}{\matulacg\matulaag}{799}}
\def\Mhoo{\hskip 4mm\genfrac{}{}{0pt}{1}{\matulae\matulae\matulab\matulab\matulab\matulab\matulab}{800}}

\def\Mhoa{\hskip 4mm\genfrac{}{}{0pt}{1}{\matulahi\matulac\matulac}{801}}
\def\Mhob{\hskip 4mm\genfrac{}{}{0pt}{1}{\matuladoa\matulab}{802}}
\def\Mhoc{\hskip 4mm\genfrac{}{}{0pt}{1}{\matulagc\matulaaa}{803}}
\def\Mhod{\hskip 4mm\genfrac{}{}{0pt}{1}{\matulafg\matulac\matulab\matulab}{804}}
\def\Mhoe{\hskip 4mm\genfrac{}{}{0pt}{1}{\matulabc\matulag\matulae}{805}}
\def\Mhof{\hskip 4mm\genfrac{}{}{0pt}{1}{\matulaca\matulaac\matulab}{806}}
\def\Mhog{\hskip 4mm\genfrac{}{}{0pt}{1}{\matulabfi\ \matulac}{807}}
\def\Mhoh{\hskip 4mm\genfrac{}{}{0pt}{1}{\matulaaoa\matulab\matulab\matulab}{808}}
\def\Mhoi{\hskip 4mm\genfrac{}{}{0pt}{1}{\matulahoi}{809}}
\def\Mhao{\hskip 4mm\genfrac{}{}{0pt}{1}{\matulae\matulac\matulac\matulac\matulac\matulab}{810}}
\def\Mhaa{\hskip 4mm\genfrac{}{}{0pt}{1}{\matulahaa}{811}}
\def\Mhab{\hskip 4mm\genfrac{}{}{0pt}{1}{\matulabi\matulag\matulab\matulab}{812}}
\def\Mhac{\hskip 4mm\genfrac{}{}{0pt}{1}{\matulabga\matulac}{813}}

\def\Mhad{\hskip 4mm\genfrac{}{}{0pt}{1}{\matulacg\matulaaa\matulab}{814}}
\def\Mhae{\hskip 4mm\genfrac{}{}{0pt}{1}{\matulaafc\matulae}{815}}
\def\Mhaf{\hskip 4mm\genfrac{}{}{0pt}{1}{\matulaag\matulac\matulab\matulab\matulab\matulab}{816}}
\def\Mhag{\hskip 4mm\genfrac{}{}{0pt}{1}{\matuladc\matulaai}{817}}
\def\Mhah{\hskip 4mm\genfrac{}{}{0pt}{1}{\matuladoi\matulab}{818}}
\def\Mhai{\hskip 4mm\genfrac{}{}{0pt}{1}{\matulaac\matulag\matulac\matulac}{819}}
\def\Mhbo{\hskip 4mm\genfrac{}{}{0pt}{1}{\matulada\matulae\matulab\matulab}{820}}
\def\Mhba{\hskip 4mm\genfrac{}{}{0pt}{1}{\matulahba}{821}}
\def\Mhbb{\hskip 4mm\genfrac{}{}{0pt}{1}{\matulaacg\matulac\matulab}{822}}
\def\Mhbc{\hskip 4mm\genfrac{}{}{0pt}{1}{\matulahbc}{823}}
\def\Mhbd{\hskip 4mm\genfrac{}{}{0pt}{1}{\matulaaoc\matulab\matulab\matulab}{824}}
\def\Mhbe{\hskip 4mm\genfrac{}{}{0pt}{1}{\matulaaa\matulae\matulae\matulac}{825}}
\def\Mhbf{\hskip 4mm\genfrac{}{}{0pt}{1}{\matulaei\matulag\matulab}{826}}

\def\Mhbg{\hskip 4mm\genfrac{}{}{0pt}{1}{\matulahbg}{827}}
\def\Mhbh{\hskip 4mm\genfrac{}{}{0pt}{1}{\matulabc\matulac\matulac\matulab\matulab}{828}}
\def\Mhbi{\hskip 4mm\genfrac{}{}{0pt}{1}{\matulahbi}{829}}
\def\Mhco{\hskip 4mm\genfrac{}{}{0pt}{1}{\matulahc\matulae\matulab}{830}}
\def\Mhca{\hskip 4mm\genfrac{}{}{0pt}{1}{\matulabgg\matulac}{831}}
\def\Mhcb{\hskip 4mm\genfrac{}{}{0pt}{1}{\matulaac\matulab\matulab\matulab\matulab\matulab\matulab}{832}}
\def\Mhcc{\hskip 4mm\genfrac{}{}{0pt}{1}{\matulaag\matulag\matulag}{833}}
\def\Mhcd{\hskip 4mm\genfrac{}{}{0pt}{1}{\matulaaci\matulac\matulab}{834}}
\def\Mhce{\hskip 4mm\genfrac{}{}{0pt}{1}{\matulaafg\matulae}{835}}
\def\Mhcf{\hskip 4mm\genfrac{}{}{0pt}{1}{\matulaai\matulaaa\matulab\matulab}{836}}
\def\Mhcg{\hskip 4mm\genfrac{}{}{0pt}{1}{\matulaca\matulac\matulac\matulac}{837}}
\def\Mhch{\hskip 4mm\genfrac{}{}{0pt}{1}{\matuladai\ \matulab}{838}}
\def\Mhci{\hskip 4mm\genfrac{}{}{0pt}{1}{\matulahci}{839}}

\def\Mhdo{\hskip 4mm\genfrac{}{}{0pt}{1}{\matulag\matulae\matulac\matulab\matulab\matulab}{840}}
\def\Mhda{\hskip 4mm\genfrac{}{}{0pt}{1}{\matulabi\matulabi}{841}}
\def\Mhdb{\hskip 4mm\genfrac{}{}{0pt}{1}{\matuladba\matulab}{842}}
\def\Mhdc{\hskip 4mm\genfrac{}{}{0pt}{1}{\matulabha\ \,\matulac}{843}}
\def\Mhdd{\hskip 4mm\genfrac{}{}{0pt}{1}{\matulabaa\matulab\matulab}{844}}
\def\Mhde{\hskip 4mm\genfrac{}{}{0pt}{1}{\matulaac\matulaac\matulae}{845}}
\def\Mhdf{\hskip 4mm\genfrac{}{}{0pt}{1}{\matuladg\matulac\matulac\matulab}{846}}
\def\Mhdg{\hskip 4mm\genfrac{}{}{0pt}{1}{\matulaaa\matulaaa\matulag}{847}}
\def\Mhdh{\hskip 4mm\genfrac{}{}{0pt}{1}{\matulaec\matulab\matulab\matulab\matulab}{848}}
\def\Mhdi{\hskip 4mm\genfrac{}{}{0pt}{1}{\matulabhc\matulac}{849}}
\def\Mheo{\hskip 4mm\genfrac{}{}{0pt}{1}{\matulaag\matulae\matulae\matulab}{850}}
\def\Mhea{\hskip 4mm\genfrac{}{}{0pt}{1}{\matulacg\matulabc}{851}}
\def\Mheb{\hskip 4mm\genfrac{}{}{0pt}{1}{\matulaga\matulac\matulab\matulab}{852}}

\def\Mhec{\hskip 4mm\genfrac{}{}{0pt}{1}{\matulahec}{853}} 
\def\Mhed{\hskip 4mm\genfrac{}{}{0pt}{1}{\matulafa\matulag\matulab}{854}}
\def\Mhee{\hskip 4mm\genfrac{}{}{0pt}{1}{\matulaai\matulae\matulac\matulac}{855}}
\def\Mhef{\hskip 4mm\genfrac{}{}{0pt}{1}{\matulaaog\matulab\matulab\matulab}{856}}
\def\Mheg{\hskip 4mm\genfrac{}{}{0pt}{1}{\matulaheg}{857}}
\def\Mheh{\hskip 4mm\genfrac{}{}{0pt}{1}{\matulaac\matulaaa\matulac\matulab}{858}}
\def\Mhei{\hskip 4mm\genfrac{}{}{0pt}{1}{\matulahei}{859}}
\def\Mhfo{\hskip 4mm\genfrac{}{}{0pt}{1}{\matuladc\matulae\matulab\matulab}{860}}
\def\Mhfa{\hskip 4mm\genfrac{}{}{0pt}{1}{\matulada\matulag\matulac}{861}}
\def\Mhfb{\hskip 4mm\genfrac{}{}{0pt}{1}{\matuladca\matulab}{862}}
\def\Mhfc{\hskip 4mm\genfrac{}{}{0pt}{1}{\matulahfc}{863}}
\def\Mhfd{\hskip 4mm\genfrac{}{}{0pt}{1}{\matulac\matulac\matulac\matulab\matulab\matulab\matulab\matulab}{864}}

\def\Mhfe{\hskip 4mm\genfrac{}{}{0pt}{1}{\matulaagc\matulae}{865}}
\def\Mhff{\hskip 4mm\genfrac{}{}{0pt}{1}{\matuladcc\ \matulab}{866}}
\def\Mhfg{\hskip 4mm\genfrac{}{}{0pt}{1}{\matulaag\matulaag\matulac}{867}}
\def\Mhfh{\hskip 4mm\genfrac{}{}{0pt}{1}{\matulaca\matulag\matulab\matulab}{868}}
\def\Mhfi{\hskip 4mm\genfrac{}{}{0pt}{1}{\matulagi\matulaaa}{869}}
\def\Mhgo{\hskip 4mm\genfrac{}{}{0pt}{1}{\matulabi\matulae\matulac\matulab}{870}}
\def\Mhga{\hskip 4mm\genfrac{}{}{0pt}{1}{\matulafg\matulaac}{871}}
\def\Mhgb{\hskip 4mm\genfrac{}{}{0pt}{1}{\matulaaoi\matulab\matulab\matulab}{872}}
\def\Mhgc{\hskip 4mm\genfrac{}{}{0pt}{1}{\matulaig\matulac\matulac}{873}}
\def\Mhgd{\hskip 4mm\genfrac{}{}{0pt}{1}{\matulabc\matulaai\matulab}{874}}
\def\Mhge{\hskip 4mm\genfrac{}{}{0pt}{1}{\matulag\matulae\matulae\matulae}{875}}
\def\Mhgf{\hskip 4mm\genfrac{}{}{0pt}{1}{\matulagc\matulac\matulab\matulab}{876}}

\def\Mhgg{\hskip 4mm\genfrac{}{}{0pt}{1}{\matulahgg}{877}}
\def\Mhgh{\hskip 4mm\genfrac{}{}{0pt}{1}{\matuladci\matulab}{878}}
\def\Mhgi{\hskip 4mm\genfrac{}{}{0pt}{1}{\matulabic\ \matulac}{879}}
\def\Mhho{\hskip 4mm\genfrac{}{}{0pt}{1}{\matulaaa\matulae\matulab\matulab\matulab\matulab}{880}}
\def\Mhha{\hskip 4mm\genfrac{}{}{0pt}{1}{\matulahha}{881}}
\def\Mhhb{\hskip 4mm\genfrac{}{}{0pt}{1}{\matulag\matulag\matulac\matulac\matulab}{882}}
\def\Mhhc{\hskip 4mm\genfrac{}{}{0pt}{1}{\matulahhc}{883}}
\def\Mhhd{\hskip 4mm\genfrac{}{}{0pt}{1}{\matulaag\matulaac\matulab\matulab}{884}}
\def\Mhhe{\hskip 4mm\genfrac{}{}{0pt}{1}{\matulaei\matulae\matulac}{885}}
\def\Mhhf{\hskip 4mm\genfrac{}{}{0pt}{1}{\matuladdc\matulab}{886}}
\def\Mhhg{\hskip 4mm\genfrac{}{}{0pt}{1}{\matulahhg}{887}}
\def\Mhhh{\hskip 4mm\genfrac{}{}{0pt}{1}{\matulacg\matulac\matulab\matulab\matulab}{888}}

\def\Mhhi{\hskip 4mm\genfrac{}{}{0pt}{1}{\matulaabg\matulag}{889}}
\def\Mhio{\hskip 4mm\genfrac{}{}{0pt}{1}{\matulahi\matulae\matulab}{890}}
\def\Mhia{\hskip 4mm\genfrac{}{}{0pt}{1}{\matulaaa\matulac\matulac\matulac\matulac}{891}}
\def\Mhib{\hskip 4mm\genfrac{}{}{0pt}{1}{\matulabbc\matulab\matulab}{892}}
\def\Mhic{\hskip 4mm\genfrac{}{}{0pt}{1}{\matuladg\matulaai}{893}}
\def\Mhid{\hskip 4mm\genfrac{}{}{0pt}{1}{\matulaadi\matulac\matulab}{894}}
\def\Mhie{\hskip 4mm\genfrac{}{}{0pt}{1}{\matulaagi\matulae}{895}}
\def\Mhif{\hskip 4mm\genfrac{}{}{0pt}{1}{\matulag\matulab\matulab\matulab\matulab\matulab\matulab\matulab}{896}}
\def\Mhig{\hskip 4mm\genfrac{}{}{0pt}{1}{\matulabc\matulaac\matulac}{897}}
\def\Mhih{\hskip 4mm\genfrac{}{}{0pt}{1}{\matuladdi\matulab}{898}}
\def\Mhii{\hskip 4mm\genfrac{}{}{0pt}{1}{\matulaca\matulabi}{899}}
\def\Mioo{\hskip 4mm\genfrac{}{}{0pt}{1}{\matulae\matulae\matulac\matulac\matulab\matulab}{900}}

\def\Mioa{\hskip 4mm\genfrac{}{}{0pt}{1}{\matulaec\matulaag}{901}}
\def\Miob{\hskip 4mm\genfrac{}{}{0pt}{1}{\matulada\matulaaa\matulab}{902}}
\def\Mioc{\hskip 4mm\genfrac{}{}{0pt}{1}{\matuladc\matulag\matulac}{903}}
\def\Miod{\hskip 4mm\genfrac{}{}{0pt}{1}{\matulaaac\matulab\matulab\matulab}{904}}
\def\Mioe{\hskip 4mm\genfrac{}{}{0pt}{1}{\matulaaha\matulae}{905}}
\def\Miof{\hskip 4mm\genfrac{}{}{0pt}{1}{\matulaaea\matulac\matulab}{906}}
\def\Miog{\hskip 4mm\genfrac{}{}{0pt}{1}{\matulaiog}{907}}
\def\Mioh{\hskip 4mm\genfrac{}{}{0pt}{1}{\matulabbg\matulab\matulab}{908}}
\def\Mioi{\hskip 4mm\genfrac{}{}{0pt}{1}{\matulaaoa\matulac\matulac}{909}}
\def\Miao{\hskip 4mm\genfrac{}{}{0pt}{1}{\matulaac\matulag\matulae\matulab}{910}}
\def\Miaa{\hskip 4mm\genfrac{}{}{0pt}{1}{\matulaiaa}{911}}
\def\Miab{\hskip 4mm\genfrac{}{}{0pt}{1}{\matulaai\matulac\matulab\matulab\matulab\matulab}{912}}
\def\Miac{\hskip 4mm\genfrac{}{}{0pt}{1}{\matulahc\matulaaa}{913}}

\def\Miad{\hskip 4mm\genfrac{}{}{0pt}{1}{\matuladeg\matulab}{914}}
\def\Miae{\hskip 4mm\genfrac{}{}{0pt}{1}{\matulafa\matulae\matulac}{915}}
\def\Miaf{\hskip 4mm\genfrac{}{}{0pt}{1}{\matulabbi\matulab\matulab}{916}}
\def\Miag{\hskip 4mm\genfrac{}{}{0pt}{1}{\matulaaca\matulag}{917}}
\def\Miah{\hskip 4mm\genfrac{}{}{0pt}{1}{\matulaag\matulac\matulac\matulac\matulab}{918}}
\def\Miai{\hskip 4mm\genfrac{}{}{0pt}{1}{\matulaiai}{919}}
\def\Mibo{\hskip 4mm\genfrac{}{}{0pt}{1}{\matulabc\matulae\matulab\matulab\matulab}{920}}
\def\Miba{\hskip 4mm\genfrac{}{}{0pt}{1}{\matulacog\ \matulac}{921}}
\def\Mibb{\hskip 4mm\genfrac{}{}{0pt}{1}{\matuladfa\matulab}{922}}
\def\Mibc{\hskip 4mm\genfrac{}{}{0pt}{1}{\matulaga\matulaac}{923}}
\def\Mibd{\hskip 4mm\genfrac{}{}{0pt}{1}{\matulaaa\matulag\matulac\matulab\matulab}{924}}
\def\Mibe{\hskip 4mm\genfrac{}{}{0pt}{1}{\matulacg\matulae\matulae}{925}}
\def\Mibf{\hskip 4mm\genfrac{}{}{0pt}{1}{\matuladfc\ \matulab}{926}}

\def\Mibg{\hskip 4mm\genfrac{}{}{0pt}{1}{\matulaaoc\ \matulac\matulac}{927}}
\def\Mibh{\hskip 4mm\genfrac{}{}{0pt}{1}{\matulabi\matulab\matulab\matulab\matulab\matulab}{928}}
\def\Mibi{\hskip 4mm\genfrac{}{}{0pt}{1}{\matulaibi}{929}}
\def\Mico{\hskip 4mm\genfrac{}{}{0pt}{1}{\matulaca\matulae\matulac\matulab}{930}}
\def\Mica{\hskip 4mm\genfrac{}{}{0pt}{1}{\matulaai\matulag\matulag}{931}}
\def\Micb{\hskip 4mm\genfrac{}{}{0pt}{1}{\matulabcc\matulab\matulab}{932}}
\def\Micc{\hskip 4mm\genfrac{}{}{0pt}{1}{\matulacaa\,\matulac}{933}}
\def\Micd{\hskip 4mm\genfrac{}{}{0pt}{1}{\matuladfg\matulab}{934}}
\def\Mice{\hskip 4mm\genfrac{}{}{0pt}{1}{\matulaag\matulaaa\matulae}{935}}
\def\Micf{\hskip 4mm\genfrac{}{}{0pt}{1}{\matulaac\matulac\matulac\matulab\matulab\matulab}{936}}
\def\Micg{\hskip 4mm\genfrac{}{}{0pt}{1}{\matulaicg}{937}}
\def\Mich{\hskip 4mm\genfrac{}{}{0pt}{1}{\matulafg\matulag\matulab}{938}}
\def\Mici{\hskip 4mm\genfrac{}{}{0pt}{1}{\matulacac\, \matulac}{939}}

\def\Mido{\hskip 4mm\genfrac{}{}{0pt}{1}{\matuladg\matulae\matulab\matulab}{940}}
\def\Mida{\hskip 4mm\genfrac{}{}{0pt}{1}{\matulaida}{941}}
\def\Midb{\hskip 4mm\genfrac{}{}{0pt}{1}{\matulaaeg\matulac\matulab}{942}}
\def\Midc{\hskip 4mm\genfrac{}{}{0pt}{1}{\matulada\matulabc}{943}}
\def\Midd{\hskip 4mm\genfrac{}{}{0pt}{1}{\matulaei\matulab\matulab\matulab\matulab}{944}}
\def\Mide{\hskip 4mm\genfrac{}{}{0pt}{1}{\matulag\matulae\matulac\matulac\matulac}{945}}
\def\Midf{\hskip 4mm\genfrac{}{}{0pt}{1}{\matuladc\matulaaa\matulab}{946}}
\def\Midg{\hskip 4mm\genfrac{}{}{0pt}{1}{\matulaidg}{947}}
\def\Midh{\hskip 4mm\genfrac{}{}{0pt}{1}{\matulagi\matulac\matulab\matulab}{948}}
\def\Midi{\hskip 4mm\genfrac{}{}{0pt}{1}{\matulagc\matulaac}{949}}
\def\Mieo{\hskip 4mm\genfrac{}{}{0pt}{1}{\matulaai\matulae\matulae\matulab}{950}}
\def\Miea{\hskip 4mm\genfrac{}{}{0pt}{1}{\matulacag\,\matulac}{951}}
\def\Mieb{\hskip 4mm\genfrac{}{}{0pt}{1}{\matulaag\matulaag\matulab\matulab\matulab}{952}}

\def\Miec{\hskip 4mm\genfrac{}{}{0pt}{1}{\matulaiec}{953}}
\def\Mied{\hskip 4mm\genfrac{}{}{0pt}{1}{\matulaec\matulac\matulac\matulab}{954}}
\def\Miee{\hskip 4mm\genfrac{}{}{0pt}{1}{\matulaaia\matulae}{955}}
\def\Mief{\hskip 4mm\genfrac{}{}{0pt}{1}{\matulabci\matulab\matulab}{956}}
\def\Mieg{\hskip 4mm\genfrac{}{}{0pt}{1}{\matulabi\matulaaa\matulac}{957}}
\def\Mieh{\hskip 4mm\genfrac{}{}{0pt}{1}{\matuladgi\matulab}{958}}
\def\Miei{\hskip 4mm\genfrac{}{}{0pt}{1}{\matulaacg\matulag}{959}}
\def\Mifo{\hskip 4mm\genfrac{}{}{0pt}{1}{\matulae\matulac\matulab\matulab\matulab\matulab\matulab\matulab}{960}}
\def\Mifa{\hskip 4mm\genfrac{}{}{0pt}{1}{\matulaca\matulaca}{961}}
\def\Mifb{\hskip 4mm\genfrac{}{}{0pt}{1}{\matulacg\matulaac\matulab}{962}}
\def\Mifc{\hskip 4mm\genfrac{}{}{0pt}{1}{\matulaaog\matulac\matulac}{963}}
\def\Mifd{\hskip 4mm\genfrac{}{}{0pt}{1}{\matulabda\matulab\matulab}{964}}

\def\Mife{\hskip 4mm\genfrac{}{}{0pt}{1}{\matulaaic\matulae}{965}}
\def\Miff{\hskip 4mm\genfrac{}{}{0pt}{1}{\matulabc\matulag\matulac\matulab}{966}}
\def\Mifg{\hskip 4mm\genfrac{}{}{0pt}{1}{\matulaifg}{967}}
\def\Mifh{\hskip 4mm\genfrac{}{}{0pt}{1}{\matulaaa\matulaaa\matulab\matulab\matulab}{968}}
\def\Mifi{\hskip 4mm\genfrac{}{}{0pt}{1}{\matulaai\matulaag\matulac}{969}}
\def\Migo{\hskip 4mm\genfrac{}{}{0pt}{1}{\matulaig\matulae\matulab}{970}}
\def\Miga{\hskip 4mm\genfrac{}{}{0pt}{1}{\matulaiga}{971}}
\def\Migb{\hskip 4mm\genfrac{}{}{0pt}{1}{\matulac\matulac\matulac\matulac\matulac\matulab\matulab}{972}}
\def\Migc{\hskip 4mm\genfrac{}{}{0pt}{1}{\matulaaci\matulag}{973}}
\def\Migd{\hskip 4mm\genfrac{}{}{0pt}{1}{\matuladhg\matulab}{974}}
\def\Mige{\hskip 4mm\genfrac{}{}{0pt}{1}{\matulaac\matulae\matulae\matulac}{975}}
\def\Migf{\hskip 4mm\genfrac{}{}{0pt}{1}{\matulafa\matulab\matulab\matulab\matulab}{976}}

\def\Migg{\hskip 4mm\genfrac{}{}{0pt}{1}{\matulaigg}{977}}
\def\Migh{\hskip 4mm\genfrac{}{}{0pt}{1}{\matulaafc\matulac\matulab}{978}}
\def\Migi{\hskip 4mm\genfrac{}{}{0pt}{1}{\matulahi\matulaaa}{979}}
\def\Miho{\hskip 4mm\genfrac{}{}{0pt}{1}{\matulag\matulag\matulae\matulab\matulab}{980}}
\def\Miha{\hskip 4mm\genfrac{}{}{0pt}{1}{\matulaaoi\matulac\matulac}{981}}
\def\Mihb{\hskip 4mm\genfrac{}{}{0pt}{1}{\matuladia\matulab}{982}}
\def\Mihc{\hskip 4mm\genfrac{}{}{0pt}{1}{\matulaihc}{983}}
\def\Mihd{\hskip 4mm\genfrac{}{}{0pt}{1}{\matulada\matulac\matulab\matulab\matulab}{984}}
\def\Mihe{\hskip 4mm\genfrac{}{}{0pt}{1}{\matulaaig\, \matulae}{985}}
\def\Mihf{\hskip 4mm\genfrac{}{}{0pt}{1}{\matulabi\matulaag\matulab}{986}}
\def\Mihg{\hskip 4mm\genfrac{}{}{0pt}{1}{\matuladg\matulag\matulac}{987}}
\def\Mihh{\hskip 4mm\genfrac{}{}{0pt}{1}{\matulaai\matulaac\matulab\matulab}{988}}

\def\Mihi{\hskip 4mm\genfrac{}{}{0pt}{1}{\matuladc\matulabc}{989}}
\def\Miio{\hskip 4mm\genfrac{}{}{0pt}{1}{\matulaaa\matulae\matulac\matulac\matulab}{990}}
\def\Miia{\hskip 4mm\genfrac{}{}{0pt}{1}{\matulaiia}{991}}
\def\Miib{\hskip 4mm\genfrac{}{}{0pt}{1}{\matulaca\matulab\matulab\matulab\matulab\matulab}{992}}
\def\Miic{\hskip 4mm\genfrac{}{}{0pt}{1}{\matulacca\matulac}{993}}
\def\Miid{\hskip 4mm\genfrac{}{}{0pt}{1}{\matulaga\matulag\matulab}{994}}
\def\Miie{\hskip 4mm\genfrac{}{}{0pt}{1}{\matulaaii\matulae}{995}}
\def\Miif{\hskip 4mm\genfrac{}{}{0pt}{1}{\matulahc\matulac\matulab\matulab}{996}}
\def\Miig{\hskip 4mm\genfrac{}{}{0pt}{1}{\matulaiig}{997}}
\def\Miih{\hskip 4mm\genfrac{}{}{0pt}{1}{\matuladii\matulab}{998}}
\def\Miii{\hskip 4mm\genfrac{}{}{0pt}{1}{\matulacg\matulac\matulac\matulac}{999}}
\def\Maooo{\hskip 4mm\genfrac{}{}{0pt}{1}{\matulae\matulae\matulae\matulab\matulab\matulab}{1000}}
\vfill
$$
\Ma\Mb\Mc\Md\Me\Mf\Mg\Mh\Mi\Mao\Maa\Mab\Mac\Mad\Mae\Maf\Mag\Mah\Mai\Mbo
$$

$$
\Mba\Mbb\Mbc\Mbd\Mbe\Mbf\Mbg\Mbh\Mbi\Mco\Mca\Mcb\Mcc\Mcd\Mce\Mcf\Mcg\Mch\Mci\Mdo
$$

$$
\hskip -3mm\Mda\Mdb\Mdc\Mdd\Mde\Mdf\Mdg\Mdh\Mdi\Meo\Mea\Meb\Mec\Med\Mee\Mef\Meg\Meh\Mei\Mfo
$$

$$
\hskip -4mm\Mfa\Mfb\Mfc\Mfd\Mfe\Mff\Mfg\Mfh\Mfi\Mgo\Mga\Mgb\Mgc\Mgd\Mge\Mgf\Mgg\Mgh\Mgi\Mho
$$

$$
\hskip -7mm\Mha\Mhb\Mhc\Mhd\Mhe\Mhf\Mhg\Mhh\Mhi\Mio\Mia\Mib\Mic\Mid\Mie\Mif\Mig\Mih\Mii\Maoo
$$
\vfill
$$
\hskip-3mm\Maoa\Maob\Maoc\Maod\Maoe\Maof\Maog\Maoh\Maoi\Maao\Maaa\Maab\Maac\Maad\Maae\Maaf\Maag
$$

$$
\hskip -5mm\Maah\Maai\Mabo\Maba\Mabb\Mabc\Mabd\Mabe\Mabf\Mabg\Mabh\Mabi\Maco\Maca\Macb\Macc\Macd
$$

$$
\hskip -4mm\Mace\Macf\Macg\Mach\Maci\Mado\Mada\Madb\Madc\Madd\Made\Madf\Madg\Madh\Madi\Maeo\Maea
$$

$$
\hskip -7mm\Maeb\Maec\Maed\Maee\Maef\Maeg\Maeh\Maei\Mafo\Mafa\Mafb\Mafc\Mafd\Mafe\Maff\Mafg\Mafh
$$

$$
\Mafi\Mago\Maga\Magb\Magc\Magd\Mage\Magf\Magg\Magh\Magi\Maho\Maha\Mahb\Mahc\Mahd
$$

$$
\Mahe\Mahf\Mahg\Mahh\Mahi\Maio\Maia\Maib\Maic\Maid\Maie\Maif\Maig\Maih\Maii\Mboo
$$
\eject
$$
\Mboa\Mbob\Mboc\Mbod\Mboe\Mbof\Mbog\Mboh\Mboi\Mbao\Mbaa\Mbab\Mbac\Mbad\Mbae
$$

$$
\Mbaf\Mbag\Mbah\Mbai\Mbbo\Mbba\Mbbb\Mbbc\Mbbd\Mbbe\Mbbf\Mbbg\Mbbh\Mbbi\Mbco
$$

$$
\Mbca\Mbcb\Mbcc\Mbcd\Mbce\Mbcf\Mbcg\Mbch\Mbci\Mbdo\Mbda\Mbdb\Mbdc\Mbdd
$$

$$
\Mbde\Mbdf\Mbdg\Mbdh\Mbdi\Mbeo\Mbea\Mbeb\Mbec\Mbed\Mbee\Mbef\Mbeg\Mbeh
$$

$$
\Mbei\Mbfo\Mbfa\Mbfb\Mbfc\Mbfd\Mbfe\Mbff\Mbfg\Mbfh\Mbfi\Mbgo\Mbga\Mbgb
$$

$$
\Mbgc\Mbgd\Mbge\Mbgf\Mbgg\Mbgh\Mbgi\Mbho\Mbha\Mbhb\Mbhc\Mbhd\Mbhe\Mbhf
$$

$$
\Mbhg\Mbhh\Mbhi\Mbio\Mbia\Mbib\Mbic\Mbid\Mbie\Mbif\Mbig\Mbih\Mbii\Mcoo
$$
\vfill
$$
\Mcoa\Mcob\Mcoc\Mcod\Mcoe\Mcof\Mcog\Mcoh\Mcoi\Mcao\Mcaa\Mcab\Mcac\Mcad\Mcae
$$

$$
\Mcaf\Mcag\Mcah\Mcai\Mcbo\Mcba\Mcbb\Mcbc\Mcbd\Mcbe\Mcbf\Mcbg\Mcbh\Mcbi\Mcco
$$

$$
\Mcca\Mccb\Mccc\Mccd\Mcce\Mccf\Mccg\Mcch\Mcci\Mcdo\Mcda\Mcdb\Mcdc\Mcdd
$$

$$
\Mcde\Mcdf\Mcdg\Mcdh\Mcdi\Mceo\Mcea\Mceb\Mcec\Mced\Mcee\Mcef\Mceg\Mceh
$$

$$
\Mcei\Mcfo\Mcfa\Mcfb\Mcfc\Mcfd\Mcfe\Mcff\Mcfg\Mcfh\Mcfi\Mcgo\Mcga\Mcgb
$$

$$
\Mcgc\Mcgd\Mcge\Mcgf\Mcgg\Mcgh\Mcgi\Mcho\Mcha\Mchb\Mchc\Mchd\Mche\Mchf
$$

$$
\Mchg\Mchh\Mchi\Mcio\Mcia\Mcib\Mcic\Mcid\Mcie\Mcif\Mcig\Mcih\Mcii\Mdoo
$$
\vfill
$$
\Mdoa\Mdob\Mdoc\Mdod\Mdoe\Mdof\Mdog\Mdoh\Mdoi\Mdao\Mdaa\Mdab\Mdac\Mdad\Mdae
$$

$$
\Mdaf\Mdag\Mdah\Mdai\Mdbo\Mdba\Mdbb\Mdbc\Mdbd\Mdbe\Mdbf\Mdbg\Mdbh\Mdbi\Mdco
$$

$$
\Mdca\Mdcb\Mdcc\Mdcd\Mdce\Mdcf\Mdcg\Mdch\Mdci\Mddo\Mdda\Mddb\Mddc\Mddd
$$

$$
\Mdde\Mddf\Mddg\Mddh\Mddi\Mdeo\Mdea\Mdeb\Mdec\Mded\Mdee\Mdef\Mdeg\Mdeh
$$

$$
\Mdei\Mdfo\Mdfa\Mdfb\Mdfc\Mdfd\Mdfe\Mdff\Mdfg\Mdfh\Mdfi\Mdgo\Mdga\Mdgb
$$

$$
\Mdgc\Mdgd\Mdge\Mdgf\Mdgg\Mdgh\Mdgi\Mdho\Mdha\Mdhb\Mdhc\Mdhd\Mdhe\Mdhf
$$

$$
\Mdhg\Mdhh\Mdhi\Mdio\Mdia\Mdib\Mdic\Mdid\Mdie\Mdif\Mdig\Mdih\Mdii\Meoo
$$
\vfill
$$
\Meoa\Meob\Meoc\Meod\Meoe\Meof\Meog\Meoh\Meoi\Meao\Meaa\Meab\Meac\Mead\Meae
$$

$$
\Meaf\Meag\Meah\Meai\Mebo\Meba\Mebb\Mebc\Mebd\Mebe\Mebf\Mebg\Mebh\Mebi\Meco
$$

$$
\Meca\Mecb\Mecc\Mecd\Mece\Mecf\Mecg\Mech\Meci\Medo\Meda\Medb\Medc\Medd
$$

$$
\Mede\Medf\Medg\Medh\Medi\Meeo\Meea\Meeb\Meec\Meed\Meee\Meef\Meeg\Meeh
$$

$$
\Meei\Mefo\Mefa\Mefb\Mefc\Mefd\Mefe\Meff\Mefg\Mefh\Mefi\Mego\Mega\Megb
$$

$$
\Megc\Megd\Mege\Megf\Megg\Megh\Megi\Meho\Meha\Mehb\Mehc\Mehd\Mehe\Mehf
$$

$$
\Mehg\Mehh\Mehi\Meio\Meia\Meib\Meic\Meid\Meie\Meif\Meig\Meih\Meii\Mfoo
$$
\vfill
$$
\hskip -3mm\Mfoa\Mfob\Mfoc\Mfod\Mfoe\Mfof\Mfog\Mfoh\Mfoi\Mfao\Mfaa\Mfab\Mfac\Mfad\Mfae
$$

$$
\hskip -5mm\Mfaf\Mfag\Mfah\Mfai\Mfbo\Mfba\Mfbb\Mfbc\Mfbd\Mfbe\Mfbf\Mfbg\Mfbh\Mfbi\Mfco
$$

$$
\Mfca\Mfcb\Mfcc\Mfcd\Mfce\Mfcf\Mfcg\Mfch\Mfci\Mfdo\Mfda\Mfdb\Mfdc\Mfdd
$$

$$
\Mfde\Mfdf\Mfdg\Mfdh\Mfdi\Mfeo\Mfea\Mfeb\Mfec\Mfed\Mfee\Mfef\Mfeg\Mfeh
$$

$$
\Mfei\Mffo\Mffa\Mffb\Mffc\Mffd\Mffe\Mfff\Mffg\Mffh\Mffi\Mfgo\Mfga\Mfgb
$$

$$
\Mfgc\Mfgd\Mfge\Mfgf\Mfgg\Mfgh\Mfgi\Mfho\Mfha\Mfhb\Mfhc\Mfhd\Mfhe\Mfhf
$$

$$
\Mfhg\Mfhh\Mfhi\Mfio\Mfia\Mfib\Mfic\Mfid\Mfie\Mfif\Mfig\Mfih\Mfii\Mgoo
$$
\eject
$$
\Mgoa\Mgob\Mgoc\Mgod\Mgoe\Mgof\Mgog\Mgoh\Mgoi\Mgao\Mgaa\Mgab\Mgac
$$

$$
\Mgad\Mgae\Mgaf\Mgag\Mgah\Mgai\Mgbo\Mgba\Mgbb\Mgbc\Mgbd\Mgbe\Mgbf
$$

$$
\Mgbg\Mgbh\Mgbi\Mgco\Mgca\Mgcb\Mgcc\Mgcd\Mgce\Mgcf\Mgcg\Mgch\Mgci
$$

$$
\Mgdo\Mgda\Mgdb\Mgdc\Mgdd\Mgde\Mgdf\Mgdg\Mgdh\Mgdi\Mgeo\Mgea\Mgeb
$$

$$
\Mgec\Mged\Mgee\Mgef\Mgeg\Mgeh\Mgei\Mgfo\Mgfa\Mgfb\Mgfc\Mgfd
$$

$$
\Mgfe\Mgff\Mgfg\Mgfh\Mgfi\Mggo\Mgga\Mggb\Mggc\Mggd\Mgge\Mggf
$$

$$
\Mggg\Mggh\Mggi\Mgho\Mgha\Mghb\Mghc\Mghd\Mghe\Mghf\Mghg\Mghh
$$

$$
\Mghi\Mgio\Mgia\Mgib\Mgic\Mgid\Mgie\Mgif\Mgig\Mgih\Mgii\Mhoo
$$
\vfill
$$
\Mhoa\Mhob\Mhoc\Mhod\Mhoe\Mhof\Mhog\Mhoh\Mhoi\Mhao\Mhaa\Mhab\Mhac
$$

$$
\Mhad\Mhae\Mhaf\Mhag\Mhah\Mhai\Mhbo\Mhba\Mhbb\Mhbc\Mhbd\Mhbe\Mhbf
$$

$$
\Mhbg\Mhbh\Mhbi\Mhco\Mhca\Mhcb\Mhcc\Mhcd\Mhce\Mhcf\Mhcg\Mhch\Mhci
$$

$$
\Mhdo\Mhda\Mhdb\Mhdc\Mhdd\Mhde\Mhdf\Mhdg\Mhdh\Mhdi\Mheo\Mhea\Mheb
$$

$$
\Mhec\Mhed\Mhee\Mhef\Mheg\Mheh\Mhei\Mhfo\Mhfa\Mhfb\Mhfc\Mhfd
$$

$$
\Mhfe\Mhff\Mhfg\Mhfh\Mhfi\Mhgo\Mhga\Mhgb\Mhgc\Mhgd\Mhge\Mhgf
$$

$$
\Mhgg\Mhgh\Mhgi\Mhho\Mhha\Mhhb\Mhhc\Mhhd\Mhhe\Mhhf\Mhhg\Mhhh
$$

$$
\Mhhi\Mhio\Mhia\Mhib\Mhic\Mhid\Mhie\Mhif\Mhig\Mhih\Mhii\Mioo
$$
\vfill
$$
\Mioa\Miob\Mioc\Miod\Mioe\Miof\Miog\Mioh\Mioi\Miao\Miaa\Miab\Miac
$$

$$
\Miad\Miae\Miaf\Miag\Miah\Miai\Mibo\Miba\Mibb\Mibc\Mibd\Mibe\Mibf
$$

$$
\Mibg\Mibh\Mibi\Mico\Mica\Micb\Micc\Micd\Mice\Micf\Micg\Mich\Mici
$$

$$
\Mido\Mida\Midb\Midc\Midd\Mide\Midf\Midg\Midh\Midi\Mieo\Miea\Mieb
$$

$$
\Miec\Mied\Miee\Mief\Mieg\Mieh\Miei\Mifo\Mifa\Mifb\Mifc\Mifd
$$

$$
\Mife\Miff\Mifg\Mifh\Mifi\Migo\Miga\Migb\Migc\Migd\Mige\Migf
$$

$$
\Migg\Migh\Migi\Miho\Miha\Mihb\Mihc\Mihd\Mihe\Mihf\Mihg\Mihh
$$

$$
\Mihi\Miio\Miia\Miib\Miic\Miid\Miie\Miif\Miig\Miih\Miii\Maooo
$$

\section*{Appendice C : un exemple d'appariement des nombres de $1$ \`a $1000$ sans facteurs carr\'es}
L'appariement se fait par ordre d\'ecroissant, en appliquant au nombre choisi $n$ une des op\'erations de coupe ou de fusion des racines d\'ecrites pr\'ec\'edemment, afin de l'apparier avec un nombre $n'$ plus petit que lui, tel que $\mu(n)+\mu(n')=0$. Les coupes au-dessus d'un sommet autre que la racine sont licites~:
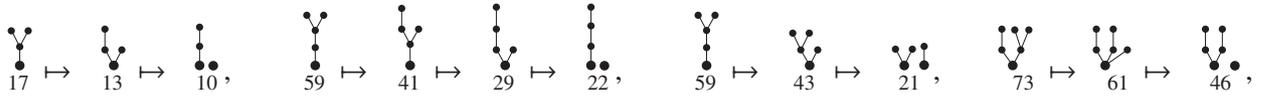
\[\Mag\mapsto\Mac\mapsto\Mao,\hskip 6mm \Mei\mapsto\Mda\mapsto\Mbi\mapsto\Mbb, \hskip 6mm \Mei\mapsto\Mdc\mapsto\Mba,\hskip 6mm\Mgc\mapsto\Mfa\mapsto \Mdf,\]
\`a condition d'\'eviter les situations d\'eriv\'ees des deux exceptions mentionn\'ees ci-dessus. Par exemple
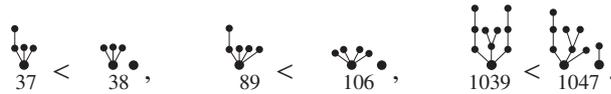
\[\Mcg < \Mch, \hskip 8mm \Mhi < \Maof,\hskip 8mm \genfrac {}{}{0pt}{1}{\matulaiglarge \hskip -10.7pt \matulaagmince}{1039} < \genfrac {}{}{0pt}{1}{\matulacdi\ \matulac}{1047}.\]

$20$ nombres sont rest\'es seuls soit \`a dessein, soit parce que tous les candidats \`a l'appariement \'etaient d\'ej\`a indisponibles. Les nombres contenant au moins un facteur carr\'e sont appari\'es avec eux-mêmes,  et ne sont pas repr\'esent\'es.

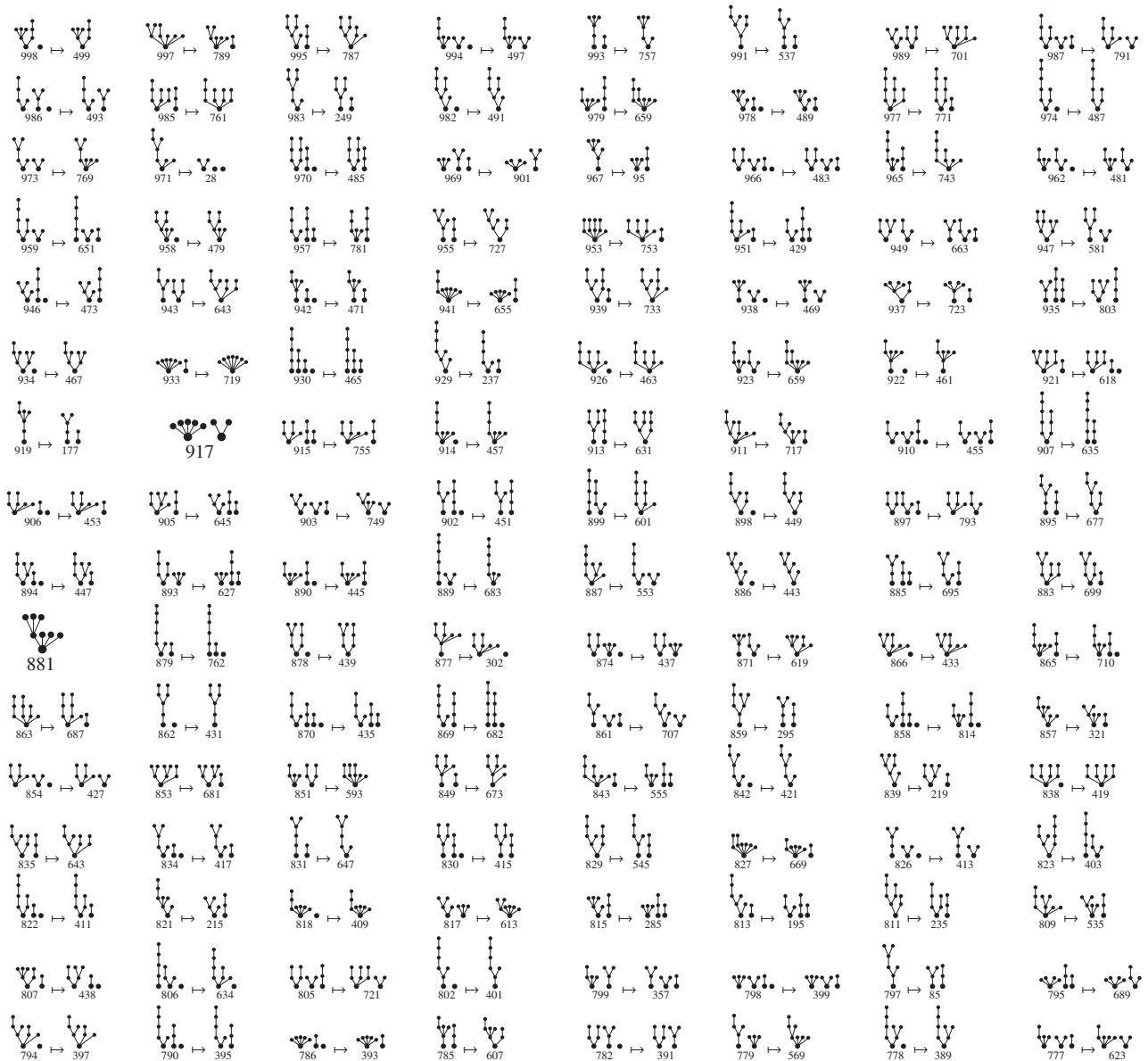
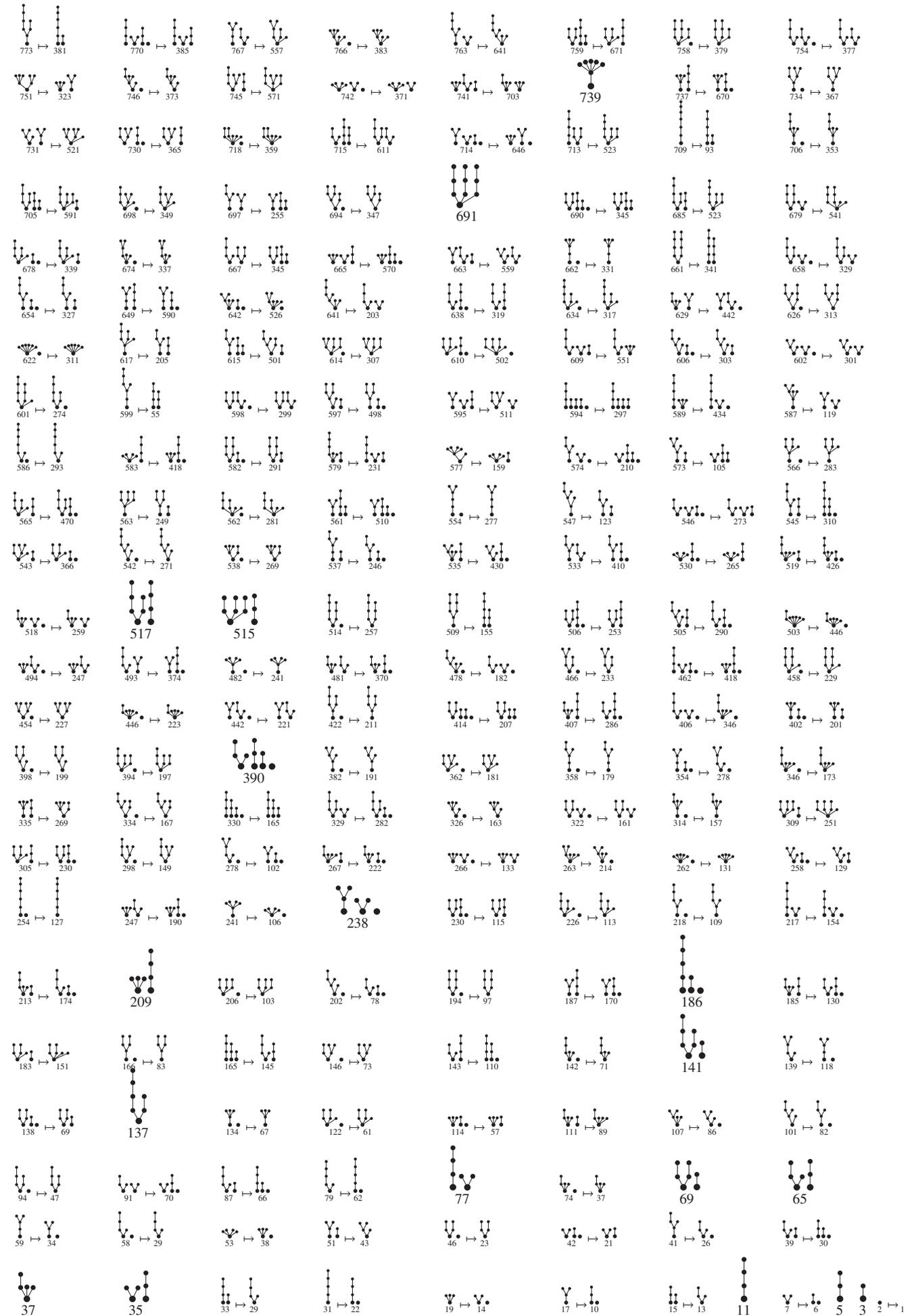
\def\pair #1#2{\genfrac{}{}{0pt}{0}{\scalebox{0.6}{$#1$}}{\scalebox{0.6}{$#2$}}\hskip 1.5mm}
\def\vers #1#2{\scalebox{0.6}{$#1\mapsto \hskip -3mm #2$}}
\begin{align*}
&\vers {\Miih}{\Mdii} &&\vers \Miig\Mghi &&\vers\Miie\Mghg &&\vers\Miid\Mdig &&\vers\Miic\Mgeg &&\vers\Miia\Mecg&&\vers\Mihi\Mgoa &&\vers\Mihg\Mgia \\
&\vers\Mihf\Mdic &&\vers\Mihe\Mgfa &&\vers\Mihc\Mbdi &&\vers\Mihb\Mdia&&\vers\Migi\Mfei &&\vers\Migh\Mdhi&&\vers\Migg\Mgga &&\vers\Migd\Mdhg\\
&\vers\Migc\Mgfi &&\vers\Miga\Mbh &&\vers\Migo\Mdhe &&\vers\Mifi\Mioa &&\vers\Mifg\Mie &&\vers\Miff\Mdhc&&\vers\Mife\Mgdc &&\vers\Mifb\Mdha\\
&\vers\Miei\Mfea &&\vers\Mieh\Mdgi &&\vers\Mieg\Mgha &&\vers\Miee\Mgbg &&\vers\Miec\Mgec &&\vers\Miea\Mdbi&&\vers\Midi\Mffc &&\vers\Midg\Meha\\
&\vers\Midf\Mdgc &&\vers\Midc\Mfdc&&\vers\Midb\Mdga &&\vers\Mida\Mfee&&\vers\Mici\Mgcc &&\vers\Mich\Mdfi &&\vers\Micg\Mgbc &&\vers\Mice\Mhoc\\
&\vers\Micd\Mdfg &&\vers\Micc\Mgai&&\vers\Mico\Mdfe &&\vers\Mibi\Mbcg &&\vers\Mibf\Mdfc &&\vers\Mibc\Mfei &&\vers\Mibb\Mdfa &&\vers\Miba\Mfah\\
&\vers\Miai\Magg &&\Miag &&\vers\Miae\Mgee &&\vers \Miad\Mdeg &&\vers\Miac\Mfca &&\vers\Miaa\Mgag&&\vers\Miao\Mdee &&\vers\Miog\Mfce\\
&\vers\Miof\Mdec &&\vers\Mioe\Mfde &&\vers\Mioc\Mgdi &&\vers\Miob\Mdea&&\vers\Mhii\Mfoa &&\vers\Mhih\Mddi &&\vers\Mhig\Mgic &&\vers\Mhie\Mfgg\\
&\vers\Mhid\Mddg &&\vers\Mhic\Mfbg&&\vers\Mhio\Mdde &&\vers\Mhhi\Mfhc &&\vers\Mhhg\Meec &&\vers\Mhhf\Mddc &&\vers\Mhhe\Mfie &&\vers\Mhhc\Mfii\\
&\Mhha &&\vers\Mhgi\Mgfb &&\vers\Mhgh\Mdci &&\vers\Mhgg\Mcob &&\vers\Mhgd\Mdcg &&\vers\Mhga\Mfai&&\vers\Mhff\Mdcc &&\vers\Mhfe\Mgao\\
&\vers\Mhfc\Mfhg &&\vers\Mhfb\Mdca&&\vers\Mhgo\Mdce &&\vers\Mhfi\Mfhb&&\vers\Mhfa\Mgog &&\vers\Mhei\Mbie &&\vers\Mheh\Mhad &&\vers\Mheg\Mcba\\
&\vers\Mhed\Mdbg &&\vers\Mhec\Mfha&&\vers\Mhea\Meic &&\vers\Mhdi\Mfgc &&\vers\Mhdc\Meee &&\vers\Mhdb\Mdba &&\vers\Mhci\Mbai &&\vers\Mhch\Mdai\\
&\vers\Mhce\Mfdc &&\vers\Mhcd\Mdag &&\vers\Mhca\Mfdg &&\vers\Mhco\Mdae &&\vers\Mhbi\Mede &&\vers\Mhbg\Mffi&&\vers\Mhbf\Mdac &&\vers\Mhbc\Mdoc\\
&\vers\Mhbb\Mdaa &&\vers\Mhba\Mbae &&\vers\Mhah\Mdoi &&\vers\Mhag\Mfac&&\vers\Mhae\Mbhe &&\vers\Mhac\Maie &&\vers\Mhaa\Mbce &&\vers\Mhoi\Mece\\
&\vers\Mhog\Mdch &&\vers\Mhof\Mfcd &&\vers\Mhoe\Mgba &&\vers\Mhob\Mdoa &&\vers\Mgii\Mceg &&\vers\Mgih\Mcii &&\vers\Mgig\Mhe &&\vers\Mgie\Mfhi\\
&\vers\Mgid\Mcig &&\vers\Mgio\Mcie &&\vers\Mghf\Mcic &&\vers\Mghe\Mfog &&\vers\Mghb\Mcia &&\vers\Mggi\Mefi&&\vers\Mggh\Mchi &&\vers\Mggg\Mfbc
\end{align*}
\begin{align*}
&\vers\Mggc\Mcha &&\vers\Mggo\Mche &&\vers\Mgfg\Meeg &&\vers\Mgff\Mchc &&\vers\Mgfc\Mfda &&\vers\Mgei\Mfga &&\vers\Mgeh\Mcgi &&\vers\Mged\Mcgg &&\\
&\vers\Mgea\Mcbc &&\vers\Mgdf\Mcgc &&\vers\Mgde\Mega &&\vers\Mgdb\Mcga &&\vers\Mgda\Mgoc &&\Mgci &&\vers\Mgcg\Mfgo &&\vers\Mgcd\Mcfg \\
&\vers\Mgca\Meba &&\vers\Mgco\Mcfe &&\vers\Mgah\Mcei &&\vers\Mgae\Mfaa &&\vers\Mgad\Mfdf &&\vers\Mgac\Mebc &&\vers\Mgoi\Mic &&\vers\Mgof\Mcec\\
&\vers\Mgoe\Meia &&\vers\Mfih\Mcdi &&\vers\Mfig\Mbee &&\vers\Mfid\Mcdg && \Mfia &&\vers\Mfio\Mcde &&\vers\Mfhe\Mebc &&\vers\Mfgi\Meda \\
&\vers\Mfgh\Mcci &&\vers\Mfgd\Mccg &&\vers\Mffg\Mcde &&\vers\Mffe\Mego &&\vers\Mffc\Meei &&\vers\Mffb\Mcca &&\vers\Mffa\Mcda &&\vers\Mfeh\Mcbi\\
&\vers\Mfed\Mcbg &&\vers\Mfdi\Meio&&\vers\Mfdb\Mebf &&\vers\Mfda\Mboc &&\vers\Mfch\Mcai &&\vers\Mfcd\Mcag &&\vers\Mfbi\Mddb &&\vers\Mfbf\Mcac\\
&\vers\Mfbb\Mcaa &&\vers\Mfag\Mboe &&\vers\Mfae\Meoa &&\vers\Mfad\Mcog &&\vers\Mfao\Meob &&\vers\Mfoi\Meea &&\vers\Mfof\Mcoc&&\vers\Mfob\Mcoa\\
&\vers\Mfoa\Mbgd &&\vers\Meii\Mee &&\vers\Meih\Mbii &&\vers\Meig\Mdih &&\vers\Meie\Meaa &&\vers\Meid\Mbig &&\vers\Mehi\Mdcd & &\vers\Mehg\Maai\\
&\vers\Mehf\Mbic &&\vers\Mehc\Mdah&& \vers\Mehb\Mbia &&\vers\Megi\Mbca &&\vers\Megg\Maei &&\vers\Megd\Mbao &&\vers\Megc\Maoe &&\vers\Meff\Mbhc\\
&\vers\Mefe\Mdgo &&\vers\Mefc\Mbdi &&\vers\Mefb\Mbha&&\vers\Mefa\Meao &&\vers\Meed\Mbgg &&\vers\Medg\Mabc &&\vers\Medf\Mbgc &&\vers\Mede\Mcao\\
&\vers\Medc\Mcff &&\vers\Medb\Mbga &&\vers\Mech\Mbfi &&\vers\Mecg\Mbdf &&\vers\Mece\Mdco &&\vers\Mecc\Mdao &&\vers\Meco\Mbfe &&\vers\Meai\Mdbf\\
&\vers\Meah\Mbei &&\Meag &&\Meae &&\vers\Mead\Mbeg &&\vers\Meoi\Maee &&\vers\Meof\Mbec &&\vers\Meoe\Mbio &&\vers\Meoc\Mddf\\
&\vers\Mdid\Mbdg&&\vers\Mdic\Mcgd &&\vers\Mdhb\Mbda &&\vers\Mdha\Mcgo &&\vers\Mdgh\Mahb &&\vers\Mdff\Mbcc &&\vers\Mdfb\Mdah &&\vers\Mdeh\Mbbi\\
&\vers\Mded\Mbbg &&\vers\Mddf\Mbbc &&\vers\Mddb\Mbba &&\vers\Mdbb\Mbaa &&\vers\Mdad\Mbog &&\vers\Mdog\Mbhf &&\vers\Mdof\Mcdf &&\vers\Mdob\Mboa\\
&\vers\Mcih\Maii &&\vers\Mcid\Maig &&\Mcio &&\vers\Mchb\Maia &&\vers\Mcfb\Maha &&\vers\Mceh\Magi &&\vers\Mced\Mbgh &&\vers\Mcdf\Magc\\
&\vers\Mcce\Mbfi &&\vers\Mccd\Mafg &&\vers\Mcco\Mafe &&\vers\Mcbi\Mbhb &&\vers\Mcbf\Mafc &&\vers\Mcbb\Mafa &&\vers\Mcad\Maeg &&\vers\Mcoi\Mbea\\
&\vers\Mcoe\Mbco &&\vers\Mbih\Madi &&\vers\Mbgh\Maob &&\vers\Mbfg\Mbbb &&\vers\Mbff\Macc &&\vers\Mbfc\Mbad &&\vers\Mbfb\Maca &&\vers\Mbeh\Mabi\\
&\vers\Mbed\Mabg &&\vers\Mbdg\Maio &&\vers\Mbda\Maof && \Mbch  &&\vers\Mbco\Maae &&\vers\Mbbf\Maac &&\vers\Mbah\Maoi &&\vers\Mbag\Maed\\
&\vers\Mbac\Magd &&\Mboi &&\vers\Mbof\Maoc &&\vers\Mbob\Mgh &&\vers\Maid\Mig &&\vers\Mahg\Mago && \Mahf &&\vers\Mahe\Maco\\
&\vers\Mahc\Maea &&\vers\Maff\Mhc &&\vers\Mafe\Made &&\vers\Madf\Mgc &&\vers\Madc\Maao &&\vers\Madb\Mga &&\Mada &&\vers\Maci\Maah\\
&\vers\Mach\Mfi &&\Macg &&\vers\Macd\Mfg &&\vers\Mabb\Mfa &&\vers\Maad\Meg &&\vers\Maaa\Mhi &&\vers\Maog\Mhf &&\vers\Maoa\Mhb\\
&\vers\Mid\Mdg &&\vers\Mia\Mgo &&\vers\Mhg\Mff &&\vers\Mgi\Mfb &&\Mgg &&\vers\Mgd\Mcg && \Mfi && \Mfe\\
&\vers\Mei\Mcd &&\vers\Meh\Mbi &&\vers\Mec\Mch &&\vers\Mea\Mdc &&\vers\Mdf\Mbc &&\vers\Mdb\Mba &&\vers\Mda\Mbf &&\vers\Mci\Mco\\
&\Mcg && \Mce &&\vers\Mcc\Mbi &&\vers\Mca\Mbb &&\vers\Mai\Mad &&\vers\Mag\Mao &&\vers\Mae\Mac \hskip 3mm \Maa\hskip -3mm && \vers\Mg\Mf  \Me  \Mc \vers\Mb\Ma
\end{align*}
\section*{Appendice D : de 1 \`a 199}
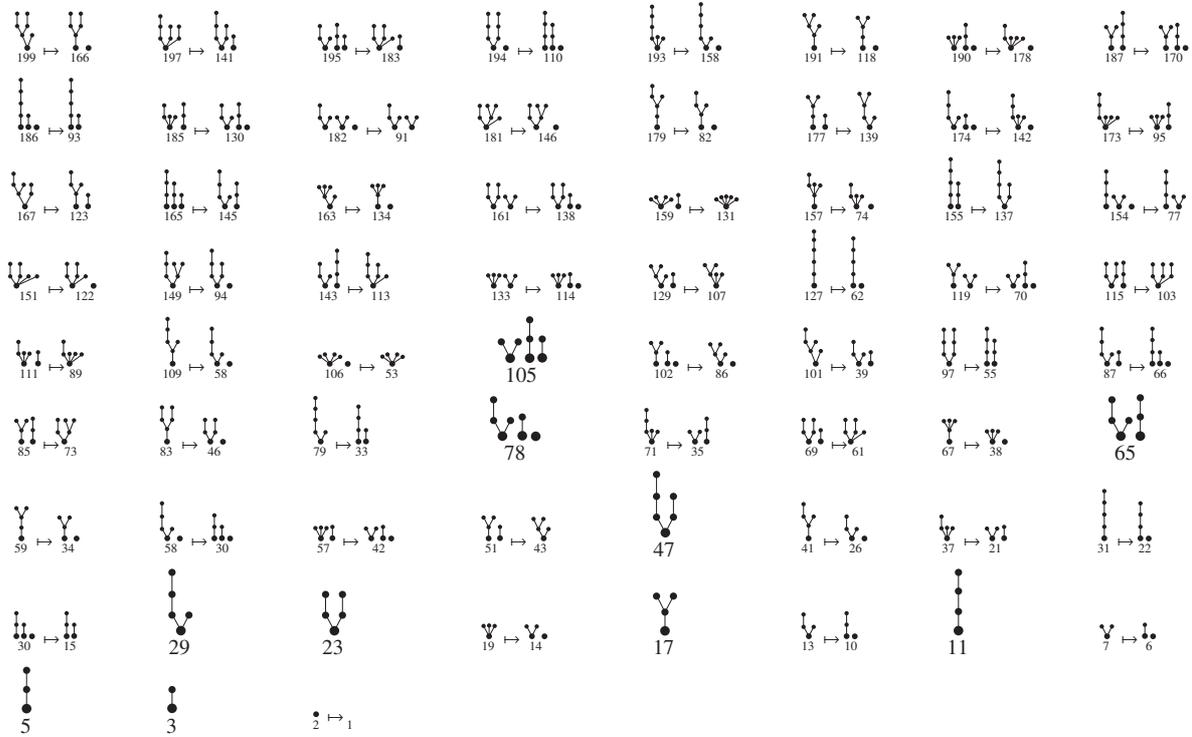
\begin{align*}
&\vers\Maii\Maff &&\vers\Maig\Mada &&\vers\Maie\Mahc&&\vers\Maid\Maao &&\vers\Maic\Maeh &&\vers\Maia\Maah &&\vers\Maio\Magh &&\vers\Mahg\Mago\\
&\vers\Mahf\Mic &&\vers\Mahe\Maco  &&\vers\Mahb\Mia &&\vers\Maha\Madf &&\vers\Magi\Mhb &&\vers\Magg\Maci &&\vers\Magd\Madb &&\vers\Magc\Mie\\
&\vers\Mafg\Mabc &&\vers\Mafe\Made &&\vers\Mafc\Macd &&\vers\Mafa\Mach &&\vers\Maei\Maca &&\vers\Maeg\Mgd &&\vers\Maee\Macg &&\vers\Maed\Mgg\\
&\vers\Maea\Mabb &&\vers\Madi\Mid  &&\vers\Madc\Maac &&\vers\Macc\Maad &&\vers\Mabi\Maog &&\vers\Mabg\Mfb &&\vers\Maai\Mgo &&\vers\Maae\Maoc \\
&\vers\Maaa\Mhi &&\vers\Maoi\Meh &&\vers\Maof\Mec &&\Maoe &&\vers\Maob\Mhf &&\vers\Maoa\Mci &&\vers\Mig\Mee &&\vers\Mhg\Mff\\
&\vers\Mhe\Mgc &&\vers\Mhc\Mdf &&\vers\Mgi\Mcc &&\Mgh &&\vers\Mga\Mce &&\vers\Mfi\Mfa &&\vers\Mfg\Mch &&\Mfe \\
&\vers\Mei\Mcd &&\vers\Meh\Mco &&\vers\Meg\Mdb &&\vers\Mea\Mdc &&\Mdg  &&\vers\Mda\Mbf &&\vers\Mcg\Mba && \vers\Mca\Mbb\\
&\vers\Mco\Mae&& \Mbi &&\Mbc &&\vers\Mai\Mad &&\Mag\ &&\vers\Mac\Mao && \Maa &&\vers\Mg\Mf\\
&\Me &&\Mc &&\vers\Mb\Ma
\end{align*}
\section*{Appendice E : de 1 \`a 96 avec les facteurs carr\'es}
Dans cet exemple, tous les nombres de $1$ \`a $96$ ont pu \^etre rang\'es par paires. Rappelons que $96$ est le dernier z\'ero de la fonction sommatoire de Liouville avant $906\, 150\, 256$.
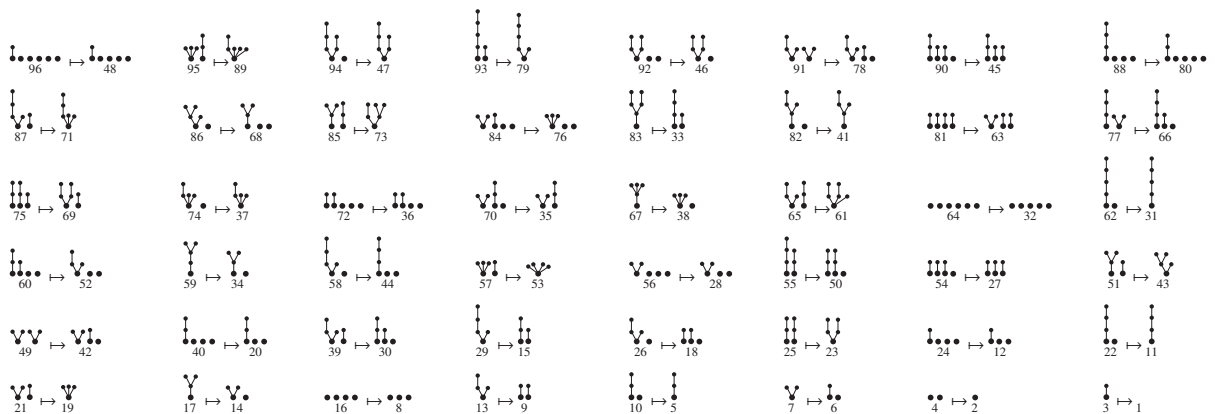
\begin{align*}
&\vers\Mif\Mdh &&\vers\Mie\Mhi &&\vers \Mid\Mdg &&\vers\Mic\Mgi &&\vers\Mib\Mdf &&\vers\Mia\Mgh &&\vers\Mio\Mde &&\vers\Mhh\Mho\\
&\vers\Mhg\Mga &&\vers\Mhf\Mfh &&\vers\Mhe\Mgc &&\vers\Mhd\Mgf &&\vers\Mhc\Mcc &&\vers\Mhb\Mda &&\vers\Mha\Mfc &&\vers\Mgg\Mff\\
&\vers\Mge\Mfi &&\vers\Mgd\Mcg &&\vers\Mgb\Mcf &&\vers\Mgo\Mce &&\vers\Mfg\Mch &&\vers\Mfe\Mfa &&\vers\Mfd\Mcb &&\vers\Mfb\Mca\\
&\vers\Mfo\Meb &&\vers\Mei\Mcd &&\vers\Meh\Mdd &&\vers\Meg\Mec &&\vers\Mef\Mbh &&\vers\Mee\Meo &&\vers\Med\Mbg &&\vers\Mea\Mdc\\
&\vers\Mdi\Mdb &&\vers\Mdo\Mbo&&\vers\Mci\Mco &&\vers\Mbi\Mae &&\vers\Mbf\Mah &&\vers\Mbe\Mbc &&\vers\Mbd\Mab &&\vers\Mbb\Maa\\
&\vers\Mba\Mai &&\vers\Mag\Mad &&\vers\Maf\Mh &&\vers\Mac\Mi &&\vers\Mao\Me &&\vers\Mg\Mf &&\vers\Md\Mb &&\vers\Mc\Ma
\end{align*}
\ignore{
Ou alors :
\begin{align*}
&\vers\Mif\Mdh &&\vers\Mie\Mhi &&\vers \Mid\Mff &&\vers\Mic\Mgi && \vers\Mib\Mfo &&\vers\Mia\Mgh &&\vers\Mio\Mgo &&\vers\Mhh\Mho\\
&\vers\Mhg\Mga &&\vers\Mhf\Mfh &&\vers\Mhe\Mgc &&\vers\Mhd\Mgf &&\vers\Mhc\Mdf &&\vers\Mhb\Meb &&\vers\Mha\Mfc &&\Mgg\\
&\vers\Mge\Mfi &&\vers\Mgd\Mdb &&\vers\Mgb\Mef &&\vers\Mfg\Mch &&\vers\Mfe\Mfa &&\vers\Mfd\Mcb &&\vers\Mfb\Mdd &&\vers\Mei\Mcd\\
&\vers\Meh\Mco &&\vers\Meg\Mec &&\vers\Mee\Mdg &&\vers\Med\Mbg &&\vers\Mea\Mdc &&\vers\Meo\Mbe &&\Mdi &&\vers\Mde\Mci\\
&\vers\Mda\Mbf &&\vers\Mdo\Mbo &&\vers\Mcg\Mba &&\vers\Mcf\Mbh &&\Mce &&\vers\Mcc\Mbi &&\vers\Mca\Mbb &&\vers\Mbd\Mab\\
&\vers\Mbc\Mae &&\vers\Mai\Mad &&\vers\Mah\Mi &&\vers\Mag\Mao &&\vers\Maf\Mh &&\Mac &&\Maa &&\vers\Mg\Mf\\
&\Me &&\vers\Md\Mb &&\vers\Mc\Ma
\end{align*}
}
\section*{Appendice F : de 1 \`a 1000 avec les facteurs carr\'es}
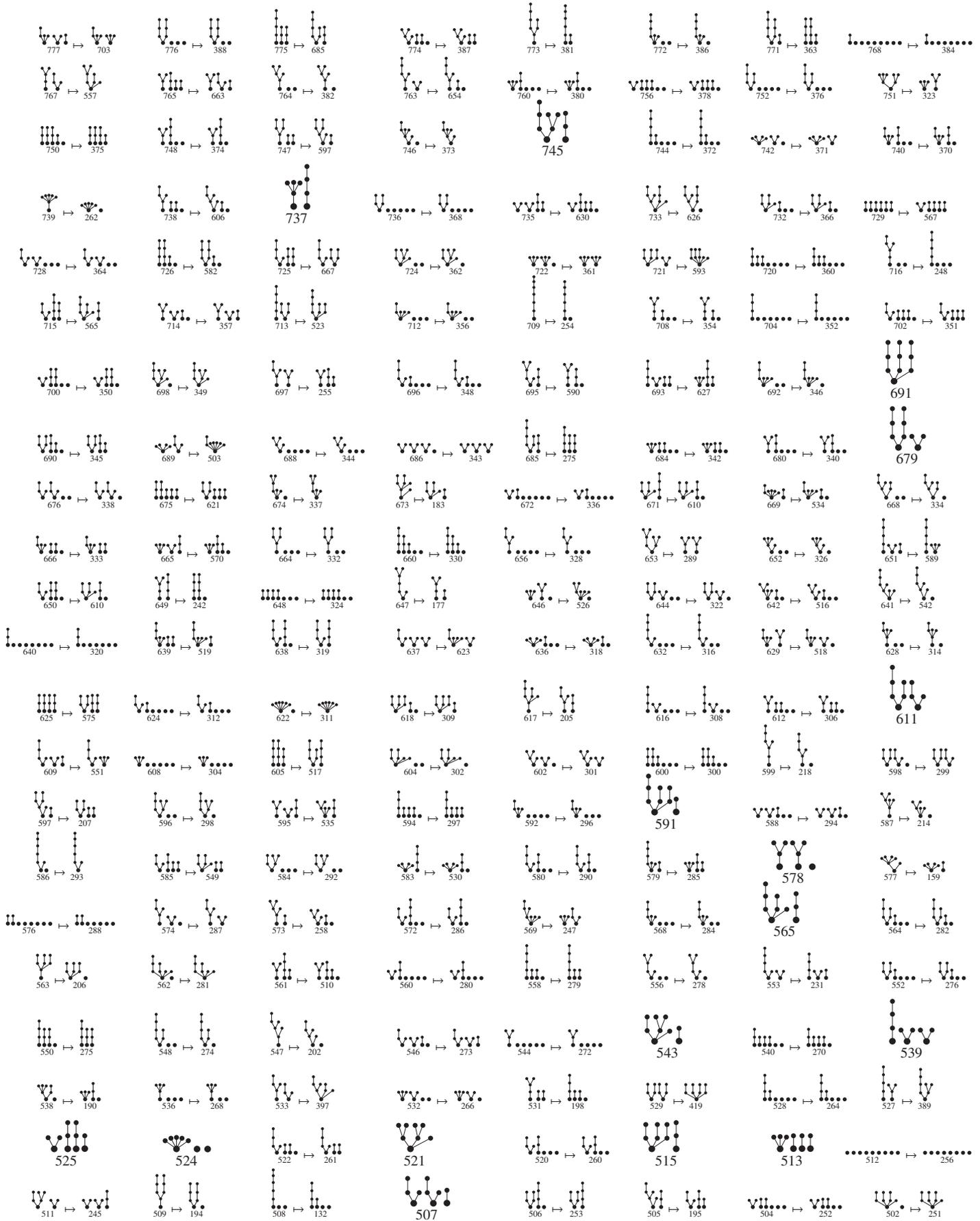
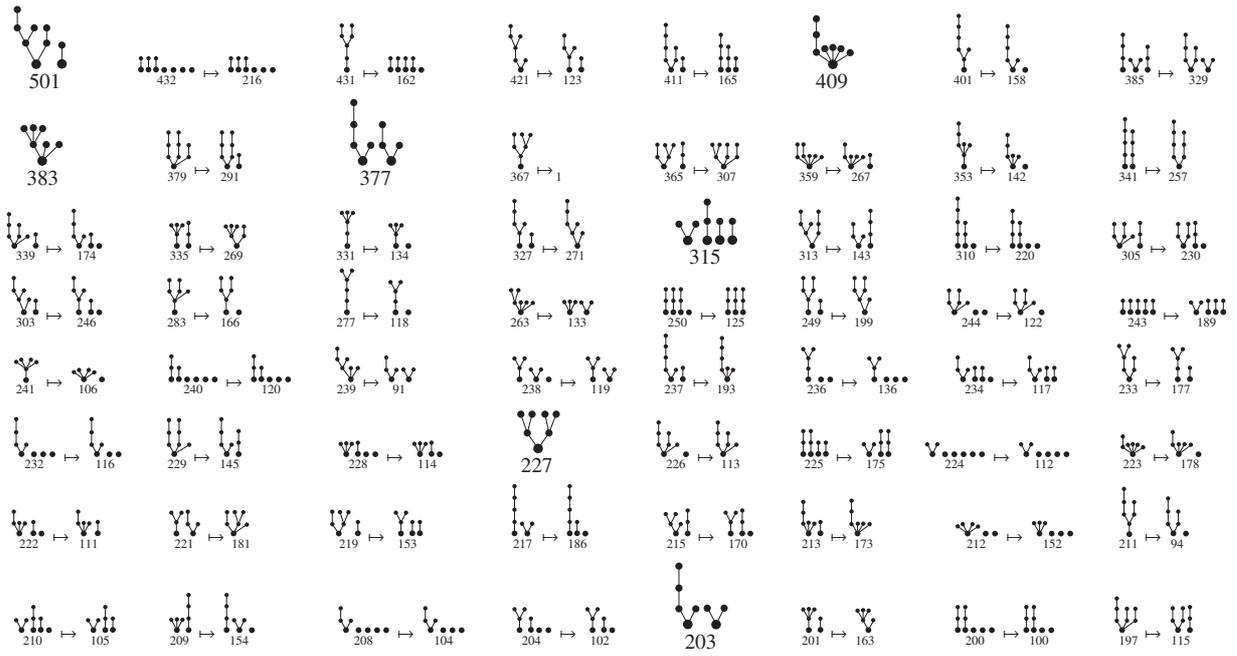
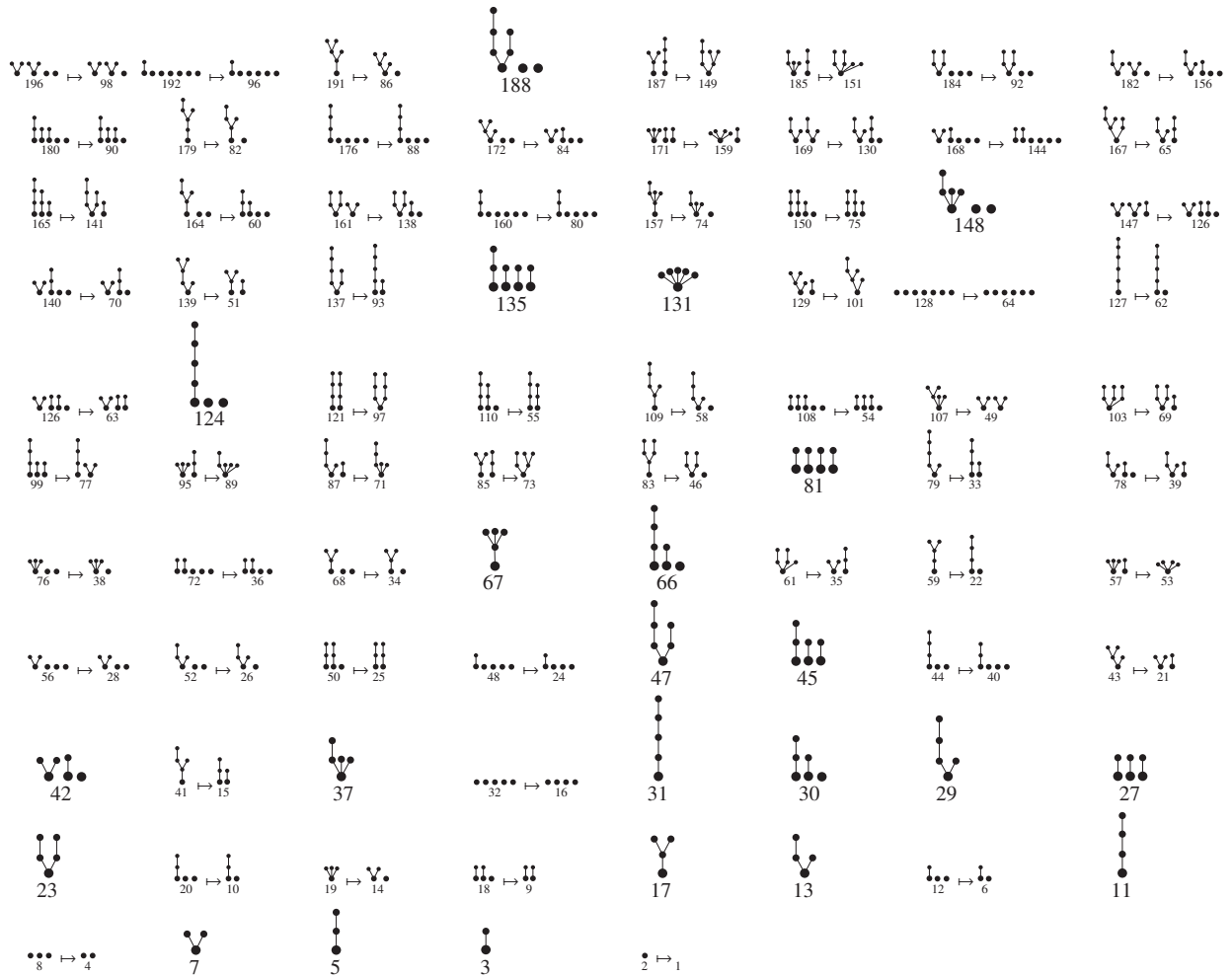
\begin{align*}
&\vers\Maooo\Meoo && \vers\Miii\Mhoa &&\vers\Miih\Mdii &&\vers\Miig\Mhff &&\vers\Miif\Mdih &&\vers\Miie\Mghg &&\vers\Miid\Mdig &&\vers\Miic\Mgeg\\
&\vers\Miib\Mdif &&\vers\Miia\Mecg &&\vers\Miio\Mdie &&\vers\Mihi\Mgoa &&\vers\Mihh\Mdid &&\vers\Mihg\Mhic &&\vers\Mihf\Mdic &&\vers\Mihe\Mgfa\\
&\vers\Mihd\Mdib &&\vers\Mihc\Mhfb &&\vers\Mihb\Mdia &&\vers\Miha\Mhac &&\vers\Miho\Mdio &&\vers\Migi\Mfei &&\vers\Migh\Mdhi &&\vers\Migg\Mead\\
&\vers\Migf\Mdhh &&\vers\Mige\Miae &&\vers\Migd\Mdhg &&\vers\Migc\Mgfi &&\vers\Migb\Mdhf &&\vers\Miga\Mhdb &&\vers\Migo\Mdhe &&\vers\Mifi\Mghi\\
&\vers\Mifh\Mdhd &&\vers\Mifg\Mffb &&\vers\Miff\Mdhc &&\vers\Mife\Mgdc &&\vers\Mifd\Mdhb &&\vers\Mifc\Mgdi &&\vers\Mifb\Mdha &&\vers\Mifa\Mffa\\
&\hskip -3mm\vers\Mifo\Mdho &&\vers\Miei\Mhbb &&\vers\Mieh\Mdgi &&\vers\Mieg\Mhda &&\vers\Mief\Mdgh &&\vers\Miee\Mgbg &&\vers\Mied\Mdgg &&\vers\Miec\Mhch\\
&\vers\Mieb\Mdgf &&\vers\Miea\Mgdc &&\vers\Mieo\Mdge &&\vers\Midi\Mgco &&\vers\Midh\Mdgd &&\vers\Midg\Meha &&\vers\Midf\Mdgc &&\vers\Mide\Mhee\\
&\vers\Midd\Mdgb &&\vers\Midc\Mfae &&\vers\Midb\Mdga &&\vers\Mida\Mhah &&\vers\Mido\Mdgo &&\vers\Mici\Mgcc &&\vers\Mich\Mdfi &&\vers\Micg\Mgbc\\
&\vers\Micf\Mdfh &&\vers\Mice\Mgde. &&\vers\Micd\Mdfg &&\vers\Micc\Mgai &&\vers\Micb\Mdff &&\vers\Mica\Miag &&\vers\Mico\Mdfe &&\vers\Mibi\Mhob\\
&\vers\Mibh\Mdfd &&\vers\Mibg\Mgec &&\vers\Mibf\Mdfc &&\vers\Mibe\Mhea &&\vers\Mibd\Mdfb &&\vers\Mibc\Mged &&\vers\Mibb\Mdfa &&\vers\Miba\Mfeg\\
&\vers\Mibo\Mdfo &&\vers\Miai\Mceh &&\vers\Miah\Mdei &&\vers\Miaf\Mdeh &&\vers\Miad\Mdeg &&\vers\Miac\Mfca &&\vers\Miab\Mdef &&\vers\Miaa\Mgid\\
&\vers\Miao\Mdee &&\vers\Mioi\Mgag &&\vers\Mioh\Mded &&\vers\Miog\Mfce &&\vers\Miof\Mdec &&\vers\Mioe\Mfde &&\vers\Miod\Mdeb &&\vers\Mioc\Mhag\\
&\vers\Miob\Mdea &&\vers\Mioa\Mfac &&\vers\Mioo\Mdeo &&\vers\Mhii\Mfoa &&\vers\Mhih\Mddi &&\vers\Mhig\Mgic &&\hskip -5mm\vers\Mhif\Mddh &&\vers\Mhie\Mfgg\\
&\vers\Mhid\Mddg &&\vers\Mhib\Mddf &&\vers\Mhia\Mghc &&\vers\Mhio\Mdde &&\vers\Mhhi\Mfhc &&\vers\Mhhh\Mddd &&\vers\Mhhg\Mggh &&\vers\Mhhf\Mddc\\
&\vers\Mhhe\Mfii &&\vers\Mhhd\Mddb &&\vers\Mhhc\Mcia &&\hskip -3mm\vers\Mhhb\Mdda &&\vers\Mhha\Mgff &&\vers\Mhho\Mddo &&\vers\Mhgi\Mgfb &&\vers\Mhgh\Mdci\\
&\vers\Mhgg\Meff &&\vers\Mhgf\Mdch &&\vers\Mhge\Mhoe &&\vers\Mhgd\Mdcg &&\vers\Mhgc\Mfhg &&\vers\Mhgb\Mdcf &&\vers\Mhga\Mfai &&\vers\Mhgo\Mdce\\
&\vers\Mhfi\Mfhb &&\vers\Mhfh\Mdcd &&\vers\Mhfg\Mgca &&\vers\Mhfe\Mgao &&\hskip -5mm\vers\Mhfd\Mdcb &&\vers\Mhfc\Mgeh &&\vers\Mhfa\Mgog &&\vers\Mhfo\Mdco\\
&\vers\Mhei\Mcee &&\vers\Mheh\Mdbi &&\vers\Mheg\Mcba &&\vers\Mhef\Mdbh &&\vers\Mhed\Mdbg &&\vers\Mhec\Mfha &&\vers\Mheb\Mdbf &&\vers\Mheo\Mdbe\\
&\vers\Mhdi\Mcfi &&\hskip -4mm\vers\Mhdh\Mdbd &&\vers\Mhdg\Mggo &&\vers\Mhdf\Mdbc &&\vers\Mhde\Mgee &&\vers\Mhdd\Mdbb &&\vers\Mhdc\Mfgh &&\hskip -2mm\vers\Mhdo\Mdbo\\
&\vers\Mhci\Mgcd &&\vers\Mhcg\Mgaa &&\vers\Mhcf\Mdah &&\vers\Mhce\Mcbe &&\vers\Mhcd\Mdag &&\vers\Mhcc\Mgdi &&\hskip -5mm\vers\Mhcb\Mdaf &&\vers\Mhca\Mfdg\\
&\vers\Mhco\Mdae &&\vers\Mhbi\Mede &&\hskip -2mm\vers\Mhbh\Mdad &&\vers\Mhbg\Mgah &&\vers\Mhbf\Mdac &&\vers\Mhbe\Mgoe &&\hskip -2mm\vers\Mhbd\Mdab
&&\vers\Mhbc\Mdoc\\
&\vers\Mhba\Mgof &&\vers\Mhbo\Mdao &&\hskip -2mm\vers\Mhai\Mgda &&\hskip -4mm\vers\Mhaf\Mdoh &&\vers\Mhae\Mfgo &&\vers\Mhad\Mdog &&\vers\Mhab\Mdof
&&\vers\Mhaa\Mbce\\
&\vers\Mhao\Mdoe &&\vers\Mhoi\Mfid &&\vers\Mhoh\Mdod &&\vers\Mhog\Mfoc &&\vers\Mhof\Mfcd &&\vers\Mhod\Mdob &&\vers\Mhoc\Mega &&\hskip -4mm\vers\Mhoo\Mdoo\\
&\vers\Mgii\Mceg &&\vers\Mgih\Mcii &&\vers\Mgig\Meed &&\vers\Mgif\Mcih &&\vers\Mgie\Mfee &&\vers\Mgib\Mcif &&\vers\Mgia\Mfeh &&\vers\Mgio\Mcie\\
&\vers\Mghh\Mcid &&\vers\Mghf\Mcic &&\vers\Mghe\Mfog &&\hskip -6mm\vers\Mghd\Mcib &&\vers\Mghb\Mfad &&\vers\Mgha\Meda &&\hskip -3mm\vers\Mgho\Mcio
&&\vers\Mggi\Mefi
\end{align*}
\begin{align*}
&\vers\Mggg\Mgoc &&\vers\Mggf\Mchh &&\vers\Mgge\Mfhe &&\vers\Mggd\Mchg &&\vers\Mggc\Mcha &&\vers\Mggb\Mchf &&\vers\Mgga\Mcfc &&\hskip -7mm\vers\Mgfh\Mchd\\
&\vers\Mgfg\Meeg &&\vers\Mgfe\Mffc &&\vers\Mgfd\Mchb &&\vers\Mgfc\Mfed &&\hskip -4mm\vers\Mgfo\Mcho &&\hskip -4mm\vers\Mgef\Mcgh &&\hskip -4mm\vers \Mgeb\Mcgf
&&\vers\Mgea\Mcbc\\
&\vers\Mgeo\Mcge &&\vers\Mgdh\Mcgd &&\vers\Mgdg\Meig &&\vers\Mgdf\Mcgc &&\Mgde &&\vers\Mgdd\Mcgb &&\hskip -3mm\vers\Mgdb\Mcga &&\vers\Mgdo\Mcgo\\
&\vers\Mgci\Mbfb &&\vers\Mgch\Mfof &&\Mgcg &&\hskip -5mm\vers\Mgcf\Mcfh &&\hskip -3mm\vers\Mgce\Mfco &&\vers\Mgcc\Mfbf &&\vers\Mgcb\Mcff &&\hskip -4mm\vers\Mgbi\Mefg\\
&\hskip -4mm\vers\Mgbh\Mcfd &&\vers\Mgbf\Mehb &&\vers\Mgbe\Mffg &&\vers\Mgbd\Mcfb &&\vers\Mgbb\Mcfa &&\vers\Mgba\Meic &&\hskip -3mm\vers\Mgbo\Mcfo &&\vers\Mgaf\Mbdh\\
&\vers\Mgae\Mefe &&\vers\Mgad\Mceg &&\vers\Mgac\Mebc &&\vers\Mgab\Mcef &&\vers\Mgoi\Mbed &&\vers\Mgoh\Mced &&\hskip -3mm\vers\Mgod\Mceb &&\vers\Mgob\Mcea\\
&\vers\Mgoo\Mceo &&\vers\Mfih\Mcdi &&\vers\Mfig\Mbee &&\vers\Mfif\Mcdh &&\vers\Mfie\Meio &&\vers\Mfic\Mfbg &&\vers\Mfib\Mcdf &&\Mfia\\
&\vers\Mfio\Mcde &&\vers\Mfhi\Meoc &&\vers\Mfhh\Mcdd &&\vers\Mfhf\Mcdc&&\vers\Mfhe\Mbge &&\vers\Mfhd\Mcdb &&\vers\Mfho\Mcdo &&\Mfgi\\
&\vers\Mfgf\Mcch &&\vers\Mfge\Mfba &&\vers\Mfgd\Mccg &&\vers\Mfgc\Mahc &&\hskip -4mm \vers\Mfgb\Mccf &&\vers\Mfga\Mfao &&\vers\Mffi\Mecd &&\vers\Mffh\Mccd\\
&\vers\Mfff\Mccc &&\vers\Mffe\Mego &&\vers\Mffd\Mccb &&\vers\Mffo\Mcco &&\hskip -4mm\vers\Mfef\Mcbh &&\vers\Mfec\Mbhi &&\vers\Mfeb\Mcbf &&\vers\Mfea\Mehi\\
&\vers\Mfeo\Mfao &&\vers\Mfdi\Mbdb &&\hskip -2mm\vers\Mfdh\Mcbd &&\vers\Mfdg\Magg &&\vers\Mfdf\Mebf &&\vers\Mfdd\Mcbb &&\vers\Mfdb\Meaf &&\vers\Mfda\Medb\\
&\hskip -6mm\vers\Mfdo\Mcbo &&\vers\Mfci\Meai &&\vers\Mfch\Mcai &&\vers\Mfcg\Mfbc &&\vers\Mfcf\Mcah &&\vers\Mfcb\Mcaf &&\vers\Mfbi\Meah &&\vers\Mfbh\Mcad\\
&\vers\Mfbe\Mege &&\hskip -4mm\vers\Mfbd\Mcab &&\vers\Mfbb\Mcaa &&\vers\Mfah\Mcoi &&\vers\Mfag\Mboe &&\vers\Mfaf\Mcoh &&\vers\Mfab\Mcof &&\Mfaa\\
&\vers\Mfoi\Meea &&\hskip -4mm\vers\Mfoh\Mcod &&\vers\Mfoe\Meag &&\vers\Mfod\Mcob &&\vers\Mfob\Mcoa &&\vers\Mfoo\Mcoo &&\vers\Meii\Mbah &&\vers\Meih\Mbii\\
&\vers\Meig\Mbog &&\vers\Meif\Mbih &&\vers\Meie\Mece &&\vers\Meid\Mbig &&\hskip -2mm\vers\Meib\Mbif &&\Meia &&\hskip -2mm \vers\Mehh\Mbid &&\vers\Mehg\Mbad\\
&\vers\Mehf\Mbic &&\vers\Mehe\Medi &&\vers\Mehd\Mbib &&\vers\Mehc\Meco &&\vers\Meho\Mbio &&\vers\Megi\Mbhe &&\Megh &&\vers\Megg\Maei\\
&\hskip -6mm\vers\Megf\Mbhh &&\vers\Megd\Mbhg &&\vers\Megc\Mbeh &&\vers\Megb\Mbhf &&\vers\Mefi\Mbdg &&\vers\Mefh\Mbhd &&\Mefe &&\vers\Mefd\Mbhb\\
&\vers\Mefc\Mbof &&\vers\Mefb\Mbha &&\vers\Mefa\Meao &&\hskip -2mm\vers\Mefo\Mbho &&\vers\Meeh\Mbgi &&\vers\Meef\Mbgh &&\vers\Meec\Mbca &&\vers\Meeb\Mbgf\\
&\vers\Meeo\Mbge &&\vers\Medh\Mbgd &&\vers\Medg\Mbob &&\vers\Medf\Mbgc &&\hskip -4mm\vers\Medd\Mbgb &&\Medc &&\hskip -2mm\vers\Medo\Mbgo
&&\Meci\\
&\vers\Mech\Maio &&\vers\Mecf\Mbfh &&\vers\Mecc\Mcig &&\vers\Mecb\Mbff &&\vers\Meca\Maih &&\vers\Mebi\Mdai &&\vers\Mebh\Mbfd&&\vers\Mebg\Mchi\\
&\Mebe &&\Mebd &&\vers\Mebb\Mbfa &&\Meba &&\vers\Mebo\Mbfo &&\Meae &&\Meac &&\hskip -7mm\vers\Meab\Mbef\\
&\vers\Meaa\Mbde &&\vers\Meoi\Maid &&\vers\Meoh\Macb &&\Meog &&\vers\Meof\Mbec &&\vers\Meoe\Maie &&\hskip -3mm\vers\Meod\Mbeb &&\vers\Meob\Mbea
 \end{align*}
\begin{align*}
&\Meoa &&\hskip -4mm\vers\Mdcb\Mbaf &&\vers\Mdca\Mafb &&\vers\Mdba\Mabc &&\vers\Mdaa\Mafe &&\Mdoi &&\vers\Mdoa\Maeh &&\vers\Mche\Mcbi\\
 &\Mchc &&\vers\Mcgi\Mbia &&\Mcgg &&\vers\Mcfg\Ma &&\vers\Mcfe\Mcog &&\vers\Mcei\Mbfg &&\vers\Mcec\Madb &&\vers\Mcda\Mbeg\\
 &\vers\Mcci\Magd &&\vers\Mcce\Mbfi &&\vers\Mcca\Macd &&\vers\Mcbg\Mbga &&\Mcae &&\vers\Mcac\Madc &&\vers\Mcao\Mbbo &&\vers\Mcoe\Mbco\\
 &\vers\Mcoc\Mbdf &&\vers\Mbhc\Maff &&\vers\Mbgg\Maah &&\vers\Mbfc\Macc &&\vers\Mbeo\Mabe &&\vers\Mbdi\Maii &&\vers\Mbdd\Mabb &&\vers\Mbdc\Mahi\\
 &\vers\Mbda\Maof &&\vers\Mbdo\Mabo &&\vers\Mbci\Mia &&\vers\Mbch\Maai &&\vers\Mbcg\Maic &&\vers\Mbcf\Macf &&\vers\Mbcd\Maag &&\vers\Mbcc\Magg\\
 &\vers\Mbcb\Maaf &&\vers\Mbbi\Made &&\vers\Mbbh\Maad &&\Mbbg &&\vers\Mbbf\Maac &&\vers\Mbbe\Mage &&\hskip -4mm\vers\Mbbd\Maab &&\vers\Mbbc\Magh\\
 &\vers\Mbbb\Maaa &&\vers\Mbba\Maha &&\vers\Mbai\Maec &&\vers\Mbag\Mahf &&\vers\Mbae\Mago &&\vers\Mbac\Magc &&\vers\Mbab\Maeb &&\vers\Mbaa\Mid\\
 &\vers\Mbao\Maoe &&\vers\Mboi\Maed &&\vers\Mboh\Maod &&\vers\Mbod\Maob &&\Mboc &&\vers\Mboa\Mafc &&\vers\Mboo\Maoo &&\vers\Maig\Maae
 \end{align*}
 \begin{align*}
 &\hskip -3mm\vers\Maif\Mih &&\hskip -5mm\vers\Maib\Mif &&\vers\Maia\Mhf &&\Mahh &&\vers\Mahg\Madi &&\vers\Mahe\Maea &&\vers\Mahd\Mib &&\vers\Mahb\Maef\\
 &\vers\Maho\Mio &&\vers\Magi\Mhb &&\vers\Magf\Mhh &&\vers\Magb\Mhd &&\vers\Maga\Maei &&\vers\Mafi\Maco &&\vers\Mafh\Madd &&\vers\Mafg\Mfe\\
 &\vers\Mafe\Mada &&\vers\Mafd\Mfo &&\vers\Mafa\Mach &&\vers\Mafo\Mho &&\vers\Maeg\Mgd &&\vers\Maeo\Mge &&\Madh && \vers\Madg\Mabf\\
 &\vers\Mado\Mgo &&\vers\Maci\Mea &&\vers\Macg\Mic &&\Mace &&\Maca &&\vers\Mabi\Maoa &&\hskip -5mm\vers\Mabh\Mfd &&\vers\Mabg\Mfb\\
 &\vers\Mabf\Mfc &&\Mabd &&\vers\Maba\Mig &&\vers\Maao\Mee &&\vers\Maoi\Meh &&\vers\Maoh\Med &&\vers\Maog\Mdi &&\vers\Maoc\Mfi\\
 &\vers\Mii\Mgg &&\vers\Mie\Mhi &&\vers\Mhg\Mga &&\vers\Mhe\Mgc &&\vers\Mhc\Mdf &&\Mha &&\vers\Mgi\Mcc &&\vers\Mgh\Mci\\
 &\vers\Mgf\Mch &&\vers\Mgb\Mcf &&\vers\Mfh\Mcd &&\Mfg &&\Mff &&\vers\Mfa\Mce &&\vers\Mei\Mbb &&\vers\Meg\Mec\\
 &\vers\Mef\Mbh &&\vers\Meb\Mbf &&\vers\Meo\Mbe &&\vers\Mdh\Mbd &&\Mdg &&\Mde &&\vers\Mdd\Mdo &&\vers\Mdc\Mba\\
 &\Mdb &&\vers\Mda\Mae &&\Mcg &&\vers\Mcb\Maf &&\Mca &&\Mco &&\Mbi &&\Mbg\\
 &\Mbc &&\vers\Mbo\Mao &&\vers\Mai\Mad &&\vers\Mah\Mi &&\Mag &&\Mac &&\vers\Mab\Mf &&\Maa\\
 &\vers\Mh\Md &&\Mg &&\Me &&\Mc &&\vers\Mb\Ma
\end{align*}
 \vskip 30mm
\section*{Plusieurs mani\`eres d'obtenir un nombre premier par fusion des racines}
$$\Made \Madc \Mada \longmapsto \Maac$$

$$\Mahe\  \Mahc \Mafi \Mafa \longmapsto \Maea$$

$$\genfrac{}{}{0pt}{1}{\matulaacg\ \matulabi}{3973}\hskip 3mm \genfrac{}{}{0pt}{1}{\matulagi\ \matuladg}{3713} \hskip 3mm \genfrac{}{}{0pt}{1}{\matulacag\ \matulaaa}{3487} \hskip 3mm \genfrac{}{}{0pt}{1}{\matulaaac\ \matulaca}{3403}\hskip 3mm\genfrac{}{}{0pt}{1}{\matulabeg\ \matulaac}{3341}\hskip 3mm\genfrac{}{}{0pt}{1}{\matulafoa\ \matulae}{3005}\hskip 3mm\genfrac{}{}{0pt}{1}{\matulaigg\ \matulac}{2931}\hskip 3mm 
\longmapsto p_{330}=2213$$
%

\end{document}